\documentclass[11pt]{article}

\usepackage{amssymb,latexsym,amsmath,epsf,amsthm,amsfonts}
\usepackage{epsfig,graphics,graphicx,afterpage,appendix}
\usepackage{textcomp,wasysym,dsfont}

\usepackage[colorlinks=true,          % link colors, set to 'false' for print version
            linkcolor=blue, 
            citecolor=blue,
            urlcolor=blue,
            ps2pdf,                   % hyper-references for ps2pdf and dvips
            dvips,
            bookmarks=true,           % generate bookmarks ...
            bookmarksnumbered=true,   % ... with numbers
            bookmarksopen=true,       % show bookmarktree open
            bookmarksopenlevel=1,     % level to which bookmarks are open
            breaklinks=true,          % allow bookmarks to be broken at ends of line
            pdfpagelayout=TwoPageRight, % default displaymode for pdf file
            pdfstartview=FitV,        % pdf view mode of start page
            pdfstartpage=1,           % initial pdf page
            pdfdisplaydoctitle=true  % display doc title on top of window
]{hyperref}
\hypersetup{
  pdfauthor   = {Matthias Heymann},
  pdftitle    = {Existence and Properties of Minimum Action Curves for Degenerate Finsler Metrics},
  pdfcreator  = {LaTeX with hyperref package},
  pdfproducer = {dvips + ps2pdf},
  pdfsubject  = {geometry, large deviation theory},
  pdfkeywords = {action functional, action minimization, rectifiable curves, minimum action curves, large deviation theory, maximum likelihood transition curves, degenerate Finsler metric, degenerate Randers metric}
}
\let\url\nolinkurl % because dvips cannot break url across lines

\allowdisplaybreaks
\numberwithin{equation}{section}

\newtheorem{theorem}{Theorem}
\newtheorem{lemma}{Lemma}
\newtheorem{corollary}{Corollary}
\newtheorem{definition}{Definition}
\newtheorem{proposition}{Proposition}
\newtheorem{remark}{Remark}

\newcommand{\q}{\hspace{-1.2pt}}
\newcommand{\qh}{\hspace{.5pt}}
\newcommand{\qb}{\hspace{.7pt}}
\newcommand{\nf}{\nabla\hspace{-1.9pt}f}
\newcommand{\nb}{\nabla\hspace{-1.3pt}\beta}
\newcommand{\nhf}{\nabla\hspace{-1.9pt}\hat f}
\newcommand{\Tp}{T_{\q+}}
\newcommand{\Tm}{T_{\q-}}
\newcommand{\G}{\mathcal{G}}
\renewcommand{\H}{\mathcal{H}}
\newcommand{\hS}{S}

\newcommand{\vte}{\vartheta}
\newcommand{\te}{\theta}
\newcommand{\that}{\vartheta}
\newcommand{\skp}[2]{{\langle #1,#2 \rangle}}
\newcommand{\Skp}[2]{{\big\langle #1,#2 \big\rangle}}

\newcommand{\Hz}{H_x}
\newcommand{\Ht}{H_\te}
\newcommand{\Htz}{H_{\te x}}
\newcommand{\Hzt}{H_{x\te}}
\newcommand{\Htt}{H_{\te\te}}
\newcommand{\Hztt}{H_{x\te\te}}

\newcommand{\Rn}{\mathbb{R}^n}

\newcommand{\R}{\mathbb{R}}
\newcommand{\N}{\mathbb{N}}
\newcommand{\One}{\mathds{1}}
\newcommand{\eps}{\varepsilon}
\newcommand{\Cxyt}[3]{\bar C_{#1}^{#2}(0,#3)}
\newcommand{\CT}{\Cxyt{x_1}{x_2}{T}}
\newcommand{\CI}{\Cxyt{x_1}{x_2}{1}}
\newcommand{\Cxx}{\tilde C_{x_1}^{x_2}(x)}
\newcommand{\Cx}{\tilde C(x)}
\newcommand{\CIA}{\Cxyt{A_1}{A_2}{1}}
\renewcommand{\AC}[1]{\bar C({#1})}
\newcommand{\tC}{\tilde C(0,1)}
\newcommand{\tCxx}{\tilde C_{x_1}^{x_2}(0,1)}
\newcommand{\tCA}{\tilde C_{A_1}^{A_2}(0,1)}
\newcommand{\tG}{\tilde\Gamma}
\newcommand{\GA}{\Gamma_{A_1}^{A_2}}
\newcommand{\Gxx}{\Gamma_{x_1}^{x_2}}
\newcommand{\tGA}{\tilde\Gamma_{A_1}^{A_2}}
\newcommand{\tGx}{\tilde\Gamma(x)}
\newcommand{\tGxx}{\tilde\Gamma_{x_1}^{x_2}}
\newcommand{\tGxxx}{\tilde\Gamma_{x_1}^{x_2}(x)}
\newcommand{\Gx}{\Gamma_{x_1}^{x_2}}
\newcommand{\Gxy}[2]{\Gamma_{#1}^{#2}}
\newcommand{\hqq}{h_{q_1}^{q_2}}
\newcommand{\prob}{P(x_1,x_2)}
\newcommand{\probA}{P(A_1,A_2)}
\newcommand{\Msa}{M_s^a}
\newcommand{\Mua}{M_u^a}
\newcommand{\Msan}{M_s^{\tilde a }}
\newcommand{\Muan}{M_u^{\tilde a }}

\newcommand{\hMtas}{\hat M_s^{\tilde a}}
\newcommand{\hMtau}{\hat M_u^{\tilde a}}
\newcommand{\Kati}{K_i^{\tilde a}}
\newcommand{\cA}{c_2} %c_8
\newcommand{\cB}{c_1} %c_1
\newcommand{\cC}{c_4} %c_2
\newcommand{\cD}{c_5} %c_3
\newcommand{\cE}{c_8} %c_5
\newcommand{\cF}{c_9} %c_6
\newcommand{\cG}{c_6} %\mu
\newcommand{\cH}{c_7} %c_4
\newcommand{\cI}{c_{10}}
\newcommand{\cJ}{c_3}
\newcommand{\cHH}{\bar c_7} %c_4
\newcommand{\dA}{d_4}
\newcommand{\dB}{d_2}
\newcommand{\dC}{d_1}
\newcommand{\dD}{d_3}
\newcommand{\dF}{d_6}
\newcommand{\dG}{d_5}
\newcommand{\dH}{d_7}
\newcommand{\E}{\tilde{D}}
\newcommand{\hE}{(\hyperlink{ass tD}{$\E$})}
\newcommand{\tHa}{(\hyperlink{H1}{H1})}
\newcommand{\tHaa}{(\hyperlink{H1'}{H1'})}

\newcommand{\tHbb}{(\hyperlink{H2'}{H2'})}
\newcommand{\tHc}{(\hyperlink{H3}{H3})}

\newcommand{\tramat}{\big(\begin{smallmatrix}\!-p\,&0\\r\!&q\end{smallmatrix}\big)}

\newcommand{\pb}{\pagebreak}
\newcommand{\enl}{\enlargethispage{\baselineskip}}
\newcommand{\enlh}{\enlargethispage{.3cm}}
\newcommand{\markit}{}

\DeclareMathOperator*{\dist}{dist}
\DeclareMathOperator*{\supp}{supp}

\DeclareMathOperator*{\esssup}{ess\,sup}
\DeclareMathOperator*{\sgn}{sgn}
\DeclareMathOperator*{\length}{length}

\begin{document}
%-------------------------------------------------

\title{Existence and Properties\\ of Minimum Action Curves\\ for Degenerate Finsler Metrics}
\author{Matthias Heymann\footnote{email: \href{mailto://heymann@math.duke.edu}{heymann@math.duke.edu}, web: \href{http://www.matthiasheymann.de}{www.matthiasheymann.de}}\\[.25cm]Duke University,\\ Mathematics Department}
\maketitle

\pdfbookmark[1]{Title \& Abstract}{abstract}

\begin{abstract}
We study a class of action functionals $S(\gamma)=\int_\gamma\ell(z,dz)$ on the space of unparameterized oriented rectifiable curves $\gamma$ in $\Rn$.
The local action $\ell(x,y)$ is a degenerate type of Finsler metric that may vanish in certain directions $y\neq0$, thus allowing for curves with positive Euclidean length but zero action.
Given two sets $A_1,A_2\subset\Rn$, we develop criteria under which $\exists\gamma^\star\in\GA:=\{\gamma\,|\,\gamma \text{ starts in }A_1\text{ and ends in }A_2\}$ such that
$S(\gamma^\star)=\inf\!\big\{S(\gamma)\,|\,\gamma\in\GA\big\}$.
We then study the properties of these minimizers $\gamma^\star$, and we prove the non-existence of minimizers in some situations.
Applied to a geometric reformulation of the quasipotential of large deviation theory, our results can prove the existence and properties of maximum likelihood transition curves between two metastable states in a stochastic process with small noise.
% Minimization problems of this type include the computation of the distance between two points or sets with respect to a Riemannian metric, the instanton by which quantum tunnelling occurs (i.e.~the minimizer of the Agmon distance), and the maximum likelihood transition curve between two metastable states of a stochastic dynamical system subject to small noise (specified via the quasipotential of large deviation theory).
%
\end{abstract}

%\vspace{0.2cm}
%\begin{center}
%\noindent\underline{Status:}\\[.2cm]
%All proofs are complete (03/08/2010).\\
%Proofreading is complete (04/12/2010).\\
%Changes completed (04/21/2010).\\[.2cm]
%Small updates may be possible until 04/25/2010\\[.3cm]
%Check back for updates on \href{http://www.matthiasheymann.de}{www.matthiasheymann.de}.
%\end{center}
             \newpage
\mbox{}\newpage

\pdfbookmark[1]{Dedication}{dedication}
\vspace*{6cm}
\begin{center}
\textit{\large
In memory of my beloved grandfather.\\[.5cm]
Julius Salzmann\\[.17cm]
${}^{\text{\textborn}}$\hspace{3pt}11/03/1908\qquad\APLnot{}\qquad\textdied\ 07/01/2009
}
\end{center}
             \newpage

\pdfbookmark[1]{Contents}{toc}
\tableofcontents                         \newpage

\part{Results} \label{part results}

\section{Introduction} \label{sec introduction}

\paragraph{Geometric Action Functionals.}
A geometric action $S$ is a mapping that assigns to every unparameterized oriented rectifiable curve $\gamma$ in $\Rn$ a number $S(\gamma)\in[0,\infty)$. It is defined via a curve integral
\begin{equation} \label{S gamma intro}
  S(\gamma):=\int_\gamma\ell(z,dz):=\int_0^1\ell(\varphi,\varphi')\,d\alpha,
\end{equation}
where $\varphi\colon[0,1]\to\Rn$ is any absolutely continuous parameterization of $\gamma$, and where the local action $\ell\in C(\Rn\times\Rn,[0,\infty))$ must have the properties
\begin{align*}
  (i)  \ \ & \forall x,y\in\Rn\ \,\forall c\geq0\colon\ \ \ell(x,cy)=c\ell(x,y), \\
  (ii) \ \ & \text{for every fixed $x\in\Rn$ the function $\ell(x,\cdot\,)$ is convex.}
\end{align*}
While (i) guarantees that the second integral in \eqref{S gamma intro} is independent of the choice of $\varphi$, (ii) is necessary to ensure that $S$ is lower semi-continuous in a certain sense. A trivial example is given by
$\ell(x,y)=|y|$, in which case $S(\gamma)$ is just the Euclidean length of $\gamma$, or more generally, by $\ell(x,y)=|y|_{g_x}$ for any Riemannian metric $g$. In fact, $\ell$ generalizes the well-studied notion of a Finsler metric \cite{Finsler} in that (a) $\ell$ only needs to be continuous (no smoothness required), and more importantly (b) $\ell^2$ need not be \emph{strictly} convex.

Now given two sets $A_1,A_2\subset\Rn$, in this work we develop criteria under which there exists a minimum action curve $\gamma^\star$ leading from $A_1$ to $A_2$, i.e.\ under which $\exists\gamma^\star\in\GA:=\{\gamma\,|\,\gamma \text{ starts in }A_1\text{ and ends in }A_2\}$ such that
\begin{equation}
 S(\gamma^\star)=\inf_{\substack{\gamma\in\GA}}S(\gamma). \label{curve min prob}
\end{equation}
We then prove properties of the minimizer $\gamma^\star$ without knowing it explicitly.
%\\[.2cm]\indent

Although our existence results can certainly be applied to the exemplary local actions given above, the present work was primarily motivated by a recently emerging problem from large deviation theory that is adding a considerable layer of difficulty: In contrast to usual Finsler metrics, in this example $\ell(x,y)$ vanishes in some direction $y=b(x)\neq0$, which allows for curves~$\gamma$ (the flowlines of the vector field $b$) with positive Euclidean length but vanishing action $S(\gamma)$.

\begin{figure}[t]
\centering
\includegraphics[width=12.5cm]{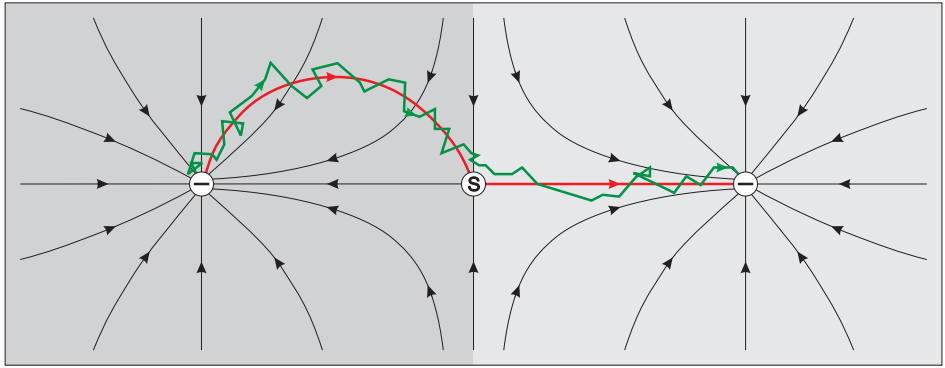}
\caption{\label{SDE transition}
\small Rare noise-induced transitions from one meta-stable state to another (green curve) stay with high probability near the minimum action curve $\gamma^\star$ (red).}
\end{figure}

\paragraph{Example: Large Deviation Theory.}
Consider for some $b\in C^1(\Rn,\Rn)$ and small $\eps>0$ the stochastic differential equation (SDE)
\begin{equation} \label{SDE0}
dX_t^\eps = b(X_t^\eps)\,dt + \sqrt{\eps} \,dW_t,\qquad X_0^\eps=x_1,
\end{equation}
where $(W_t)_{t\geq0}$ is a Brownian motion, and where the zero-noise-limit, i.e.~the ODE $\dot x=b(x)$, has two stable equilibrium points $x_1,x_2\in\Rn$. The presence of small noise allows for rare transitions from $x_1$ to $x_2$ that would be impossible without the noise (green curve in Fig.~\ref{SDE transition}), and one is interested in the frequency and the most likely path of these transitions. Both questions are answered within the framework of large deviation theory \cite{FW,SW}, the key object being the quasipotential
\begin{align}
  V(x_1,x_2) &= \inf_{\substack{T>0\\\chi\in\CT}} S_T(\chi), \label{double minimization0} \\
  \text{where}\hspace{1.4cm}
  S_T(\chi)  &= \frac12\int_0^T|b(\chi)-\dot\chi|^2\,dt,\hspace{1.4cm} \label{SDE ST}
\end{align}
and where $\CT$ denotes the space of all absolutely continuous functions $\chi\colon[0,T]\to\Rn$ fulfilling $\chi(0)=x_1$ and $\chi(T)=x_2$.

An unpleasant feature of this formulation is that the minimization problem \eqref{double minimization0} does not have a minimizer $(T^\star,\chi^\star)$, the main reason being that by \cite[Lemma 3.1]{FW} $\dot\chi^\star$ would need to vanish at $x_1$ and $x_2$, and typically also at some critical point along the way, and so $T^\star$ would have to be (doubly) infinite. This is a major problem for both analytical and numerical work, and so in \cite{CPAM,thesis} the use of the alternative representation
\begin{equation} \label{qp min prob}
  V(x_1,x_2)=\inf_{\gamma\in\Gx}S(\gamma)
\end{equation}
was suggested, where the geometric action $S(\gamma)$ is given by
\begin{equation} \label{local geo action0}\hspace{3.2cm}
  \ell(x,y) = |b(x)||y|-\skp{b(x)}{y}, \qquad\ \ \text{(SDE)}
\end{equation}
which can be seen as a degenerate version of a Randers metric \cite[Ch.\ 11]{Finsler}.
The minimizer $\gamma^\star$ of \eqref{qp min prob}, i.e.~the \emph{maximum likelihood transition curve} (the red curve in Fig.~\ref{SDE transition}), seems more feasable to exist in this formulation since the time parameterization has been eliminated from the problem.

This geometric reformulation of the quasipotential generalizes also to other types of stochastic dynamics such as SDEs with multiplicative noise or continuous-time Markov jump processes \cite{CPAM,thesis}, with modified (in the latter case not Randers-like) local action $\ell$. It was shown to effectively remove the numerical difficulties \cite{CPAM,thesis,PRL,JCP}, and our goal in this monograph is now to demonstrate also its analytical advantages.

\paragraph{Existence of Minimizers; the Drift Vector Field.}
Since minimizers $\gamma^\star$ of \eqref{curve min prob} have numerically been found to generally have cusps as they pass certain critical points (even in the basic case where $\ell$ is given by \eqref{local geo action0} with some smooth $b$, see Fig.~\ref{SDE transition} or e.g.\ \cite[Fig.\ 4.1]{CPAM}), any a~priori assumptions on the smoothness of $\gamma^\star$ in our existence proof would be counterproductive. This forbids the variational approach using the Euler-Lagrange equations associated to \eqref{curve min prob}, and so instead we will opt for a lower semi-continuity argument.

% Since minimizers $\gamma^\star$ of \eqref{curve min prob} have numerically been found to be non-smooth in general (they typically have cusps as they pass certain critical points along the way), any a priori assumptions on the smoothness of $\gamma^\star$ in our existence proof would be counterproductive. This forbids the variational approach using the Euler-Lagrange equations associated to \eqref{curve min prob}, and so instead we will opt for a lower semi-continuity argument.

A first result which is relatively easy to obtain is the following (Proposition~\ref{min prop}): \textit{If there exists a minimizing sequence $(\gamma_n)_{n\in\N}$ of \eqref{curve min prob} that is contained in some compact set $K\subset\Rn$ and has uniformly bounded curve lengths, then there exists a minimizer $\gamma^\star\in\GA$.}
%Indeed, if we parameterize the curves $\gamma_n$ by arclength, i.e.~with absolutely continuous functions $\varphi_n$ fulfilling $|\varphi_n'|\equiv\length(\varphi_n)$ a.e.~on $[0,1]$, then we can invoke Arzel\`a-Ascoli's theorem, and so there exists a subsequence $(\varphi_{n_k})_{k\in\N}$ that converges uniformly to some limit $\varphi^\star$. This limit can easily be shown to be absolutely continuous as well and thus a parameterization of a rectifiable curve $\gamma^\star$, and the lower-semicontinuity of $S$ then implies that $\gamma^\star$ is a minimizer.
%
In practice however, this criterion alone is of little use since minimizing sequences are not at our direct disposal and so their curve lengths can be hard to control. Instead, we would rather like to have criteria that are based on some explicitly available key ingredient of the local action~$\ell$. What could this key ingredient be?

An essential property of \eqref{local geo action0} is that $\ell(x,y)$ vanishes whenever $y$ aligns with $b(x)$. In fact, such behavior is generic to large deviation geometric actions: For general stochastic dynamics, the drift vector field $b$ given by the zero-noise limit $\dot x=b(x)$ is the direction which the system can follow without the aid of the noise (as $\eps\searrow0$), and so any curve segment that follows a drift flowline has zero cost.

%This observation complicates our existence proof significantly, since it allows for long curves with vanishing or small action, and thus for minimizing sequences $(\gamma_n)_{n\in\N}$ with unbounded curve lengths. For this reason the drift vector field $b(x)$ (or a generalization thereof in the case of general geometric actions) will be the key object of our criteria. Since its magnitude $|b|$ is a velocity and $S(\gamma)$ does not depend on the speed at which $\gamma$ is traversed, our key criteria will solely depend on the flowline diagram of $b$.

This observation complicates our existence proofs (which are based on Proposition \ref{min prop}) significantly, since it allows for long curves with vanishing or small action, and thus for minimizing sequences $(\gamma_n)_{n\in\N}$ with unbounded curve lengths. For this reason, the flowline diagram of the drift vector field~$b$ (or of a generalization thereof in the case of general geometric actions) will be the key object of our main criteria, Propositions \ref{blm prop 1} and \ref{blm prop 2}.
%(Its magnitude $|b|$ does not enter these criteria since $|b|$ is a velocity and $S(\gamma)$ does not depend on the speed at which $\gamma$ is traversed.)

Surprisingly, the drift $b$ is in fact all that these criteria depend on, while other aspects such as the nature of the noise in the case of large deviation geometric actions are largely irrelevant (except for the brute force estimate needed in Lemma \ref{main corollary}). One may now argue that this indicates that our criteria may waste valuable information, potentially leaving us undecided where in fact a minimizer exists. However, we will give an example in which no minimizer exists and where the location that is responsible for this non-existence coincides exactly with the location where our criteria fail. This suggests that if our criteria fail, they do so for a reason.

\paragraph{Properties of Minimum Action Curves.}
Then turning our attention to the properties of minimizers, we consider a subclass of geometric actions that still contains the large deviation geometric actions mentioned above.
For our main result, suppose that the drift~$b$ has two basins of attraction (see e.g.~Figures~\ref{SDE transition}, \ref{fig two basins}~(b) or \ref{fig transition through crit points}), and let~$\gamma^\star$ be the minimum action curve leading from one attractor to the other.

Since for the class of actions in question $\gamma^\star$ can follow the flowlines of~$b$ at no cost, it is not surprising that the second (``downhill'') part of~$\gamma^\star$ will be a flowline connecting a saddle point to the second attractor. In particular, the \emph{last} hitting point of the separatrix is a point with zero drift (the saddle point).
Here we prove also the non-obvious fact that also the \emph{first} hitting point must have zero drift. In practice, such knowledge can be used either to gain confidence in the output of algorithms that compute $\gamma^\star$ numerically (such as the geometric minimum action method, gMAM, see \cite{CPAM,thesis}), or to speed up such algorithms by restricting their search to only those curves with these properties.

Finally, we will demonstrate how the same result (Corollary \ref{first hit corollary}) that is used to prove this property can also be used to prove the non-existence of minimizers is some situations.

\paragraph{The Structure of This Monograph.} This monograph is split into three parts: In Part~\ref{part results} we lay out all our results on the existence of minimum action curves, we demonstrate on several examples how to use our criteria in practice, we discuss when minimizers do \textit{not} exist, and finally we prove 
the above-mentioned properties of minimum action curves. The reader who is only interested in gaining enough working knowledge to use our existence criteria in practice will find it sufficient to read only this first part.

Part~\ref{part proofs} contains essential proofs of a local existence property to which the global statement had been reduced in Part~\ref{part results}. The reader who wants to know why the criteria in Part~\ref{part results} work should also read this second~part.

Part~\ref{part superprop} contains the proof of a very technical lemma that is needed in the second part in order to deal with curves that are passing a saddle point. The reader can decide to skip this part without losing much insight.

\paragraph{Notation.}
For a point $x\in\Rn$ and a radius $r>0$ we define the open and the closed balls
\[ B_r(x):=\big\{w\in\Rn\,\big|\,|w-x|<r\big\}\qquad\text{\!and\!}\qquad \bar B_r(x):=\big\{w\in\Rn\,\big|\,|w-x|\leq r\big\}. \]
Similarly, for a set $A\subset\Rn$ and a distance $r>0$ we define the open and the closed neighborhoods $N_r(A)$ and $\bar N_r(A)$ as
\[ N_r(A)\!:=\!\big\{w\in\Rn\,\big|\dist(w,A)<r\big\} \;\;\text{and}\,\; \bar N_r(A)\!:=\!\big\{w\in\Rn\,\big|\dist(w,A)\leq r\big\}. \]
Furthermore, we denote by $\bar A$ the closure of $A$ in $\Rn$, and by $A^c:=\Rn\setminus A$, $A^\circ:=(\overline{A^c})^c$ and $\partial A:=\bar A\setminus A^\circ$ the complement, the interior and the boundary of $A$ in $\Rn$, respectively.
For a point $x$ on a $C^1$-manifold~$M$ we denote by $T_xM$ the tangent space of $M$ at $x$.

For a function $f$ and a subset $A$ of its domain we denote by $f|_A$ the restriction of $f$ to~$A$, and we use notation such as $f\equiv c$ to emphasize that $f$ is constant. Expressions of the form $\One_{\textit{cond}}$ denote the indicator function that returns the value $1$ whenever the condition $cond$ is fulfilled and $0$ otherwise.

Finally, throughout the entire paper we let $\E\subseteq D\subseteq\Rn$ be two fixed connected sets, where $D$ is open, and where $\E$ is closed in $D$. An additional technical assumption on $\E$ will be made at the beginning of Section~\ref{first ex result subsec}.
$D$~will serve as our state space, i.e.\ as the set that the curves $\gamma$ live in, and $\E$ will be used for an additional constraint in our minimization, i.e.\ we will in fact minimize over $\GA:=\{\gamma\subset\E\,|\,\gamma \text{ starts in }A_1\text{ and ends in }A_2\}$. (For simplicity we suppress the dependence of $\GA$ on $\E$ in our notation.)
If no such constraint is desired, just choose $\E:=D$. The reader is encouraged to consider this simple unconstrained case $\E=D$ whenever on first reading he may feel overwhelmed by some definition or statement involving~$\E$.

\paragraph{Acknowledgments.} The work of M.~Heymann is partially supported by the National Science Foundation via grant DMS-0616710.
I want to thank Weinan E, Gerard Ben Arous, Eric Vanden-Eijnden, Lenny Ng, Marcus Werner and Stephanos Venakides for some useful suggestions and comments.
I also want to thank the Duke University Mathematics Department and in particular Jonathan Mattingly and Mike Reed for providing me with the inspiring environment and the freedom without which this work would not have been possible.

\section{Geometric Action Functionals} \label{main results}

\subsection{Rectifiable Curves and Absolutely Continuous Functions} \label{rec curves subsec}
An unparameterized oriented curve $\gamma$ is an equivalence class of functions\linebreak $\varphi\in C([0,T],D)$,\, $T>0$, that are identical up to continuous non-decreasing changes of their parameterizations, or more formally, whose Fr\'e\-chet distance to each other vanishes. \textit{In this paper we will tacitly assume that all our curves are unparameterized and oriented.}

A curve $\gamma$ is called \textit{rectifiable} \cite[p.115]{Stein} if for some (and thus for every) parameterization $\varphi\in C([0,T],D)$ of $\gamma$ we have
\[
  \length(\gamma):=\length(\varphi):=\sup_{\substack{N\in\N\\0=t_0<\dots<t_N=T}}\,
  \sum_{i=1}^N\big|\varphi(t_i)-\varphi(t_{i-1})\big|<\infty.
\]
It is easy to see that $\length(\varphi)$ is in fact the same for any parameterization $\varphi$ of $\gamma$, and that it is finite if and only if all the component functions of $\varphi$ are of bounded variation \cite[Thm.\ 3.1]{Stein}. We will denote the set of rectifiable curves by $\Gamma$.

A function $\varphi\colon[0,T]\to D$ is said to be \textit{absolutely continuous} \cite[p.127]{Stein} if for every $\eps>0$ there exists a $\delta>0$ such that for any finite collection of disjoint intervals $[t_{i-1},t_i)\subset[0,T]$, $i=1,\dots,N$, we have
\[ \sum_{i=1}^N(t_i-t_{i-1})<\delta\qquad\Longrightarrow\qquad\sum_{i=1}^N\big|\varphi(t_i)-\varphi(t_{i-1})\big|<\eps. \]
We will denote the space of absolutely continuous functions with values in our fixed set $D$ by $\AC{0,T}$. One can show \cite[Prop.\ 1.12(ii) and Thm.\ 3.11]{Stein} that a function $\varphi$ is in $\AC{0,T}$ if and only if there exists an $L^1$-function which we denote by $\varphi'$ such that $\varphi(t)=\varphi(0)+\int_0^t\varphi'(\tau)\,d\tau$ for $\forall t\in[0,T]$. In that case, $\varphi$ is differentiable in the classical sense at almost every $t\in[0,T]$, with derivative $\varphi'(t)$.

Clearly, every function $\varphi\in\AC{0,T}$ describes a rectifiable curve $\gamma$ since for every partition $0=t_0<\cdots<t_N=T$ we have
\[ \sum_{i=1}^N\big|\varphi(t_i)-\varphi(t_{i-1})\big|=\sum_{i=1}^N\bigg|\int_{t_i}^{t_{i-1}}\varphi'\,dt\bigg|\leq\int_0^T|\varphi'|\,dt<\infty, \]
and it is not hard to show \cite[Thm.\ 4.1]{Stein} that $\length(\gamma)=\int_0^T|\varphi'|\,dt$.
The reverse is not true: Not every function $\varphi$ that describes a rectifiable curve $\gamma\in\Gamma$ is necessarily absolutely continuous (a counterexample can be constructed using the Cantor function \cite[p.125]{Stein}). However, we have the following:

\begin{lemma}[Parameterization by arclength] \label{arclength lemma}
(i) Any curve $\gamma\in\Gamma$ can be parameterized by a unique function $\varphi_\gamma\in\AC{0,1}$ with $|\varphi_\gamma'|\equiv\length(\gamma)$ a.e..\\[.1cm]
(ii) If $\varphi\in\AC{0,T}$ is any absolutely continuous parameterization of $\gamma$ then $\varphi=\varphi_\gamma\circ\beta$ for some absolutely continuous function $\beta\colon[0,T]\to[0,1]$, and we have $\varphi'=(\varphi_\gamma'\circ\beta)\cdot\beta'$ and $\beta'\geq0$ a.e.~on $[0,1]$.
\end{lemma}
\begin{proof}
(i) This is a trivial modification of \cite[p.136]{Stein}.\\[.15cm]
(ii) In the proof in \cite[p.136]{Stein} it is shown that for any parameterization $\varphi\in C([0,T],D)$ of $\gamma$ the function $\varphi_\gamma$ fulfills $\varphi(t)=\varphi_\gamma(\beta(t))$ for $\forall t\in[0,T]$, where $\beta\colon[0,T]\to[0,1]$ is defined by $\beta(t):=\length\!\big(\varphi|_{[0,t]}\big)/\length(\gamma)$.
For any collection of disjoint intervals $[t_{i-1},t_i)\subset[0,T]$, $i=1,\dots,N$, we have
\begin{align*}
   \sum_{i=1}^N\big(\beta(t_i)-\beta(t_{i-1})\big)
    &=\frac{1}{\length(\gamma)}\sum_{i=1}^N\length\!\big(\varphi|_{[t_{i-1},t_i]}\big)\\
    &\hspace{-3.1cm}=\frac{1}{\length(\gamma)}\sum_{i=1}^N\,
       \sup_{\substack{M_i\in\N\\t_{i-1}=s_0^i<\dots<s_{M_i}^i=t_i}}\,
       \sum_{k=1}^{M_i}\big|\varphi(s_k^i)-\varphi(s_{k-1}^i)\big| \\
    &\hspace{-3.1cm}=\frac{1}{\length(\gamma)}\!
\sup_{\substack{M_1\in\N\\t_0=s_0^1<\dots<s_{M_1}^1=t_1}}\hspace{-.1cm}\cdots\hspace{-.3cm}
\sup_{\substack{M_N\in\N\\t_{N-1}=s_0^N<\dots<s_{M_N}^N=t_N}}
\sum_{i=1}^N\sum_{k=1}^{M_i}\big|\varphi(s_k^i)-\varphi(s_{k-1}^i)\big|,
\end{align*}
and since for $\varphi\in\AC{0,T}$ the last double sum can be made arbitrarily small by ensuring that $\sum_{i=1}^N\sum_{k=1}^{M_i}(s_k^i-s_{k-1}^i)=\sum_{i=1}^N(t_i-t_{i-1})$ is sufficiently small, this shows that $\beta$ is absolutely continuous. Clearly, $\beta'\geq0$ a.e.~since $\beta$ is non-decreasing, and for $\forall t\in[0,T]$ we have
\begin{align*}
\int_0^t\varphi'\,d\tau
  &= \varphi(t)-\varphi(0)
   = \varphi_\gamma(\beta(t))-\varphi_\gamma(\beta(0)) \\
  &= \int_{\beta(0)}^{\beta(t)}\varphi_\gamma'\,d\alpha
   = \int_0^t\varphi_\gamma'(\beta(\tau))\beta'(\tau)\,d\tau
\end{align*}
(for the last step, see \cite[p.149, Ex.21]{Stein}),
which implies that $\varphi'=(\varphi_\gamma'\circ\beta)\cdot\beta'$ a.e.~on $[0,T]$.
\end{proof}
The following lemma is a result on the uniform convergence of absolutely continuous functions. We will use the notation $\varphi\subset G$ (for a function $\varphi\in\AC{0,1}$ and a set $G\subset\Rn$) to indicate that $\varphi(\alpha)\in G$ for $\forall\alpha\in[0,1]$.
Similarly, for a curve $\gamma\in\Gamma$ we write $\gamma\subset G$ to indicate that $\varphi_\gamma\subset G$.

\begin{lemma} \label{min prop0}
(i) If a sequence $(\varphi_n)_{n\in\N}\subset\AC{0,1}$ fulfills $\varphi_n\subset K$ for $\forall n\in\N$ and some compact set $K\subset D$, and if
\begin{equation} \label{arzela 2}
M:=\sup_{n\in\N}\ \esssup_{\alpha\in[0,1]}|\varphi_n'(\alpha)|<\infty,
\end{equation}
then there exists a uniformly converging subsequence.\\[.2cm]
(ii) If a sequence $(\varphi_n)_{n\in\N}\subset\AC{0,1}$ fulfilling the conditions of part (i) converges uniformly then its limit $\varphi$ is in $\AC{0,1}$ and fulfills $|\varphi'|\leq M$ a.e..
\end{lemma}
\begin{proof}
(i) The sequence $(\varphi_n)_{n\in\N}$ is equicontinuous since by \eqref{arzela 2} we have
\[ |\varphi_n(\alpha_1)-\varphi_n(\alpha_0)|
   = \bigg|\int_{\alpha_0}^{\alpha_1}\varphi_n'\,d\alpha\bigg|
   \leq \int_{\alpha_0}^{\alpha_1}|\varphi_n'|\,d\alpha
   \leq M(\alpha_1-\alpha_0)  \]
for $\alpha_0<\alpha_1$ and $\forall n\in\N$, and so we can apply the Arzel\`a-Ascoli theorem.\\[.2cm]
(ii) By the same estimate, for any collection of disjoint intervals $[\alpha_{i-1},\alpha_i)$ $\subset[0,1]$, $i=1,\dots,N$, we have
\begin{align*}
\sum_{i=1}^N\big|\varphi(\alpha_i)-\varphi(\alpha_{i-1})\big|
    &= \lim_{n\to\infty} \sum_{i=1}^N\big|\varphi_n(\alpha_i)-\varphi_n(\alpha_{i-1})\big|
     \leq M\sum_{i=1}^N(\alpha_i-\alpha_{i-1}).
\end{align*}
This shows that $\varphi$ is absolutely continuous, and (taking $N=1$ and recalling that $\varphi'$ is the classical derivative a.e.) that $|\varphi'|\leq M$ a.e.. Since $K$ is compact and $\varphi_n\subset K$ for $\forall n\in\N$, we have $\varphi\subset K\subset D$ and thus $\varphi\in\AC{0,1}$.
\end{proof}

\paragraph{Curves that pass points in infinite length.}
Sometimes we will have to work with curves that do not have finite length (i.e.~that are not rectifiable). We denote by $\tC\supset\AC{0,1}$ the space of all functions in $C([0,1],D)$ that are absolutely continuous in neighborhoods of all but at most finitely many $\alpha_i\in[0,1]$, and we denote by $\tG\supset\Gamma$ the set of all curves that can be parameterized by a function $\varphi\in\tC$.

Note that for $\forall\varphi\in\tC$, $\varphi'$ is still defined a.e., but one can see that for these exceptional values $\alpha_i$ we have $\int_{[0,1]\cap[\alpha_i-\eps,\alpha_i+\eps]}|\varphi'|\,d\alpha=\infty$ for $\forall\eps>0$.\footnote{The key argument for this can be found at the end of the proof of Proposition~\ref{blm prop 2}.}
We therefore say that the curve $\gamma\in\tG$ given by $\varphi$ \textit{``passes the points $\varphi(\alpha_i)$ in infinite length.''}

Of particular use in our work is, for fixed $x\in D$, the set $\tGx$ of all curves that are either of finite length (i.e.~rectifiable) or that pass $x$ once in infinite length (note that $\Gamma\subset\tGx\subset\tG$).
More precisely, these are the curves that can be parameterized by functions in the set $\Cx$, which we define to be the set of functions $\varphi\in C([0,1],D)$ such that\\[.2cm]
\hspace*{.3cm}
\begin{tabular*}{10cm}[t]{ll}
either &$\varphi\in\AC{0,1}$, \\[.2cm]
or     &
  \begin{minipage}[t]{10cm}$\varphi(\tfrac12)=x$,\\
                          and $\varphi|_{[0,1/2-a]}$ and $\varphi|_{[1/2+a,1]}$ are abs.~cont.~for $\forall a\in(0,\tfrac12)$.\end{minipage}
\end{tabular*}\\[.2cm]
See the end of this section and Fig.~\ref{curve illustration} for an illustration of these classes of curves.\\[.2cm]
In preparation for Lemma \ref{min prop0b}, which is the equivalent of Lemma \ref{min prop0} for sequences of functions in $\Cx$, we introduce the following notation:
For a curve $\gamma$ and a point $x$ we say that \textit{$\gamma$ passes $x$ at most once} if for any parameterization $\varphi\in C([0,1])$ of $\gamma$ we have
\begin{equation} \label{pass point once def}
  \big(\exists 0\leq\alpha_1<\alpha_2\leq1\colon\ \varphi(\alpha_1)=\varphi(\alpha_2)=x\big)
\quad\Rightarrow\quad\forall\alpha\in[\alpha_1,\alpha_2]\colon\ \varphi(\alpha)=x.
\end{equation}
For a Borel set $E\subset D$ and a curve $\gamma\in\tilde\Gamma$ we define
\[ \length(\gamma|_E):=\int_\gamma\One_{z\in E}\,|dz|=
\int_0^1|\varphi'|\One_{\varphi\in E}\,d\alpha\ \in[0,\infty] \]
for any parameterization $\varphi\in\tC$ of $\gamma$.

\begin{lemma} \label{min prop0b}
Let $x\in D$, let the sequence $(\gamma_n)_{n\in\N}\subset\tGx$ fulfill $\gamma_n\subset K$ for $\forall n\in\N$ and some compact set $K\subset D$, suppose that every curve $\gamma_n$ passes~$x$ at most once, and suppose that there exists a function $\eta\colon(0,\infty)\to[0,\infty)$ such that
\begin{equation} \label{length away from x cond}
  \forall n\in\N\ \,\forall u>0:\ \ \length\!\big(\gamma_n|_{\bar B_u(x)^c}\big)\leq\eta(u).
\end{equation}
Then there exist parameterizations $\varphi_n\in\Cx$ of the curves $\gamma_n$ such that a subsequence $(\varphi_{n_k})_{k\in\N}$ converges pointwise on $[0,1]$ and uniformly on the sets $[0,\frac12-a]\cup[\frac12+a,1]$, $a\in(0,\frac12)$. The limit $\varphi$ is in $\Cx$, and the corresponding curve $\gamma\in\tGx$ fulfills\vspace{-.2cm}
\begin{equation} \label{length away from x result}
  \hspace{1.18cm}\forall u>0\colon\ \length\!\big(\gamma|_{\bar B_u(x)^c}\big)\leq\eta(u).
\end{equation}
%}
%  \forall\eps>0:\quad &\length\big(\gamma_n|_{\bar B_\eps(x)^c}\big)<\eta(\eps-).
%\end{align}
\end{lemma}
\begin{proof}
See Appendix \ref{weak convergence appendix}. This proof uses Lemma \ref{lower semi lemma 2} (i).
\end{proof}

\noindent Introducing some final notation, for two sets $A_1,A_2\subset\E$ we write
\begin{align*}
 \GA  &:=\big\{\gamma\in\Gamma\,\big|\,\text{$\gamma\subset\E$, $\gamma$ starts in $A_1$ and ends in $A_2$}\big\}, \\*
 \CIA &:=\big\{\varphi\in\AC{0,1}\,\big|\,\varphi\subset\E,\ \varphi(0)\in A_1,\ \varphi(1)\in A_2 \big\},
\end{align*}
and for two points $x_1,x_2\in\E$ we similarly define $\Gx$ and $\CI$.
The sets $\tGA$, $\tCA$, $\tGxx$, $\tCxx$, $\tGxxx$ and $\Cxx$ are defined analogously.
\begin{figure}[t]
\centering
\includegraphics[width=6.1cm]{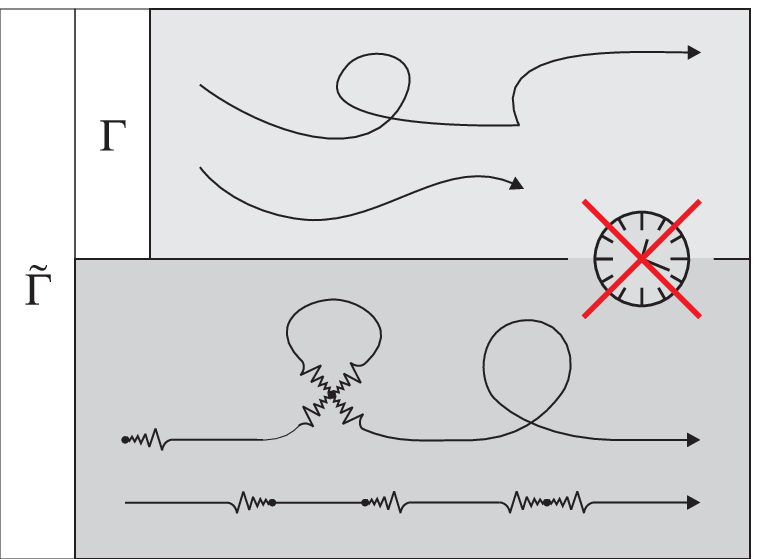}\hspace{.05cm}
\includegraphics[width=6.1cm]{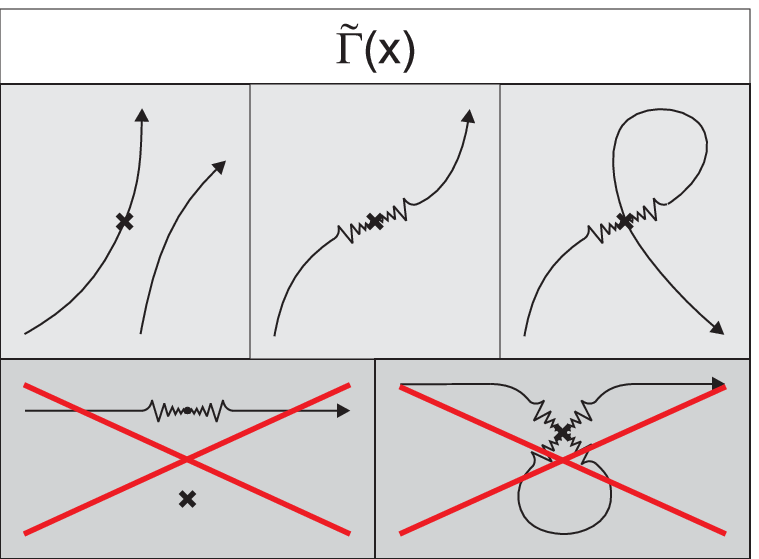}
\caption{\label{curve illustration}
\small Illustration of the various classes of curves.}
\end{figure}
\paragraph{Summary of the various classes of curves (see Fig.~\ref{curve illustration}).} All curves are unparameterized and oriented, and they may have loops and cusps. The class $\Gamma$ contains only curves with finite length, while curves in $\tG\supset\Gamma$ may reach and/or leave finitely many points in infinite length, also repeatedly. For some fixed $x\in D$ (marked by the cross), $\tG(x)$ contains all of $\Gamma$, plus all the curves that pass $x$ once in infinite length; they cannot pass any other point in infinite length, and they cannot pass $x$ twice in infinite length. The sub- and superscripts $x_1$ and $x_2$ or $A_1$ and $A_2$ add constraints to the start and end points of these functions and curves and require them to take values in~$\E$.

\subsection{The Class \texorpdfstring{$\G$}{G} of Geometric Actions, Drift Vector Fields} \label{not def sub}

In this section we will define the class $\G$ of geometric action functionals, and we will generalize the concept of a ``drift vector field'' $b(x)$ from the large deviation geometric action of the SDE \eqref{SDE0}, given by \eqref{local geo action0}, to general geometric actions $S\in\G$.

\begin{definition} \label{tech lemma 1}
We denote by $\mathcal{G}$ the set of all functionals $S\colon\tG\to[0,\infty]$ of the form
\begin{equation} \label{geo action formula}
 S(\gamma):=\int_\gamma\ell(z,dz):=\int_0^1\ell(\varphi,\varphi')\,d\alpha,
\end{equation}
where $\varphi\in\tC$ is an arbitrary parameterization of $\gamma$, and where the local action $\ell\in C(D\times\Rn,[0,\infty))$ has the following properties:\\[.2cm]
\begin{tabular}{rl}
(i)   & $\forall x\in D\ \,\forall y\in\Rn\ \,\forall c\geq0\colon\ \ \ell(x,cy)=c\ell(x,y)$,\\[.1cm]
(ii)  & for every fixed $x\in D$ the function $\ell(x,\cdot\,)$ is convex. %,\\[.1cm]
%(iii) & $\forall x\in D\colon$\hspace{-.2cm}
%\begin{minipage}[t]{10cm}
%\begin{tabular}[t]{ll}
%  either &\!\!$\forall y\in\Rn\colon\ \ell(x,y)=0$,\\
%  or     &\!\!$\exists$ at most one $b\in\Rn$ with $|b|=1$ and $\ell(x,b)=0$.
%\end{tabular}
%\end{minipage}
\end{tabular}\\[.2cm]
For $\varphi\in\tilde C(0,1)$ we will sometimes use the notation $S(\varphi):=\int_0^1\ell(\varphi,\varphi')\,d\alpha$, and for any interval $[\alpha_1,\alpha_2]\subset[0,1]$ we will denote by $S(\varphi|_{[\alpha_1,\alpha_2]}):=$\linebreak $\int_{\alpha_1}^{\alpha_2}\ell(\varphi,\varphi')\,d\alpha$ the action of the curve segment parameterized by $\varphi|_{[\alpha_1,\alpha_2]}$.
\end{definition}
\noindent
As we will see next, (i) is needed to show that \eqref{geo action formula} is independent of the specific choice of $\varphi$, while (ii) is essential to show that $S$ is lower semi-continuous in a certain sense (Lemma \ref{lower semi lemma 2}). Observe also that (i) implies that $\ell(x,0)=0$ for $\forall x\in D$.

\begin{lemma} \label{lower semi lemma}
Functionals $S\in\G$ and their local actions $\ell(x,y)$ have the following properties:\\[.2cm]
\begin{tabular}{rl}
\!\!\!(i)  & \hspace{-.2cm}\begin{minipage}[t]{11.8cm} $S(\gamma)$ is well-defined, i.e.~\eqref{geo action formula} is independent of the specific choice of~$\varphi$.
\end{minipage}\\[.1cm]
\!\!\!(ii)  & \hspace{-.2cm}\begin{minipage}[t]{11.9cm}For\, $\forall\!$ compact $K\subset D\ \,\exists \cB\!=\!\cB(K)\!>\!0\ \,\forall x\in K\ \,\forall y\in\Rn\colon\,\ell(x,y)\!\leq\!\cB|y|$.
In particular, we have for $\forall\gamma\in\tG$ with $\gamma\subset K\colon$ $S(\gamma)\leq\cB\cdot\length(\gamma)$.\end{minipage}
\end{tabular}
\end{lemma}
\begin{proof}
(i) Given a curve $\gamma\in\Gamma$ and any parameterization $\varphi\in\AC{0,1}$ of $\gamma$, we use the representation $\varphi=\varphi_\gamma\circ\beta$ of Lemma \ref{arclength lemma} (ii) and Definition \ref{tech lemma 1}~(i) to find that
\begin{align*}
 \int_0^1\ell(\varphi,\varphi')\,d\alpha &= \int_0^1\ell\big(\varphi_\gamma\circ\beta,(\varphi_\gamma'\circ\beta)\beta'\big)\,d\alpha \\
  &= \int_0^1\ell(\varphi_\gamma\circ\beta,\varphi_\gamma'\circ\beta)\beta'\,d\alpha 
  = \int_0^1\ell(\varphi_\gamma,\varphi_\gamma')\,d\beta,
\end{align*}
where the last step follows again from \cite[p.149, Ex.21]{Stein}. By the uniqueness of $\varphi_\gamma$, the right-hand side only depends on $\gamma$. The proof for general curves $\gamma\in\tilde\Gamma$ is based on the same calculation.%\\[.1cm]
\pb

\noindent
(ii) Given any $K$, set $\cB:=1+\max_{x\in K,|y|=1}\ell(x,y)>0$, use Definition~\ref{tech lemma 1}~(i) to show that
$\ell(x,y)=|y|\ell\big(x,\frac{y}{|y|}\big)\leq \cB|y|$ for $\forall y\neq0$, and recall that $\ell(x,0)=0$. In particular, if $\varphi\in\tC$ is a parameterization of some $\gamma\in\tG$ with $\gamma\subset K$ then $S(\gamma)=\int_0^1\ell(\varphi,\varphi')\,d\alpha\leq\cB\int_0^1|\varphi'|\,d\alpha=\cB\cdot\length(\gamma)$.
\end{proof}

\begin{lemma}[Lower semi-continuity] \label{lower semi lemma 2}
For $\forall S\in\G$ we have the following:\\[.2cm]
\begin{tabular}{rl}
\!\!\!(i) & \hspace{-.2cm}\begin{minipage}[t]{11.8cm}If a sequence $(\varphi_n)_{n\in\N}\subset\AC{0,1}$ fulfilling \eqref{arzela 2} has a uniform limit $\varphi\in\AC{0,1}$ then
$\liminf_{n\to\infty}S(\varphi_n)\geq S(\varphi)$.\end{minipage}\\[.6cm]
\!\!\!(ii)  & \hspace{-.2cm}\begin{minipage}[t]{11.8cm}The limit $\gamma$ constructed in Lemma \ref{min prop0b} fulfills $\liminf_{n\to\infty}S(\gamma_n)\geq S(\gamma)$.\end{minipage}
\end{tabular}
\end{lemma}
\begin{proof}
See Appendix~\ref{lower semi proof app}.
\end{proof}

\begin{definition} \label{drift vector field}
Let $S\in\G$. A vector field $b\in C^1(D,\Rn)$ is called \underline{a drift of $S$} if for $\forall\!$ compact $K\subset D$ $\exists \cA\!=\!\cA(K)>0\ \forall x\in K\ \forall y\in\Rn:$
\begin{equation} \label{drift lower bound}
  \ell(x,y)\geq \cA\big(|b(x)||y|-\skp{b(x)}{y}\big).
\end{equation}
\end{definition}
The right-hand side of \eqref{drift lower bound} is a constant multiple of the local large deviation geometric action \eqref{local geo action0} of the SDE \eqref{SDE0} with drift $b(x)$ and homogeneous noise, and thus we see that for the geometric action associated to \eqref{SDE0}, the vector field $b(x)$ in \eqref{SDE0} is clearly a drift also in this generalized sense (take $\cA=1$).
The inequality \eqref{drift lower bound}, which will only be used in the key estimate Lemma \ref{key estimate} and its weaker version Lemma~\ref{inf length at crit points}, effectively reduces our proofs for an arbitrary action $S\in\G$ to the case of the action given by \eqref{local geo action0}, and it is ultimately the reason why the conditions of our main criteria, Propositions \ref{blm prop 1} and \ref{blm prop 2}, solely depend on the drift and not on any other aspect of the action~$S$.

The drift vector field $b(x)$ in Definition \ref{drift vector field} is not a uniquely defined object: If $b$ is a drift of some action $S\in\G$ and if $\beta\in C^1(D,[0,\infty))$ then $\beta\cdot b$ is a drift of $S$ as well (with modified constants $\cA$), and in particular the vector field $b(x)\equiv0$ is a drift of \textit{any} action $S\in\G$.
Note however that (i) if $\beta(x)>0$ for $\forall x\in D$ then the vector fields $b$ and $\beta b$ have the same flowline diagrams, and we will find that our criteria will not distinguish between these two choices; (ii) if on the other hand $\beta(x)=0$ and $b(x)\neq0$ for some $x\in D$ then the flowline diagrams of $b$ and $\beta\cdot b$ are different, and our criteria may only apply to $b$ but not to $\beta\cdot b$.
In general, a good choice for the drift (i.e.~one that lets us get the most out of our criteria) will be one with only as many roots as necessary.

\begin{definition} \label{flow def}
For a given vector field $b\in C^1(D,\Rn)$ we define the flow $\psi\in C^1(D\times\R,D)$
as the unique solution of the ODE
\begin{equation} \label{psi ODE system}
\begin{cases}
\partial_t   \psi(x,t)\hspace{.05cm}=b(\psi(x,t)) &\text{for}\quad x\in D,\ t\in\R,\\
\hspace{.325cm}\psi(x,0)=x            &\text{for}\quad x\in D.
\end{cases}
\end{equation}
%If it is not clear from the context which vector field $b$ we are referring to, we will make the dependence of $\psi$ on $b$ explicit by writing $\psi_b(x,t)$.
\end{definition}
By a standard result from the theory of ODEs \cite[\S 7.3, Corollary 4]{Arnold}, our regularity assumption on $b$ implies that the solution $\psi(x,t)$ is well-defined \textit{locally} (i.e.~for small $t$), unique, and $C^1$ in $(x,t)$.
However, since $b$ will always play the role of a drift, we may assume that $\psi(x,t)$ is in fact defined \textit{globally}, i.e.~for $\forall t\in\R$: Indeed, if this is not the case then we can instead consider the modified drift $\beta\cdot b$, for some function $\beta\in C^1(D,(0,\infty))$ that vanishes so fast near the boundary $\partial D$ that the associated flow $\tilde\psi$ only reaches $\partial D$ in infinite time (i.e.~$\tilde\psi(x,t)$ is defined for $\forall (x,t)\in D\times\R$), and the only aspect of the flow that will be relevant to us (the flowline diagram) remains invariant under this change.
%and the conditions of Propositions \ref{blm prop 1} and \ref{blm prop 2} (the two main statements utilizing the drift) can be seen to stay valid under this change.

Finally, recall that under this additional assumption we have $\psi(\psi(x,t),s)$\linebreak $=\psi(x,t+s)$ and $\partial_t\nabla\psi(x,t)=\nabla b(\psi(x,t))$ for $\forall x\in D$ and $\forall t,s\in\R$.\\[.2cm]
%
%As for example shown in \cite{Perko}, the function $\nabla\psi(x,t)$ is $C^1$ in $t$, with
%\begin{equation}
%  \partial_t\nabla\psi(x,t)=\nabla b\big(\psi(x,t)\big)\nabla\psi(x,t)\qquad\text{for $\forall x\in D\ \forall t\in\R$}.
%\end{equation}
A special role in our theory will be played by so-called critical points.
\begin{definition} \label{critical point def}
For a given $S\in\G$ with local action $\ell(x,y)$, a point $x\in D$ is called a \underline{critical point} if\, $\forall y\in\Rn\colon\ \ell(x,y)=0$.
\end{definition}

\subsection{The Subclass \texorpdfstring{$\H$}{H} of Hamiltonian Geometric Actions} \label{hamilton section}
We will now consider a particular way of constructing a geometric action from a Hamiltonian $H(x,\te)$, which was introduced in \cite{CPAM} in the context of large deviation theory.\footnote{This paper also proposed an efficient algorithm (called the geometric minimum action method, or gMAM) for numerically computing minimizing curves of such geometric actions.}
\enl

\begin{lemma} \label{action construction lemma}
Let the Hamiltonian $H\in C(D\times\Rn,\R)$ fulfill the assumptions\\[.2cm]
\begin{tabular}{rl}
  \!\!\hypertarget{H1}{(H1)} &\hspace{-.35cm}  $\forall x\in D\colon\ H(x,0)\leq0$, \\[.1cm]
  \!\!\hypertarget{H2}{(H2)} &\hspace{-.35cm}  the derivatives $\Ht$ and $\Htt$ exist and are continuous in~$(x,\te)$,\\[.1cm]
  \!\!\hypertarget{H3}{(H3)} &\hspace{-.35cm}  $\forall\!$ compact $K\!\subset\! D$ $\exists m_K\!\!>\!0$
$\forall x\!\in\!\! K$ $\forall\te,\xi\!\in\!\Rn\colon \skp{\xi}{\Htt(x,\te)\xi}\geq m_K|\xi|^2$\!.
\end{tabular}\\[.23cm]
Then the function $\ell\colon D\times\Rn\to[0,\infty)$ defined by
\begin{subequations}
\begin{align}
\ell(x,y) :=&\,\max\!\big\{\skp{y}{\te}\,\big|\,\te\in\Rn,\,H(x,\te)\leq0\big\} \label{sup rep}\\*
 =&\,\max\!\big\{\skp{y}{\te}\,\big|\,\te\in\Rn,\,H(x,\te)=0\big\} \label{sup rep 2}
\end{align}
\end{subequations}
has the properties of Definition \ref{tech lemma 1}, and so it defines a geometric action $S\in\G$.
\end{lemma}
\begin{proof}
The sets $L_x:=\{\te\in\Rn\,|\,H(x,\te)\leq0\}$ are bounded, in fact uniformly for all $x$ in any compact set $K\subset D$, since for $\forall x\in K\ \forall\te\in L_x\ \exists\tilde\te\in\Rn\colon$
\begin{align}
  0\geq H(x,\te) &= H(x,0)+\skp{\Ht(x,0)}{\te}+\tfrac12\Skp{\te}{\Htt(x,\tilde\te)\te} \nonumber\\
           &\geq -\max_{x\in K}|H(x,0)|-\max_{x\in K}|\Ht(x,0)|\cdot|\te|+\tfrac12m_K|\te|^2.\label{that bounded}
\end{align}
This shows that $\ell$ is finite-valued, and since $0\in L_x$ by \tHa\ we have $\ell(x,y)\geq\skp{y}{0}=0$ for $\forall y\in\Rn$.
The fact that the representations \eqref{sup rep} and \eqref{sup rep 2} are equivalent is obvious for $y=0$; for $y\neq0$ observe that for $\forall\te\in\Rn$ with $H(x,\te)<0$ the boundedness of $L_x$ implies that there $\exists c>0$ such that $H(x,\te+cy)=0$, and $\skp{y}{\te+cy}\geq\skp{y}{\te}$.
The relation $\ell(x,cy)=c\ell(x,y)$ for $\forall c\geq0$ is clear, and $\ell(x,\cdot\,)$ is convex as the supremum of linear functions. The continuity at any point $(x_0,y_0=0)$ follows from the estimate $\ell(x,y)\leq M|y|$ for $\forall y\in\Rn$ and all $x$ in some ball $\bar B_\eps(x_0)\subset D$, where $M:=\sup\!\big\{|\te|\,\big|\,\te\in\bigcup_{x\in\bar B_\eps(x_0)}L_x\big\}$. The continuity everywhere else will follow from Lemma~\ref{second action rep}~(i).\!
\end{proof}

\begin{definition} \label{Hamiltonian action def}
(i) We denote the class of all Hamiltonian geometric actions, i.e.\ of all actions $S$ constructed as in Lemma \ref{action construction lemma}, by $\H\subset\G$.\\[.1cm]
(ii) We denote by $\H_0\subset\H$ the class of all geometric actions $S\in\H$
that are constructed from a Hamiltonian $H$ which fulfills the stronger assumption\\[.15cm]
 \hspace*{.2cm}\hypertarget{H1'}{(H1')}\, $\forall x\in D\colon\ H(x,0)=0$.
% such that $\forall x\in D$ $\exists y\in\Rn\setminus\{0\}:\ \ell(x,y)=0$.%\\[.1cm]
%(iii) We denote by $\H_0^+\subset\H_0$ the class of all geometric actions $S\in\H_0$ that can be constructed from a Hamiltonian $H$ that has a continuous derivative~$\Hztt$.
\end{definition}
Note that since $\ell$ depends on $H$ only through its $0$-level sets, different Hamiltonians $H$ can induce the same geometric action $S\in\H$. In particular, for $\forall\beta\in C(D,(0,\infty))$ the Hamiltonians $H(x,\te)$ and $\beta(x) H(x,\te)$ induce the same action $S$.
The next lemma shows how Definition \ref{critical point def} can be expressed in terms of $H$, and that Assumption \tHaa\ does not depend on the choice of~$H$.
% In particular, the class $\H_0$ is obtained by requiring equality in Assumption (H1) (thus justifying its name).
%
\begin{lemma} \label{critical point lemma}
Let $S\in\H$, and let $H$ be a Hamiltonian that induces $S$.\\[.1cm]
(i) A point $x\in D$ is critical if and only if
\begin{equation} \label{Hamilton critical criterion}
   \Ht(x,0)=0\qquad\text{and}\qquad H(x,0)=0,
\end{equation}
and in that case \eqref{Hamilton critical criterion} holds in fact for every Hamiltonian that induces $S$.\\[.1cm]
(ii) $\forall x\in D\colon\ \big(H(x,0)=0\ \Leftrightarrow\ \exists y\in\Rn\setminus\{0\}\colon\ \ell(x,y)=0\big)$. In particular, if some $H$ inducing $S$ fulfills \tHaa\ then all of them do.
\end{lemma}
\begin{proof}
See Appendix \ref{crit point proof app}. For part (ii) see also Fig.~\ref{that illustration}~(b).
\end{proof}
To actually compute $\ell(x,y)$ from a given Hamiltonian $H$, and for many proofs, the following alternative representation of $\ell$ is oftentimes useful. It can be derived by carrying out the constraint maximization in \eqref{sup rep 2} with the method of Lagrange multipliers.
\begin{lemma} \label{second action rep}
(i) For every fixed $x\in D$ and $y\in\Rn\setminus\{0\}$ the system
\begin{equation} \label{vte eq}
\Ht(x,\that) = \lambda y,\qquad H(x,\that)=0,\qquad \lambda\geq0
\end{equation}
has a unique solution $(\that(x,y),\lambda(x,y))$, the functions $\that\colon D\times(\Rn\setminus\{0\})\to\Rn$ and $\lambda\colon D\times(\Rn\setminus\{0\})\to[0,\infty)$ are continuous, and the function $\ell$ defined in \eqref{sup rep} can be written as
\begin{equation} \label{general geometric local action}
\ell(x,y)= \begin{cases}\skp{y}{\that(x,y)} & \text{if }y\neq0, \\ 0 & \text{if }y=0. \end{cases}
\end{equation}
(ii) If $S\in\H$ is induced by $H$ then a point $x\in D$ is critical if and only if $\exists y\neq0\colon\ \lambda(x,y)=0$. In that case, we have in fact $\lambda(x,y)=0$ for $\forall y\neq0$.
\end{lemma}
\begin{proof}
See Appendix \ref{second action rep app}.
\end{proof}
See Fig.~\ref{that illustration} (a) for a geometric interpretation of \eqref{sup rep}-\eqref{sup rep 2} and \eqref{vte eq}-\eqref{general geometric local action}: By Assumption \tHc\ the function $H(x,\cdot\,)$ and thus also its 0-sublevel set $\{\te\in\Rn\,|\,H(x,\te)\leq0\}$ is strictly convex, and by Assumption~\tHa\ it contains the origin. The maximizer in \eqref{sup rep}, $\te=\that(x,y)$, is the unique point on its boundary where the outer normal aligns with $y$, and the local action $\ell(x,y)$ is $|y|$ times the component of $\that(x,y)$ in the direction~$y$.
%(b) If $H(x,0)=0$ (i.e.~if $\te=0$ lies on the boundary) and if $y$ aligns with $\Ht(x,0)$ then we have $\that=0$ and thus $\ell(x,y)=0$.

%
\begin{figure}[t]
\centering
\includegraphics[width=6.1cm]{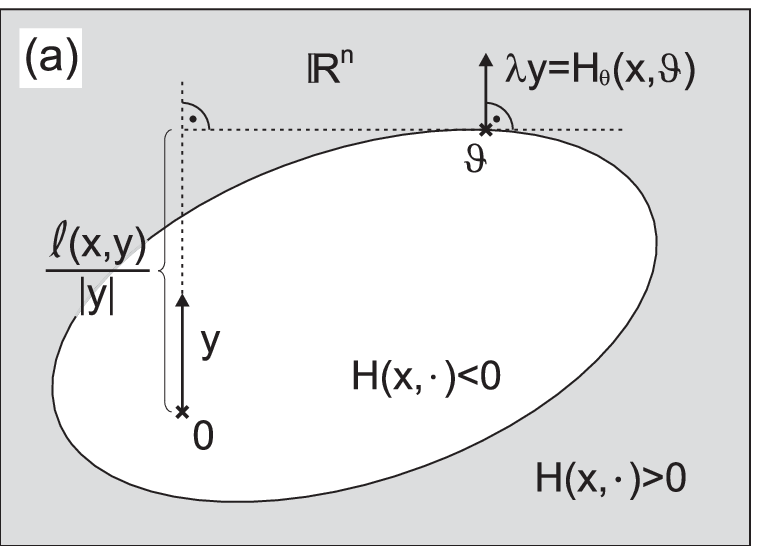}\hspace{.05cm}
\includegraphics[width=6.1cm]{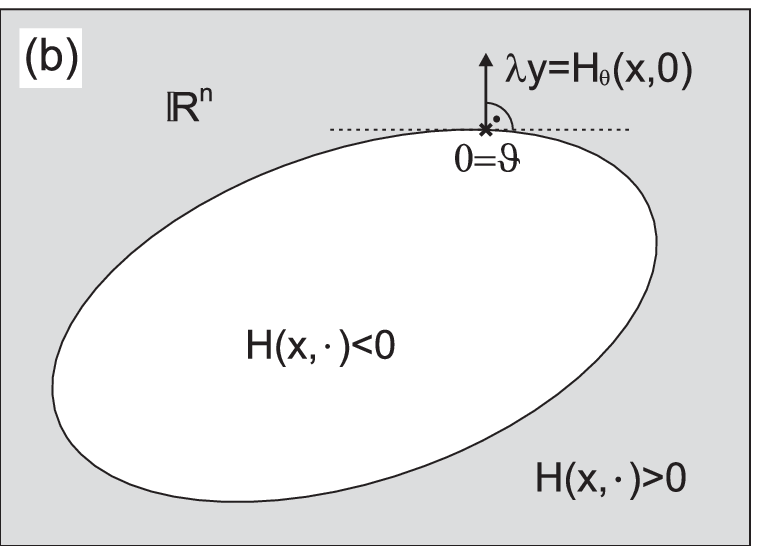}
\caption{\label{that illustration}
\small (a) Illustration of \eqref{sup rep}-\eqref{sup rep 2} and \eqref{vte eq}-\eqref{general geometric local action}, for fixed $x\in D$ and $y\in\Rn\setminus\{0\}$, in the case $H(x,0)<0$.\ \ (b) If $H(x,0)=0$ and if $y$ aligns with $\Ht(x,0)$ then we have $\that=0$.}
\end{figure}

The following lemma provides a quick way to obtain a drift for any Hamiltonian geometric action.

\begin{lemma} \label{hS in G lemma}
If $S\in\H$ is induced by $H$ then $b(x) := \Ht(x,0)$ fulfills the estimate in Definition \ref{drift vector field}, and thus if $b$ is $C^1$ then it is a drift of~$S$. We call a drift obtained in this way \underline{a natural drift of $S$}.
\end{lemma}

\begin{proof}
Let $b(x):=\Ht(x,0)$, and let $K\subset D$ be compact. Define $a:=\sup_{x\in K}|b(x)|$ and
$\cA:=\big[2+\sup\big\{|\Htt(x,\te)|\,\big|\,x\in K,\,|\te|\leq a\big\}\big]^{-1}\!\in(0,\frac12]$, and let $x\in K$ and $y\in\Rn$.

If $y=0$ then \eqref{drift lower bound} is trivial since both sides vanish. Also, if $y\neq0$ and $\lambda(x,y)=0$ then by Lemmas \ref{second action rep} (ii) and \ref{critical point lemma}~(i) we have $b(x)=0$, so \eqref{drift lower bound} is trivial again. Therefore let us now assume that $y\neq0$ and that $\lambda(x,y)>0$.

Setting $\te_0:=\cA \big(\frac{|b(x)|}{|y|}y-b(x)\big)$, a Taylor expansion of $H(x,\te_0)$ around $\te=0$ gives us a $\te'$ on the straight line between $0$ and $\te_0$ (thus fulfilling $|\te'|\leq|\te_0|\leq 2\cA |b(x)|\leq2a\cA \leq a$) such that
\begin{align*}
H(x,\te_0)
 &= H(x,0)+\Skp{\Ht(x,0)}{\te_0}+\tfrac12\Skp{\te_0}{\Htt(x,\te')\te_0} \\
 &\leq 0 + \Skp{b(x)}{\te_0} + \tfrac12\cA^{-1}|\te_0|^2 \\
 &= \Skp{b(x)+\tfrac12\cA^{-1}\te_0}{\te_0} \\
 &= \Skp{\tfrac12\big(\tfrac{|b(x)|}{|y|}y+b(x)\big)}{\,\cA\big(\tfrac{|b(x)|}{|y|}y-b(x)\big)} \\
 &= \tfrac12\cA\big(\big|\tfrac{|b(x)|}{|y|}y\big|^2-|b(x)|^2\big)=0.
\end{align*}
Another Taylor expansion, this time around $\te=\that:=\that(x,y)$, now gives us a $\te''$ such that
\begin{align*}
0 &\geq H(x,\te_0) \\
  &= H(x,\that) + \Skp{\Ht(x,\that)}{\te_0-\that} +
     \tfrac12\Skp{\te_0-\that}{\Htt(x,\te'')(\te_0-\that)} \\
  &\geq 0 + \lambda(x,y)\skp{y}{\te_0-\that}+0,
\end{align*}
where we used both equations in \eqref{vte eq}, and Assumption \tHc. Since $\lambda(x,y)>0$, this implies that
\[
\ell(x,y)
 = \skp{\that}{y} \geq \skp{\te_0}{y} = \cA \Skp{\tfrac{|b(x)|}{|y|}y-b(x)}{y}
 = \cA \big(|b(x)||y|-\skp{b(x)}{y}\big).\vspace{-.15cm}
\]
\end{proof}
Note that since there is not a unique Hamiltonian associated to $S$, there is not a unique natural drift either; in particular, the remark following Definition~\ref{Hamiltonian action def} implies that with $b$ also $\beta b$ is a natural drift for $\forall\beta\in C^1(D,(0,\infty))$, with the same flowline diagram. The next remark shows that for actions $S\in\H_0$ in fact \textit{every} natural drift has the same flowline diagram.

\begin{remark} \label{H0 remark}
For $S\in\H_0$ we have the following:\\[.1cm]
(i) All natural drifts $b$ share the same roots since by Lemma \ref{critical point lemma}~(i) and \tHaa\ we have $b(x)=0$ if and only if $x$ is a critical point. In particular, this means that natural drifts are optimal in the sense that by \eqref{drift lower bound} they only vanish where necessary. \\[.1cm]
(ii) At non-critical points $x$, the direction $y:=\frac{b(x)}{|b(x)|}$ is the same for every natural drift $b$, since Lemma \ref{flow lemma}~(i)-(ii) will characterize it as the unique unit vector $y$ such that $\ell(x,y)=0$.\\[.1cm]
Thus, for any fixed $S\in\H_0$ all natural drifts have the same flowline diagram.
\end{remark}

In contrast, for actions $S\in\H\setminus\H_0$ the natural drift is not always the optimal choice: In Examples ~\hyperlink{example Riemann}{2} and ~\hyperlink{example quantum}{3} below the natural drift will even turn out to be the trivial (and thus useless) drift $b\equiv0$. (See Example~\hyperlink{example 10}{10} in Section~\ref{quantum ex vanishing drift sec} for how to find a better one.)

Finally, the next lemma states the key property of Hamiltonian geometric actions in particular in the context of large deviation theory: It shows how a double minimization problem such as \eqref{double minimization0}-\eqref{SDE ST} can be reduced to a simple minimization problem over a Hamiltonian geometric action.
\begin{lemma}\label{double inf lemma}
Let $H$ be a Hamiltonian fulfilling \tHa-\tHc, and define for $\forall T>0$ the functional $S_T\colon\AC{0,T}\to[0,\infty]$ by
\begin{align}
  S_T(\chi) &:= \int_0^TL(\chi,\dot\chi)\,dt, && \hspace{-.5cm}\text{where} \label{ST def}\\
  L(x,y)    &:= \sup_{\te\in\Rn}\big(\skp{y}{\te}-H(x,\te)\big) && \hspace{-.5cm}\text{for\,\ $\forall x\in D\ \,\forall y\in\Rn$} \label{LH def}
\end{align}
is the Legendre transform of $H(x,\cdot\,)$. Then for $\forall A_1,A_2\subset D$ we have
\begin{equation} \label{double min}
  \inf_{\substack{T>0\\\chi\in\Cxyt{A_1}{A_2}{T}}}S_T(\chi)\  = \inf_{\gamma\in\GA}S(\gamma),
\end{equation}
where $S\in\H$ is the geometric action induced by $H$.
\end{lemma}
\begin{proof}
Using the bijection $(T,\chi)\leftrightarrow(\gamma,T,\beta)$ given in Lemma \ref{arclength lemma} (ii) that assigns to every $\chi\in\AC{0,T}$ its curve $\gamma\in\Gamma$ and its parameterization $\beta\in\AC{[0,T],[0,1]}$ via the relation $\chi=\varphi_\gamma\circ\beta$,
we have
\begin{equation} \label{infinf proof}
  \inf_{\substack{T>0\\\chi\in\Cxyt{A_1}{A_2}{T}}}S_T(\chi) = \inf_{\gamma\in\GA}\inf_{\substack{T>0\\\beta\in\AC{[0,T],[0,1]}}}S_T(\varphi_\gamma\circ\beta) = \inf_{\gamma\in\GA}S(\gamma),
\end{equation}
where the functional
\[
S(\gamma) :=\inf_{\substack{T>0\\\beta\in\AC{[0,T],[0,1]}}}S_T(\varphi_\gamma\circ\beta)
\]
was found in \cite{CPAM} to have the integral representation \eqref{geo action formula} with the local action given by \eqref{sup rep}-\eqref{sup rep 2} and \eqref{general geometric local action}.\footnote{At the beginning of \cite{CPAM}, additional smoothness assumptions on $H$ were made,
%(the existence of continuous derivatives $\Hz$ and $\Htz=(\Hzt)^T$),
but they do not enter the proof of this representation.}
\end{proof}

\noindent We conclude this section with three examples of Hamiltonian geometric actions.

\paragraph{Example 1: Large Deviation Theory.}
\hypertarget{example LDT}{Stochastic} dynamical systems with small noise parameter $\eps>0$ often satisfy a large deviation principle whose action functional $S_T$ is of the form \eqref{ST def}-\eqref{LH def}. Examples include (i) stochastic differential equations (SDEs) in $\Rn$ \cite{FW}
\begin{equation} \label{SDE1}
dX_t^\eps = b(X_t^\eps)\,dt + \sqrt{\eps}\sigma(X_t^\eps)\,dW_t,\qquad X_0^\eps=x_1,
\end{equation}
where $b(x)$ is the drift vector field and $\sigma(x)$ is the diffusion matrix of the SDE, and (ii) continuous-time Markov jump processes in $\Rn$ \cite{SW} with jump vectors $\eps e_i\in\Rn$, $i=1,\dots,N$, and corresponding jump rates $\eps^{-1}\nu_i(\eps x)>0$.
Here we assume that $b$, $A:=\sigma\sigma^T$ and $\nu_i$ are $C^1$ functions,
and that for each fixed $x\in D$, $A(x)$ is a positive definite matrix.
The Hamiltonians used in \eqref{ST def}-\eqref{LH def} to define $S_T$ are
\begin{subequations}
\begin{align}
  \hspace{.8cm}
  H(x,\te) &= \skp{b(x)}{\te}+\tfrac12\skp{\te}{A(x)\te}, &&\text{(SDE)} \label{H SDE}\\*
  H(x,\te) &= \sum_{i=1}^N \nu_i(x)\big(e^{\skp{e_i}{\te}}-1\big).\hspace{-.5cm} &&\text{(Markov jump process)} \label{H KMC}
\end{align}
\end{subequations}
The central object of large deviation theory for answering various questions about rare events in the zero-noise-limit $\eps\to0$, such as the transition from one stable equilibrium point of $b$ to another, is the quasipotential $V(x_1,x_2)$.
Originally defined by \eqref{double minimization0} using the above choice of $S_T$, Lemma~\ref{double inf lemma} allows us to rewrite it as
\begin{equation} \label{qp curve rep}
  V(x_1,x_2)=\inf_{\gamma\in\Gamma_{x_1}^{x_2}}S(\gamma),
\end{equation}
where $S\in\H_0$ is the Hamiltonian geometric action defined via \eqref{sup rep}-\eqref{sup rep 2}, or equivalently, \eqref{vte eq}-\eqref{general geometric local action}.
The minimizing curve $\gamma^\star$ in \eqref{qp curve rep} (if it exists) can be interpreted as the maximum likelihood transition curve.

In the SDE case, \eqref{vte eq} can in fact be solved explicitly: Using for any positive definite symmetric matrix $M$ the notation $\skp{w_1}{w_2}_{\!M}:=\skp{w_1}{Mw_2}$ and $|w|_M^2:=\skp{w}{w}_{\!M}$, the solution of \eqref{vte eq} is given by $\lambda=|b(x)|_{A(x)^{-1}}/|y|_{A(x)^{-1}}$ and $\that=A(x)^{-1}(\lambda y-b(x))$, and so we obtain the local geometric action
\begin{equation}
%  L(x,y)&=\tfrac12|b(x)-y|^2_{A^{-1}(x)}, &&\text{(SDE)} \\[.1cm]
\hspace{.5cm}
  \ell(x,y)=|b(x)|_{A^{-1}(x)}|y|_{A^{-1}(x)}-\skp{b(x)}{y}_{\!A^{-1}(x)}.\hspace{.6cm} \text{(SDE)}
\end{equation}
For Markov jump processes no explicit expression for $\ell(x,y)$ exists.

Finally, we observe that in the SDE case \eqref{H SDE} the expression $\Ht(x,0)$ for the natural drift given in Lemma \ref{hS in G lemma} indeed recovers the given vector field $b(x)$, while in the case \eqref{H KMC} of a Markov jump process we obtain the zero-noise-limit of Kurtz's Theorem~\cite{SW}, i.e.
\[
  \hspace{4.8cm}
  b(x)=\sum_{i=1}^N\nu_i(x)e_i\,.\qquad\text{(Markov jump process)}\vspace{-.6cm}
\]
\hfill\openbox

\paragraph{Example 2: Riemannian metric.}
\hypertarget{example Riemann}{Suppose} that $A\in C(D,\R^{n\times n})$ is a function whose values are positive definite symmetric matrices $A(x)$, and that the metric $g$ is defined by $\skp{y_1}{y_2}_{g_x}:=\skp{y_1}{A(x)y_2}$ for $\forall y_1,y_2\in\Rn$, where the second scalar product is just the Euclidean one. Then the action $S\in\G$ given by
\begin{align}
  \ell(x,y)&=|y|_{g_x} \label{Riemannian local action} \\
\intertext{is a Hamiltonian action, $S\in\H\setminus\H_0$, with associated Hamiltonian}
\hspace{4.5cm}
  H(x,\te)&=|\te|^2_{g^{-1}_x}-1, \qquad\quad\text{(Riemannian metric)} \nonumber
\end{align}
where the metric $g^{-1}$ is defined as above using the matrices $A(x)^{-1}$ instead of $A(x)$. Indeed, as one can easily check, for this choice of $H$ the equations \eqref{vte eq} are fulfilled by $\lambda=2/|y|_{g_x}$ and $\that:=A(x)y/|y|_{g_x}$, and thus the local geometric action defined in \eqref{general geometric local action} yields \eqref{Riemannian local action}.

Note that the natural drift for this Hamiltonian is $b(x)\equiv0$. As we shall see however, this will be made up for by the fact that $H(x,0)<0$ for $\forall x\in D$, see  Proposition \ref{blm prop 0} and Example \hyperlink{example 10}{10} in Section \ref{quantum ex vanishing drift sec}.
\hfill\openbox

\paragraph{Example 3: Quantum Tunnelling.}
\hypertarget{example quantum}{The} instanton by which quantum tunnelling arises is the minimizer $\gamma^\star$ of the Agmon distance \cite[Eq.~(1.4)]{Simon}, i.e.~of \eqref{qp curve rep}, where $S\in\G$ is given by the local action
\begin{equation} \label{agmon local action}
  \ell(x,y) = \sqrt{2U(x)}|y|.
\end{equation}
Here, $x_1$ and $x_2$ are the minima of the potential $U\in C(D,[0,\infty))$, and it is assumed that $U(x_1)=U(x_2)=0$.

If $U$ did not have any roots then this would be a special case of Example~\hyperlink{example Riemann}{2}, with $A(x):=2U(x)I$, which leads us to the Hamiltonian $H(x,\te)=|\te|^2/(2U(x))-1$. According to the remark following \eqref{vte eq}, we can multiply $H$ by the function $U(x)$ without changing the associated action, and so we find that \eqref{agmon local action} is given by
\[
  \hspace{4cm}
  H(x,\te)=\tfrac12|\te|^2-U(x). \qquad\text{(quantum tunnelling)}
\]
We can now check that this choice in fact leads to \eqref{agmon local action} even if $U$ does have roots (with $\lambda=\sqrt{2U(x)}/|y|$ and $\that=\sqrt{2U(x)}\,y/|y|$), and so we have $S\in\H\setminus\H_0$. Again, the natural drift is $b(x)\equiv0$.
%
%Note that the natural drift for this Hamiltonian is $b(x)\equiv0$. As we will see in Example \hyperlink{example 6}{6}, however, this will be made up for by the fact that $H(x,0)=-U(x)<0$ for $\forall x\in D\setminus\{x_1,x_2\}$.
\hfill\openbox

\section{Existence of Minimum Action Curves} \label{existence results chapter}

\subsection{A First Existence Result} \label{first ex result subsec}

\begin{definition} \label{minimizer def}
(i) For a given geometric action $S\in\G$ and two sets $A_1,A_2\subset\E$ we denote by $P(A_1,A_2)$ the minimization problem $\inf_{\gamma\in\GA}S(\gamma)$. For two points $x_1,x_2\in\E$ we write in short $\prob:=P(\{x_1\},\{x_2\})$. \\[.2cm]
(ii) We say that $P(A_1,A_2)$ has a strong (weak) minimizer if $\exists\gamma^\star\in\GA$ \emph($\gamma^\star\in\tGA$\emph) such that
\begin{align*}
 S(\gamma^\star) &= \inf_{\gamma\in\GA}S(\gamma). \\
\intertext{(iii) We say that $(\gamma_n)_{n\in\N}\subset\GA$ is a minimizing sequence of $P(A_1,A_2)$ if}
 \lim_{n\to\infty}S(\gamma_n) &= \inf_{\gamma\in\GA}S(\gamma).
\end{align*}
\end{definition}
\noindent Recall that (by our definition at the end of Section \ref{rec curves subsec}) the class of curves $\GA$ only contains curves that are contained in $\E$, and so $P(A_1,A_2)$ is the problem of finding the best curve leading from $A_1$ to $A_2$ \underline{in $\E$}. 
To avoid that this additional constraint negatively affects our construction of minimizers by forcing us to move along curves whose lengths we cannot control, we have to require some regularity of $\E$: For the rest of this paper we will make the following assumption.\\[.2cm]
\noindent \textbf{\hypertarget{ass tD}{Assumption:}} \textit{The set $\E$ has the following property:}\\[.1cm]
\hspace*{.2cm}($\E$)\hspace{.4cm}$\forall x\in\E\ \,\forall\nu>0\ \,\exists r>0\ \,\forall w\in\bar B_r(x)\cap\E\ \,\exists\gamma\in\Gamma_x^w\colon\ \length(\gamma)\leq\nu$.\\[.2cm]
This assumption says that nearby points in $\E$ can be connected by short curves in $\E$. Using a compactness argument, it also implies that any two points in $\E$ can be connected by a rectifiable curve $\gamma\subset\E$, which by Lemma \ref{lower semi lemma}~(ii) (with $K:=\gamma$) has finite action. In particular, any (weak or strong) minimizer must have finite action.

The next lemma gives some sufficient (but by no means necessary) conditions that can help to prove the Assumption \hE\ for a given set $\E$ of interest.

\begin{lemma} \label{E Lemma}
If $\E=D$, or if $\E=\bigcup_{i=1}^m\E_i$ for some sets $\E_1,\dots,\E_m\subset D$ that are convex and closed in $D$, then the Assumption \hE\ is fulfilled.
\end{lemma}
\begin{proof} Let $x\in\E$ and $\nu>0$. If $\E=D$ then we can choose $r\in(0,\nu]$ so small that $\bar B_r(x)\subset\E$, and for any $w\in\bar B_r(x)\cap\E=\bar B_r(x)$ we can let $\gamma$ be the straight line from~$x$ to~$w$. Then we have $\gamma\subset\bar B_r(x)\subset\E$ and thus $\gamma\in\Gamma_x^w$, and furthermore $\length(\gamma)=|w-x|\leq r\leq\nu$.

If $\E=\bigcup_{i=1}^m\E_i$ for some sets $\E_i$ that are convex and closed in $D$, let $I:=\{i\,|\,x\in\E_i\}\neq\varnothing$ and choose $r\in(0,\nu]$ so small that
$\bar B_r(x)\subset D\setminus\bigcup_{i\notin I}\E_i$. Then we have
$\bar B_r(x)\cap\E=\bigcup_{i=1}^m(\bar B_r(x)\cap\E_i)=\bigcup_{i\in I}(\bar B_r(x)\cap\E_i)$, and so for $\forall w\in\bar B_r(x)\cap\E\ \exists i\in I$ such that $w$ is in the convex set $\E_i$. Since also $x\in\E_i$, the straight connection line $\gamma$ from $x$ to $w$ fulfills $\gamma\subset\E_i\subset\E$ and thus $\gamma\in\Gamma_x^w$, and again we have $\length(\gamma)=|w-x|\leq r\leq\nu$.
\end{proof}

The following lemma explains why in Definition \ref{minimizer def} we do not distinguish between minimizing over $\GA$ and over $\tGA$.
\begin{lemma} \label{weak vs strong inf}
For any geometric action $S\in\G$ and any two sets $A_1,A_2\subset\E$ we have
\begin{equation}
  \inf_{\gamma\in\GA}S(\gamma)=\inf_{\gamma\in\tGA}S(\gamma).
\end{equation}
\end{lemma}
\begin{proof}
The inequality ``$\geq$'' is clear since $\GA\subset\tGA$. To show also the inequality ``$\leq$'', let any $\tilde\gamma\in\tGA$ and $\eps>0$ by given. We must construct a curve $\gamma\in\GA$ with $S(\gamma)\leq S(\tilde\gamma)+\eps$.

To do so, let $\rho>0$ be so small that $K:=\bar N_\rho(\tilde\gamma)\subset D$, and let $\cB>0$ be the corresponding constant given by Lemma \ref{lower semi lemma} (ii). Suppose there are $m$ points along $\tilde\gamma$ that are passed in infinite length. We then define $\gamma\in\GA$ by replacing the at most $2m$ infinitely long curve segments preceding and/or following these $m$ points by rectifiable curves $\gamma_i\subset\E$ with $\length(\gamma_i)\leq\nu:=\min\{\frac{\eps}{2m\cB},\rho\}$, as given by Assumption \hE. Since for every $i$ we have $\gamma_i\subset\bar N_\rho(\tilde\gamma)$ and thus $S(\gamma_i)\leq\cB\length(\gamma_i)\leq\frac\eps{2m}$ by Lemma \ref{lower semi lemma} (ii), we have $S(\gamma)\leq S(\tilde\gamma)+\sum_iS(\gamma_i)\leq S(\tilde\gamma)+\eps$, completing the proof.
\end{proof}

In this chapter we will explore conditions on $S$ that guarantee the existence of a (weak or strong) minimizer $\gamma^\star$. We begin with a first result that was already stated in the introduction.

\begin{proposition} \label{min prop}
Let $S\in\G$, let the two sets $A_1,A_2\subset\E$ be closed in $D$, and suppose that there exists a compact set $K\subset\E$ such that the minimization problem $\probA$ has a minimizing sequence $(\gamma_n)_{n\in\N}$ with $\gamma_n\subset K$ for $\forall n\in\N$ and with $\sup_{n\in\N}\length(\gamma_n)<\infty$.
Then $\probA$ has a strong minimizer $\gamma^\star$ fulfilling $\length(\gamma^\star)\leq\liminf_{n\to\infty}\length(\gamma_n)$.
\end{proposition}

\begin{proof}
Let $M':=\liminf_{n\to\infty}\length(\gamma_n)$, and let us pass on to a subsequence, which we again denote by $(\gamma_n)_{n\in\N}$, such that $\lim_{n\to\infty}\length(\gamma_n)=M'$.
For $\forall n\in\N$, let $\varphi_n$ be the arclength parameterization of $\gamma_n$ given by Lemma \ref{arclength lemma} (i), i.e.\ the one fulilling $|\varphi_n'|\equiv\length(\gamma_n)$ a.e.. Our conditions on $(\gamma_n)_{n\in\N}$ now imply that the sequence $(\varphi_n)_{n\in\N}$ fulfills the conditions of Lemma \ref{min prop0} (i), and so there exists a subsequence $(\varphi_{n_k})_{k\in\N}$ that converges uniformly to some function $\varphi^\star\subseteq K\subset\E\subset D$ which by Lemma \ref{min prop0} (ii) is in $\AC{0,1}$. Since $A_1$ and $A_2$ are closed in $D$, we have $\varphi^\star\in\CIA$.
By Lemma \ref{lower semi lemma 2} (i), the curve $\gamma^\star\in\GA$ parameterized by $\varphi^\star$ fulfills
\[
S(\gamma^\star)=S(\varphi^\star)\leq\lim_{k\to\infty} S(\varphi_{n_k})=\lim_{k\to\infty} S(\gamma_{n_k})=\inf_{\gamma\in\GA}S(\gamma),
\]
i.e.~$\gamma^\star$ is a strong minimizer of $P(A_1,A_2)$.

Finally, observe that for $\forall\eps>0\ \exists k_0\in\N:\ \sup_{k\geq k_0}\length(\gamma_{n_k})\leq M'+\eps$, and applying Lemma \ref{min prop0} (ii) to the tail sequence $(\varphi_{n_k})_{k\geq k_0}$ we find that $|{\varphi^\star}'|\leq M'+\eps$ a.e.~and thus $\length(\gamma^\star)\leq M'+\eps$. Since $\eps>0$ was arbitrary, this shows that $\length(\gamma^\star)\leq M'$.
\end{proof}
\subsection{Points with Local Minimizers, Existence Theorem} \label{main results subsection 1}
As we shall see in Theorem \ref{comp thm}, by using a compactness argument the minimization problem $\probA$ can be reduced to the special case $\prob$ where $x_1$ and $x_2$ are close to each other. The following definition therefore lies at the heart of this entire work, and thus the reader is strongly advised not to proceed until this definition is fully understood. The illustrations in Fig.~\ref{local minimizers illustration} may help in this respect.
\begin{figure}[t]
\centering
\includegraphics[height=4.7cm]{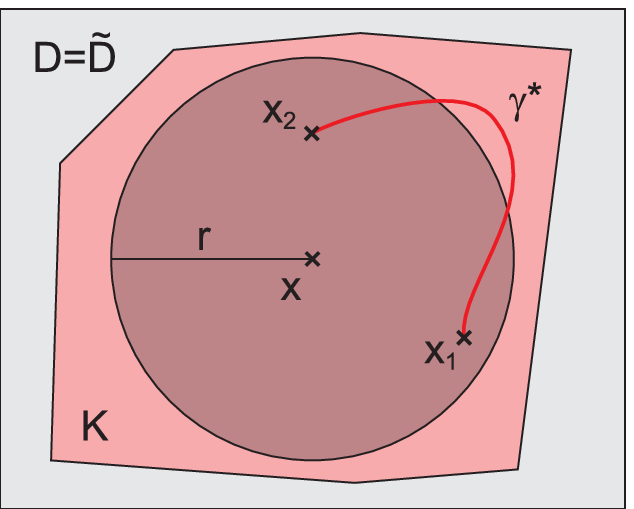}\hspace{.05cm}
\includegraphics[height=4.7cm]{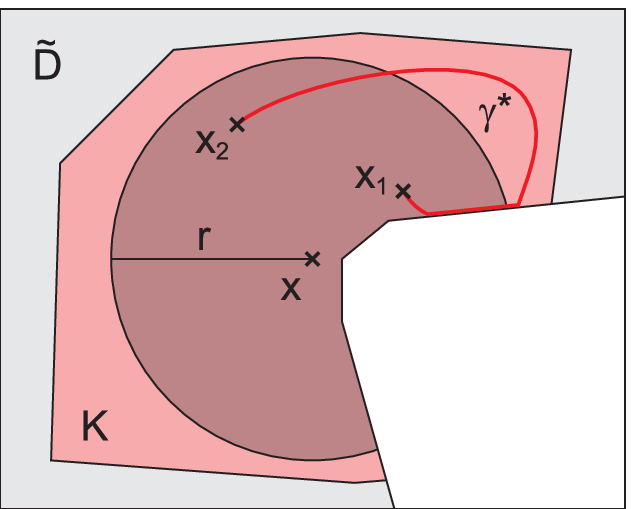}
\caption{\label{local minimizers illustration}
\small Illustration of Definition \ref{blm def}. The left graphic illustrates the case $\E=D$; the right graphic shows how for $\E\subsetneq D$ we only need to consider points $x_1,x_2\in\E$, and that the corresponding minimizing curve $\gamma^\star$ is then constrained to lie within~$\E$. In either case, independently of $x_1$ and $x_2$, $\gamma^\star$ must lie within some fixed compact set $K\subset\E$ and satisfy a length condition. }
\end{figure}
\begin{definition} \label{blm def}
(i) We say that a point $x\in\E$ \underline{has strong local minimizers} if $\exists r,\eta>0$ $\exists$ compact $K\subset\E$ $\forall x_1,x_2\in\bar B_r(x)\cap\E$ the minimization problem $P(x_1,x_2)$ has a strong minimizer $\gamma^\star\in\Gx$ with $\gamma^\star\subset K$ and $\length(\gamma^\star)\leq\eta$.\\[.1cm]
(ii) We say that a point $x\in\E$ \underline{has weak local minimizers} if there exist a constant $r>0$, a function $\eta\colon(0,\infty)\to[0,\infty)$ and a compact set $K\subset\E$ such that for $\forall x_1,x_2\in\bar B_r(x)\cap\E$ the minimization problem $\prob$ has~a~weak minimizer $\gamma^\star\in\tGxxx$ with $\gamma^\star\subset K$ and
$\forall u>0\colon\,\length\!\big(\gamma^\star|_{\bar B_u(x)^c}\big)$ $\leq\eta(u)$.
\end{definition}
%Note that we do not require that $\varphi^\star$ is entirely contained in $\bar B_r(x)$. Also
%
Observe that strong implies weak: Indeed, if $x$ has strong local minimizers then we can choose the function $\eta(u)$ in part~(ii) to be the constant $\eta$ given in part~(i), and so $x$ has weak local minimizers.

It is important to understand that the only aspect of this property that justifies the use of the word ``local'' is that $x_1$ and $x_2$ are close to $x$; the corresponding minimization problem $\prob$ still considers curves that lead far away from $x$. Thus, checking that a given point $x$ has local minimizers generally requires \textit{global} knowledge of $\ell(x,y)$ (although an exception is given in Proposition \ref{blm prop 0}).

\enl
\begin{remark} \label{blm remark}
(i) The set of points with strong local minimizers is open in $\E$.\\[.1cm]
(ii) To prove that a point $x\in\E$ has strong local minimizers, it suffices to show that for
$\forall\eta>0$ $\exists r>0\ \forall x_1,x_2\in\bar B_r(x)\cap\E$ the minimization problem $P(x_1,x_2)$ has a minimizer $\gamma^\star\in\Gx$ with $\length(\gamma^\star)\leq\eta$.
Indeed, this implies that $\gamma^\star\subset\bar B_{r+\eta}(x)\cap\E=:K\subset D$, and $K$ is compact if $r$ and $\eta$ are chosen so small that $\bar B_{r+\eta}(x)\subset D$.\\[.1cm]
(iii) For the same reasons, if $D=\Rn$ then the requirement $\gamma^\star\subset K$ in Definition \ref{blm def} (i) may be dropped entirely since then $K:=\bar B_{r+\eta}(x)\cap\E$ is a compact set with $\gamma^\star\subset K$.
\end{remark}

As we will see in Sections \ref{criteria subsec} and \ref{sec examples}, showing that a given point has (weak or strong) local minimizers is rather easy once the flowlines of a good choice for the drift $b(x)$ of $S$ are understood. In fact, oftentimes one can show that \textit{every} point $x\in\E$ has local minimizers.

The following theorem which is proven at the end of this section extends the local property of Definition \ref{blm def} to a global one by using a compactness argument.

\begin{theorem}[Existence Theorem] \label{comp thm}
(i) Let $S\in\G$, let $K\subset\E$ be a compact set consisting only of points that have weak local minimizers. Let the two sets $A_1,A_2\subset\E$ be closed in $D$, and let us assume that the minimization problem $\probA$ has a minimizing sequence $(\gamma_n)_{n\in\N}$ such that $\gamma_n\subset K$ for $\forall n\in\N$.
Then $\probA$ has a weak minimizer.\\[.1cm]
(ii) If (in addition to the above conditions) all points in $K$ have strong local minimizers then $\probA$ has a strong minimizer.
\end{theorem}
\begin{proof}
Postponed to the end of this section.
\end{proof}

The decisive advantage of Theorem \ref{comp thm} over Proposition \ref{min prop} is that the bounded-length-condition of the minimizing sequence is no longer required, and instead we have to show that $K$ consists of points with local minimizers.
The remaining condition, $\gamma_n\subset K$ for $\forall n\in\N$, boils down to the following estimate.

\begin{lemma} \label{main corollary}
Let $S\in\G$, let $K\subset\E$ be compact, let $A_1,A_2\subset\E$, and suppose that there exists some curve $\gamma_0\in\GA$ with $\gamma_0\subset K$ such that
\begin{equation} \label{main cor condition}
  \hS(\gamma_0)\leq\inf_{\substack{\gamma\in\GA\\\gamma\nsubseteq K}}\hS(\gamma),
\end{equation}
i.e.~no curve leading from $A_1$ to $A_2$ and leaving $K$ along its way has a smaller action than $\gamma_0$.
Then $\probA$ has a minimizing sequence $(\gamma_n)_{n\in\N}$ with $\gamma_n\subset K$ for $\forall n\in\N$.
\end{lemma}

\begin{proof}[Proof of Lemma \ref{main corollary}]
Let $(\gamma_n)_{n\in\N}$ be \textit{any} minimizing sequence. If we replace every curve $\gamma_n$ that is not entirely contained in $K$ by $\gamma_0$ then because of \eqref{main cor condition} we only reduce the action. Thus we obtain a new minimizing sequence that is now entirely contained in $K$.
\end{proof}

\paragraph{Example 4.}
\hypertarget{example 4}{In} the case that $A_1$ is bounded and $S$ is the SDE geometric action given by \eqref{local geo action0} with a drift of the form $b=-\nabla V$, for some potential $V\in C^1(\Rn,\R)$ with $\lim_{x\to\infty}V(x)=\infty$, it suffices in Lemma \ref{main corollary} to choose $K=\bar B_R(0)$ for some sufficiently large $R>0$.

To see this, choose the fixed curve $\gamma_0\in\GA$ arbitrarily, and let $\gamma\in\GA$ with $\gamma\subsetneq K$. Let $\gamma'$ denote the curve segment of $\gamma$ until its first exit of $B_R(0)$, and let $x_1$ and $x_2$ be the start and end points of $\gamma'$, respectively. Then we have
\begin{align*}
   S(\gamma)
     &\geq S(\gamma')
    =    \int_{\gamma'}\big(|\nabla V(z)||dz|+\skp{\nabla V(z)}{dz}\big) \\
    &\geq 2\int_{\gamma'}\skp{\nabla V(z)}{dz}
     = 2\int_{\gamma'} dV(z)
     = V(x_2)-V(x_1) \\
    &\geq \min\!\big\{V(x)\,\big|\,|x|=R\big\} - \max\!\big\{V(x)\,\big|\,x\in A_1\big\},
\end{align*}
which can be made larger than $S(\gamma_0)$ by choosing $R$ large enough.
\hfill\openbox

\begin{proof}[Proof of Theorem \ref{comp thm}]
Although the construction for part (i) directly implies the statement of part (ii), we will show part (ii) separately first (since its proof uses a much easier argument at its end) and then extend the proof to cover part (i). See Fig.~\ref{comp thm illustration} for an illustration of the proof of part~(ii).\\[.2cm]
(ii) Let $S\in\G$, and let the sets $K,A_1,A_2\subset\E$ have the properties described in Theorem \ref{comp thm}, where $K$ only consists of points with strong local minimizers.
For $\forall x\in K$ Definition \ref{blm def}~(i) provides us with values $r_x,\eta_x>0$ and compact sets $K_x\subset\E$ such that for $\forall x_1,x_2\in\bar B_{r_x}(x)\cap\E$ there exists a minimizer $\gamma^\star_{x_1,x_2}\in\Gx$ of the minimization problem $\prob$ with $\gamma^\star_{x_1,x_2}\subset K_x$ and $\length(\gamma^\star_{x_1,x_2})\leq\eta_x$.
Since $\{B_{r_x}(x)\,|\,x\in K\}$ is an open covering of $K$, there exists a finite subcovering, i.e.~there exist points $x_1,\dots,x_k\in K$ such that $K\subset\bigcup_{j=1}^{k}B_{r_j}(x_j)$, where $r_j:=r_{x_j}$. We define $M:=\sum_{j=1}^k\eta_{x_j}$.
\begin{figure}[t!]
\centering
\includegraphics[width=12.2cm]{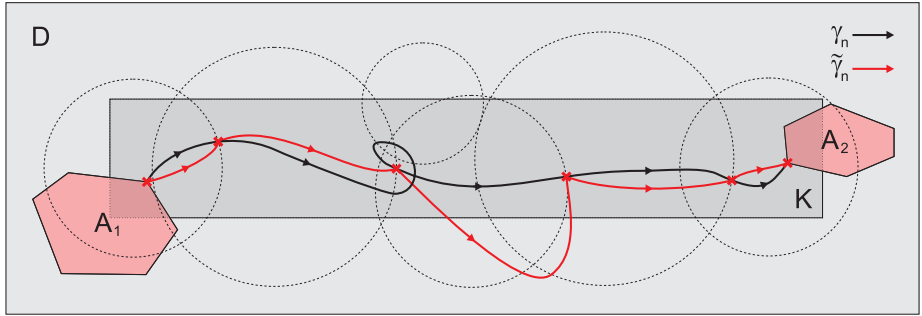}
\caption{\label{comp thm illustration}
\small Illustration of the proof of Theorem \ref{comp thm} (ii), with $\E=D$. Every curve $\gamma_n$ of the given minimizing sequence is cut into at most $k$ pieces whose start and end point is contained in the same ball $\bar B_{r_j}(x_j)$. Using Definition~\ref{blm def}, these pieces are then replaced by new curve segments with minimal action and controllable length.}
\end{figure}

Now let $(\gamma_n)_{n\in\N}\subset\GA$ be a minimizing sequence with $\gamma_n\subset K$ for $\forall n\in\N$.
For each fixed $n\in\N$ we will now define a modified curve $\tilde\gamma_n$ by cutting $\gamma_n$ into at most $k$ pieces whose start and end points lie within the same ball, and then by replacing these pieces by the corresponding optimal curves with the same start and end points.

To make this description rigorous, let the functions $\varphi_n\in\CIA$ be some parameterizations of the curves $\gamma_n$, and fix $n\in\N$. We then define (for some $m\leq k$) the numbers $0=\alpha_1<\dots<\alpha_m=1$, the distinct indices $j_1,\dots,j_m\in\{1,\dots,k\}$ and finally $j_{m+1}=j_m$ by induction, as follows:
\pb
\begin{itemize}
\item Let $\alpha_1=0$, and let $j_1$ be such that $\varphi_n(0)\in B_{r_{j_1}}(x_{j_1})$.
\item For $i\geq1$, let $\alpha_{i+1}:=\sup\!\big\{\alpha\in[0,1]\,\big|\,\varphi_n(\alpha)\in B_{r_{j_i}}(x_{j_i})\big\}$, and let
    \[ \begin{cases}j_{i+1}\text{ be such that }\varphi_n(\alpha_{i+1})\in B_{r_{j_{i+1}}}(x_{j_{i+1}}) & \text{if}\quad\alpha_{i+1}<1,\\
    j_{i+1}:=j_i,\ m:=i&\text{if}\quad\alpha_{i+1}=1.\end{cases} \]
\end{itemize}
In other words, we split the curve $\gamma_n$ into $m$ pieces whose endpoints fulfill $\varphi_n(\alpha_i),\varphi_n(\alpha_{i+1})\in\bar B_{r_{j_i}}(x_{j_i})$ for $\forall i=1,\dots,m$.
Since also $\varphi_n\subset K\subset\E$, by definition of the radii $r_j$ the $m$ minimization problems $P\big(\varphi_n(\alpha_i),\varphi_n(\alpha_{i+1})\big)$ ($i=1,\dots,m$) have strong minimizers $\gamma^\star_{n,i}\subset K_{x_{j_i}}\subset\E$ with $\length(\gamma^\star_{n,i})\leq\eta_{x_{j_i}}$, and in particular we have $\hS(\gamma^\star_{n,i})\leq\hS(\varphi_n|_{[\alpha_i,\alpha_{i+1}]})$. The concatenated curve $\tilde\gamma_n:=\gamma^\star_{n,1}+\dots+\gamma^\star_{n,m}\in\GA$ thus fulfills
\begin{align}
\hS(\tilde\gamma_n)&=\sum_{i=1}^m\hS(\gamma^\star_{n,i})
  \leq\sum_{i=1}^m\hS\big(\varphi_n|_{[\alpha_i,\alpha_{i+1}]}\big)
    =\hS(\varphi_n)=\hS(\gamma_n),\label{action piece} \\*
\length(\tilde\gamma_n&)=\sum_{i=1}^m\length(\gamma^\star_{n,i})
  \leq \sum_{i=1}^m\eta_{x_{j_i}}\leq \sum_{j=1}^k\eta_{x_j}=M. \label{length piece}
\end{align}
Because of \eqref{action piece}, the modified sequence $(\tilde\gamma_n)_{n\in\N}$ is still a minimizing sequence, and \eqref{length piece} tells us that the curves $\tilde\gamma_n$ have uniformly bounded lengths. Furthermore, we have $\tilde\gamma_n\subset\bigcup_{i=1}^mK_{x_{j_i}}\subset\bigcup_{j=1}^kK_{x_j}$, which is a compact subset of $\E$. Therefore we can apply Proposition \ref{min prop} and conclude that $\probA$ has a minimizer $\gamma^\star$, with
\[ \length(\gamma^\star)\leq\liminf_{n\to\infty}\length(\tilde\gamma_n) \leq M. \]
(i) For this part we begin as in the proof of part (ii), by choosing a finite collection of balls $B_{r_j}(x_j)$ covering $K$, now given by Definition \ref{blm def}~(ii) whenever $x_{j}$ only has weak local minimizers. Given the minimizing sequence $(\gamma_n)_{n\in\N}\subset\GA$, we cut each curve $\gamma_n$ into smaller segments as in part (ii). The number of pieces $m$ and the indices $j_1,\dots,j_m$ may depend on $n$, but since there are only finitely many combinations, we may pass on to a subsequence (which we again denote by $(\gamma_n)_{n\in\N}$), such that $m$ and $j_1,\dots,j_m$ are in fact the same for every curve $\gamma_n$.

We then construct a new sequence $(\tilde\gamma_n)_{n\in\N}\subset\tilde\Gamma_{A_1}^{A_2}$ with $S(\tilde\gamma_n)\leq S(\gamma_n)$ for $\forall n\in\N$
as in the proof of part (ii), only that now if $x_{j_i}$ only has weak local minimizers then the curve segment $\gamma^\star_{n,i}$ must be obtained from Definition~\ref{blm def}~(ii), and so we have $\gamma^\star_{n,i}\in\tilde\Gamma(x_{j_i})$ in this case.
We can assume that each segment $\gamma^\star_{n,i}$ visits the point $x_{j_i}$ at most once (otherwise we can cut out the piece between the first and the last hitting point of $x_{j_i}$, which can only decrease the action of the curve).

If $x_{j_1}$ has strong local minimizers then we can apply Lemma \ref{min prop0}, just as in the proof of Proposition \ref{min prop}, to show that some subsequence of the arc\-length parameterizations $(\varphi_{n,1})_{n\in\N}\subset\bar C(0,1)$ of $(\gamma^\star_{n,1})_{n\in\N}$ converges uniformly to the parameterization of some $\gamma^\star_{\infty,1}\in\Gamma$. If instead $x_{j_1}$ only has weak local minimizers then we apply Lemma \ref{min prop0b} to show that a subsequence of some parameteriations $(\varphi_{n,1})_{n\in\N}\subset\tilde C(x_{j_i})$ of $(\gamma^\star_{n,1})_{n\in\N}$ converges pointwise on $[0,1]$ and uniformly on each set $[0,\frac12-a]\cup[\frac12+a,1]$, $a\in(0,\frac12)$, to the parameterization of some some $\gamma^\star_{\infty,1}\in\tilde\Gamma(x_{j_1})$. In either case, since $\gamma^\star_{n,1}\subset\E$ for $\forall n\in\N$ and since $\E$ is closed in $D$, we have $\gamma^\star_{\infty,1}\subset\E$.

We repeat this procedure for $x_{j_2},\dots,x_{j_m}$, each time passing on to a further subsequence, and in this way obtain curve pieces $\gamma^\star_{\infty,1},\dots,\gamma^\star_{\infty,m}$ that by construction connect to a curve $\gamma^\star\in\tGA$. Using both parts of Lemma \ref{lower semi lemma 2}, its action fulfills
\begin{align*}
 S(\gamma^\star) &= \sum_{i=1}^mS(\gamma^\star_{\infty,i})
    \leq \sum_{i=1}^m\liminf_{n\to\infty}S(\gamma^\star_{n,i})
    \leq \liminf_{n\to\infty}\sum_{i=1}^mS(\gamma^\star_{n,i}) \\
   &=    \liminf_{n\to\infty}S(\tilde\gamma_{n})
    \leq \liminf_{n\to\infty}S(\gamma_{n})
    = \inf_{\gamma\in\GA}S(\gamma) 
    = \inf_{\gamma\in\tGA}S(\gamma),
\end{align*}
where in the last step we used Lemma \ref{weak vs strong inf}. Since $\gamma^\star\in\tGA$, equality must hold, and so $\gamma^\star$ is a weak minimizer.
\end{proof}

\begin{remark}
Denoting the minimizer by $\gamma^\star$, the proof implies that\\[.1cm]
in (i), there exists a finite set $W\subset K$ of points that only have weak but not strong local minimizers, depending only on $K$ but not on $A_1$ and $A_2$, such that every point that $\gamma^\star$ passes in infinite length is in $W$;\\[.1cm]
in (ii), we have $\length(\gamma^\star)\leq M$, where $M>0$ is a constant only depending on $K$ but not on $A_1$ and $A_2$.
\end{remark}

\begin{remark}
Theorem~\ref{comp thm} and Lemma \ref{main corollary} can easily be generalized to cover also the minimization over sets of the form
\begin{align*}
\Gamma_{A_1,\dots,A_k}&:=\big\{\gamma\subset\E\,\big|\,\text{\emph{$\gamma$ visits $A_1,\dots,A_k$ in this order}}\big\} \\*
\text{or}\qquad\Gamma_{A_1,\dots,A_k}'&:=\big\{\gamma\subset\E\,\big|\,\text{\emph{$\gamma$ visits $A_1,\dots,A_k$ in any order}}\big\}
\end{align*}
for any given $k\in\N$ and any given sets $A_1,\dots,A_k\subset\E$ that are closed in $D$. In this case, $(\gamma_n)_{n\in\N}$ must be a minimizing sequence of the corresponding associated minimization problem.
\end{remark}

\subsection{Finding Points with Local Minimizers} \label{criteria subsec}

This leaves us with the question how one can show that a given point $x\in\E$ has local minimizers. We have developed three criteria, given by Propositions \ref{blm prop 0}, \ref{blm prop 1} and \ref{blm prop 2}, which were designed to cover the three cases listed in Table~\ref{case table}. In the special case of a Hamiltonian geometric action $S\in\H$ with the choice of a natural drift $b(x)=\Ht(x,0)$, these three cases can be expressed in terms of the Hamiltonian $H$ associated to~$S$.
The proofs of most statements that are listed in this section will be carried out in Part~\ref{part proofs}.

We will from now on assume that $S\in\G$ and that $b$ is a drift of $S$, and we will denote by $\psi(x,t)$ the flow of $b$ given in Definition \ref{flow def}. Our first result is the following.

\begin{table}[t]
\begin{center}
\begin{tabular}[t]{|l||l|l|}
\hline
 & \hspace{1.8cm}$S\in\G$ & \hspace{0.4cm}$S\in\H$ with natural drift $b$\\
\hline\hline
Prop.~\ref{blm prop 0} & $\ell(x,y)>0$ for $\forall y\neq0$ & $H(x,0)<0$ \\
\hline
Prop.~\ref{blm prop 1} & 
\begin{minipage}[t]{4.25cm}$\ell(x,y)=0$ for some $y\neq0$\\
$b(x)\neq0$
\end{minipage}\hspace*{.4cm}
 & $H(x,0)=0$ and $\Ht(x,0)\neq0$\hspace*{.4cm}\\
\hline
Prop.~\ref{blm prop 2} & 
\begin{minipage}[t]{4.25cm}$\ell(x,y)=0$ for some $y\neq0$\\
$b(x)=0$
\end{minipage} &
$H(x,0)=0$ and $\Ht(x,0)=0$\\
\hline
\end{tabular}
\end{center}
\vspace{-.55cm}
\caption{\label{case table}The situations for which our criteria Propsitions \ref{blm prop 0}-\ref{blm prop 2} were designed.
}
\end{table}

\begin{proposition} \label{blm prop 0}
Let $x\in\E$ be such that $\ell(x,y)>0$ for $\forall y\in\Rn\setminus\{0\}$. Then $x$ has strong local minimizers.
\end{proposition}
\begin{proof}
See Part \ref{part proofs}, Section \ref{subsec proof prop 2}.
\end{proof}
By Lemma \ref{critical point lemma}~(ii) for actions $S\in\H$ the condition of Proposition \ref{blm prop 0} is fulfilled if and only if $H(x,0)<0$ for some (and thus every) Hamiltonian $H$ that induces $S$.
Unfortunately, this means that Proposition~\ref{blm prop 0} cannot be applied to actions $S\in\H_0$, and in particular it cannot be applied to the large deviation geometric actions for SDEs and for Markov jump processes, as given in Example~\hyperlink{example LDT}{1}. For actions $S\in\H\setminus\H_0$ such as the ones given in Examples~\hyperlink{example Riemann}{2} and~\hyperlink{example quantum}{3}, however, this criterion is essential (and the easiest one to use); see Example~\hyperlink{example 10}{10} in Section \ref{quantum ex vanishing drift sec}.

To control the potential problems that can arise if $\ell(x,y)=0$ for some $y\neq0$, we now introduce the concept of admissible manifolds. Loosely speaking, an admissible manifold $M$ is a compact $C^1$-manifold of codimension 1 with the property that the flowlines of the drift $b$ are never tangent to $M$ and always cross $M$ in the same direction (``in'' or ``out'').

\begin{definition} \label{admissible def}
Given a vector field $b\in C^1(D,\Rn)$, a set $M\subset D$ is called an \underline{admissible manifold of $b$} if there exists a function $f_M\in C(D,\R)$ such that\\[.2cm]
\begin{tabular}{rl}
(i)   & $M=f_M^{-1}(\{0\})$, \\[.2cm]
(ii)  & $M$ is compact,\\[.2cm]
(iii) & $f_M$ is $C^1$ in a neighborhood of $M$, and \\[.2cm]
(iv)  & $\forall x\in M\colon\,\skp{\nf_M(x)}{b(x)}>0$.
\end{tabular}
\end{definition}

Property (iv) says that the drift vector field $b(x)$ flows from the set $f_M^{-1}\big((-\infty,0)\big)$ into the set $f_M^{-1}\big((0,\infty)\big)$ at every point of their common boundary $M=f_M^{-1}(\{0\})$, crossing $M$ at a non-vanishing angle.
%Observe also that $\nu(x):=\frac{\nabla f_M(x)}{|\nabla f_M(x)|}$ (for $\forall x\in M$) is the outer normal vector field of $D^-$, which is well-defined since $\nabla f_M(x)\neq0$ for $\forall x\in M$ (a consequence of property (iv)).
Note that $M$ is a proper $C^1$-manifold since by part (iv) we have $\nf_M\neq0$ on $M$.
Also by part (iv) we have the following:

\begin{remark} \label{on manifold b neq0}
  If $M$ is an admissible manifold of $b$ then $\forall x\in M\colon\, b(x)\neq0$.
\end{remark}

To get a better idea of how admissible manifolds look in $\R^2$, the reader may briefly skip ahead and take a look at Figures \ref{fig two basins}-\ref{fig counterexamples} on pages \pageref{fig two basins}-\pageref{fig counterexamples}. There, the black and the blue lines are the flowlines of the vector field $b(x)$, and the solid red lines are admissible manifolds. Dashed red lines are examples of curves that are \textit{not} admissible manifolds since they are crossed by the flowlines in either direction (both ``in'' and ``out'').

A simple explicit example can be given for the drift of Example \hyperlink{example 4}{4}, i.e.\ if $b=-\nabla V$ for some potential $V\in C^1(\Rn,\R)$ with $\lim_{x\to\infty}V(x)=\infty$: Here, the level sets $M_c:=V^{-1}(\{c\})$, $c\in\R$, are admissible manifolds, provided that $\nabla\nabla V\neq0$ on $M_c$. Indeed, the reader can easily check that all four properties in Definition \ref{admissible def} are fulfilled, with $f_{M_c}=-V+c$.

Lemma \ref{admissible ball example} below gives the simplest general example of an admissible manifold, as found repeatedly in Figures \ref{fig two basins}-\ref{fig counterexamples}: the surface of a small deformed ball around a stable or unstable equilibrium point.
To prepare for this lemma, we introduce two functions $f_s$ and $f_u$ that are defined on the basins of attraction/repulsion of $x$, denoted by $B_s$ and $B_u$, respectively. These functions measure the ``distance'' of a point $w$ to the equilibrium point $x$ in terms of the length of the flowline starting from $w$ until it reaches $x$ as $t\to\infty$ ($f_s$) or as $t\to-\infty$ ($f_u$), respectively.

\begin{definition} \label{fs fu def}
Let $x\in D$ be such that $b(x)=0$ and that all the eigenvalues of the matrix $\nabla b(x)$ have negative (positive) real part, and let $B_s$ \emph($B_u$\emph) be the basin of attraction (repulsion) of $x$.
Then we define the function\linebreak $f_s\colon B_s\to[0,\infty)$ \emph($f_u\colon B_u\to[0,\infty)$\emph) by
\begin{subequations}
\begin{align}
  f_s(w) &:= \int_0^\infty|b(\psi(w,t))|\,dt = \int_0^\infty|\dot\psi(w,t)|\,dt,\qquad w\in B_s, \label{fs def 1} \\*
  f_u(w) &:= \int_{-\infty}^0\!|b(\psi(w,t))|\,dt = \int_{-\infty}^0\!|\dot\psi(w,t)|\,dt,\qquad w\in B_u. \label{fu def 1}
\end{align}
\end{subequations}
\end{definition}

\begin{lemma} \label{admissible ball example}
Let $x\in D$ be such that $b(x)=0$ and that all the eigenvalues of the matrix $\nabla b(x)$ have negative (positive) real parts. Then for sufficiently small $a>0$ the level set $\Msa:=f_s^{-1}(\{a\})$ \emph($\Mua:=f_u^{-1}(\{a\})$\emph) is an admissible manifold.
\end{lemma}
\begin{proof}
See Part~\ref{part proofs}, Section \ref{subsec proof lemma 1}.
\end{proof}

The following Proposition \ref{blm prop 1}, which is our second criterion for showing that a given point $x\in D$ has local minimizers, is our first result that makes use of the concept of admissible manifolds. In practice this criterion covers most cases which cannot be treated with Proposition \ref{blm prop 0}.

\begin{proposition} \label{blm prop 1}
Let $M$ be an admissible manifold and $x\in\psi(M,\R)\cap\E$.
Then $x$ has strong local minimizers.
\end{proposition}
\begin{proof}
See Part~\ref{part proofs}, Section \ref{subsec proof prop 3}.
\end{proof}

Proposition \ref{blm prop 1} says that every admissible manifold $M$ that we find will give us a whole region $\psi(M,\R)\cap\E$ of points with strong local minimizers, consisting of all the flowlines emanating from $M$.
% or equivalently, that the flowline starting from $x$ must intersect $M$ at some time $t\in\R$.
An immediate consequence is the following:

\begin{corollary} \label{blm 3}
Let $x\in\E$ be such that $b(x)=0$ and that all the eigenvalues of the matrix $\nabla b(x)$ have negative (positive) real parts, and denote by $B_s$ \emph($B_u$\emph) the basin of attraction (repulsion) of $x$. Then every point in $(B_s\setminus\{x\})\cap\E$ \emph($(B_u\setminus\{x\})\cap\E$\emph) has strong local minimizers.
\end{corollary}
\begin{proof}
This follows from Lemma \ref{admissible ball example} and Proposition \ref{blm prop 1} since
for small $a>0$ we have $\psi(M_s^a,\R)=B_s\setminus\{x\}$ and $\psi(M_u^a,\R)=B_u\setminus\{x\}$. (The reader who wants to prove these intuitive equations rigorously will find the necessary tools in Lemma~\ref{distance function}.)
\end{proof}

By Remark \ref{on manifold b neq0}, admissible manifolds cannot contain any points $x$ with $b(x)=0$, and thus the flowlines emanating from $M$ cannot contain any such points either. As a consequence, to show that a given point $x\in\E$ has local minimizers, Proposition \ref{blm prop 1} can only be useful if $b(x)\neq0$. For points with $b(x)=0$ (and in particular for the missing point $x$ in Corollary \ref{blm 3}) we have the following criterion.
\begin{proposition} \label{blm prop 2}
Let $x\in\E$ be such that $b(x)=0$, and that all the eigenvalues of the matrix $\nabla b(x)$ have nonzero real part. Let us denote by $M_s$ and $M_u$ the global stable and unstable manifolds of $x$, respectively, i.e.
\begin{subequations}
\begin{align}
  M_s&:=\big\{w\in D \,\big|\, \lim_{t\to \infty}\psi(w,t)=x\big\}, \label{Ms def} \\
  M_u&:=\big\{w\in D \,\big|\, \lim_{t\to-\infty}\psi(w,t)=x\big\}. \label{Mu def}
\end{align}
\end{subequations}
(i) If $x$ is an attractor or repellor of $b$ then $x$ has weak local minimizers. If in addition
\begin{align}
 &\exists \eps,\cJ>0 &&\forall w\in\bar B_\eps(x)\cap\E\hspace{-.15cm} &&\exists\gamma\in\Gamma_x^w\colon &\hspace{-.2cm}\length(\gamma)&\leq \cJ|w-x|, \label{strong condition 2} \\
 &\exists\rho,\cC,\delta>0\hspace{-.15cm} &&\forall w\in\bar B_\rho(x) &&\forall y\in\Rn\hspace{.5pt}\colon &\ell(w,y)&\leq \cC |w-x|^\delta|y|, \label{strong condition}
\end{align}
then $x$ has strong local minimizers.\\[.2cm]
(ii) If $x$ is a saddle point, and if there exist admissible manifolds $M_1,\dots,M_m$ such that
\begin{equation} \label{main condition}
(M_s\cup M_u)\setminus\{x\}\subset\bigcup_{i=1}^m\psi(M_i,\R),
\end{equation}
then $x$ has weak local minimizers.
If in addition the state space is two-dimensional, i.e.~$D\subset\R^2$, and if \eqref{strong condition 2}-\eqref{strong condition} are fulfilled then $x$ has strong local minimizers.
\end{proposition}
\begin{proof}
See Part~\ref{part proofs}, Section \ref{subsec proof prop 4}.
\end{proof}
The condition \eqref{strong condition 2} on the shape of the set $\E$ near $x$ is a stronger version of Assumption \hE, and it is violated only in degenerate cases that are rarely of interest in practice. Lemma \ref{prop cond equiv}~(i) will give some useful criteria.
The condition \eqref{strong condition} can also easily be checked, even in the case of a Hamiltonian geometric action when no explicit formula for $\ell(x,y)$ may be available; see Lemma \ref{prop cond equiv}~(ii).
\begin{lemma} \label{prop cond equiv}
(i) If $x\in\E^\circ$ (which is true in particular if $\E=D$), or if $\E=\bigcup_{i=1}^m\E_i$ for some sets $\E_1,\dots,\E_m\subset D$ that are convex and closed in~$D$, then the condition \eqref{strong condition 2} is fulfilled.\\[.1cm]
(ii) Suppose that $S\in\H$ is induced by a Hamiltonian $H$ such that $H(\,\cdot\,,0)$ and $\Ht(\,\cdot\,,0)$ are locally H\"older continuous at $x$. Then the condition \eqref{strong condition} is fulfilled if and only if $x$ is a critical point.
%In particular, if $S\in\H_0$ and if $b$ is a natural drift then \eqref{strong condition} follows from our assumption $b(x)=0$.
\end{lemma}
\begin{proof}
(i) As in the proof of Lemma \ref{E Lemma}. (ii) See Appendix \ref{prop cond proof app}.
\end{proof}
The condition \eqref{main condition} says that every point in the stable and unstable manifold of $x$ (except for $x$ itself) has to lie on a flowline emanating from one of a finite collection of admissible manifolds, or equivalently, that every flowline in the stable and the unstable manifold must intersect one of these finitely many admissible manifolds. See the next section for examples.

Finally, it should be pointed out that it is Proposition \ref{blm prop 2}~(ii) that is responsible for the excessive length of our proofs (and in particular for all of Part~\ref{part superprop}). In particular, a lot of effort in part (ii) went into proving the existence of \textit{strong} local minimizers at least in the two-dimensional case, which allows us to conclude that the problem $P(A_1,A_2)$ of minimizing $S(\gamma)$ over all $\gamma\in\GA$ has a solution $\gamma^\star$ that actually lies in $\GA$ and not only in the larger class $\tGA$. For remarks on the possible extension of our results to higher dimensions, see the Conclusions in Chapter \ref{sec conclusions}.

\subsection{Examples in \texorpdfstring{$\R^2$}{R2}} \label{sec examples}
Let us see in some two-dimensional examples, $D=\R^2$, how these criteria are used in practice. In Figures \ref{fig two basins}-\ref{fig counterexamples}, the black and the blue lines are the flowlines of $b$, the roots of $b$ are denoted by the symbols $\ominus$ (attractor), $\oplus$ (repellor) and $\circledS$ (saddle point).
Basins of attraction are shown in various shades of gray, basins of repulsion are drawn in gray lines at various angles. The stable and unstable manifolds of the saddle points are drawn in blue. Finally, a representative selection of admissible manifolds is drawn as red solid curves. In Fig.~\ref{fig counterexamples}, dashed red curves illustrate why it is impossible to draw admissible manifolds through certain points.

Throughout the discussion of these examples (i.e.~in the remainder of Section \ref{sec examples}) we will assume that for every root $x$ of $b$ (i.e.~for every attractor, repellor, or saddle point) all the eigenvalues of the matrix $\nabla b(x)$ have non-zero real parts. Also, for simplicity we will discuss the case $\E=D$, so that the condition \eqref{strong condition 2} is trivially fulfilled by Lemma \ref{prop cond equiv}~(i). But our arguments will not change if $\E\subsetneq D$, except that then proving that the roots of~$b$ have \textit{strong} (as opposed to weak) local minimizers requires checking the additional condition \eqref{strong condition 2}, e.g.\ by using Lemma \ref{prop cond equiv}~(i).

\subsubsection{Two basins of attraction} \label{subsub two basins}
In our first two examples we consider systems in which the drift vector field~$b$ has two stable equilibrium points whose basins of attraction partition the state space into two regions.

\paragraph{Example 5.}
Fig.~\ref{fig two basins} (a) shows the flowlines of a vector field $b$ with two attractors, and with one saddle point on the separatrix. The points in the two basins of attraction (shaded in light gray and dark gray) all have local minimizers by Corollary \ref{blm 3} and Proposition \ref{blm prop 2} (i). The three red lines are admissible manifolds (the two small ones can be obtained from Lemma \ref{admissible ball example}), and we observe that every flowline on the stable and the unstable manifold of the saddle point (blue) intersects one of them. Proposition \ref{blm prop 1} thus implies that every point on these flowlines has local minimizers, and Proposition \ref{blm prop 2} (ii) implies that the saddle point itself has local minimizers as well.
We conclude that in this system \textit{every} point in $\E$ has local minimizers.

In fact, all points (with the possible exception of the roots of $b$) have \textit{strong} local minimizers. To guarantee that the three roots have \textit{strong} local minimizers as well, one only needs to check the condition \eqref{strong condition} at these points. In the case of an action $S\in\H$ induced by some Hamiltonian $H$ such that $H(\,\cdot\,,0)$ and $\Ht(\,\cdot\,,0)$ are locally H\"older continuous, by Lemma~\ref{prop cond equiv} this is equivalent to \eqref{Hamilton critical criterion}. In particular, if $S\in\H_0$ and $b$ is a natural drift then there is nothing to check.
These remarks about the distinction between strong and weak local minimizers also apply to the Examples \hyperlink{example 6}{6}-\hyperlink{example 9}{9}.
\hfill\openbox

\paragraph{Example 6.}
\hypertarget{example 6}{Fig.}~\ref{fig two basins} (b) shows another system with two attractors, only now there are two saddle points and one repellor on the separatrix. The points in the two basins of attraction are again shaded in light gray and dark gray, the basin of repulsion is drawn in gray diagonal lines. By Corollary \ref{blm 3} and Proposition \ref{blm prop 2} (i) every point in these three regions has local minimizers, which leaves us only with the two saddle points, and with the outer halves of their respective stable manifolds. Again we observe that every flowline of the stable and unstable manifolds of the two saddle points (blue) intersects one of the four admissible manifolds drawn in the figure. As in the previous example, Proposition \ref{blm prop 1} thus implies that every point on these flowlines has local minimizers, and Proposition~\ref{blm prop 2} (ii) implies that the two saddle points have local minimizers as well.
We conclude that also in this system every point in $\E$ has local minimizers.
\hfill\openbox

\begin{figure}[p]
\centering
\includegraphics[width=12.5cm]{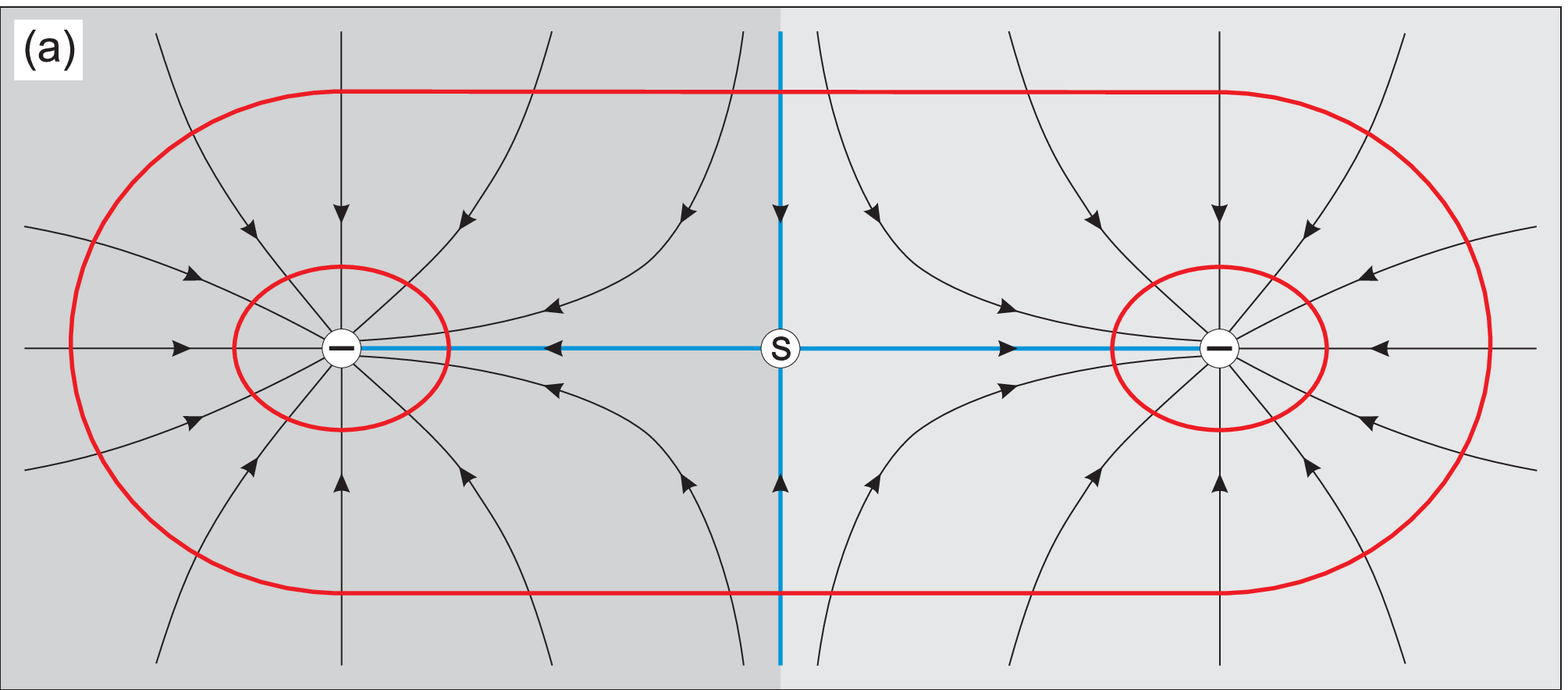} \\[.2cm]
\includegraphics[width=12.5cm]{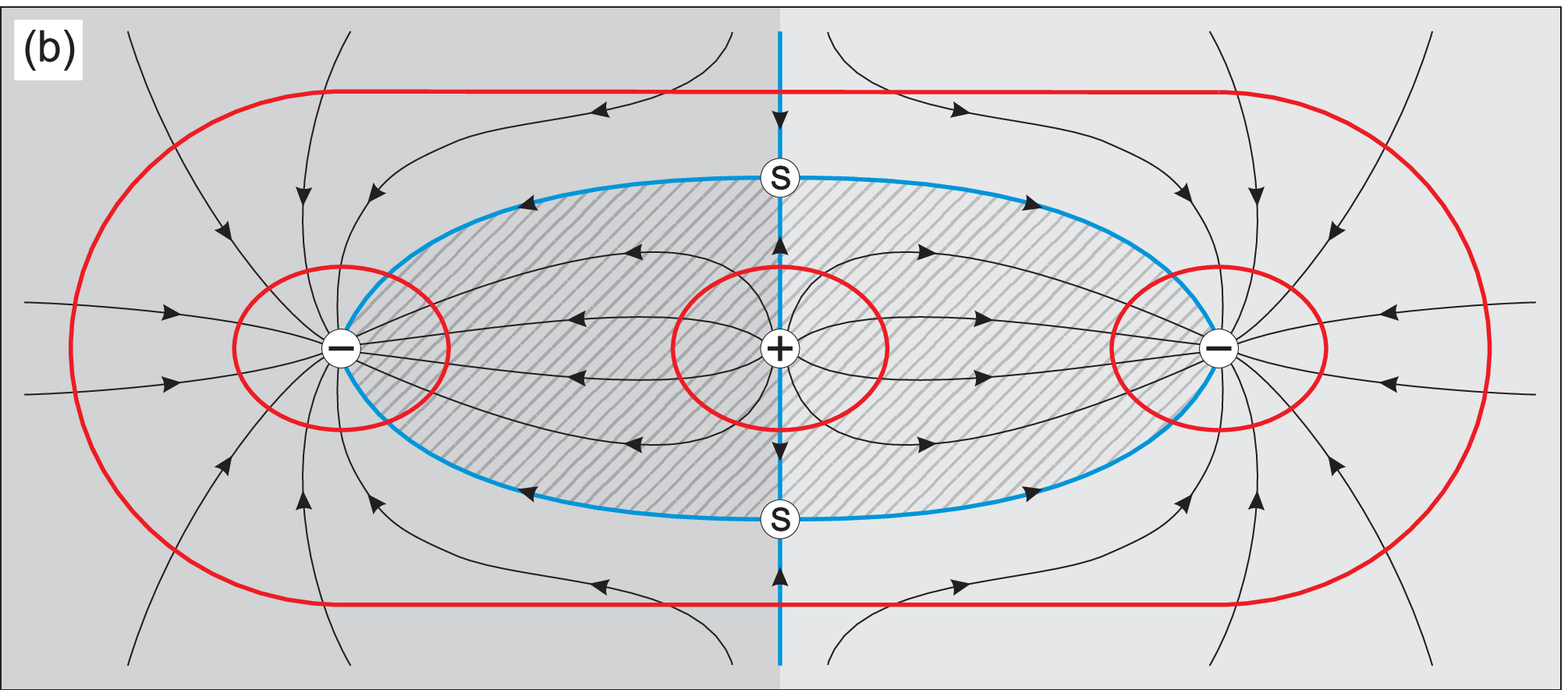} \\[.2cm]
\includegraphics[width=12.5cm]{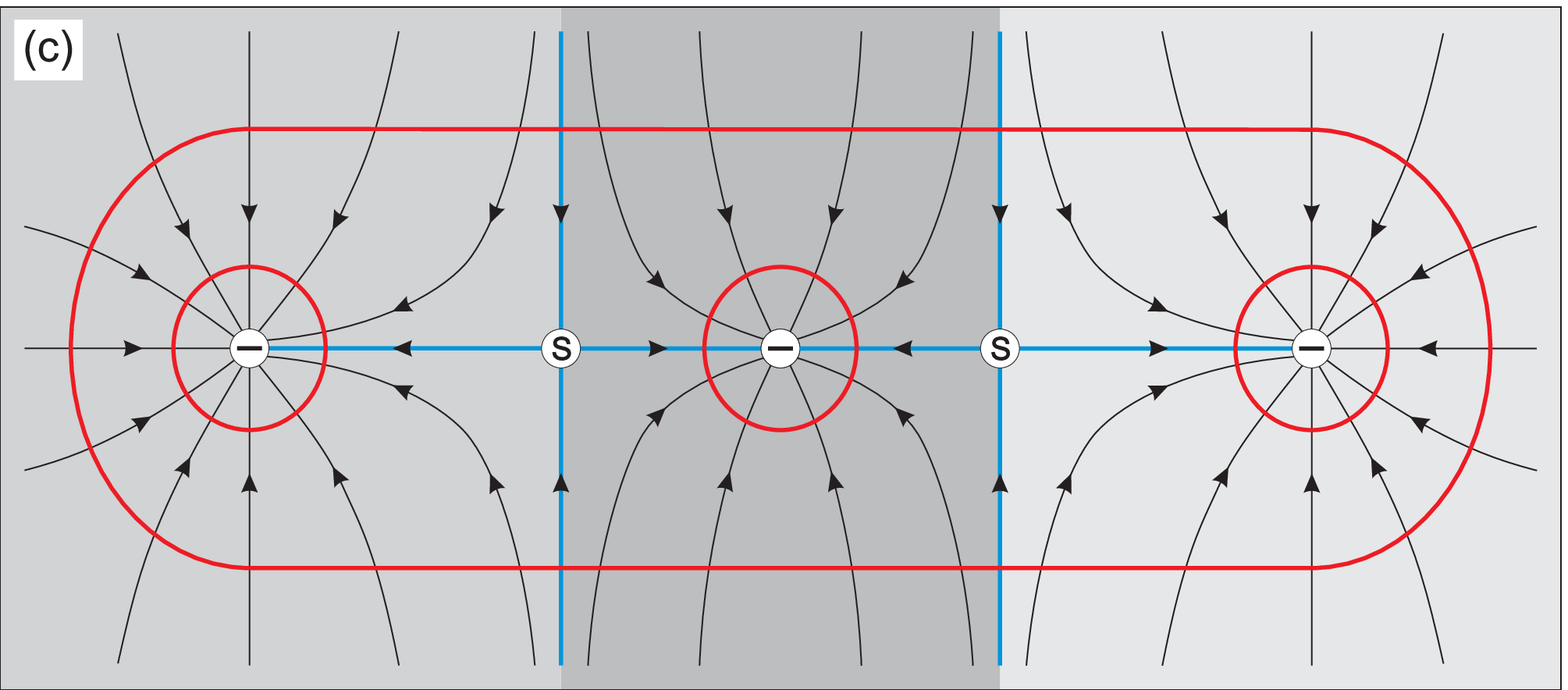}
\caption{\label{fig two basins}
\small Two systems with two attractors, and one system with three attractors.
}
\end{figure}
\begin{figure}[p]
\includegraphics[width=12.5cm]{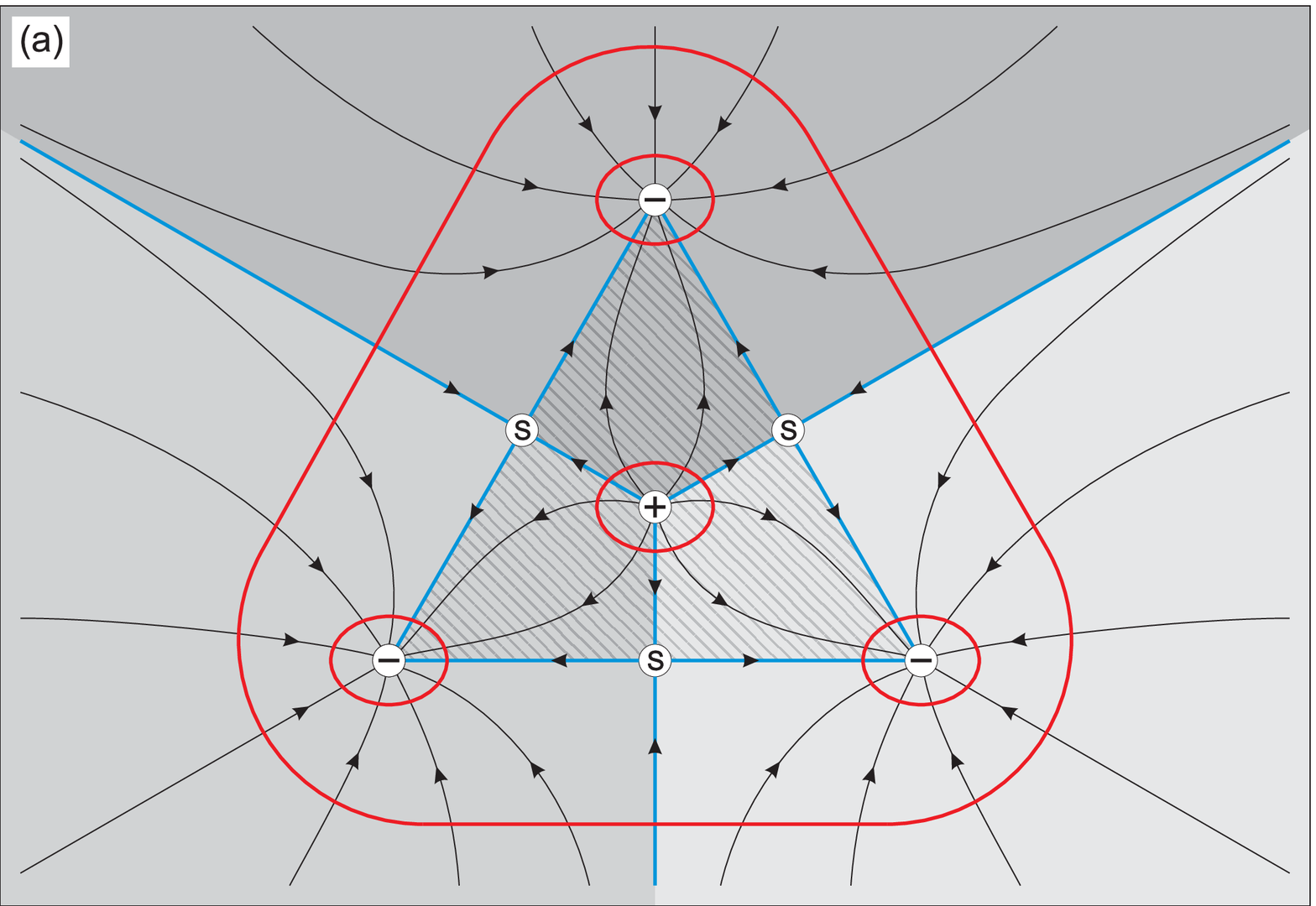} \\[.2cm]
\includegraphics[width=12.5cm]{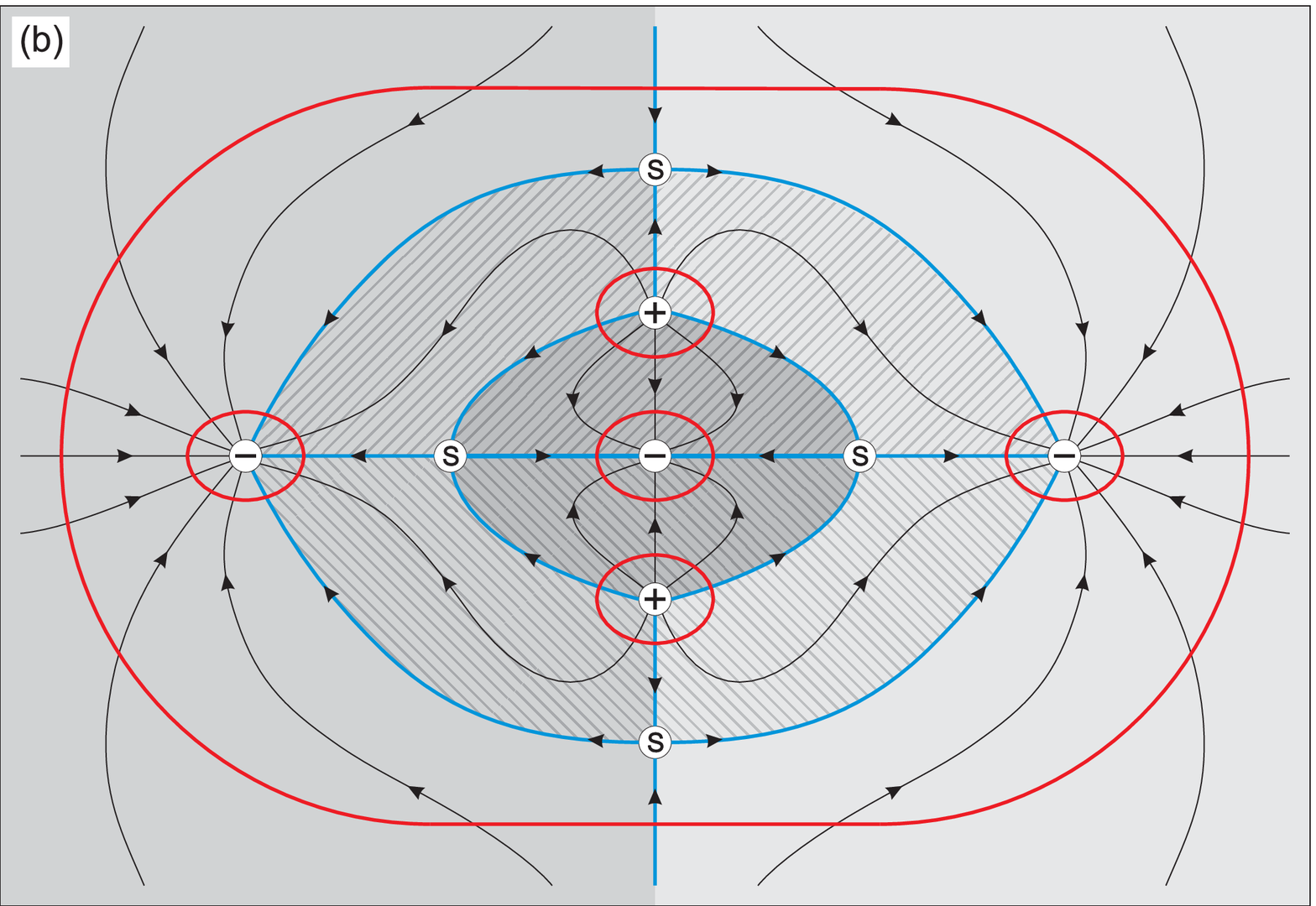}
\caption{\label{fig three basins}
\small Two more systems with three attractors.
}
\end{figure}

\subsubsection{Three basins of attraction} \label{subsub three basins}
We now discuss three examples of systems with three attractors. In each case, we will again find that every point in the state space has local minimizers.

\paragraph{Example 7.}
Fig.~\ref{fig two basins} (c) shows a system with three attractors, with all three basins of attraction aligned in a row. As usual, Corollary \ref{blm 3} and Proposition \ref{blm prop 2}~(i) cover the three basins of attraction, Proposition \ref{blm prop 1} covers the stable manifolds of the saddle points since they intersect the outer admissible manifold, and Proposition \ref{blm prop 2} (ii) covers the saddle points themselves since every flowline of their stable and unstable manifolds intersects an admissible manifold. We conclude again that every point in $\E$ has local minimizers.
\hspace*{0cm}\hfill\openbox

\paragraph{Example 8.}
Fig.~\ref{fig three basins} (a) shows a system with three attractors that form a triangle with a repellor at its center. There are a total of three saddle points, one on each of the three branches of the separatrix. All the points in the three basins of attraction and in the basin of repulsion have local minimizers by Corollary \ref{blm 3} and Proposition \ref{blm prop 2} (i). Again we are left only with the three saddle points, and with the outer halves of their stable manifolds. Both can be treated with Propositions \ref{blm prop 1} and \ref{blm prop 2} (ii) as in the previous examples, and we find again that every point in $\E$ has local minimizers.
\hfill\openbox

\paragraph{Example 9.}
\hypertarget{example 9}{Fig.}~\ref{fig three basins} (b) shows yet another system with three attractors. This time, one basin of attraction is enclosed by the two others, and we count a total of two repellors and four saddle points. After applying Corollary \ref{blm 3} and Proposition \ref{blm prop 2} (i) to the three basins of attraction and the two basins of repulsion, we are only left with the four saddle points, and with the outer halves of the stable manifolds of the two outer saddle points. We can proceed as before, and apply Propositions \ref{blm prop 1} and \ref{blm prop 2} (ii) to show that also these remaining points have local minimizers.
\hfill\openbox

\subsubsection{An example with trivial natural drift} \label{quantum ex vanishing drift sec}
\paragraph{Example 10.} \hypertarget{example 10}{For} the geometric action given by \eqref{Riemannian local action}, i.e.\ the curve length with respect to a Riemannian metric, and for the quantum tunnelling geometric action given by \eqref{agmon local action} in Section~\ref{hamilton section} we only found the natural drift $b(x)\equiv0$, and so we must argue differently.
In the first case we have $\ell(x,y)>0$ for $\forall y\neq0$ by our assumption that $A(x)$ is positive definite, and so every point $x\in\E$ has strong local minimizers by Proposition \ref{blm prop 0}.

For the quantum tunnelling geometric action this argument applies only to all points $x\in\E\setminus\{x_1,x_2\}$ (where $U(x)>0$), and we will have to deal with the points $x=x_1,x_2$ separately.
%Note that since below we will not make use of the Hamiltonian structure of $S$, we only need to require that $U$ is continuous (instead of $C^1$).
Let us now assume that $\exists c,\eps>0$ $\forall x\in B_\eps(x_i)\colon$
$|U(x)|\geq c|x-x_i|^2$, $i=1,2$.

Then the vector fields $b_i(x):=\zeta_i(x)(x-x_i)$,
for some cutoff functions $\zeta_i\in C^1(D,[0,1])$ with $\supp \zeta_i\subset B_\eps(x_i)$ and $\zeta_i(x_i)=1$, are drift vector fields of $S$ since
\begin{align*}
\ell(x,y)=\sqrt{2U(x)}\,|y|&\geq\sqrt{2c}\,\zeta_i(x)|x-x_i||y|=\sqrt{2c}\,|b_i(x)||y|\\
         &\geq\sqrt{c/2}\,\big(|b_i(x)||y|-\skp{b_i(x)}{y}\big).
\end{align*}
Since $x_i$ is a repellor of $b_i(x)$ with $\nabla b_i(x_i)=I$, we can apply Proposition~\ref{blm prop 2}~(i) to conclude that $x_1$ and $x_2$ have weak local minimizers.
If in addition $U$ is H\"older continuous at $x_1$ and $x_2$ then the condition \eqref{strong condition} is fulfilled, and $x_1$ and $x_2$ have in fact strong local minimizers. (Observe that the alternative criterion for \eqref{strong condition} given by Lemma \ref{prop cond equiv} leads to the same condition.)
\hfill\openbox

\subsubsection{Examples to which our criteria do not apply} \label{subsub counterexamples}
We will now present three examples in which for some points the conditions of our criteria are not fulfilled. As a consequence, unless we can otherwise show that there exists a minimizing sequence that stays in a compact set $K\subset\E$ away from these points, the question of whether a minimizer exists will be left undecided at present: Without further thought it may still be possible that (i) the points in question in fact \textit{do} have local minimizers, and our criteria from the previous section are only not strong enough to show it, or (ii) the points \textit{do not} have local minimizers, but Theorem \ref{comp thm} which requires this property for all points in the compact set $K\subset \E$ is asking for more than necessary. In both cases a minimizer may still exist.

Fortunately, for the first of the following examples we will discover later in Chapter \ref{crit point section} that (at least for actions $S$ in the subclass $\H_0^+\subset\H_0$ defined at the beginning of Chapter \ref{crit point section}) both Theorem \ref{comp thm} and our criteria in fact fail for a reason, and that the above possibilities (i) and (ii) are not the case: Proposition \ref{limit cycle prop} will show that for these actions the points in question do not have local minimizers and that a minimizer does not exist. For the second example we will have a partial result of that kind. These insights are an important contribution to our theory because they indicate why the conditions of our criteria are necessary, and they suggest that they are not unnecessarily strong.

These first two examples have in common that there is a loop consisting of one or more flowlines that can be traversed at no cost. Such loops are bound to lead to problems since they allow for infinitely long curves with zero action, thus making it hard to control the curve lengths of a minimizing sequence.

\begin{figure}[p]
\centering
\includegraphics[width=6.14cm]{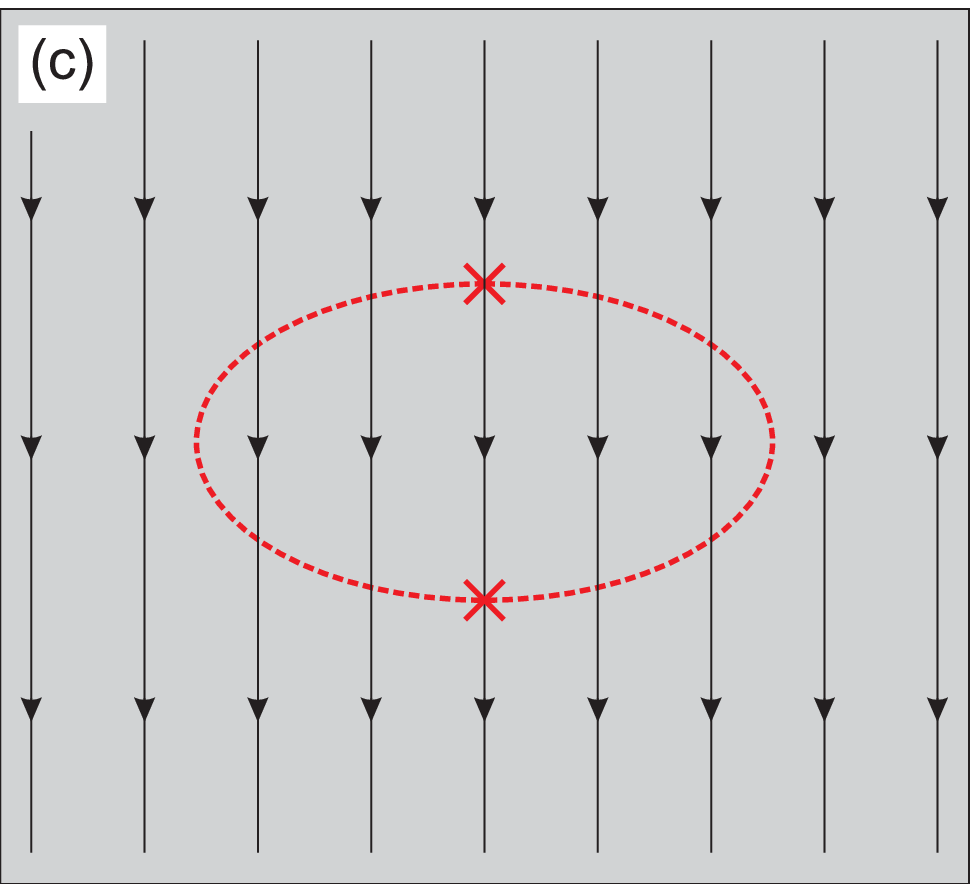} \\[.3cm]
\includegraphics[width=6.14cm]{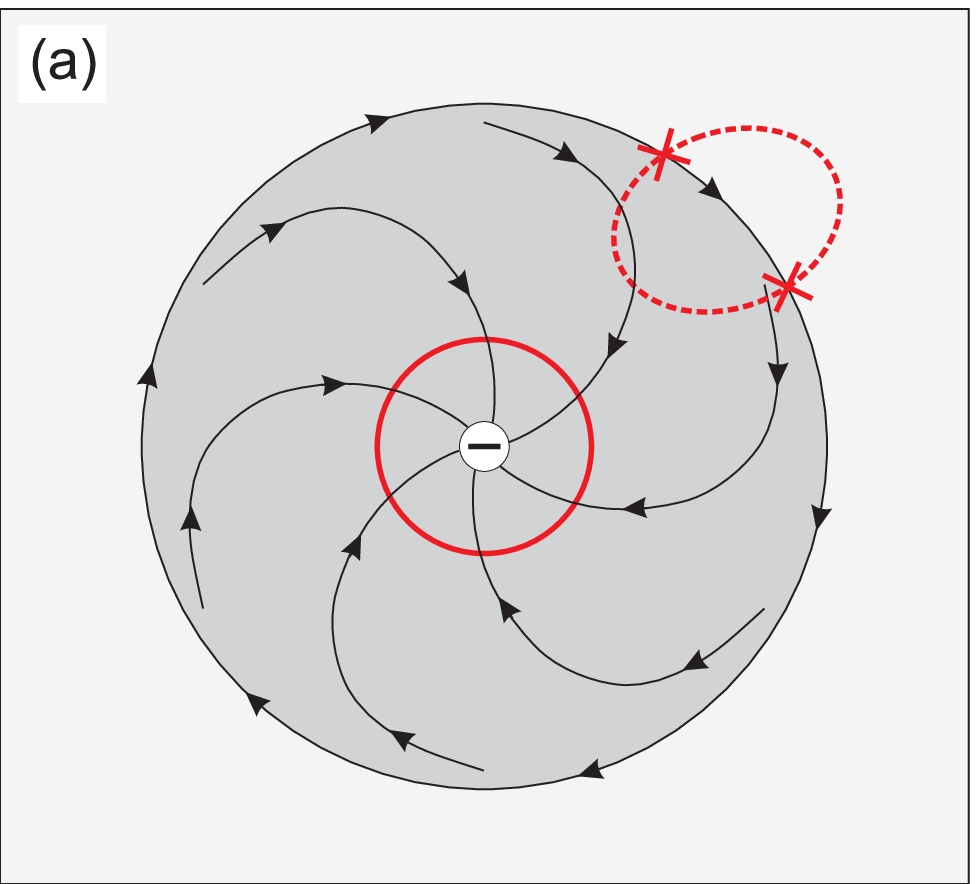} \hspace{.1cm}
\includegraphics[width=6.14cm]{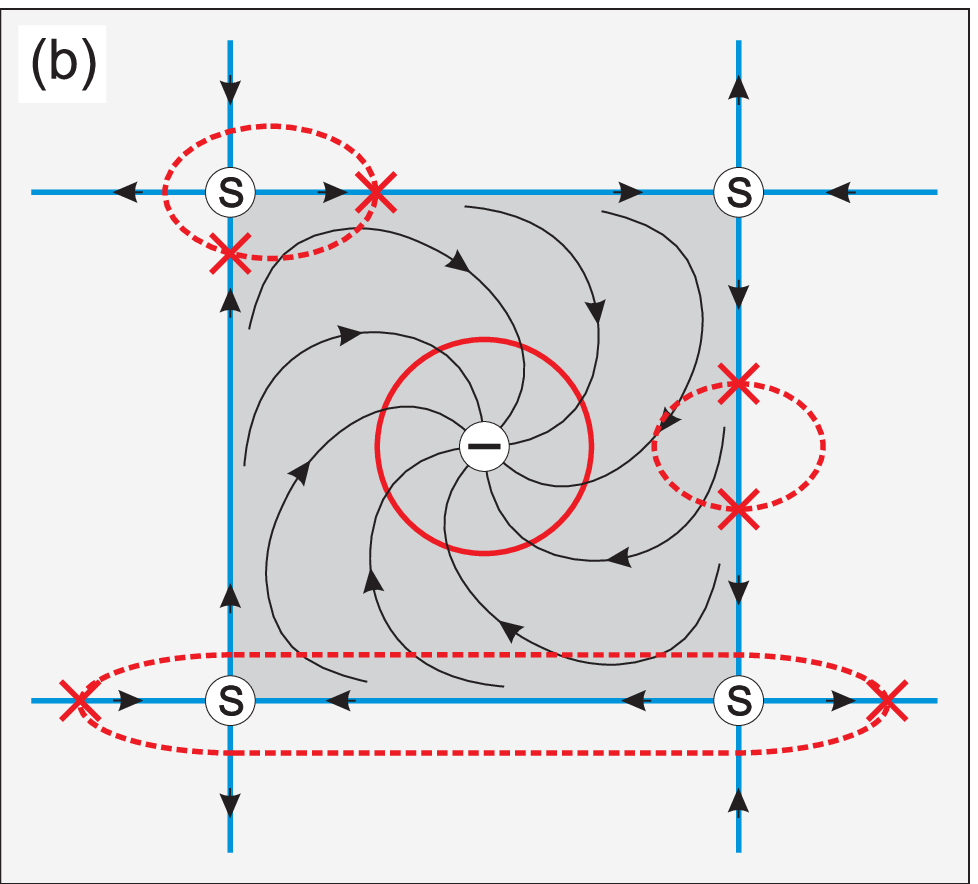} \\[-.2cm]
\caption{\label{fig counterexamples}
\small Three systems to which our criteria cannot be applied.
}
\end{figure}
\begin{figure}[p]
\centering
\vspace{.15cm}
\includegraphics[width=6.14cm]{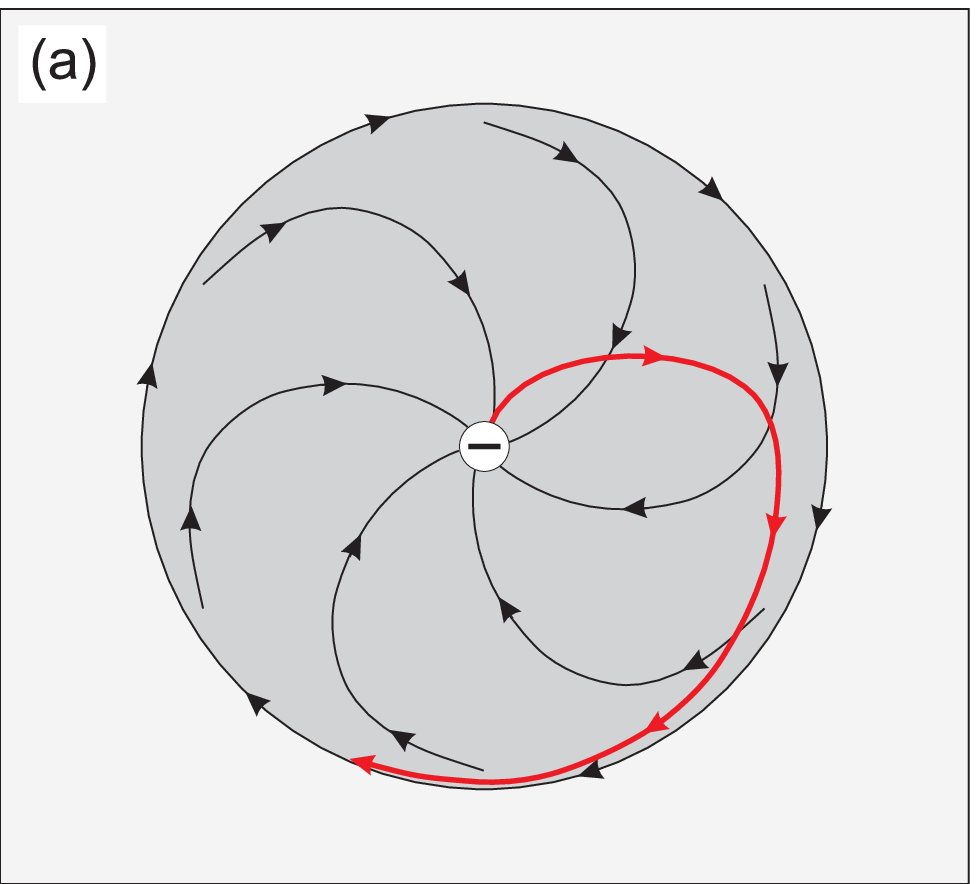} \hspace{.1cm}
\includegraphics[width=6.14cm]{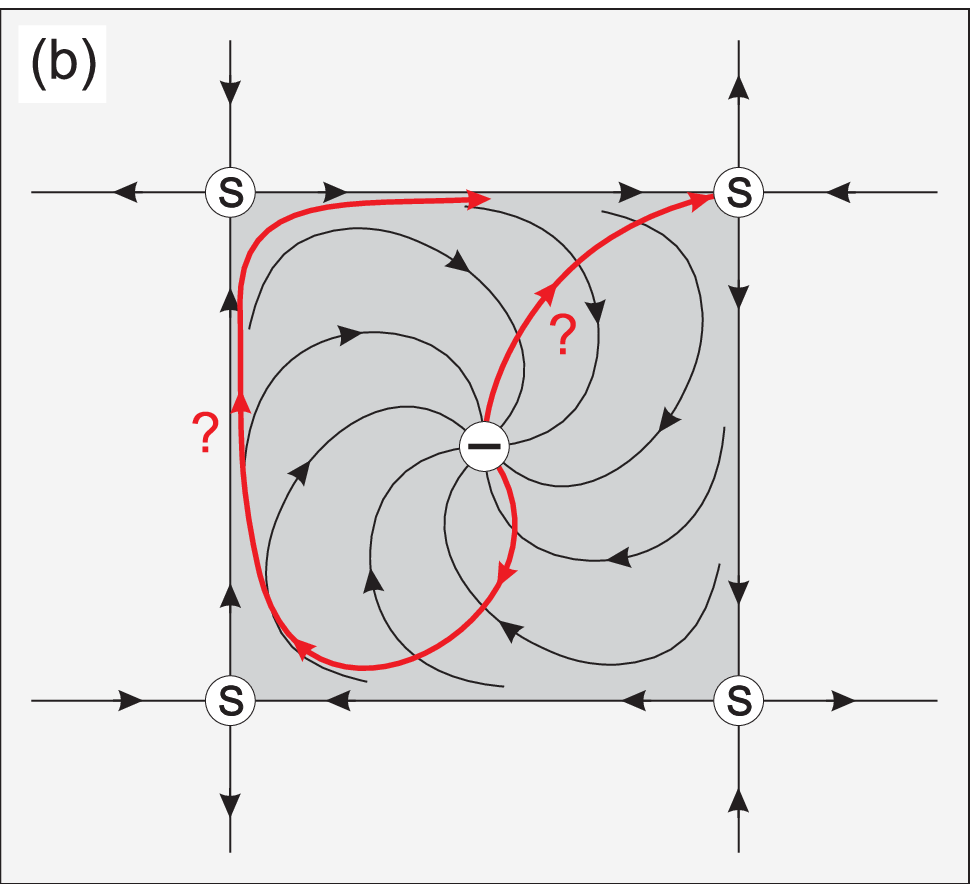} \\[-.2cm]
\caption{\label{fig counterexamples curves}
\small The (generalized) minimum action curves for two of these cases.
}
\end{figure}

\paragraph{Limit cycles.}
Fig.~\ref{fig counterexamples} (a) shows a system consisting of a limit cycle which encloses the basin of attraction of a stable equilibrium point. We are interested in a curve of minimal action that leads from the attractor to the limit cycle, and so the vector field outside of the limit cycle is irrelevant~to~us.

All the points in the basin of attraction can again be treated by Corollary~\ref{blm 3} and Proposition \ref{blm prop 2} (i), but (independently of the drift vector field outside of the limit cycle) our criteria will fail to show that the points on the limit cycle itself have local minimizers: Proposition \ref{blm prop 1} would require us to find an admissible manifold that crosses the limit cycle, but this is impossible.

Indeed, any closed loop $M$ that may be a candidate for an admissible manifold crossing the limit cycle (such as the red dashed line in Fig.~\ref{fig counterexamples} (a)) would have to intersect the limit cycle at least twice (it is not allowed to be tangent to the limit cycle by Definition \ref{admissible def}~(iv)), or put differently, the limit cycle would have to intersect $M$ at least twice. But this would mean that the flowline on the limit cycle enters the interior of $M$ at one place and exits it at another (at the two red crosses), contradicting of Definition \ref{admissible def}~(iv). This observation is proven rigorously in Corollary \ref{no limit cycle corollary} of Part~\ref{part proofs}.

In Section \ref{nonex sec} we will prove that all this happens for a reason: Proposition~\ref{limit cycle prop} says that for actions $S\in\H_0^+$, points on limit cycles never have (weak or strong) local minimizers, and that no minimizer from the attractor (in fact from any point in the basin of attraction) to the limit cycle exists. Instead, the cheapest way to approach the limit cycle is to circle around infinitely in the direction of the flow, see Fig.~\ref{fig counterexamples curves}~(a); this however is not a curve in $\tG$ and is thus not considered a valid minimizer in our, present framework.

\paragraph{Closed chains of flowlines.} The next example in Fig.~\ref{fig counterexamples} (b) is similar in character: Again we have a closed curve that can be traversed at no cost, only that this time it consists of four flowlines that lead from saddle point to saddle point, and we are looking for a curve of minimal action that leads from the attractor to this loop.
As before, our criteria fail to show that any of the points on the loop has local minimizers: Both Proposition \ref{blm prop 1} and \ref{blm prop 2} (ii) would require us to find an admissible manifold crossing the loop, but for the same reasons as in the previous example this can easily be seen to be impossible.

This time however, the issue can at present not be resolved entirely.
Corollary \ref{first hit corollary} in Section \ref{nonex sec} only allows us to conclude for actions $S\in\H_0^+$ that \textit{if} a minimizer exists then it will reach the loop at one of the saddle points. Further work would be necessary to prove that such a solution indeed exists, and to decide if it is more advantageous to rather approach the loop by circling around infinitely in the direction of the flow, see Fig.~\ref{fig counterexamples curves}~(b).

At least Lemma \ref{four saddle failed proof} explains why our criteria are insufficient for showing that those points on the loop with non-zero drift have local minimizers: The proofs of these criteria work by proving the stronger requirements of Remark \ref{blm remark} (ii), and for actions $S\in\H_0^+$ those are not fulfilled.

\paragraph{Non-contracting state space.} The examples of Sections \ref{subsub two basins} and \ref{subsub three basins} had in common that the state space was contracting in the sense that there exists a bounded region which every flowline eventually leads into as $t\to\infty$. This last example, a constant vector field $b(x):\equiv b_0\neq0$ illustrated in Fig. \ref{fig counterexamples}~(c), discusses what can happen if that is not the case.

For reasons similar to the ones in the previous two examples we fail to find even a single admissible manifold, and so we cannot apply Proposition~\ref{blm prop 1}. However, at least in the simple case of the geometric action for an SDE with non-vanishing constant drift and with additive noise it is not difficult to adjust the technique of this paper and to show that every point has strong local minimizers: At the beginning of Section \ref{sec flowline tracing functions} we will show how in this case one can effectively use the non-compact admissible manifold $M=\{b_0\}^\perp$.

It may be possible to extend the results of this paper to cover also cases like this one in more generality: One could drop the assumption that admissible manifolds need to be compact and instead list all the entities that need to be bounded on them, leading to a more technical definition of admissible manifolds. This however would go beyond the scope of our work at this point.

\section{Properties of Minimum Action Curves} \label{crit point section}

Let us begin by defining the subclass $\H_0^+\subset\H_0$ of geometric actions to which most results in this chapter apply. Observe that this class includes the large deviation geometric actions in Example~\hyperlink{example LDT}{1}.

\begin{definition}
We define $\H_0^+\subset\H_0$ as the class of all Hamiltonian geometric actions that are induced by a Hamiltonian that fulfills the Assumptions \tHaa, \tHc, and the following stronger smoothness assumption:\\[.2cm]
\begin{tabular}{rl}
\!\hypertarget{H2'}{(H2')} & \hspace{-.2cm}\begin{minipage}[t]{11.3cm}The derivatives $\Hz$, $\Ht$, $\Hzt=(\Htz)^T$, $\Htt$ and $\Hztt$ exist and are continuous in $(x,\te)$.\end{minipage}
\end{tabular}
\end{definition}
\noindent Note that for $S\in\H_0^+$ we cannot guarantee that \textit{every} Hamiltonian that induces $S$ will fulfill \tHbb. Also recall that by Lemma \ref{critical point lemma}~(i), for these actions a point $x\in\E$ is critical if and only if $b(x)=0$.\\[.2cm]
\indent The goal of this chapter is to study some properties of geometric actions and their minimizers. Our main results (for simplicity stated for the case $\E=D$) are summarized below. While the first result applies to general geo\-metric actions, the last three only hold for actions $S\in\H_0^+$ with a corresponding natural drift $b$. 
\begin{itemize}
\item The only points that a curve $\gamma\in\tG$ with $S(\gamma)<\infty$ can pass in infinite length are those at which every drift of $S$ vanishes.
\item If $L$ is a limit cycle of $b$ and if $A_1\subset D\setminus L$ then the minimization problem $P(A_1,L)$ does not have a solution. We give a quantitative explanation why curves rather like to approach $L$ by circling around infinitely in the direction of the flow.
\item Points on limit cycles of $b$ do not have local minimizers.
\item Minimum action curves leading from one attractor of $b$ to another reach and leave the separatrix between the two basins of attraction at critical points (see Fig.~\ref{fig transition through crit points}).
\end{itemize}

\subsection{Points that are Passed in Infinite Length}

To prepare for Corollary \ref{first hit corollary}, we need to understand which points can be passed in infinite length without accumulating infinite action. Here we find that such points must be roots of any drift~$b$. 
%The intuitive reason for this is that for any other point an infinitely long part of the curve would move against the direction of the drift and thus accumulate infinite action.
A refined statement relating the length of a curve to its action is given by Lemma \ref{key estimate} in Part~\ref{part proofs}.

\begin{lemma} \label{inf length at crit points}
Let $S\in\G$, let $\gamma\in\tG$ with $S(\gamma)<\infty$, and let $x$ be a point~on~$\gamma$ that is passed in infinite length. Then for every drift~$b$ of~$S$ we have $b(x)=0$.
\end{lemma}
\begin{proof}
Suppose that $b_0:=b(x)\neq0$. Let $\eps>0$ be so small that $\bar B_\eps(x)\subset D$,
\[ c:=\min_{w\in\bar B_\eps(x)}|b(w)|>0 \qquad\text{and}\qquad\min_{w\in\bar B_\eps(x)}\Skp{\hat b_0}{\widehat{b(w)}}\geq\tfrac12, \]
where we use the notation $\hat v:=\frac{v}{|v|}$ for $\forall v\in\Rn\setminus\{0\}$, and let $\cA:=\cA(\bar B_\eps(x))$. By passing on to a small segment of $\gamma$ around $x$, it is enough to consider the case $\gamma\in\tGx$, and we may assume that $\gamma\subset\bar B_\eps(x)$. We will obtain a contradiction by showing that $S(\gamma)=\infty$.

To do so, let $\varphi\in\Cx$ be a parameterization of $\gamma$, and define for\linebreak $\forall a\in(0,\frac12)$ the sets $I_a:=[0,\frac12-a]\cup[\frac12+a,1]$ and $I_a^-:=\{\alpha\in I_a\,|\,\varphi'(\alpha)\neq0\}$ and the number $L_a:=\int_{I_a}|\varphi'|\,d\alpha$. Then
\begin{align*}
  \int_{I_a^-}|\varphi'|\big|\widehat{b(\varphi)}-\widehat{\varphi'}\big|\,d\alpha
   &\geq \int_{I_a^-} |\varphi'|\Skp{\hat b_0}{\widehat{b(\varphi)}-\widehat{\varphi'}}\,d\alpha 
   \geq \int_{I_a}\big(\tfrac12|\varphi'| - \Skp{\hat b_0}{\varphi'}\big)\,d\alpha \\
   &=    \tfrac12 L_a - \Skp{\hat b_0}{\big[\varphi(\tfrac12-a)-\varphi(0)\big]+\big[\varphi(1)-\varphi(\tfrac12+a)\big]} \\
   &\geq \tfrac12 L_a - 4\eps,
\end{align*}
which is positive for small $a$ since $\lim_{a\searrow0}L_a=\length(\gamma)=\infty$. By \eqref{drift lower bound} and the Cauchy-Schwarz inequality this implies that
\begin{align*}
  S(\gamma)
    &\geq \int_{I_a^-} \ell(\varphi,\varphi')\,d\alpha
    \geq \cA \int_{I_a^-} \big(|b(\varphi)||\varphi'|-\Skp{b(\varphi)}{\varphi'}\big) \,d\alpha\nonumber \\
    &= \frac{\cA }{2}\int_{I_a^-}|b(\varphi)||\varphi'|\big|\widehat{b(\varphi)}-\widehat{\varphi'}\big|^2\,d\alpha
    \geq \frac{\cA c}{2} \int_{I_a^-}|\varphi'|\big|\widehat{b(\varphi)}-\widehat{\varphi'}\big|^2\,d\alpha \nonumber \\
    &\geq \frac{\cA c}{2} \cdot\frac{\big(\int_{I_a^-}|\varphi'|\big|\widehat{b(\varphi)}-\widehat{\varphi'}\big|\,d\alpha\big)^2}
{\int_{I_a^-}|\varphi'|\,d\alpha}
    \geq \frac{\cA c\big(\tfrac12 L_a-4\eps\big)^2}{2L_a},
\end{align*}
and letting $a\searrow0$ shows that $S(\gamma)=\infty$.
\end{proof}

\subsection{The Advantage of Going With the Flow} \label{going with the flow sec}
The next lemma says that the drift $b$ is the only candidate for a direction into which one can move at no cost, and that for actions $S\in\H_0$ one can indeed follow the the natural drift flowlines at no cost. Note that the latter is obvious for the geometric action given by~\eqref{local geo action0}.

\begin{lemma} \label{flow lemma}
(i) Let $S\in\G$, let $b$ be a drift of $S$, and let $x\in D$ and $y\in\Rn\setminus\{0\}$. If $\ell(x,y)=0$ then either $b(x)=0$ or $y=cb(x)$ for some $c>0$.\\[.1cm]
(ii) Let $S\in\H_0$, let $b$ be a natural drift, and let $x\in D$ and $y\in\Rn$. If $b(x)=0$ or $y=cb(x)$ for some $c\geq0$ then $\ell(x,y)=0$. \\[.1cm]
%, and we have $\lambda(x,y)=\tfrac1c$.
(iii) If $S\in\H_0$ and $\gamma\in\tG$ is a flowline of a natural drift then $S(\gamma)=0$.
\end{lemma}
\begin{proof}
(i) If $\ell(x,y)=0$ then \eqref{drift lower bound} implies that either $b(x)=0$ or $y=cb(x)$ for some $c\geq0$. Since $y\neq0$, we must have $c>0$. \\[.2cm]
(ii) If $0=b(x)=\Ht(x,0)$ then $x$ is a critical point by Lemma \ref{critical point lemma}~(i), so that $\ell(x,y)=0$ for $\forall y\in\Rn$. If $b(x)\neq0$ and $y=cb(x)=c\Ht(x,0)$ for some $c>0$ then $(\that,\lambda)=(0,\frac1c)$ solves \eqref{vte eq},
so that $\that(x,y)=0$ and thus $\ell(x,y)=\skp{\that(x,y)}{y}=0$ by \eqref{general geometric local action}. If $c=0$ then $y=0$, and so we have $\ell(x,y)=0$ again. \\[.2cm]
(iii) Given any parameterization $\varphi\in\tC$ of $\gamma$, we have $\varphi'=cb(\varphi)$ a.e.~on $[0,1]$ for some function $c(\alpha)\geq0$, and so part (ii) implies that $\ell(\varphi,\varphi')=0$ a.e.~on $[0,1]$, i.e.~$S(\gamma)=0$.
\end{proof}

Now suppose that $S\in\H_0$. The next lemma says that if the end of a given curve does not follow the natural drift flowlines (so that its action is positive) then we may reduce its action by bending it slightly into the direction of the drift.
This is less obvious than it seems at first since the sheared curves given by \eqref{bent function} may also be longer, and so a precise calculation is necessary to show that the benefits from the change in direction outweigh the potential increase in length.

\begin{lemma} \label{going with the flow lemma}
Let $S\in\H_0^+$, and let $b$ be a natural drift of $S$ obtained from a Hamiltonian that fulfills the Assumption \tHbb. Let $\gamma\in\Gamma$, let $x$ be its end point, and let $\varphi\in\AC{0,1}$ be its arclength parameterization. Suppose that $b(x)\neq0$, and that
\begin{equation} \label{dont follow flow}
\exists\tau>0\,\ \text{$\exists\!$ arbitrarily large $\alpha\in[0,1)\colon$}
\ \,\varphi(\alpha)\notin\psi\big(x,(-\tau,0]\big).
\end{equation}
Then for sufficiently large $\alpha_0\in[0,1)$ the family of curves $\gamma_\eps\in\Gamma$ given by
\begin{equation} \label{bent function}
\varphi_\eps(\alpha):=
\begin{cases}
\varphi(\alpha) & \text{if }\alpha\in[0,\alpha_0],\\
\varphi(\alpha)+\eps(\alpha-\alpha_0)b(\varphi(\alpha)) &\text{if }\alpha\in[\alpha_0,1],
\end{cases}
\end{equation}
defined for small $\eps\geq0$, fulfills $\partial_\eps S(\gamma_\eps)|_{\eps=0}<0$.
\end{lemma}
\begin{proof}
See Appendix \ref{going with the flow proof}.
\end{proof}

\subsection{Some Results on the Non-Existence of Minimizers} \label{nonex sec}

Lemma \ref{going with the flow lemma} has many useful consequences. The first one is that under certain conditions on $A_2$, 
%is that if $S\in\H_0^+$ and if $A_2\subset\E^\circ$ is flow-invariant under the natural drift then
any solution of $\probA$ must first reach $A_2$ at a critical point, since otherwise we could use Lemma~\ref{going with the flow lemma} to construct a curve with a lower action.
In particular, (under these conditions) this means that if $A_2$ does not contain any critical points then no minimizer can exist.

\begin{corollary} \label{first hit corollary}
Let $S\in\H_0^+$. Let $A_2\subset\E$ be closed in $D$, let $A_1\subset\E\setminus A_2$, 
and suppose that the minimization problem $\probA$ has a weak solution $\gamma^\star\in\tGA$.
Denoting by $\hat x$ its first hitting point of $A_2$, let us also assume that $\hat x\in\E^\circ$ and that the flow $\psi$ of some natural drift $b$ of $S$ fulfills
\begin{equation} \label{xA flow}
 \psi\big(\hat x,(-\tau,\tau)\big)\subset A_2 \text{\qquad for some $\tau>0$.}
\end{equation}
(In particular, these conditions on $\hat x$ are fulfilled if $A_2\subset\E^\circ$ and if $A_2$ is flow-invariant under $b$.)
Then $\hat x$ is a critical point.
\end{corollary}
\begin{proof}
We may assume that $\hat x$ is the end point of $\gamma^\star$ (otherwise we may instead consider the minimizer obtained by cutting off the segment after~$\hat x$). Also, because of Remark \ref{H0 remark}, \eqref{xA flow} is in fact fulfilled for the flow of \textit{any} natural drift of $S$, and thus we may assume that $b$ is constructed from a Hamiltonian that fulfills Assumption \tHbb.

Suppose that $b(\hat x)\neq0$. Then since $S(\gamma^\star)<\infty$ by the remark following Assumption \hE, Lemma \ref{inf length at crit points} says that $\gamma^\star$ cannot pass $\hat x$ in infinite length, and thus we can write $\gamma^\star=\gamma^1+\gamma^2$, where $\gamma^2$ is a rectifiable curve ending in $\hat x$ such that $\gamma^2\subset\E^\circ$ and $\length(\gamma^2)>0$. Now consider the family of curves $\gamma_\eps$ constructed from $\gamma=\gamma^2$ as in Lemma \ref{going with the flow lemma}.
The condition \eqref{dont follow flow} is fulfilled since $\gamma^2$ does not visit  $\psi(\hat x,(-\tau,0])\subset A_2$ prior to $\hat x$, and so we have $\partial_\eps S(\gamma_\eps)|_{\eps=0}<0$, which implies that $S(\gamma_\eps)\leq S(\gamma^2)-c\eps$ for some $c>0$ and all sufficiently small $\eps\geq0$. Now defining $x_\eps:=\psi(\hat x,\eps(1-\alpha_0))$, which by \eqref{xA flow} is in $A_2$ for $\eps\in[0,\tau)$, we have
\begin{align*}
x_\eps &= \psi(\hat x,0)+\eps(1-\alpha_0)\dot\psi(\hat x,0)+o(\eps) \\
&=\hat x+\eps(1-\alpha_0)b(\hat x)+o(\eps)\\
&= \varphi_\eps(1) + o(\eps),
\end{align*}
i.e.~the straight line $\bar\gamma_\eps$ from $\varphi_\eps(1)$ (that is the end point of~$\gamma_\eps$) to $x_\eps\in A_2$ has a length and thus by Lemma~\ref{lower semi lemma}~(ii) also an action of order $o(\eps)$. Finally, for sufficiently small $\eps>0$ we have $\gamma_\eps,\bar\gamma_\eps\subset\E^\circ$ and thus $\tilde\gamma^\star:=\gamma^1+\gamma_\eps+\bar\gamma_\eps\in\tGA$, and the above estimates show that
\pb
\begin{align*}
 S(\tilde\gamma^\star)
   &= S(\gamma^1)+S(\gamma_\eps)+S(\bar\gamma_\eps)
   \leq S(\gamma^1)+S(\gamma^2)-c\eps+o(\eps) \\*
   &= S(\gamma^\star)-c\eps+o(\eps)
    < S(\gamma^\star)
\end{align*}
for small $\eps>0$, contradicting the minimizing property of $\gamma^\star$.
\end{proof}

Two examples of flow-invariant sets $A_2$ to which we can apply Corollary~\ref{first hit corollary} are limit cycles and closed chains of flowlines, as shown in Fig.~\ref{fig counterexamples}~(a) and~(b), which leads us to the results that were discussed in Section \ref{subsub counterexamples}.

\begin{proposition} \label{limit cycle prop}
Let $S\in\H_0^+$, let $b$ be a natural drift, and let $L\subset\E^\circ$ be a limit cycle of $b$, i.e.
\[ \exists x\in L\,\ \exists T>0\colon\quad b(x)\neq0,\ \ L=\psi(x,[0,T)) \text{\ \ and\ \ } \psi(x,T)=x. \]
(i) If $A_1\subset\E\setminus L$ and $A_2\subset L$ then the minimization problem $\probA$ does not have any solutions.\\[.1cm]
(ii) Points $x\in L$ do not have local minimizers.
\end{proposition}

\begin{proof}
(i) First suppose that $A_2=L$. If $P(A_1,L)$ had a solution $\gamma^\star$ then according to Corollary~\ref{first hit corollary} its first hitting point of $L$ would be a critical point. But there are no critical points on $L$, so $P(A_1,L)$ cannot have a solution.

Now let $A_2\subset L$, and suppose that $\probA$ had a solution $\gamma^\star$. Then we obtain a contradiction by showing that $\gamma^\star$ is also a solution of $P(A_1,L)$, which was just proven not to exist. 
Indeed, if there were a curve $\gamma_1\in\tilde\Gamma_{A_1}^L$ with $S(\gamma_1)<S(\gamma^\star)$ then the curve $\gamma_2\in\tGA$, constructed by attaching to $\gamma_1$ a piece of $L$ leading from the end point of $\gamma_1$ to some point on $A_2$ in the direction of the flow, would by Lemma \ref{flow lemma} (iii) have the same action, $S(\gamma_2)=S(\gamma_1)<S(\gamma^\star)$, contradicting the minimizing property of~$\gamma^\star$.\\[.2cm]
(ii) Suppose that some point $x\in L$ had weak local minimizers. Then there would be an $r>0$ such that $\bar B_r(x)\subset\E$ and that for $\forall x_1,x_2\in\bar B_r(x)$ the minimization problem $\prob$ has a weak solution $\gamma^\star$. In particular, we could choose $x_1\in\bar B_r(x)\setminus L$ and $x_2:=x\in L$. But part (i) says that for this choice $\prob$ does not have a solution.
\end{proof}

\begin{remark}
The proof of Proposition \ref{limit cycle prop} (i) via Lemma \ref{going with the flow lemma}, which argues that every curve leading to $L$ can be improved by bending its end in the natural drift direction, indicates why curves like to approach $L$ by circling around infinitely in the direction of the flow (see Fig.~\ref{fig counterexamples curves}~(a)). Using the tools of this paper, proving the existence of a ``minimizing spiral'' is not difficult and will be subject to a future publication.
\end{remark}

The next result explains why our techniques are insufficient to prove that the points on the chain of flowlines in Fig.~\ref{fig counterexamples}~(b) have local minimizers: They were designed to show the stronger property of Remark (ii), which in this example does not hold for actions $S\in\H_0^+$.

\begin{lemma} \label{four saddle failed proof}
Let $S\in\H_0^+$, and suppose that the natural drift flowlines are as in Fig.~\ref{fig counterexamples} (b). Let $A_2$ be the set consisting of the four flowlines connecting the critical points (including their end points), and suppose that $A_2\subset\E^\circ$. Then any non-critical point $x\in A_2$ does not fulfill the property of Remark~\ref{blm remark}~(ii).
\end{lemma}
\begin{proof}
Let $x\in A_2$ with $b(x)\neq0$, and let $\eta>0$ be so small that $\bar B_{2\eta}(x)$ does not contain any critical point. If the property in Remark \ref{blm remark} (ii) were true then there would be an $r\in(0,\eta]$ such that $\bar B_r(x)\subset\E$ and that for $\forall x_1,x_2\in\bar B_r(x)$, $\prob$ has a solution $\gamma^\star$ with $\length(\gamma^\star)\leq\eta$ and thus $\gamma^\star\subset\bar B_{r+\eta}(x)$. In particular, we can pick $x_1\in\bar B_r(x)\setminus A_2$ and $x_2:=x\in A_2$.
As in part (i) we could then show that the corresponding solution $\gamma^\star$ of $\prob$ is also a solution of $P(x_1,A_2)$, and by Corollary~\ref{first hit corollary} $\gamma^\star$ would first hit $A_2$ at a critical point. But this is not possible since $\gamma^\star\subset\bar B_{r+\eta}(x)\subset\bar B_{2\eta}(x)$.
\end{proof}

\subsection{How to Move From One Attractor to Another}
\label{sec move from attr to attr}

Still assuming that $S\in\H_0^+$ and that $b$ is a corresponding natural drift, as another consequence of Corollary~\ref{first hit corollary} we will learn how minimum action curves cross the separatrix as they move from one attractor of $b$ to another, as illustrated in Fig.~\ref{fig transition through crit points}.
Clearly, the point at which the curve \textit{leaves} the separatrix and enters the second basin of attraction must have zero drift. Indeed, after leaving the separatrix, the curve can at no cost follow a flowline of $b$ into the second attractor, and that flowline can only touch the separatrix at a point where~$b$ vanishes.

It is however not that obvious that also the \textit{first} hitting point of the separatrix must have zero drift. Consider for example the geometric action given by \eqref{local geo action0}, where the flowline diagram of $b$ is as in Fig.~\ref{SDE transition} or Fig.~\ref{fig transition through crit points}, and where $|b|$ is very small along a channel that leads from the first attractor to a point on the separatrix far away from any critical point. Curves can then follow that channel at very little cost, and it seems unclear whether it is then advantageous to go the long way towards a critical point in order to cross the separatrix.

\begin{figure}[t]
\centering
\vspace{.15cm}
\includegraphics[width=12.5cm]{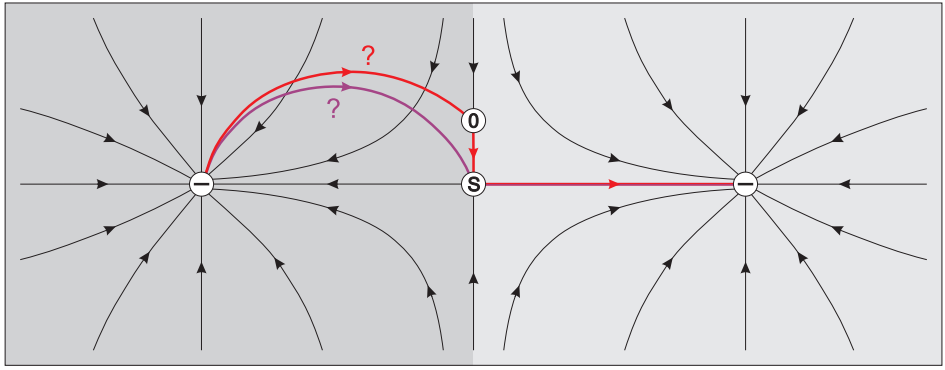} \\[-.2cm]
\caption{\label{fig transition through crit points}
\small Minimum action curves reach and leave the separatrix between two basins of attraction at critical points. However, the first and last hitting points do not need to coincide, as illustrated in this example with an additional degenerate equilibrium point on the separatrix.
}
\end{figure}

The answer to this question is given in Theorem \ref{main theorem}. Note that in contrast to the previous chapter, here we do not make any assumptions on the eigenvalues of $\nabla b$ at the attractors or at the saddle point.
\pb

\begin{theorem} \label{main theorem}
Let $S\in\H_0^+$, let $b$ be a natural drift, let $x_1,x_2\in D$ be two distinct attractors of $b$, let the open sets $B_1,B_2\subset D$ denote their basins of attraction, let $X:=\partial B_1\cap\partial B_2\cap D$ denote their separatrix, and assume that $X\subset\E^\circ$. Let $A_1,A_2\subset\E$ be such that $A_1\subset B_1$ and $x_2\in A_2\subset B_2$.

If the minimization problem $\probA$ has a weak solution $\gamma^\star\subset B_1\cup B_2\cup X$ then its first and last hitting point of $X$ are critical points.
\end{theorem}
\begin{proof}
Let us denote the first and the last hitting points of $X$ by $z_1:=\varphi^\star(\alpha_1)$ and $z_2:=\varphi^\star(\alpha_2)$, where $\varphi^\star\in\tC$ is a parameterization of $\gamma^\star\in\tGA$ and
\begin{align*}
\alpha_1 := \min\hspace{2.2pt}\!\big\{\alpha\in[0,1]\,\big|\,\varphi^\star(\alpha)\in X\big\}\ \in(0,1),&\\*
\alpha_2 := \max\!\big\{\alpha\in[0,1]\,\big|\,\varphi^\star(\alpha)\in X\big\}\ \in(0,1).&
\end{align*}
\textit{First hitting point:}
$X$ is closed in $D$ by definition, we have $X\subset\E^\circ$ by assumption, and $X=\bar B_1\cap\bar B_2\cap D$ is flow-invariant since $\bar B_1\cap D$ and $\bar B_2\cap D$ are. Therefore, to conclude that $z_1$ is a critical point it is by Corollary~\ref{first hit corollary} enough to show that the curve given by $\varphi^\star|_{[0,\alpha_1]}$ is a weak solution of the minimization problem $P(A_1,X)$.

To do so, assume that there were a curve $\gamma_1\in\tilde\Gamma_{A_1}^X$ with $S(\gamma_1)<S(\varphi^\star|_{[0,\alpha_1]})\leq S(\gamma^\star)$. One could then obtain  a contradiction by constructing a curve in $\tGA$ with an action less than $S(\gamma^\star)$, as follows: First follow $\gamma_1$ from $A_1$ to $X$, then move from the endpoint of $\gamma_1$ into $B_2$ along a line segment $\gamma_2$ so short that $S(\gamma_1)+S(\gamma_2)<S(\gamma^\star)$ (using Assumption \hE\ and Lemma \ref{lower semi lemma} (ii)), and finally follow the drift~$b$ into $x_2\in A_2$ at no additional cost (using Lemma \ref{flow lemma}~(iii)).\vspace{.2cm}

\noindent\textit{Last hitting point:}
To make the arguments at the beginning of this section rigorous, first we argue that $s:=S(\varphi^\star|_{[\alpha_2,1]})=0$. Indeed, if $s$ were positive then in contradiction to the minimizing property of $\gamma^\star$ we could construct a curve in $\tGA$ with an action less than $S(\gamma^\star)$, as follows: First move along the curve segment given by $\varphi^\star|_{[0,\alpha_2+\delta]}$, where $\delta>0$ is chosen so small that $S(\varphi^\star|_{[\alpha_2,\alpha_2+\delta]})<s$ and thus $S(\varphi^\star|_{[0,\alpha_2+\delta]})<S(\gamma^\star)$;
since $\varphi^\star(\alpha_2+\delta)\in B_2$ by definition of $\alpha_2$, we can then follow the drift from $\varphi^\star(\alpha_2+\delta)$ into $x_2\in A_2$ at no additional cost.

This shows that $s=0$, and we can conclude that $\ell(\varphi^\star,{\varphi^\star}')=0$ a.e.~on $[\alpha_2,1]$. Now if we had $b(z_2)\neq0$ and thus $b(\varphi^\star)\neq0$ on some interval $[\alpha_2,\tilde\alpha]$, $\tilde\alpha>\alpha_2$, then Lemma \ref{flow lemma} (i) would imply that ${\varphi^\star}'=cb(\varphi^\star)$ a.e.~on $[\alpha_2,\tilde\alpha]$ for some function $c(\alpha)\geq0$, i.e.\ $\varphi^\star$ follows a flowline of $b$ on this interval. Since $\varphi^\star(\tilde\alpha)\in B_2$ and $b(\varphi^\star)\neq0$ on $[\alpha_2,\tilde\alpha]$, we would thus obtain the contradiction $z_2=\varphi^\star(\alpha_2)\in B_2\subset D\setminus X$.
\end{proof}

\section{Conclusions} \label{sec conclusions}
We have defined the class $\G$ of geometric action functionals on the space~$\Gamma$ of rectifiable curves (in fact on a larger space~$\tG$ that contains also infinitely long curves), and we have shown that the Hamiltonian geometric actions that arose in \cite{CPAM,thesis} in the context of large deviation theory belong to~$\G$. We have extended the notion of a drift vector field $b$ from the large deviation geometric action of an SDE \eqref{SDE0} to general actions $S\in\G$, such that any curve with vanishing action must be a flowline of $b$.
%as the only candidate for a direction in which a curve may move without accumulating any action.

We developed conditions under which there exists a curve $\gamma^\star$ with
\[ S(\gamma^\star)=\inf_{\gamma\in\GA}S(\gamma), \]
i.e.~a solution to the problem of minimizing some given action $S\in\G$ over all curves $\gamma$ leading from the set $A_1$ to the set $A_2$. The curve $\gamma^\star$ is called a strong solution if it has finite length, and it is called a weak solution if it passes certain critical points in infinite length.
Using a compactness argument, we reduced this existence problem to a local property (``a point $x$ has local minimizers''), and we listed several criteria (whose proofs are the content of Parts~\ref{part proofs}-\ref{part superprop}) with which one can check this property for a given point~$x$, provided that the flowline diagram of an underlying drift is well-understood.

We then demonstrated in various examples how these criteria are oftentimes sufficient to show that every point in the state space has local minimizers. We also included some examples in which our criteria are insufficient, and we obtained some results that explain why. In particular, in one example we proved that no minimizer $\gamma^\star$ exists.

Finally, we showed various properties of geometric actions and their minimizers. Our main result here was that for certain actions, minimum action curves leading from one attractor of the drift to another reach and leave the separatrix between the two basins of attraction at a point with zero drift. In particular, this result applies to maximum likelihood transition curves in large deviation theory.

\paragraph{Future Work, Open Problems.} 
In a short follow-up paper the author will further investigate the drift $b$ in Fig.~\ref{fig counterexamples curves}~(a) and prove the existence of a ``minimizing spiral'' leading from the attractor to the limit cycle.
In the case of the drift in Fig.~\ref{fig counterexamples curves}~(b) a minimizer will exist, too, but it is not clear whether it will be in the form of a curve $\gamma\in\tG$ that ends in one of the saddle points, or again in the form of a minimizing spiral. To answer this question, one will need new ideas to decide whether the points on the chain of flowlines have local minimizers.

Another interesting open question is whether it is possible to extend the criterion for \emph{strong} local minimizers in Proposition~\ref{blm prop 2}~(ii) also to dimensions $n>2$. 
While it would certainly suffice to extend Lemma \ref{superprop}~(vi)-(vii) correspondingly, after several failed attempts the author now believes that Lemma \ref{superprop}~(vi) is false in higher dimensions, and so a change in strategy may be necessary. One possible alternative approach may be to omit the line \eqref{omittable line} in the proof of Proposition~\ref{blm prop 2} and instead use a generalized version of Lemma \ref{key estimate} that directly applies to our function~$F$; in this way one would need to control the gradients $\nf_i$ only where $F=f_i$.

\pagebreak
\begin{appendices}
  \section{Proofs of some Lemmas}
\label{Appendix Part I}

\subsection{Proof of Lemma \ref{min prop0b}} \label{weak convergence appendix}
\begin{proof}
Let $(\gamma_n)_{n\in\N}\subset\tGx$ be given with the properties stated, and let $s_0:=\liminf_{n\to\infty}S(\gamma_n)$. In a first step, let us pass on to a subsequence (which we again denote by $(\gamma_n)_{n\in\N}$), such that $\lim_{n\to\infty}S(\gamma_n)=s_0$ (we will only need this property for the proof of Lemma \ref{lower semi lemma 2} (ii)).
Let $(\tilde\varphi_n)_{n\in\N}\subset\Cx$ be a corresponding sequence of parameterizations.

To facilitate the proof of Proposition \ref{blm prop 2} in Section \ref{subsec proof prop 4}, which will build on the construction of the present proof, let us rewrite our assumption \eqref{length away from x cond} more generally as
\begin{equation} \label{length estimate alternative}
\forall n\in\N\ \forall u>0\colon\ \int_{\gamma_n}\One_{F(z)>u}\,|dz|\leq\eta(u),
\end{equation}
where $F(w):=|w-x|$ for $\forall w\in D$. We point out that the only properties of $F$ that we will use are that (i) $F$ is continuous on $D$, and (ii) $\exists c>0$ $\forall w\in K\colon\ F(w)\geq c|w-x|$.

To begin, we first pick for $\forall n\in\N$ a value $\alpha_{\min}^n\in[0,1]$ such that $F(\tilde\varphi_n(\alpha_{\min}^n))=\min_{\alpha\in[0,1]}F(\tilde\varphi_n(\alpha))$.
Since $\tilde\varphi_n\subset K$ for $\forall n\in\N$, we may (by passing on to a subsequence if necessary) assume that $\lim_{n\to\infty}\tilde\varphi_n(\alpha_{\min}^n)$ exists.
Next we define for $\forall k\in\N_0$
\begin{align*}
 d_k^-  &:=\tfrac12-2^{-(k+1)},  \hspace{-.4cm}  &d_k^+&:=\tfrac12+2^{-(k+1)}, \\*
 Q_k^-  &:=[d_k^-,d_{k+1}^-],    \hspace{-.4cm}  &Q_k^+&:=[d_{k+1}^+,d_k^+],  \\
 Q_k^\pm&:=Q_k^-\cup Q_k^+,      \hspace{-.4cm}  &J_k  &:=\textstyle\bigcup_{i=0}^kQ_i^\pm
                                                         =[0,d_{k+1}^-]\cup[d_{k+1}^+,1],
\end{align*}
we choose a strictly decreasing sequence $(u_k)_{k\in\N_0}\subset(0,\infty)$ such that
\begin{equation} \label{u0 prop}
  u_0 \geq \max\!\Big\{\sup_{n\in\N} F(\tilde\varphi_n(0)),\ \sup_{n\in\N} F(\tilde\varphi_n(1))\Big\}
\end{equation}
(this is possible since the right-hand side is bounded by $\max_{w\in K}F(w)$) and that $u_k\searrow0$ as $k\to\infty$, 
and we define for $\forall n\in\N$ and $\forall k\in\N_0$ the compact set
\[ I_{n,k}:=\big\{\alpha\in[0,1]\,\big|\,F\big(\tilde\varphi_n(\alpha)\big)\leq u_k\big\}. \]
Then we define for $\forall n\in\N$ the surjective, weakly increasing function\linebreak $\alpha_n\colon[0,1]\to[0,1]$ as follows: At the points $d_k^-$ and $d_k^+$ we set
\begin{equation} \label{alpha n def}
  \alpha_n(d_k^-) := \begin{cases}\min I_{n,k}    & \text{if $I_{n,k}\neq\varnothing$,}\\
                                    \alpha_{\min}^n & \text{else},\end{cases}  \qquad
  \alpha_n(d_k^+) := \begin{cases}\max I_{n,k}    & \text{if $I_{n,k}\neq\varnothing$,}\\
                                    \alpha_{\min}^n & \text{else},\end{cases}
\end{equation}
for $\forall k\in\N_0$, and we set $\alpha_n(\frac12):=\alpha_{\min}^n$.

Before we define $\alpha_n(s)$ at the remaining points $s\in[0,1]$, observe that $\alpha_n(0)=0$ and $\alpha_n(1)=1$, since \eqref{u0 prop} implies that $0,1\in I_{n,0}$.
Also note that every function $\alpha_n$ as defined so far is non-decreasing since for each fixed $n\in\N$ the sequence of sets $(I_{n,k})_{k\in\N_0}$ is decreasing, and since $\alpha_{\min}^n\in I_{n,k}$ whenever $I_{n,k}\neq\varnothing$ (which implies that $\alpha_n(d_k^-)\leq\alpha_{\min}^n\leq\alpha_n(d_k^+)$\linebreak for $\forall k\in\N_0$).

Finally, observe that for $\forall k\in\N$ and $\forall n\in\N$ we have
\begin{subequations}
\begin{align}
    &\text{either} &&\hspace{-1.8cm}\forall \alpha\in [0,\alpha_n(d_k^-)]\colon\,\ F(\tilde\varphi_n(\alpha))\geq u_k \label{Fphik geq 1-} \\*
    &\text{or}     &&\hspace{-1.8cm}\alpha_n(d_k^-)=0 \label{Fphik geq 2-}
\end{align}
\end{subequations}
(or both), and the same is true with $[0,\alpha_n(d_k^-)]$ replaced by $[\alpha_n(d_k^+),1]$ in \eqref{Fphik geq 1-}, and with \eqref{Fphik geq 2-} replaced by $\alpha_n(d_k^+)=1$.
Indeed, if $\alpha_n(d_k^-)>0$ then for $\forall\alpha\in[0,\alpha_n(d_k^-))$ we have $\alpha\notin I_{n,k}$, i.e.~$F(\tilde\varphi_n(\alpha))>u_k$, which implies \eqref{Fphik geq 1-}. The modified statement is shown analogously.

In either case, the curve segments given by $\tilde\varphi_n|_{[0,\alpha_n(d_k^-)]}$ are rectifiable for $\forall k\in\N$: If \eqref{Fphik geq 1-} holds then this follows from \eqref{length estimate alternative} with $u=\frac{u_k}2$, and if \eqref{Fphik geq 2-} holds then this segment degenerates to a single point. Similarly, the segments given by $\tilde\varphi_n|_{[\alpha_n(d_k^+),1]}$ are rectifiable for $\forall k\in\N$ by the corresponding modified versions of \eqref{Fphik geq 1-}-\eqref{Fphik geq 2-}.

We can thus define $\alpha_n(s)$ at the remaining points $s\in[0,1]$ by requiring that the function $\varphi_n(s):=\tilde\varphi_n(\alpha_n(s))$, restricted to the sets $Q_k^-$ and~$Q_k^+$, $k\in\N_0$, is the arclength parameterization of the curves given by $\tilde\varphi_n|_{[\alpha_n(d_k^-),\alpha_n(d_{k+1}^-)]}$ and $\tilde\varphi_n|_{[\alpha_n(d_{k+1}^+),\alpha_n(d_k^+)]}$, respectively. In particular, on each set $Q_k^-$ and $Q_k^+$, $\varphi_n$ is absolutely continuous and $|\varphi_n'|$ is constant a.e..

By construction, $\varphi_n|_{[0,\frac12]}$ and $\varphi_n|_{[\frac12,1]}$ traverse the curves given by $\tilde\varphi_n|_{[0,\hat\alpha_n]}$ and $\tilde\varphi_n|_{[\check\alpha_n,1]}$, where $\hat\alpha_n:=\lim_{k\to\infty}\alpha_n(d_k^-)$ and $\check\alpha_n=\lim_{k\to\infty}\alpha_n(d_k^-)$ (these limits exist since $(\alpha_n(d_k^-))_{k\in\N_0}$ and $(\alpha_n(d_k^+))_{k\in\N_0}$ are monotone bounded sequences).
Therefore, to see that $\varphi_n$ is in fact a parameterization of the entire curve $\gamma_n$, we need to show that $\tilde\varphi_n$ is constant on $[\hat\alpha_n,\check\alpha_n]$.

Now if (for fixed $n\in\N$) there $\exists k_0\in\N_0\ \forall k\geq k_0\colon\ I_{n,k}=\varnothing$ then we have $\forall k\geq k_0\colon\ \alpha_n(d_k^-)=\alpha_{\min}^n=\alpha_n(d_k^+)$ and thus $\hat\alpha_n=\check\alpha_n$, and we are done. Otherwise we have $\alpha_n(d_k^-)\in I_{n,k}$ for $\forall k\in\N_0$, and thus $F\big(\tilde\varphi_n(\alpha_n(d_k^-))\big) \leq u_k\to0$ as $k\to\infty$. This shows that $F(\tilde\varphi_n(\hat\alpha_n))=0$ and thus $\tilde\varphi_n(\hat\alpha_n)=x$, and similarly one can show that $\tilde\varphi_n(\check\alpha_n)=x$. Because of our assumption that $\gamma_n$ passes the point $x$ at most once we can now use \eqref{pass point once def} to conclude that $\tilde\varphi_n$ is constant on $[\hat\alpha_n,\check\alpha_n]$ also in this case.

This shows that $\varphi_n$ is a parameterization of $\gamma_n$ (and in particular continuous). Furthermore, we have $\varphi_n\in\Cx$. To see this, first note that by construction $\varphi_n$ is absolutely continuous on $[0,\frac12-a]\cup[\frac12+a,1]$ for $\forall a\in(0,\frac12)$. If $\varphi_n(\frac12)\neq x$ then 
$F(\tilde\varphi_n(\alpha_{\min}))=F(\varphi_n(\frac12))>0$, so that for large $k\in\N$ we have $I_{n,k}=\varnothing$ and thus $\alpha_n(d_k^-)=\alpha_n(d_k^+)$ by \eqref{alpha n def}; this in turn implies that $\alpha_n$ and thus $\varphi_n$ is constant on $[d_k^-,d_k^+]$, and thus that $\varphi_n\in\AC{0,1}$.\\[.2cm]
\indent Now let us construct a converging subsequence of $(\varphi_n)_{n\in\N}$. First observe that our definition $\varphi_n=\tilde\varphi_n\circ\alpha_n$ and the monotonicity of $\alpha_n$ translate \eqref{Fphik geq 1-}-\eqref{Fphik geq 2-} into the following: For $\forall k\in\N$ and $\forall n\in\N$ we have
\begin{subequations}
\begin{align}
    &\text{either} &&\hspace{-1.8cm}\forall s\in [0,d_k^-]\colon\,\ F(\varphi_n(s))\geq u_k \label{Fphik geq 1} \\
    &\text{or}     &&\hspace{-1.8cm}\text{$\varphi_n$ is constant on $[0,d_k^-]$} \label{Fphik geq 2}
\end{align}
\end{subequations}
(or both), and the same is true with $[0,d_k^-]$ replaced by $[d_k^+,1]$.

We can now find a subsequence of functions $\varphi_n$ that for $k=1$ either all fulfill \eqref{Fphik geq 1} or that all fulfill \eqref{Fphik geq 2}; we can then find a further subsubsequence such that the same is true for $k=2$, etc., and by a diagonalization argument we can pass on to a subsequence which we again denote by $(\varphi_n)_{n\in\N}$ such that for $\forall k\in\N$ $\exists n_k\in\N$ such that
\begin{equation}\label{Fphik geq B1}
\begin{split}
    \text{either} \hspace{.7cm}&\forall n\geq n_k\,\ \forall s\in [0,d_k^-]\colon\,\ F(\varphi_n(s))\geq u_k  \\
    \text{or}     \hspace{1.3cm}&\forall n\geq n_k\colon\,\ \text{$\varphi_n$ is constant on $[0,d_k^-]$}
\end{split}
\end{equation}
(or both). Finally, by following the same strategy one more time we may also assume that the same is true also with $[0,d_k^-]$ replaced by $[d_k^+,1]$. This property \eqref{Fphik geq B1} is not important to us now, but we will need it in the proof of Proposition~\ref{blm prop 2}.

Now using that for $\forall n\in\N$, $|\varphi_n'|$ is constant a.e.~on the intervals $Q_k^-$ and $Q_k^+$, and using \eqref{Fphik geq 1} and \eqref{Fphik geq 2}, which say that either $|\varphi_n'|$ vanishes a.e.~on $[0,d_{k+1}^-]\supset Q_k^-$ or the indicator function in \eqref{deriv indicator} below takes the value $1$ on $[0,d_{k+1}^-]\supset Q_k^-$, we find for $\forall k\in\N_0$ and almost every $s\in Q_k^-$ that
\begin{align}
|\varphi_n'(s)|
  &=    |Q_k^-|^{-1}\int_{Q_k^-}|\varphi_n'|\,d\alpha 
   =    (2^{-(k+2)})^{-1}\int_{Q_k^-}|\varphi_n'|\One_{F(\varphi_n)\geq u_{k+1}}\,d\alpha \label{deriv indicator}\\
  &\leq 2^{k+2}\int_0^1|\varphi_n'|\One_{F(\varphi_n)>u_{k+2}}\,d\alpha
  \leq 2^{k+2}\eta(u_{k+2}), \nonumber
\end{align}
and analogously one can derive this $n$-independent upper bound also for almost every $s\in Q_k^+$. This shows that for every fixed $k\in\N_0$ we have
\begin{equation} \label{essup finite}
  \sup_{n\in\N}\,\esssup_{s\in J_k}|\varphi_n'(s)|
   = \sup_{0\leq j\leq k}\,\sup_{n\in\N}\,\esssup_{s\in Q_j^\pm}|\varphi_n'(s)|
   \leq\sup_{0\leq j\leq k}2^{j+2}\eta(u_{j+2})
<\infty.
\end{equation}
By Lemma \ref{min prop0} (i) we can therefore extract a subsequence of $(\varphi_n)_{n\in\N}$ that converges uniformly on $J_1$, then extract a further subsubsequence converging uniformly on $J_2$, etc., and using a diagonalization argument we can find a subsequence which for simplicity we will again denote by $(\varphi_n)_{n\in\N}$ that converges uniformly on every $J_k$, and in particular pointwise on $\bigcup_{k=0}^\infty J_k=[0,\frac12)\cup(\frac12,1]$. Since also $\varphi_n(\frac12)=\tilde\varphi_n\big(\alpha_{\min}^n\big)$ converges as $n\to\infty$, $(\varphi_n)_{n\in\N}$ converges in fact pointwise on all of $[0,1]$. Let us denote the limit by $\varphi:[0,1]\to K$.\\[.2cm]
By Lemma \ref{min prop0} (ii) the function $\varphi$ is absolutely continuous on each set $J_k$, which provides us with an almost everywhere defined function ${\varphi}'\colon[0,1]\to\Rn$ which is integrable on each set $J_k$. To see that
\begin{equation} \label{general outer length} \int_0^1|{\varphi}'|\One_{F(\varphi)>u}\,d\alpha\leq\eta(u)\qquad\text{ for $\forall u>0$,}
\end{equation}
we fix $u>0$, and we define for $\forall v>u$ and $\forall q\in\R$ the continuous function $h_v(q):=\min(\max(\frac{q-u}{v-u},0),1)\leq\One_{q>u}$. Applying Lemma \ref{lower semi lemma 2} (i) to the functional $S\in\G$ given by $\ell(x,y):=h_v(F(x))|y|$, we find that for $\forall k\in\N$ we have 
\begin{align*}
 \int_{J_k}h_v(F(\varphi))|{\varphi}'|\,d\alpha
 &\leq\liminf_{n\to\infty} \int_{J_k}h_v(F(\varphi_n))|\varphi_n'|\,d\alpha \\
 &\leq\liminf_{n\to\infty} \int_0^1\One_{F(\varphi_n)>u}|\varphi_n'|\,d\alpha \\
 &=   \liminf_{n\to\infty} \int_{\gamma_n}\One_{F(z)>u}\,|dz| \leq\eta(u)
\end{align*}
by \eqref{length estimate alternative}. Taking the limits $k\to\infty$ and $v\searrow u$ and using monotone convergence now imply \eqref{general outer length}.\\[.2cm]
\indent It remains to show that $\varphi\in\Cx$. To prepare, let us first show that for $\forall k\in\N_0$ we have
\begin{subequations}
\begin{align}
    \text{either}\hspace{.7cm}  &F(\varphi(d_k^-))\leq u_k  \label{Fphi leq 1}\\*
    \text{or}\hspace{1.3cm}     &\text{$\varphi$ is constant on $[d_k^-,\tfrac12]$} \label{Fphi leq 2}
\end{align}
\end{subequations}
(or both), and the same holds with $d_k^-$ replaced by $d_k^+$ in \eqref{Fphi leq 1}, and with $[d_k^-,\tfrac12]$ replaced by $[\tfrac12,d_k^+]$ in \eqref{Fphi leq 2}.

Indeed, if for some fixed $k\in\N_0$ we have $F(\varphi(d_k^-))>u_k$ then for large $n\in\N$ we have $F\big(\tilde\varphi_n(\alpha_n(d_k^-))\big)=F(\varphi_n(d_k^-))>u_k$, i.e.~$\alpha_n(d_k^-)\notin I_{n,k}$ and thus $\alpha_n(d_k^-)=\alpha_{\min}^n=\alpha_n(\tfrac12)$ by \eqref{alpha n def}. The monotonicity of $\alpha_n$ then implies for large $n\in\N$ that $\alpha_n$ and thus $\varphi_n$ are constant on $[d_k^-,\tfrac12]$, and taking the limit $n\to\infty$ implies \eqref{Fphi leq 2}. The modified statements can be shown analogously.

Next, let us show that $\varphi$ is continuous. Since $\varphi$ is even absolutely continuous on every set $J_k$, we only have to show continuity at $s=\tfrac12$, and by symmetry of our construction we only have to show that $\varphi(\tfrac12-)=\varphi(\tfrac12)$.
Now if for some $k\in\N$ \eqref{Fphi leq 2} holds then this is clear, therefore let us assume that \eqref{Fphi leq 1} holds for $\forall k\in\N$. Taking the limit $k\to\infty$ in \eqref{Fphi leq 1} implies that $\liminf_{s\nearrow1/2}F(\varphi(s))=0$.
Thus, if the limit $\lim_{s\nearrow1/2}F(\varphi(s))$ would not exist then there would be a sequence $(s_m)_{m\in\N}\in(0,\frac12)$ with $s_m\nearrow\frac12$ such that for some $u>0$ and $\forall m\in\N$ we have $F(\varphi(s_m))\geq2u$.
Now $F^{-1}([0,u])\cap K$ is compact, so that
\[ \dist\!\Big(F^{-1}([0,u])\cap K,\ F^{-1}([2u,\infty))\Big)>0, \]
and thus the fact that $\varphi(s)$ moves back and forth between these two sets infinitely many times as $s\nearrow\frac12$ would imply that $\int_0^{1/2}|{\varphi}'|\One_{u<F(\varphi)<2u}\,d\alpha$ $=\infty$,
contradicting \eqref{general outer length}. This proves that $\lim_{s\nearrow1/2}F(\varphi(s))=0$, and since by construction $F\circ\varphi$ takes its minimum at $s=\frac12$, we have $F(\varphi(\frac12))=0$. Property (ii) of $F$ now implies that $\lim_{s\nearrow1/2}\varphi(s)=x=\varphi(\frac12)$, concluding the proof of the continuity of $\varphi$.

Finally, to show that $\varphi\in\Cx$, assume that $\varphi(\frac12)\neq x$. Then neither \eqref{Fphi leq 1} nor its modified version can hold for $\forall k\in\N$ (since taking the limit $k\to\infty$ in \eqref{Fphi leq 1} would imply that $F(\varphi(\frac12))=0$ and thus $\varphi(\frac12)=x$), and so $\varphi$ must be constant on some interval $[d_{k_1}^-,d_{k_2}^+]$. Since $\varphi$ is absolutely continuous on every set $J_k$, this implies that $\varphi\in\AC{0,1}$, terminating the proof.
\end{proof}

\subsection{Proof of Lemma \ref{lower semi lemma 2}} \label{lower semi proof app}
\begin{proof}
(i) Denoting by $M>0$ the bound given in \eqref{arzela 2}, it suffices to define a family of functions $\ell^\delta\colon D\times\bar B_M(0)\to[0,\infty)$, $\delta>0$, such that\\[.2cm]
\begin{tabular}[t]{rl}
(a)   & $\forall \delta>0$ $\forall x\in D$ $\forall y\in\bar B_M(0)\colon$ $0\leq\ell^\delta(x,y)\leq\inf_{w\in\bar B_\delta(x)\cap D}\ell(w,y)$, \\[.1cm]
(b)  & $\forall\delta>0$ $\forall x\in D\colon$ $\ell^\delta(x,\cdot\,)$ is convex,\\[.1cm]
(c) & $\forall x_0\in D$ $\forall y_0\in\bar B_M(0)\colon$ $\liminf_{(x,y,\delta)\to(x_0,y_0,0+)}\ell^\delta(x,y)\geq\ell(x_0,y_0)$.
\end{tabular}\\[.2cm]
The proof then follows the lines of \cite[Lemma 5.42]{SW} (where the distance function induced by the norm $||\cdot||_\infty$ is denoted by $d_c$). The only necessary modification of that proof is that because of the property in Definition \ref{tech lemma 1}~(i) we do not have an equivalent to \cite[Lemmas 5.17 and 5.18]{SW}, and so we had to guarantee the uniform absolute continuity of the sequence $(\varphi_n)_{n\in\N}$ by requiring the uniform bound on $|\varphi_n'|$ in \eqref{arzela 2}. That same bound is also the reason why (other than in \cite{SW}) here it suffices to define the functions $\ell^\delta(x,y)$ only for $|y|\leq M$.

To do so, we define for $\forall\delta>0$, $\forall x\in D$ and $\forall y\in\bar B_M(0)$
\begin{align*}
\ell^\delta(x,y) &:= \sup_{(\te,a)\in\Theta_{x,\delta}}\big[\skp{\te}{y}+a\big],\quad\text{where} \\
\Theta_{x,\delta} &:= \Big\{(\te,a)\in\Rn\times\R\ \Big|\ \forall v\in\bar B_M(0)\colon\,\skp{\te}{v}+a\leq\inf_{w\in\bar B_\delta(x)\cap D}\ell(w,v)\Big\}, \nonumber
\end{align*}
i.e.\ $\ell^\delta(x,\cdot\,)$ is the convex hull of the function $v\mapsto\inf_{w\in\bar B_\delta(x)\cap D}\ell(w,v)$ restricted to $v\in\bar B_M(0)$.

(a,b) First observe that $(\te,a)=(0,0)$ fulfills $\skp{\te}{v}+a=0\leq\ell(w,v)$ for every $w$ and $v$, and so we have $(0,0)\in\Theta_{x,\delta}$ and thus $\ell^\delta(x,y)\geq\skp{0}{y}+0=0$. The upper bound in (a) follows right from the definition of $\ell^\delta(x,y)$ and $\Theta_{x,\delta}$. Finally, $\ell^\delta(x,\cdot\,)$ is convex as the supremum over affine functions.

(c) Let $x_0\in D$ and $y_0\in\bar B_M(0)$. If $\ell(x_0,y_0)=0$ then by the lower bound in part (a) there is nothing to prove, therefore let us assume that $\ell(x_0,y_0)>0$. Since $\ell(x_0,\cdot\,)$ is convex, $\exists\te\in\Rn\ \exists a\in\R$ such that
\[
  \ell(x_0,y_0)=\skp{\te}{y_0}+a
  \quad\text{and}\quad
  \forall y\in\Rn\colon\ \ell(x_0,y)\geq\skp{\te}{y}+a.
\]
In particular, for $\forall c\geq0$ we can apply the latter to $y=cy_0$ to find that $c\skp{\te}{y_0}+a\leq c\ell(x_0,y_0)=c(\skp{\te}{y_0}+a)$ and thus $(1-c)a\leq0$. This shows that $a=0$, and so we have
\begin{align} \label{lower plane}
  \ell(x_0,y_0)=\skp{\te}{y_0}
  \quad\text{and}\quad
  \forall y\in\Rn\colon\ \ell(x_0,y)-\skp{\te}{y}\geq0.
\end{align}
Given any $\eps>0$, there thus $\exists\eta>0$ such that
\begin{equation} \label{tilde te in Te}
   \forall w\in\bar B_\eta(x_0)\ \forall v\in\bar B_M(0)\colon\
   \ell(w,v)-\skp{\te}{v}\geq-\eps.
\end{equation}
Now let $(x,y,\delta)\in\bar B_{\eta/2}(x_0)\times\bar B_M(0)\times(0,\tfrac{\eta}{2})$. Since for $\forall w\in\bar B_\delta(x)\cap D$ we have $w\in\bar B_\eta(x_0)$, \eqref{tilde te in Te} implies that $(\te,-\eps)\in\Theta_{x,\delta}$, so that
\begin{align*}
  \ell^\delta(x,y)
    &\geq \skp{\te}{y}-\eps \\
    &=    \skp{\te}{y_0} + \skp{\te}{y-y_0}-\eps \\
    &=    \ell(x_0,y_0)  + \skp{\te}{y-y_0}-\eps \\
    &\geq \ell(x_0,y_0)  - |\te||y-y_0|-\eps.
\end{align*}
by the first statement of \eqref{lower plane}. This shows that
\[
   \liminf_{(x,y,\delta)\to(x_0,y_0,0+)}\ell^\delta(x,y)
   \geq\ell(x_0,y_0)-\eps,
\]
and since $\eps>0$ was arbitrary, the proof of property (c) and thus of Lemma \ref{lower semi lemma 2}~(i) is complete.\\[.2cm]
(ii) Since the convergence is uniform on each set $I_a:=[0,\frac12-a]\cup[\frac12+a,1]$, $a\in(0,\frac12)$, and since \eqref{arzela 2} is fulfilled for the sequences $\big(\varphi_{n_k}|_{I_a}\big)_{k\in\N}$ by \eqref{essup finite}, Lemma \ref{lower semi lemma 2} (i) allows us to estimate the combined action of the two pieces of the function $\varphi|_{I_a}$ by
\[
  \int_{I_a}\ell(\varphi,{\varphi}')\,d\alpha
    =    \hS(\varphi|_{I_a}) 
    \leq \liminf_{k\to\infty}\hS(\varphi_{n_k}|_{I_a})
    \leq  \liminf_{k\to\infty}\hS(\gamma_{n_k})
    =     \liminf_{n\to\infty}\hS(\gamma_n).
\]
In the last step we used that at the beginning of the proof of Lemma \ref{min prop0b} we had made sure that $\lim_{k\to\infty}S(\gamma_{n_k})=\liminf_{n\to\infty}S(\gamma_n)$.
Letting $a\searrow0$ and then using the monotone convergence theorem now imply that
\[
S(\gamma) =    \int_0^1\ell(\varphi,{\varphi}')\,d\alpha
          \leq \liminf_{n\to\infty}\hS(\gamma_n). \qedhere
\]
\end{proof}

%
%(iv) Given $K$, let again $A:=\sup_{x\in K}|b(x)|$, and let $\mu:=\inf_{x\in K}H(x,0)$ and $\cB :=\max\big(\frac{4(A+1)}{m_K},\frac{-\mu}{A+1}\big)>0$, where $m_K$ is given by Assumption 3. Finally, let $x\in K$ and $y\in\Rn\setminus\{0\}$ be given. We find that for every $\te\in\Rn$ with $|\te|>\cB $ we have (for some new $\tilde\te\in\Rn$) that
%\begin{eqnarray*}
%H(x,\te)
% &=& H(x,0) + \Skp{\Ht(x,0)}{\te} + \frac12\Skp{\te}{\Htt(x,\tilde\te)\te} \\
% &\geq& \mu - A|\te| + \frac12 m_K|\te|^2 \\
% &>& \mu -(A+1)|\te| + \frac12 m_K\cB |\te| \\
% &\geq& \mu+(A+1)|\te|\\
% &>& \mu +(A+1)\cB  \\
% &\geq& 0=H\big(x,\that(x,y)\big),
%\end{eqnarray*}
%i.e.~$\that(x,y)\neq\te$. Therefore we must have $|\that(x,y)|\leq \cB $, and the estimate for $\ell(x,y)$ follows from its definition \eqref{general geometric local action}.\\[.2cm]
%

\subsection{Proof of Lemma \ref{critical point lemma}} \label{crit point proof app}
\begin{proof}
(i) If \eqref{Hamilton critical criterion} holds for some $H$ then the function $H(x,\cdot\,)$, which is strictly convex by Assumption \tHc, achieves its minimum value $0$ at the point $\te=0$, implying that $\{\te\in\Rn\,|\,H(x,\te)\leq0\}=\{0\}$ and thus $\ell(x,y)=0$ for $\forall y\in\Rn$.
Conversely, if $\forall y\in\Rn\colon\ \ell(x,y)=0$ and $H$ is any Hamiltonian inducing $S$ then we have $\forall\te\neq0\colon\ H(x,\te)>0$ (for if there were a $\te\neq0$ with $H(x,\te)\leq0$ then we had $\ell(x,y=\te)\geq\skp{\te}{\te}>0$), and so by Assump\-tion~\tHa\ $H(x,\cdot\,)$ achieves its minimum value $0$ at the point $\te=0$, which implies \eqref{Hamilton critical criterion}. \\[.1cm]
(ii) Let $x\in D$. By Assumption \tHa\ we have $H(x,0)\leq0$. If $H(x,0)<0$ then given any $y\neq0$ we have $H(x,\te=\eps y)<0$ for some small $\eps>0$, and thus $\ell(x,y)\geq\skp{y}{\eps y}>0$.
Now assume that $H(x,0)=0$. If $x$ is a critical point then we have $\ell(x,y)$ $=0$ even for $\forall y\in\Rn$. Otherwise by part~(i) we have $y:=\Ht(x,0)\neq0$, and since for $\forall\te\in\Rn$ with $H(x,\te)\leq0$ there $\exists\tilde\te\in\Rn$ such that
\[
  0\geq H(x,\te)=H(x,0)+\skp{\Ht(x,0)}{\te}+\tfrac12\Skp{\te}{\Htt(x,\tilde\te)\te}
   \geq0+\skp{y}{\te}+0
\]
by Assumption \tHc, we find that $\ell(x,y)\leq0$ and thus $\ell(x,y)=0$.
\end{proof}

\subsection{Proof of Lemma \ref{second action rep}} \label{second action rep app}
\begin{proof}
First let us show the existence of a solution of \eqref{vte eq}. If $x$ is a critical point then $(\that,\lambda)=(0,0)$ solves \eqref{vte eq} for $\forall y\in\Rn$ by Lemma \ref{critical point lemma}~(i) (this also shows the first direction of part (ii)). If $x$ is not critical then we have $\Ht(x,\te)\neq0$ whenever $H(x,\te)=0$ (for otherwise $H(x,\cdot\,)$ would take its minimum value $0$ at $\te$, and since the minimizer is unique by Assumption \tHc, Assumption \tHa\ would imply that $\te=0$, i.e.\ $x$ is a critical point by Lemma \ref{critical point lemma}~(i)). Thus, for fixed $y\neq0$, any $\te^\star\in\Rn$ that is a solution the constraint maximization problem \eqref{sup rep 2} (and thus also of \eqref{sup rep}) must solve $\nabla_{\!\te,\mu}\big[\skp{y}{\te}-\mu H(x,\te)\big]=0$  for some $\mu\in\R$, i.e.
\[ y=\mu\Ht(x,\te^\star) \qquad\text{and}\qquad H(x,\te^\star)=0. \]
Clearly, $\mu\neq0$ since $y\neq0$. In fact, $\mu>0$ since otherwise we would have $\skp{\Ht(x,\te^\star)}{y}=|y|^2/\mu<0$ and thus $H(x,\te^\star+\eps y)<0$ for some $\eps>0$,
but then $\skp{y}{\te^\star+\eps y}>\skp{y}{\te^\star}$ would contradict the fact that $\te^\star$ is a maximizer of \eqref{sup rep}.
Therefore $(\that,\lambda):=(\te^\star,\mu^{-1})$ solves \eqref{vte eq}.

Next we will show the uniqueness, and that the representation \eqref{general geometric local action}, which is trivial for $y=0$, holds also for $y\neq0$. Let $x\in D$ and $y\in\Rn\setminus\{0\}$, and let $(\that,\lambda)$ be a solution of \eqref{vte eq}. Since $\lambda=|\Ht(x,\that)|/|y|$, the uniqueness of $(\that,\lambda)$ will follow from the uniqueness of $\that$.

If $\lambda=0$ then \eqref{vte eq} says that $H(x,\cdot\,)$ takes its minimum value $0$ at $\that$, and thus again by Assumptions \tHa\ and \tHc\ we must have $\that=0$ (proving uniqueness). By Lemma \ref{critical point lemma}~(i), \eqref{vte eq} now says that $x$ is a critical point, so \eqref{general geometric local action} returns the correct value $\ell(x,y)=0$. This also shows the reverse direction of part~(ii).

If $\lambda>0$ then for $\forall\te\in L_x:=\{\te\in\Rn\,|\,H(x,\te)\leq0\}$ there $\exists\tilde\te\in\Rn$ such that by \eqref{vte eq} and Assumption \tHc\ we have
\begin{align}
  0\geq H(x,\te)
   &=H(x,\that)+\skp{\Ht(x,\that)}{\te-\that}
     +\tfrac12\Skp{\te-\that}{\Htt(x,\tilde\te)(\te-\that)} \nonumber \\
   &\geq 0+\lambda\skp{y}{\te-\that}+\tfrac12m_{\{x\}}|\te-\that|^2 \nonumber \\[.1cm]
\Rightarrow\hspace{.7cm}
\skp{y}{\that}
   &\geq\skp{y}{\te}+\tfrac12m_{\{x\}}\lambda^{-1}|\te-\that|^2
    \geq\skp{y}{\te}. \label{vte uniqueness est}
\end{align}
Since also $\that\in L_x$, this implies that $\ell(x,y)=\skp{y}{\that}$, i.e.\ \eqref{general geometric local action}.
If $(\that',\lambda')$ is another solution of \eqref{vte eq} then we have $\skp{y}{\that}=\ell(x,y)=\skp{y}{\that'}$, and so setting $\te:=\that'$ in the left inequality in \eqref{vte uniqueness est} implies that $\that=\that'$.

Finally, to show the continuity, suppose that for some $(x,y)\in D\times(\Rn\setminus\{0\})$ there exists a sequence $(x_n,y_n)\to(x,y)$ such that $(\that_n,\lambda_n):=(\that(x_n,y_n),\lambda(x_n,y_n))$ stays bounded away from $(\that(x,y),\lambda(x,y))$. Since $\that_n\in L_{x_n}$ and the sets $L_{x_n}$ are uniformly bounded by what was shown at the beginning of the proof of Lemma \ref{action construction lemma}, the sequence $(\that_n)$ is bounded. Thus, since $\lambda_n=|\Ht(x_n,\that_n)|/|y_n|$, also the sequence $(\lambda_n)$ is bounded, and so there is a converging subsequence $(\that_{n_k},\lambda_{n_k})$. Now letting $k\to\infty$ in the system \eqref{vte eq} for $(x_{n_k},y_{n_k})$ and using the uniqueness shown above, we see that its limit must be $(\that(x,y),\lambda(x,y))$, and we obtain a contradiction.
\end{proof}

\subsection{Proof of Lemma \ref{prop cond equiv} (ii)} \label{prop cond proof app}
\begin{proof}
``$\Leftarrow$'': If \eqref{strong condition} holds then choosing $w=x$ implies that $x$ is a critical point by Definition \ref{critical point def}. \\[.2cm]
%If $x$ is a critical point then applying the stronger version of \eqref{drift lower bound}, i.e.~\eqref{stronger est}, to $K:=\{x\}$ tells us that for $\forall y\in\Rn$ we have
%\[ 0=\ell(x,y)\geq\cA\big(|b(x)||y|-\skp{b(x)}{y} + \mu|y|\big), \]
%where $\cA>0$ and $\mu=-H(x,0)\geq0$. Choosing $y=-b(x)$ implies $0\geq 2|b(x)|^2+\mu|b(x)|$ and thus $b(x)=0$, and choosing any $y\neq0$ then implies that $0=\mu=H(x,0)$. This proves \eqref{Hamilton critical criterion}.
%
``$\Rightarrow$'':
If $x$ is a critical point then it fulfills \eqref{Hamilton critical criterion}, and so by our assumption there $\exists a,\delta,\rho>0$ such that for $\forall w\in K:=\bar B_\rho(x)\subset D$ we have $|H(w,0)|\leq a|w-x|^{2\delta}$ and $|\Ht(w,0)|\leq a|w-x|^{2\delta}$. Because of \eqref{that bounded} the second equation in \eqref{vte eq} implies that $c:=\sup_{w\in K,\,y\in\Rn}|\that(w,y)|<\infty$. Finally, let $m_K>0$ be the constant given by Assumption \tHc, and let $\cC:=(2a(1+c)m_K^{-1})^{1/2}$.

Now let $w\in\bar B_\rho(x)$ and $y\in\Rn$. If $y=0$ then $\ell(x,y)=0$ and there is nothing to prove. Otherwise we abbreviate $\that:=\that(w,y)$, and a Taylor expansion gives us a $\tilde\te\in\Rn$ such that
\begin{align*}
 0 = H(w,\that)
  &= H(w,0) + \Skp{\Ht(w,0)}{\that} + \tfrac12\Skp{\that}{\Htt(w,\tilde\te)\that} \\
  &\geq -a|w-x|^{2\delta}-a|w-x|^{2\delta}\,|\that| + \tfrac12m_K|\that|^2 \\
  &\geq -a(1+c)|w-x|^{2\delta} + \tfrac12m_K|\that|^2 \\[.1cm]
\hspace{-1cm}\Longrightarrow\hspace{1.7cm}|\that|
  &\leq \big(2a(1+c)m_K^{-1}\big)^{\!1/2}\,|w-x|^\delta=\cC |w-x|^\delta.
\hspace{.8cm}
\end{align*}
The estimate \eqref{strong condition} thus follows from \eqref{general geometric local action}.
% and choosing $w=x$ in \eqref{strong condition} implies that $x$ is a critical point.
\end{proof}

\subsection{Proof of Lemma \ref{going with the flow lemma}}
\label{going with the flow proof}

\begin{proof}
For greater transparency, we will first lead the proof for the special case of the local action \eqref{local geo action0}.\\[.2cm]
\textit{SDE case.}
Let $B\subset D$ be a closed ball around $x$ that is so small that $d_1:=\min_{w\in B}|b(w)|$ $>0$, and further define $d_2:=\max_{w\in B}|b(w)|$ and $d_3:=\max_{w\in B}|\nabla b(w)|$.
Let $\tilde\alpha\in[0,1)$ be so large that $\varphi|_{[\tilde\alpha,1]}\subset B$, and define for $\alpha\in[\tilde\alpha,1]$
\[ \eta(\alpha):=\big|\widehat{\varphi'}-\widehat {b (\varphi)}\big|^2 = 2\big(1-\Skp{\widehat{\varphi'}}{\widehat {b (\varphi)}}\big), \]
where we use the notation $\hat w=\frac{w}{|w|}$ for $\forall w\in\Rn\setminus\{0\}$.
Note that $\eta(\alpha)$ is well-defined a.e.~on $[\tilde\alpha,1]$ because $b(\varphi)\neq0$ on $[\tilde\alpha,1]$ (by our choice of $B$ and~$\tilde\alpha$), and because $|{\varphi}'|\equiv\length(\gamma)>0$ a.e.\ on $[0,1]$.

First we claim that there are arbitrarily large values $\alpha_0\in[\tilde\alpha,1)$ such that $\int_{\alpha_0}^1\eta(\alpha)\,d\alpha>0$. Indeed, if this were not true then there would exist an $\alpha_0\in[\tilde\alpha,1)$ such that $\eta=0$ and thus $\widehat{\varphi'}=\widehat{b(\varphi)}$ a.e.~on $[\alpha_0,1]$. But this would mean that on $[\alpha_0,1]$, $\varphi$ traverses a flowline of $b$ that ends in $x$, and so we have $\varphi(\alpha)\in\psi(x,(-\tau,0])$ for $\forall$ sufficiently large $\alpha\in[0,1)$, contradicting \eqref{dont follow flow}.

We pick $\alpha_0<1$ so large that $d_2d_3\length(\gamma)(1-\alpha_0)\leq\frac14d_1^2$ and formally compute
\begin{align}
\partial_\eps S(\gamma_\eps)|_{\eps=0}
 &= \lim_{\eps\to0}\tfrac{1}{\eps}\int_0^1 \big[\ell(\varphi_\eps,\varphi_\eps')-\ell(\varphi,\varphi')\big]\,d\alpha \nonumber \\
 &= \lim_{\eps\to0}\int_{\alpha_0}^1\tfrac{1}{\eps}\big[\ell(\varphi_\eps,\varphi_\eps')-\ell(\varphi,\varphi')\big]\,d\alpha \nonumber\\
 &= \int_{\alpha_0}^1\partial_\eps\ell(\varphi_\eps,\varphi_\eps')|_{\eps=0}\,d\alpha. \label{int limit}
\end{align}
The last step of exchanging limit and integral will be justified rigorously when we treat the general case. Since
\begin{equation*}
\varphi_\eps' = \varphi' + \eps\big(b(\varphi)+(\alpha-\alpha_0)\nabla b(\varphi)\varphi' \big),
\end{equation*}
a.e.~on $[\alpha_0,1]$, the integrand of \eqref{int limit} is
\begin{align*}
\partial_\eps\ell(\varphi_\eps,\varphi_\eps')|_{\eps=0}
 =\,& \partial_\eps\big(|b(\varphi_\eps)|\,|\varphi_\eps'|-\skp{b(\varphi_\eps)}{\varphi_\eps'}\big)\big|_{\eps=0} \\
 =\,& |\varphi'|\Skp{\widehat{b(\varphi)}}{\nabla b(\varphi)(\alpha-\alpha_0)b(\varphi)} \\*
 &  +  |b(\varphi)|\Skp{\widehat{\varphi'}}{b(\varphi)+(\alpha-\alpha_0)\nabla b(\varphi)\varphi'} \\*
 & - \Skp{{\varphi}'}{\nabla b(\varphi)(\alpha-\alpha_0)b(\varphi)}\\*
 & -\Skp{b(\varphi)}{b(\varphi)+(\alpha-\alpha_0)\nabla b(\varphi)\varphi'} \\
 =\,& -|b(\varphi)|^2\big(1-\Skp{\widehat{\varphi'}}{\widehat {b(\varphi)}}\big) \\*
 & + (\alpha-\alpha_0)|b(\varphi)|\,|\varphi'|\Skp{\widehat {b (\varphi)} -\widehat{\varphi'}}{\nabla b(\varphi)\big(\widehat {b(\varphi)} -\widehat{\varphi'}\big)} \\
 \leq\,& -\tfrac12d_1^2\eta(\alpha) + d_2d_3(1-\alpha_0)|\varphi'|\big|\widehat {b(\varphi)} -\widehat{\varphi'}\big|^2 \\
=\,&  \eta(\alpha)\big[{-\tfrac12}d_1^2+d_2d_3(1-\alpha_0)\length(\gamma)\big]\\
\leq\,& -\tfrac14d_1^2\eta(\alpha).
\end{align*}
Plugging this into \eqref{int limit}, we obtain
\[ \partial_\eps S(\gamma_\eps)|_{\eps=0} \leq -\tfrac14d_1^2\int_{\alpha_0}^1\eta(\alpha)\,d\alpha <0. \]
\textit{General case.}
We choose $B$ and $\tilde\alpha$ as before, but now we define $\eta(\alpha):=|\that(\varphi,\varphi')|^2$.
Again, there are arbitrarily large $\alpha_0\in[\tilde\alpha,1)$ with $\int_{\alpha_0}^1\eta(\alpha)\,d\alpha>0$ since $\,\eta(\alpha)=0\ \,\Rightarrow\ \,\that(\varphi,\varphi')=0\ \,\Rightarrow\ \,\Ht(\varphi,0)
=\lambda(\varphi,\varphi')\varphi'\ \,\Rightarrow\ \,\widehat{\varphi'}
=\widehat{\Ht(\varphi,0)}=\widehat{b(\varphi)}$
(the second step followed from the definition \eqref{vte eq} of $\that(x,y)$, in the third step we used that $\Ht(\varphi,0)\neq0$ by our choice of $B$ and $\tilde\alpha$).
By implicit differentiation in \eqref{vte eq}, \cite[Appendix E]{CPAM} shows that for $\forall x\in D$ and $\forall y\neq0$ we have\footnote{In this calculation we consider the gradients $H_x$, $H_\te$, $\nabla_{\!x}\ell$ and $\nabla_{\!y}\ell$ as column vectors.}
\begin{align*}
  \that_x(x,y)^Ty &= -\lambda^{-1}(x,y)\Hz(x,\that(x,y)) \qquad\text{wherever\, $\lambda(x,y)\neq0$,} \\
  \that_y(x,y)^Ty &= 0.
\end{align*}
From \eqref{general geometric local action} we therefore obtain
\begin{eqnarray*}
\nabla_{\!x}\ell(x,y) &=& \that_x^T(x,y)y=-\lambda^{-1}(x,y)\Hz(x,\that(x,y)), \\
\nabla_{\!y}\ell(x,y) &=& \that_y^T(x,y)y+\that(x,y)=\that(x,y)
\end{eqnarray*}
wherever $y\neq0$ and $\lambda(x,y)\neq0$,
and thus, abbreviating $\vte_\eps=\that(\varphi_\eps,\varphi_\eps')$ and $\lambda_\eps=\lambda(\varphi_\eps,\varphi_\eps')$, we have a.e.~on $[\alpha_0,1]$
\begin{align}
\partial_\eps\ell(\varphi_\eps,\varphi_\eps')
 &= -\lambda_\eps^{-1}\Skp{\Hz(\varphi_\eps,\vte_\eps)}{\partial_\eps\varphi_\eps} + \Skp{\vte_\eps}{\partial_\eps\varphi_\eps'} \nonumber\\
 &= -\lambda_\eps^{-1}\Skp{\Hz(\varphi_\eps,\vte_\eps)}{(\alpha-\alpha_0)b(\varphi)} \nonumber\\*
 & \hspace*{.43cm}+ \Skp{\vte_\eps}{b(\varphi)+(\alpha-\alpha_0)\nabla b(\varphi)\varphi'} \label{d eps ell}
\end{align}
Setting $\eps=0$ and abbreviating $\that=\that(\varphi,\varphi')$ and $\lambda=\lambda(\varphi,\varphi')$, we find
\begin{equation} \label{deps ell}
\partial_\eps\ell(\varphi_\eps,\varphi_\eps')\big|_{\eps=0}
 = \Skp{\that}{b(\varphi)}+ (\alpha-\alpha_0)\Big[\Skp{\that}{\nabla b(\varphi)\varphi'} -\lambda^{-1}\Skp{\Hz(\varphi,\that)}{b(\varphi)}\Big].
\end{equation}
To show that the first term is negative, we make a Taylor expansion and find that for some $\tilde\vte$ we have
\begin{align}
0 = H(\varphi,\that) &= H(\varphi,0) + \Skp{\Ht(\varphi,0)}{\that}+\tfrac12\Skp{\that}{\Htt(\varphi,\tilde\vte)\that}\nonumber \\
 &\geq 0+\skp{b(\varphi)}{\that} + \tfrac12m_B|\that|^2 \nonumber\\
\Rightarrow\quad \skp{\that}{b(\varphi)} &\leq-\tfrac12m_B|\that|^2,\label{vte estimate}
\end{align}
where we used Assumptions \tHaa\ and \tHc. To control the second term in \eqref{deps ell}, we make two more Taylor expansions and use the equations $\Hz(x,0)$ $=0$ (a~consequence of Assumption \tHaa) and \eqref{vte eq} to show that
\begin{align*}
\Hz(\varphi,\that) &= \Hz(\varphi,0) + \Hzt(\varphi,0)\that + O(|\that|^2) = 0+\nabla b(\varphi)^T\that + O(|\that|^2),\\*
b(\varphi) &= \Ht(\varphi,0) = \Ht(\varphi,\that) + O(|\that|) = \lambda\varphi' + O(|\that|).
\end{align*}
Note that to bound the first remainder term we had to require the existence of a continuous derivative $\Hztt$, and we also needed a uniform bound on $\varphi$ (which is in $B$) and on $\that$ (which then follows from what was shown at the beginning of the proof of Lemma \ref{action construction lemma}).
The square bracket term in \eqref{deps ell}~is~thus
\begin{align}
[\dots]
 &= \Skp{\that}{\nabla b(\varphi)\varphi'}-\lambda^{-1}\Skp{\nabla b(\varphi)^T\that+O(|\that|^2)}{\lambda\varphi'+O(|\that|)} \nonumber\\
 &= O(|\that|^2), \label{bracket estimate}
\end{align}
where we used that $\lambda^{-1}$ is bounded. (The latter follows from Lemma \ref{second action rep}~(ii) and the continuity of $\lambda$, since
$\varphi$ is in the compact set $B$ which does not contain any critical points, and since $|\varphi'|\equiv\length(\gamma)>0$ a.e..)
Now combining \eqref{deps ell}, \eqref{vte estimate} and \eqref{bracket estimate}, and choosing $\alpha_0$ sufficiently close to $1$, we find that
\[
\partial_\eps\ell(\varphi_\eps,\varphi_\eps')\big|_{\eps=0}
 \leq -\tfrac12m_B|\that|^2 + (1-\alpha_0)\cdot \tilde c|\that|^2 \leq -c|\that|^2 = -c\eta(\alpha)
\]
for some constants $\tilde c,c>0$, and thus $\partial_\eps S(\gamma_\eps)|_{\eps=0}\leq -c\int_{\alpha_0}^1\eta(\alpha)\,d\alpha<0$.

It remains to justify the exchange of limit and integral in \eqref{int limit}. Using the mean value theorem and Lebesgue, this boils down to finding a bound on \eqref{d eps ell} that is uniform in both $\eps>0$ and $\alpha\in[\alpha_0,1]$. But this is a straight forward estimate since $\vte_\eps$ and $\lambda_\eps^{-1}$ are uniformly bounded in $\alpha$ and $\eps$ (for reasons similar to the ones used for $\that$ and $\lambda^{-1}$ above).
\end{proof}

\end{appendices}

\addtocontents{toc}{\newpage}
\newpage
\part{Proofs} \label{part proofs}

\section{Finding Points with Local Minimizers} \label{sec proofs}

\subsection{Proof of Proposition \ref{blm prop 0}} \label{subsec proof prop 2}
The key to the proof of Proposition \ref{blm prop 0} is that the condition $\forall y\neq0\colon\!\!$\linebreak $\ell(x,y)>0$ implies that we can locally estimate $|y|\leq\frac1\mu\ell(x,y)$ for some $\mu>0$, which in turn will provide us with a quick way to locally bound the length of a curve by its action. Since minimizing sequences have bounded actions, their lengths must therefore be bounded as well, and we can apply Proposition \ref{min prop}.

\begin{proof}[Proof of Proposition \ref{blm prop 0}]
We will prove the stronger condition of Remark \ref{blm remark}~(ii).
Let $\eta>0$ be given. Since $\min_{|y|=1}\ell(x,y)>0$, there exists an $\eps>0$ such that $\bar B_\eps(x)\subset D$ and
\[ \mu:=\min_{\substack{w\in \bar B_\eps(x)\\|y|=1}}\ell(w,y)>0. \]
Using Definition \ref{tech lemma 1} (i), this implies that
\begin{equation} \label{easy lower action bound}
 \forall w\in\bar B_\eps(x)\ \forall y\neq0\colon\ \ 
\ell(w,y)=|y|\ell\big(w,\tfrac{y}{|y|}\big)\geq\mu|y|,
\end{equation}
and for $y=0$ this relation is trivial.
Let $\cB =\cB(\bar B_\eps(x))>0$ be the constant given by Lemma \ref{lower semi lemma} (ii), let $\nu:=\min\!\big\{\eps,\frac{\mu\eps}{5\cB},\frac{\eta\mu}{2\cB}\big\}$, and finally use Assumption \hE\ to choose $r\in(0,\tfrac12\eps]$ so small that for $\forall w\in\bar B_r(x)\cap\E$ $\exists\gamma\in\Gamma_x^w$: $\length(\gamma)\leq\nu$.

Now let $x_1,x_2\in\bar B_r(x)\cap\E$. For $i=1,2$ let $\bar\gamma^i\in\Gamma^{x_i}_x$ with $\length(\bar\gamma^i)\leq\nu$ and thus in particular $\bar\gamma^i\subset\bar B_\nu(x)\subset\bar B_\eps(x)$, and let $\bar\gamma:=-\bar\gamma^1+\bar\gamma^2\in\Gx$. Since $\bar\gamma\subset\bar B_\eps(x)$, we can use Lemma \ref{lower semi lemma} (ii) to find that
\begin{equation} \label{rough inf est}
   \inf_{\gamma\in\Gx}\hS(\gamma) \leq\hS(\bar\gamma)\leq\cB\length(\bar\gamma)
   \leq 2\cB\nu.
\end{equation}
Next, let $(\varphi_n)_{n\in\N}\subset\CI$ be a parameterization of a minimizing sequence $(\gamma_n)_{n\in\N}$ of $P(x_1,x_2)$. We claim that
\begin{equation}\label{length bound 1}
\exists n_0\in\N\ \forall n\geq n_0\colon\ \gamma_n\subset\bar B_\eps(x).
\end{equation}
Indeed, if this were not the case then we could find a subsequence $(\varphi_{n_k})_{n\in\N}$ such that $\forall k\in\N\ \exists\alpha\in[0,1]\colon\ |\varphi_{n_k}(\alpha)-x|=\eps$. Letting $\alpha_k:=$\linebreak $\min\!\big\{\alpha\in[0,1]\,\big|\,|\varphi_{n_k}(\alpha)-x|\geq \eps\big\}\in(0,1)$ and applying \eqref{easy lower action bound}, we would then have
\begin{align}
\hS(\gamma_{n_k})
  &\geq \int_0^{\alpha_k}\ell(\varphi_{n_k},\varphi_{n_k}')\,d\alpha \nonumber \\
  &\geq \mu\int_0^{\alpha_k}|\varphi_{n_k}'|\,d\alpha \nonumber \\
  &\geq \mu\bigg|\int_0^{\alpha_k}\varphi_{n_k}'\,d\alpha\bigg| \nonumber \\
  &=    \mu|\varphi_{n_k}(\alpha_k)-\varphi_{n_k}(0)| \nonumber \\
  &=    \mu\big|\big(\varphi_{n_k}(\alpha_k)-x\big)+(x-x_1)\big| \nonumber \\
  &\geq \mu\big(|\varphi_{n_k}(\alpha_k)-x|-|x-x_1|\big) \nonumber \\
  &\geq \mu(\eps-r)\geq\tfrac12\mu\eps. \label{prop lower bound 1}
\end{align}
Taking the limit $k\to\infty$ and using the minimizing property of $(\gamma_n)_{n\in\N}$ and  \eqref{rough inf est}, we would thus find that
$\frac12\mu\eps\leq2\cB\nu$, which contradicts our definition of $\nu$. This proves \eqref{length bound 1}, which allows us for $\forall n\geq n_0$ to apply \eqref{easy lower action bound} on the \textit{entire} curve $\gamma_n$, and so we find that we have
\begin{equation} \label{action length estimate 1}
\hS(\gamma_n) = \int_0^1\ell(\varphi_n,\varphi_n')\,d\alpha \geq \mu\int_0^1|\varphi_n'|\,d\alpha = \mu\cdot\length(\gamma_n)
\end{equation}
for $\forall n\geq n_0$, and thus
\[
  \sup_{n\geq n_0}\length(\gamma_n) \leq \frac{1}{\mu}\sup_{n\geq n_0}\hS(\gamma_n)<\infty. \]
We can now apply Proposition \ref{min prop} and conclude that the problem $\prob$ has a strong minimizer $\gamma^\star\in\Gx$ fulfilling
\begin{align*}
\length(\gamma^\star)
  &\leq \liminf_{n\to\infty}\length(\gamma_n)
   \leq  \frac{1}{\mu}\liminf_{n\to\infty}\hS(\gamma_n) =   \frac{1}{\mu}\inf_{\gamma\in\Gx}\hS(\gamma)
   \q\leq\q \frac{2\cB\nu}{\mu}\q\leq\q\eta,
\end{align*}
where we used \eqref{action length estimate 1}, the minimizing property of $(\gamma_n)_{n\in\N}$, \eqref{rough inf est}, and the definition of $\nu$.
\end{proof}

\subsection{Proof of Lemma \ref{admissible ball example}} \label{subsec proof lemma 1}

To prepare for the proof of Lemma \ref{admissible ball example} we first need to collect some properties of the functions $f_s$ and $f_u$ of Definition \ref{fs fu def}.

\begin{lemma} \label{distance function}
The functions $f_s$ and $f_u$ of Definition \ref{fs fu def} are finite-valued and continuous. Furthermore,\\[.2cm]
\begin{tabular}{rl}
(i)  & $f_s\in C^1(B_s\setminus\{x\})$ and $f_u\in C^1(B_u\setminus\{x\})$; \\[.2cm]
(ii) &
\begin{minipage}[t]{11.38cm}
\vspace{-.87cm}
\begin{subequations}
\begin{eqnarray}
 \hspace{-4.0cm}\forall w\in\qb B_s\setminus\{x\}\colon\ \hspace{1pt}
   \skp{\nf_s(w)}{b(w)}=&\!\!\!-\!&\!\!\!\!|b(w)|, \label{fs decreasing}\\
 \hspace{-4.0cm}\forall w\in B_u\setminus\{x\}\colon\ 
   \skp{\nf_u(w)}{b(w)}=& &\!\!\!\!|b(w)|; \label{fu increasing}
\end{eqnarray}
\end{subequations}
\end{minipage} \\[.8cm]
(iii) & \begin{minipage}[t]{11.38cm}
\vspace{-.87cm}
\begin{subequations}
\begin{eqnarray}
\hspace{-6.12cm}
\forall w\in\qb B_s\colon &f_s(w)\geq|w-x|, \label{fs norm est 1 lower} \\
\hspace{-6.12cm}
\forall w\in B_u\colon &f_u(w)\geq|w-x|; \label{fs norm est 2 lower}
\end{eqnarray}
\end{subequations}
\end{minipage} \\[.8cm]
(iv) & \begin{minipage}[t]{11.38cm}
\vspace{-.87cm}
\begin{subequations}
\begin{eqnarray}
\hspace{-1.53cm}
\forall\!\text{ compact }K\subset B_s\ \hspace{1pt}\exists\cD\geq1\ \forall w\in K\colon
 &f_s(w)& \!\!\!\leq \cD |w-x|, \label{fs norm est 1 upper} \\
\hspace{-1.53cm}
\forall\!\text{ compact }K\subset B_u\ \exists\cD\geq1\ \forall w\in K\colon
 &f_u(w)& \!\!\!\leq \cD |w-x|. \label{fs norm est 2 upper}
\end{eqnarray}
\end{subequations}
\end{minipage}
\end{tabular}
\end{lemma}
\begin{proof}
See Appendix \ref{fs fu appendix}.
\end{proof}

\begin{proof}[Proof of Lemma \ref{admissible ball example}]
Let us assume first that $x$ is an unstable equilibrium point. Let $a>0$ be so small that $\bar B_{2a}(x)\subset B_u$, abbreviate $M:=M_u^a=f_u^{-1}(\{a\})$, and define
\begin{equation}
  f_M(w):=\begin{cases}
    \min\{f_u(w)-a,a\} & \text{if $w\in B_u$,} \\
    a                  & \text{else.}
  \end{cases}
\end{equation}
Then $f_M$ is continuous on $D$. Indeed, $f_u$ is continuous on $B_u$, and for $\forall w\in B_u\setminus\bar B_{2a}(x)$ we have $f_u(w)\geq|w-x|>2a$ by \eqref{fs norm est 2 lower} and thus $f_M(w)=a$. It now remains to show the properties (i)-(iv) of Definition~\ref{admissible def}.\\[.2cm]
(i) $f_M(w)=0\ \Leftrightarrow\ \big(w\in B_u$ and $f_u(w)=a\big)\ \Leftrightarrow\ w\in f_u^{-1}(\{a\})=M$.\\[.1cm]
(ii) $M$ is closed as a level set of the continuous function $f_M$. $M$ is bounded since $M\subset\bar B_a(x)$: Indeed, if $w\in M$ then $|w-x|\leq f_u(w)=a$ by \eqref{fs norm est 2 lower}.\\[.1cm]
(iii) Let $w_0\in M$, i.e.~$f_u(w_0)=a$. In particular, we must have $w_0\neq x$, since $f_u(x)=0$ by definition of $f_u$. Since $B_u$ is open, there exists an $\eps>0$ such that $B_\eps(w_0)\subset B_u\setminus\{x\}$, and thus $f_u$ is $C^1$ on $B_\eps(w_0)$ by Lemma \ref{distance function} (i). Since $f_u$ is continuous, we can also choose $\eps>0$ so small  that $\forall w\in B_\eps(w_0)\colon f_u(w)\in(\frac{a}{2},2a)$, which in particular implies that $f_M=f_u-a$ on $B_\eps(w_0)$, and thus that $f_M$ is $C^1$ on $B_\eps(w_0)$ as well.
Since $w_0\in M$ was arbitrary, this shows that there exists a neighborhood of $M$ on which $f_M$ is $C_1$, with $\nf_M=\nf_u$.\\[.1cm]
(iv) Consequently, we have for $\forall w\in M$ that $\skp{\nf_M(w)}{b(w)}=$\linebreak $\skp{\nf_u(w)}{b(w)}=|b(w)|$ by \eqref{fu increasing}. Since $M\subset B_u\setminus\{x\}$ as seen in part~(iii), we have for $\forall w\in M$ that $b(w)\neq0$ and thus $\skp{\nf_M(w)}{b(w)}>0$.\\[.2cm]
If $x$ is a \textit{stable} equilibrium point then the proof is carried out analogously, except that we replace $f_u$ by $f_s$ and then multiply the definition of $f_M$ by $-1$. In this way, in the proof of (iii) we will find that $\nf_M=-\nf_s$ on $M$, but since in the proof of part (iv) we will now have to use \eqref{fs decreasing} instead of \eqref{fu increasing}, we will still find that $\skp{\nf_M(w)}{b(w)}=-\skp{\nf_s(w)}{b(w)}=+|b(w)|>0$.
\end{proof}

\subsection{Admissible Manifolds}
In preparation for the proofs of Propositions \ref{blm prop 1} and \ref{blm prop 2}, we will now collect some properties of admissible manifolds. Before proceeding, the reader is~advised to review Definition \ref{admissible def} which we will soon use without further reference.

\begin{lemma} \label{limit cycle lemma}
If $M$ is an admissible manifold then
\begin{equation} \label{fM sgn 1}
  \forall x\in M\ \,\forall t\in\R\colon\ \,\sgn\!\big(f_M(\psi(x,t))\big)=\sgn(t).
\end{equation}
In particular, we have $\psi(x,t)\in M$ if and only if $t=0$, which shows that admissible manifolds cannot be crossed by the same flowline more than once.
\end{lemma}

%\begin{remark} \label{limit cycle remark}
%More specifically, we show that for $\forall x\in M$ we have $\forall t>0:$ $f_M\big(\psi(x,t)\big)>0$ and $\forall t<0:$ $f_M\big(\psi(x,t)\big)<0$; after introducing some more notation, this will be rephrased later in Lemma \hyperlink{Lemma14prime}{\ref{limit cycle lemma}'}.
%\end{remark}
\begin{proof}
Let $x\in M$. Clearly, \eqref{fM sgn 1} holds for $t=0$ by Definition \ref{admissible def}~(i). Suppose now that there were a $t>0$ such that $f_M(\psi(x,t))\leq0$. Then
\[ T:=\inf\!\big\{t>0\,\big|\,f_M(\psi(x,t))\leq0\big\} \]
would be well-defined, and since
\[
 \partial_t f_M(\psi(x,t))\big|_{t=0}
   = \Skp{\nf_M(\psi(x,0))}{\dot\psi(x,0)}
   = \skp{\nf_M(x)}{b(x)}
   > 0
\]
by Definition \ref{admissible def}~(iv), we would have $T>0$,
\begin{align}
  f_M(\psi(x,t))>0&\qquad\text{for $\forall t\in(0,T)$} \label{eq 1} \\
\intertext{and $w:=\psi(x,T)\in f_M^{-1}(\{0\})=M$. Since $\psi(x,t)=\psi(w,t-T)$, \eqref{eq 1} can be rewritten as}
  f_M(\psi(w,t))>0&\qquad\text{for $\forall t\in(-T,0)$.} \nonumber
\end{align}
But this would mean that
\begin{align*}
\skp{\nf_M(w)}{b(w)}
    &= \partial_tf_M(\psi(w,t))\big|_{t=0} \\
    &= \lim_{t\searrow0}\tfrac1t\big[\underbrace{f_M(\psi(w,0))}_{=f_M(w)=0}
           -\underbrace{f_M(\psi(w,-t))}_{>0\ \text{for}\ t\in(0,T)}\big]
    \leq0,
\end{align*}
which contradicts property (iv) of Definition \ref{admissible def}. Consequently, we must have $f_M(\psi(x,t))>0$ for $\forall t>0$, and with an analogous argument one can show that $f_M(\psi(x,t))<0$ for $\forall t<0$, concluding the proof of \eqref{fM sgn 1}.

In particular, if a flowline crosses $M$ at some point $x$ then \eqref{fM sgn 1} implies that for $\forall t\neq0$ we have $f_M(\psi(x,t))\neq0$ and thus $\psi(x,t)\notin M$.
\end{proof}

\begin{corollary} \label{no limit cycle corollary}
If $x\in D$ lies on a limit cycle of $b$ then there is no admissible manifold $M$ with $x\in\psi(M,\R)$.
\end{corollary}
\begin{proof}
If $x\in D$ lies on a limit cycle then we have $\psi(x,T)=x$ for some $T>0$. If there existed an admissible manifold $M$, a $w\in M$ and a $t\in\R$ such that $\psi(w,t)=x$ then we would have $\psi(w,T)=\psi(x,T-t)=\psi(x,-t)=w\in M$, which contradicts Lemma \ref{limit cycle lemma}.
\end{proof}

In particular, this shows that we cannot use Proposition \ref{blm prop 1} to prove that a given point on a limit cycle has local minimizers. Proposition \ref{limit cycle prop} (ii) of Section \ref{nonex sec} explains why this had to be the case: For actions $S\in\H_0^+$ points on limit cycles do not have local minimizers.

The next lemma (which is used in the proofs of Corollary \ref{move cor} and Lemma~\ref{superprop}) allows us to deform a given admissible manifold and turn it into a new one. With a smart choice of the function $\beta(x)$ this new manifold can have additional useful properties.

\begin{definition}
For any $\beta\in C^1(D,\R)$ we denote by $\psi_\beta\in C^1(D\times\R,D)$ the flow corresponding to the vector field $\beta\cdot b$.
\end{definition}

\begin{lemma} \label{move M lemma}
Let $\beta\in C^1(D,\R)$. If $M$ is an admissible manifold and $T\in\R$ then also the set $M':=\psi_\beta(M,T)$ is an admissible manifold.
\end{lemma}
\begin{proof}
We will show that the continuous function $f_{M'}(x):=f_M(\psi_\beta(x,-T))$, $x\in D$, has the four properties of Definition \ref{admissible def}.\\[.2cm]
(i) $f_{M'}(x)=0$ $\ \,\Leftrightarrow\ \,$ $f_M(\psi_\beta(x,-T))=0$ $\ \,\Leftrightarrow\ \,$ $\psi_\beta(x,-T)\in M$ $\ \,\Leftrightarrow\ \,$ $x\in$\linebreak$\psi_\beta(M,T)=M'$.\\[.2cm]
(ii) $M'=\psi_\beta(M,T)$ is compact as the continuous image of a compact set.\\[.2cm]
(iii) Denote by $N$ an open neighborhood of $M$ on which $f_M$ is $C^1$. Then $f_{M'}(x)$ is $C^1$ wherever $\psi_\beta(x,-T)\in N$, i.e.~where $x\in\psi_\beta(N,T)=:N'\supset M'$. Since $\psi_\beta(\,\cdot\,,T)$ has a continuous inverse (namely $\psi_\beta(\,\cdot\,,-T)$), $N'$ is an open neighborhood of $M'$.\\[.2cm]
(iv) Suppose that there exists an $x_0\in M'$ such that $\skp{\nf_{M'}(x_0)}{b(x_0)}\leq0$, and let $w:=\psi_\beta(x_0,-T)\in M$.
The functions
\[ f_t(x):=f_M(\psi_\beta(x,-t)), \quad t\in\R,\ x\in D,\  \]
are $C^1$ in $(t,x)$ wherever $\psi_\beta(x,-t)\in N$, and thus in particular where
$x=\psi_\beta(w,t)$. Therefore the function
\[
   g(t):=\Skp{\nf_t(\psi_\beta(w,t))}{\,b(\psi_\beta(w,t))},
   \quad t\in\R,
\]
is well-defined and continuous, and since $f_{0}=f_M$ and $f_{T}=f_{M'}$, it fulfills
\[
   g(0)=\skp{\nf_M(w)}{b(w)}>0
   \qquad\text{and}\qquad
   g(T)=\skp{\nf_{M'}(x_0)}{b(x_0)}\leq0
\]
(the first estimate is property (iv) of the admissible manifold $M$). This shows that $\exists t_0\in(0,T]\colon\ g(t_0)=0$, and abbreviating $v:=\psi_\beta(w,t_0)$, we find that
\begin{align*}
  0 &= \beta(v)g(t_0)
     =  \Skp{\nf_{t_0}(v)}{\beta(v)b(v)}
     =  \partial_\tau f_{t_0}(\psi_\beta(v,\tau))\big|_{\tau=0}  \\
    &= \partial_\tau f_M\big(\psi_\beta(v,\tau-t_0)\big)\big|_{\tau=0}
     =  \partial_\tau f_M(\psi_\beta(w,\tau))\big|_{\tau=0}  \\
    &=  \Skp{\nf_M(w)}{\beta(w)b(w)}
     =  \beta(w)g(0)
\end{align*}
and thus $\beta(w)=0$. In particular, this implies that
\begin{equation} \label{psi w const}
  \psi_\beta(w,t)=w \quad \text{for $\forall t\in\R$,}
\end{equation}
which enables us to compute an explicit formula for the function $h(t):=\nabla\psi_\beta(w,t)$: We have $h(0)=I$ (since $\psi_\beta(x,0)=x$ for $\forall x\in D$) and
\begin{align*}
  \dot h(t)
    &= \nabla_{\!\!x}\dot\psi_\beta(x,t)\big|_{x=w}
     =  \nabla_{\!\!x}\big[(\beta b)(\psi_\beta(x,t))\big]\big|_{x=w} \\
    &= \big(\nabla(\beta b)\big)\big(\underbrace{\psi_\beta(w,t)}_{=w}\big)\underbrace{\nabla\psi_\beta(w,t)}_{=h(t)}
     =  \big(\underbrace{\beta(w)}_{=0}\nabla b(w)+b(w)\otimes\nb(w)\big)h(t),
\end{align*}
and so $\nabla\psi_\beta(w,t)=h(t)=\exp\!\big(b(w)\otimes t\nb(w)\big)$ for $\forall t\in\R$. Again using \eqref{psi w const}, we thus obtain the contradiction
\begin{align*}
  g(t_0)
    &= \skp{\nf_{t_0}(w)}{b(w)} \\
    &= \nf_M\big(\psi_\beta(w,-t_0)\big)\nabla\psi_\beta(w,-t_0)b(w) \\
    &= \nf_M(w)e^{b(w)\otimes (-t_0)\nb(w)}b(w) \\
    &= \skp{\nf_M(w)}{b(w)}e^{-t_0\skp{\nb(w)}{b(w)}} \\
    &> 0.
\qedhere
\end{align*}
\end{proof}
In other words, if one lets the points on $M$ follow the flow $\beta b$ for a fixed amount of time then one obtains a new admissible manifold. As a direct consequence we obtain Corollary \ref{move cor}, which in turn will reduce the proof of Proposition \ref{blm prop 1} to points $x\in M$ only.

\begin{corollary} \label{move cor}
If $x\in\psi(M,\R)$ for some admissible manifold $M$ then there exists another admissible manifold $M'$ such that $x\in M'$.
\end{corollary}
\begin{proof}%[Proof of Corollary \ref{move cor}]
Let $x=\psi(w,T)$ for some $w\in M$ and some $T\in\R$\,. Then $x\in M':=\psi(M,T)$, and by Lemma \ref{move M lemma} (applied to $\beta:\equiv1$) $M'$ is an admissible manifold.
\end{proof}

The following lemma defines two functions $z(x)$ and $t(x)$ on the set $\psi(M,\R)$ (that is the union of all the flowlines of $b$ emanating from $M$). These functions are used extensively throughout the rest of this paper, in particular in the proof of Lemma \ref{superprop} to define a function $\beta$ for use in Lemma~\ref{move M lemma}, and in the proofs of Lemmas~\ref{man2func cor} and~\ref{superprop} to define certain ``flowline tracing functions'' from admissible manifolds.

\begin{lemma} \label{man2func lemma}
Let $M$ be an admissible manifold. Then $\psi(M,\R)$ is open, and there exist two functions $z\in C^1(\psi(M,\R),M)$ and $t\in C^1(\psi(M,\R),\R)$ whose values are the unique ones fulfilling
\begin{equation}\label{t z eq}
  \forall x\in\psi(M,\R)\colon\quad z(x)\in M \quad\text{and}\quad \psi(z(x),t(x))=x.
\end{equation}
Furthermore, we have for $\forall x\in\psi(M,\R)$
\begin{align}
  \nabla\! z(x)\,b(x)&=0,  \label{grad z grad t eq 1}\\
  \skp{\nabla t(x)}{b(x)}&=1,    \label{grad z grad t eq 2}
\end{align}\vspace{-.6cm}
\begin{equation} \label{man move remark eq}
  x\in M\quad\Leftrightarrow\quad t(x)=0 \quad\Leftrightarrow\quad z(x)=x.
\end{equation}
\end{lemma}

\begin{proof}
Let us abbreviate $A:=\psi(M,\R)$.
The existence (but not the smoothness) of two functions $z(x)$ and $t(x)$ fulfilling $\psi(z(x),t(x))=x$ is clear by our choice of their domain $\psi(M,\R)$. To show uniqueness, let $x\in A$, $z_1,z_2\in M$ and $t_1,t_2\in\R$ fulfill $\psi(z_1,t_1)=x=\psi(z_2,t_2)$. Then we have $\psi(z_1,t_1-t_2)=z_2\in M$, and Lemma~\ref{limit cycle lemma} tells us that $t_1-t_2=0$, i.e.~$t_1=t_2$. This in turn implies that $z_2=\psi(z_1,t_1-t_2)=\psi(z_1,0)=z_1$.

To see that the functions $z$ and $t$ are $C^1$ on $A$, let $x\in A$. Let $\eps>0$ be so small that $f_M$ is $C^1$ on $B_\eps(z(x))$. Since $\psi(x,-t(x))=z(x)$, there exists a neighborhood $U$ of $(x,t(x))$ such that
$\forall(w,\tau)\in U\colon\ \psi(w,-\tau)\in B_\eps(z(x))$. In particular, the function $F(w,\tau):=f_M(\psi(w,-\tau))$ is $C^1$ on $U$. Since
\begin{align*}
  F(x,t(x)) &=f_M\big(\psi(x,-t(x))\big)=f_M(z(x))=0 \\[.15cm]
\text{and}\quad\ 
  \partial_\tau F(x,t(x))
     &= -\Skp{\nf_M\big(\psi(x,-t(x))\big)}{\,b\big(\psi(x,-t(x))\big)}\hspace{1cm} \\
     &= -\Skp{\nf_M(z(x))}{b(z(x))} \neq 0
\end{align*}
by Definition \ref{admissible def}~(i) and~(iv), we can apply the Implicit Function Theorem to obtain a $C^1$-function $\tilde t(w)$, defined in a neighborhood $V$ of $x$, such that for $\forall w\in V$ we have $0=F(w,\tilde t(w))=f_M\big(\psi(w,-\tilde t(w))\big)$, i.e.~$\tilde z(w):=\psi(w,-\tilde t(w))\in M$. By definition of $\tilde z$ we have $\psi(\tilde z(w),\tilde t(w))=w$ for $\forall w\in V$, which tells us that (i) $V\subset A$, proving that $A$ is open, and (ii) $\tilde t=t|_V$ and $\tilde z=z|_V$ (because of the uniqueness of the functions $z$ and $t$). Since $\tilde t$ and $\tilde z$ are $C^1$, the latter shows that $t$ and $z$ are $C^1$ on $V$, and thus on all of $A$.

To show \eqref{grad z grad t eq 1} and \eqref{grad z grad t eq 2}, we evolve both sides of \eqref{t z eq} by some small time $\tau$ and find that $\psi\big(z(x),t(x)+\tau\big)=\psi(x,\tau)$, i.e.
\[ z(\psi(x,\tau))=z(x)\qquad\text{and}\qquad t(\psi(x,\tau))=t(x)+\tau. \]
Differentiating with respect to $\tau$ and setting $\tau=0$, we obtain
\begin{align*}
  0 &= \nabla\! z(\psi(x,0))\,\dot\psi(x,0)=\nabla\! z(x)\, b(x) \hspace{1.3cm}\text{and}\hspace{-1.3cm} \\
  1 &= \Skp{\nabla t(\psi(x,0))}{\dot\psi(x,0)}=\skp{\nabla t(x)}{b(x)}.
\end{align*}
It remains to show \eqref{man move remark eq}. If $x\in M$ then the equation $\psi(x,0)=x$ and the uniqueness of the representation \eqref{t z eq} imply that $t(x)=0$. If $t(x)=0$ then by \eqref{t z eq} we have $x=\psi(z(x),0)=z(x)$. Finally, if $z(x)=x$ then $x\in M$ since $z$ takes values in $M$. 
\end{proof}

\noindent With this new notation we can now rephrase Lemma \ref{limit cycle lemma} as follows.

%\noindent\hypertarget{Lemma14prime}{\textbf{Lemma \ref{limit cycle lemma}'.}}
\begin{corollary} \label{improved prime lemma}
Let $M$ be an admissible manifold, and let $t(x)$ be the corresponding function given by Lemma \ref{man2func lemma}. Then we have
\begin{align}
   \forall x\in\psi(M,\R)\ \,\forall t\in\R\colon\ \,
   &\sgn\!\big(f_M(\psi(x,t))\big)=\sgn(t(x)+t), \label{prime lemma eq 1}\\
   \forall x\in\psi(M,\R)\colon\ \,
   &\sgn(f_M(x))=\sgn(t(x)). \label{prime lemma eq 2}
\end{align}
\end{corollary}
\begin{proof}
Using \eqref{t z eq} we can write
\[
   \sgn\!\big(f_M(\psi(x,t))\big)=
   \sgn\!\big(f_M\big(\psi(z(x),t(x)+t)\big)\big),
\]
and since $z(x)\in M$, we can apply Lemma \ref{limit cycle lemma} to obtain \eqref{prime lemma eq 1}. To prove \eqref{prime lemma eq 2}, set $t=0$.
\end{proof}

\subsection{Flowline Tracing Functions} \label{sec flowline tracing functions}

The purpose of this section is to find a replacement for the local bound $\ell(x,y)\geq \mu|y|$ that was used in \eqref{prop lower bound 1} and \eqref{action length estimate 1} to bound the length of a curve in terms of its action.
Without the condition of Proposition \ref{blm prop 0}, our only lower bound on $\ell(x,y)$ is \eqref{drift lower bound}, which vanishes if $y=cb(x)$ for some $c\geq0$. As a result, curves that follow the flowlines of $b$ could be arbitrarily long and have zero action.
We thus need to exclude the possibility that the curve follows the flowlines of $b$ for arbitrarily long distances, for example because these flowlines lead far away from the desired endpoint.

To quantify this idea, consider for example the constant vector field $b(x)\equiv b_0\in\Rn\setminus\{0\}$. In this case, if $\gamma\subset K$ and if the start and end point of $\gamma$ are confined to a ball $\bar B_r(x)$ then we have
\begin{align}
\hS(\gamma)&=\int_0^1\ell(\varphi,\varphi')\,d\alpha
\geq\cA\int_0^1\big(|b_0||\varphi'|-\skp{b_0}{\varphi'}\big)\,d\alpha \nonumber\\
&=\cA\big(|b_0|\length(\gamma)-\skp{b_0}{\varphi(1)-\varphi(0)}\big) \nonumber \\
\Rightarrow\quad\length(\gamma)&=\tfrac{1}{\cA|b_0|}\hS(\gamma)+\Skp{\tfrac{b_0
}{|b_0|}}{\varphi(1)-\varphi(0)}
\leq \tfrac{1}{\cA|b_0|}S(\gamma)+2r, \label{simple key est}
\end{align}
where $\cA=\cA(K)$, and again we have found a bound for the length of $\gamma$ in terms of its action.

For non-constant vector fields $b$ however, things are not that easy. We will have to lay out a non-cartesian coordinate grid that is compatible with this idea, i.e.~one whose ``$b$-coordinate'' increases at unit speed along the flowlines of $b$. The manifold consisting of all the points with vanishing $b$-coordinate can be crossed by the flowlines of $b$ only in one direction, which leads us to the definition of admissible manifolds. The notion of such a coordinate grid is made precise by the following definition.

\begin{definition} \label{tracing def}
A function $f\colon D\to\R$ is said to \underline{trace the flowlines} of the vector field $b\colon D\to\Rn$ between the values $q_1$ and $q_2$ (for two real numbers $q_1<q_2$) if\\[.2cm]
\begin{tabular}[t]{rl}
(i)    & $f$ is continuous on $D$, \\[.1cm]
(ii)   & $f$ is continuously differentiable on $E:=f^{-1}\big((q_1,q_2)\big)$,\\[.1cm]
(iii)  & \begin{minipage}[t]{11.3cm} we have
          \begin{minipage}[t]{10cm}
             either (iii.1)\ $\forall x\in E\colon\ \skp{\nf(x)}{b(x)}=|b(x)|$, \\
             or \hspace{.6cm}(iii.2)\ $\forall x\in E\colon\ \skp{\nf(x)}{b(x)}=-|b(x)|.$
          \end{minipage}
\end{minipage}
\end{tabular}\\
\end{definition}

Property (iii) says that on the region $E$, $f$ increases or decreases at unit speed in the direction of the flow $b$, and thus for $x\in E$, $f(x)$ can be interpreted as the value of the $b$-coordinate of $x$.
Note that if a function $f$ traces the flowlines of $b$ between $q_1$ and $q_2$ and if
$(\tilde q_1,\tilde q_2)\subset(q_1,q_2)$, then $f$ also traces the flowlines of $b$ between $\tilde q_1$ and $\tilde q_2$.

The following lemma, which is used in the proof of Proposition \ref{blm prop 1}, shows how to construct a flowline tracing function from an admissible manifold. A corresponding statement for Proposition \ref{blm prop 2} is given by Lemma \ref{superprop}.
\pb

\begin{lemma} \label{man2func cor}
Let $M$ be an admissible manifold. Then there exists an $\eps>0$ and a function $f\in C(D,\R)$ such that\\[.2cm]
\begin{tabular}{rl}
(i)   & $f^{-1}(\{0\})=M$, \\[.1cm]
(ii)   & \begin{minipage}[t]{11.5cm} $f$ traces the flowlines of $b$ between the values $-\eps$ and $\eps$, \end{minipage} \\[.1cm]
(iii)  & $\bar E$ is a compact subset of $D$, where $E:=f^{-1}\big((-\eps,\eps)\big)$,  \\[.1cm]
(iv)  & $\forall\,x\in\bar E\colon\ \,b(x)\neq0$, and\\[.1cm]
(v) & $\sup_{x\in E}|\nf(x)|<\infty$.
\end{tabular}\\[-.05cm]
\end{lemma}
\begin{proof}
Abbreviate $A:=\psi(M,\R)$, let $z\in C^1(A,M)$ and $t\in C^1(A,\R)$ be the functions given by Lemma \ref{man2func lemma}, and define the function $g\in C^1(A,\R)$ by
\begin{equation} \label{g def}
  g(x):=\int_0^{t(x)}\big|b\big(\psi(z(x),\tau)\big)\big|\,d\tau \qquad\text{for $\forall x\in A$,}
\end{equation}
i.e.\ $|g(x)|$ is the length of the flowline segment between $x$ and $z(x)$.
First note that by Remark~\ref{on manifold b neq0} we have $b(z(x))\neq0$ and thus $b\big(\psi(z(x),\tau)\big)\neq0$ for $\forall\tau\in\R$. This shows that $g$ is $C^1$ and (using \eqref{g def} and \eqref{prime lemma eq 2}) that
\begin{equation} \label{g t M}
  \sgn(g(x))=\sgn(t(x))=\sgn(f_M(x)) \qquad\text{for $\forall x\in A$.}
\end{equation}
Since $A$ is open by Lemma \ref{man2func lemma} and contains the compact set $M$, there $\exists\eps>0$ such that $\bar N_{2\eps}(M)\subset A$.
%Since $\bar N_\delta(M)$ is compact and $b(x)\neq0$ for $\forall x\in \bar N_\delta(M)$ (in fact, for $\forall x\in A$), one can use an appropriate smooth cutoff function to define a function $\beta\in C^1(D,\R)$ such that $\beta(x)=1/|b(x)|$ for $x\in \bar N_\delta(x)$.
%Let $\tilde A:=\psi_\beta(M,\R)\subset A$, and let $t\in C^1(\tilde A,\R)$ and $z\in C^1(\tilde A,M)$ be the functions $t(x)$ and $z(x)$ defined in Lemma \ref{man2func lemma} corresponding to this function $\beta$.
%Now let $K:=\psi_\beta(M,[-\eps,\eps])$ (this is a compact set as the continuous image of the compact set $M\times[-\eps,\eps]$), and let $\eps:=\delta/\max_{y\in K}|\beta(y)b(y)|$ (which is well-defined since $\beta b\neq0$ on $M\subset K$).
Since for $\forall x\in G:=g^{-1}\big((-2\eps,2\eps)\big)$ we have
\begin{align*}
  |x-z(x)|
    &= \big|\psi\big(z(x),t(x)\big)-\psi\big(z(x),0\big)\big| \\
    &= \bigg|\int_0^{t(x)}\dot\psi\big(z(x),\tau\big)\,d\tau\bigg|
     = \bigg|\int_0^{t(x)}b\big(\psi(z(x),\tau)\big)\,d\tau\bigg| \\
    &\leq \bigg|\int_0^{t(x)}\big|b\big(\psi(z(x),\tau)\big)\big|\,d\tau\bigg|
     = |g(x)|<2\eps,
\end{align*}
we have $G\subset \bar N_{2\eps}(M)\subset A$. Finally, we set $D^-:=f_M^{-1}\big((-\infty,0)\big)$ and $D^+:=f_M^{-1}\big((0,\infty)\big)$, and we define the function $f\colon D\to\R$ as
\begin{equation}\label{f g def}
  f(x):=\begin{cases}
            g(x)  & \text{if $x\in G$,} \\
            -2\eps & \text{if $x\in D^-\setminus G$,} \\
            2\eps  & \text{if $x\in D^+\setminus G$.}
         \end{cases}
\end{equation}
Note that $f$ is well-defined since the three cases are defining $f$ on disjoint sets whose union is all of $D$. Indeed, since $f_M^{-1}(\{0\})=M\subset A$, \eqref{g t M} implies
\begin{samepage}
\begin{equation} \label{Dpm}
    f_M^{-1}(\{0\})=g^{-1}(\{0\})
\end{equation}
and thus $D\setminus(D^-\cup D^+)=f_M^{-1}(\{0\})=g^{-1}(\{0\})\subset G$.
It remains to show that $f$ has the desired properties (i)-(v) of Lemma \ref{man2func cor}.
\end{samepage}
%\vspace{.2cm}

\noindent
(i) Using \eqref{f g def}-\eqref{Dpm} we find that $f^{-1}(\{0\})=g^{-1}(\{0\})=f_M^{-1}(\{0\})=M$.\\[.2cm]
(ii) To check that $f$ traces the flowlines of $b$ between the values $-\eps$ and $\eps$, we have to check the three properties of Definition \ref{tracing def}:\\[.2cm]
(ii.1) For any set $B\subset D$ let us temporarily (i.e.\ for this part (ii.1) only) use the notation $\bar B$ to denote its \emph{closure in $D$}. Clearly, $f$ is continuous on each of the three parts of the domain. To see that $f$ is also continuous on the boundaries of these regions, we use that $G$ is open, \eqref{Dpm}, \eqref{g t M}, and that $\bar G\subset g^{-1}\big([-2\eps,2\eps]\big)$ (since $\bar G\subset \bar N_{2\eps}(M)\subset A$), to obtain
\begin{align*}
  \overline{(D^-\setminus G)}\cap\overline{(D^+\setminus G)}\hspace{.23cm}
    &=\hspace{.23cm} \big(\overline{D^-}\cap\overline{D^+}\big)\cap G^c \\
    &\subset\hspace{.23cm} f_M^{-1}\big((-\infty,0]\big)\cap f_M^{-1}\big([0,\infty)\big)\cap G^c \\
    &=\hspace{.23cm} f_M^{-1}(\{0\})\cap g^{-1}\big((-2\eps,2\eps)\big)^c \\[-.15cm]
    &\hspace{-.23cm}\stackrel{\eqref{Dpm}}{=} g^{-1}(\{0\})\cap g^{-1}\big((-2\eps,2\eps)\big)^c
    =\varnothing, \\[.2cm]
  \overline{(D^-\setminus G)}\cap\bar G\hspace{.23cm}
    &=\hspace{.23cm}  \overline{D^-} \cap G^c \cap\bar G \\
    &\subset\hspace{.23cm} f_M^{-1}\big((-\infty,0]\big)\cap g^{-1}\big((-2\eps,2\eps)\big)^c \cap g^{-1}\big([-2\eps,2\eps]\big)  \\
    &=\hspace{.23cm} f_M^{-1}\big((-\infty,0]\big) \cap g^{-1}\big(\{-2\eps,2\eps\}\big) \\[-.15cm]
    &\hspace{-.23cm}\stackrel{\eqref{g t M}}{=} g^{-1}(\{-2\eps\}),\hspace{3.7cm}\text{and similarly} \\[.2cm]
  \overline{(D^+\setminus G)} \cap \bar G\hspace{.23cm}
    &\subset\hspace{.23cm} g^{-1}(\{2\eps\}).
\end{align*}
(ii.2) $f$ is $C^1$ on $G$ since $f|_G=g|_G$ and $g$ is $C^1$. Since $G=g^{-1}\big((-2\eps,2\eps)\big)=f^{-1}\big((-2\eps,2\eps)\big)\supset E$, this shows that $f$ is $C^1$ on $E$.\\[.2cm]
(ii.3) This also shows that for $\forall x\in G\supset E$ we have
\begin{align*}
\nf(x) &= \nabla g(x) \\
    &= \big|b\big(\psi(z(x),t(x))\big)\big|\nabla t(x) \\*
    & \hspace{1.3cm}{}+ \int_0^{t(x)}\Big(\frac{b^T\nabla b}{|b|}\Big)\big(\psi(z(x),\tau)\big)\nabla\psi\big(z(x),\tau\big)\,d\tau\,\cdot\nabla\! z(x),
\end{align*}
so \eqref{t z eq}-\eqref{grad z grad t eq 2} imply that $\skp{\nf(x)}{b(x)}=|b(x)|$. \\[.2cm]
(iii) The continuity of $f$ implies that $\bar E\subset f^{-1}\big([-\eps,\eps]\big) =g^{-1}\big([-\eps,\eps]\big)\subset G\subset\bar N_{2\eps}(M)$. Since $\bar N_{2\eps}(M)$ is compact, this shows that $\bar E$ is a compact subset of $G\subset D$.\\[.2cm]
(iv) This is a consequence of Remark~\ref{on manifold b neq0} since $\bar E\subset G\subset A=\psi(M,\R)$.\\[.2cm]
(v) This follows directly from our proofs of parts (ii.2) and (iii) where we showed that $f$ is $C^1$ on the set $G$ which contains the compact set $\bar E$.
\end{proof}

As we see, we cannot expect to cover all of $D$ with our grid, but only some set $E=f^{-1}((q_1,q_2))$, and so our generalized version of the estimate \eqref{simple key est}, given in Lemma \ref{key estimate}, must be restricted to $E$ as well.
To do so, we need to introduce the continuous function $\hqq$, which is equal to the identity on $[q_1,q_2]$ and constant outside of $[q_1,q_2]$. Two properties are given in Lemma~\ref{h lemma}.

\begin{definition}
For any two real numbers $q_1<q_2$ we define the function $\hqq\colon\R\to[q_1,q_2]$ by
\[ \hqq(a):=\min\!\big(\!\max(a,q_1),q_2\big). \]
\end{definition}

\begin{lemma} \label{h lemma}
For $\forall a_1,a_2\in\R$ we have the estimates
\begin{subequations}
\begin{align}
  \big|\hqq(a_1)-\hqq(a_2)\big|&\leq q_2-q_1,  \label{hqq est 1} \\
  \big|\hqq(a_1)-\hqq(a_2)\big|&\leq|a_1-a_2|. \label{hqq est 2}
\end{align}
\end{subequations}
%(ii) For every real-valued absolutely continuous function $x(\alpha)$, also the function $\hqq(x(\alpha))$ is absolutely continuous, and we have
%\begin{equation} \label{dhqq}
%\partial_\alpha \big[\hqq\big(x(\alpha)\big)\big] = x'(\alpha)\cdot\One_{x(\alpha)\in(q_1,q_2)} \quad\text{a.e..}
%\end{equation}
\end{lemma}
\begin{proof}
The estimate \eqref{hqq est 1} holds because $h_{q_1}^{q_2}$ maps into $[q_1,q_2]$, \eqref{hqq est 2} just says that $h_{q_1}^{q_2}$ is Lipschitz continuous with Lipschitz constant~$1$, which can easily be checked by splitting $\R$ into $(-\infty,q_1]$, $[q_1,q_2]$ and $[q_2,\infty)$.
\end{proof}

\begin{lemma} \label{key estimate}
Let $x_1,x_2\in\E$, $\gamma\in\Gxx$, $q_1<q_2$, let $f\colon D\to\R$ be a function that traces $b$ between the values $q_1$ and $q_2$, let $E:=f^{-1}\big((q_1,q_2)\big)$, and assume that $\bar E$ is a compact subset of $D$. Let $\cA :=\cA (\bar E)$ be the constant given by Definition~\ref{drift vector field}, and assume that $\cG:=\cG(\bar E):=\min_{x\in\bar E}|b(x)|>0$ and $\cH :=\cH (f,q_1,q_2):=\sup_{x\in E}|\nf(x)|<\infty$. Then we have
\begin{equation} \label{key est short}
 \length\!\big(\gamma|_{f^{-1}((q_1,q_2))}\big)
 \leq
 \frac{2\cH ^2}{\cA \cG} S(\gamma)
  + 2\big|\hqq(f(x_1))-\hqq(f(x_2))\big|.
\end{equation}
\end{lemma}

\begin{proof}
Let us abbreviate $L:=\length(\gamma|_E)$ and $\Delta:=\hqq(f(x_2))-\hqq(f(x_1))$.
If $L-|\Delta|\leq0$ then
\[ L-2|\Delta|\leq 2(L-|\Delta|)\leq0\leq\frac{2\cH ^2}{\cA\cG}S(\gamma), \]
so \eqref{key est short} is clear. Therefore let us now assume that $L-|\Delta|>0$ and thus in particular $L>0$. Let $\varphi\in\Cxyt{x_1}{x_2}{1}$ be a parameterization of $\gamma$, and let
\[ Q := \big\{\alpha\in[0,1]\,\big|\,\varphi(\alpha)\in E\text{ and }\varphi'(\alpha)\neq0\big\}. \]
Using \eqref{drift lower bound} and the Cauchy-Schwarz inequality, and using the notation\linebreak $\hat w:=\frac{w}{|w|}$ for $\forall w\in\Rn\setminus\{0\}$, we find that
\begin{align}
  S(\gamma)
    &\geq \int_0^1 \ell(\varphi,\varphi')\One_{\alpha\in Q}\,d\alpha \nonumber \\
    &\geq \cA \int_0^1 \big(|b(\varphi)||\varphi'|-\Skp{b(\varphi)}{\varphi'}\big) \One_{\alpha\in Q}\,d\alpha\nonumber \\
    &= \frac{\cA }{2}\int_0^1|b(\varphi)||\varphi'|\big|\widehat{b(\varphi)}-\widehat{\varphi'}\big|^2\One_{\alpha\in Q}\,d\alpha \nonumber \\
    &\geq \frac{\cA \cG}{2}\int_0^1|\varphi'|\big|\widehat{b(\varphi)}-\widehat{\varphi'}\big|^2\One_{\alpha\in Q}\,d\alpha \nonumber \\
    &\geq \frac{\cA \cG}{2}\cdot\frac{\big(\int_0^1|\varphi'|\big|\widehat{b(\varphi)}-\widehat{\varphi'}\big|\One_{\alpha\in Q}\,d\alpha\big)^2}{\int_0^1|\varphi'|\One_{\alpha\in Q}\,d\alpha}\nonumber \\
    &= \frac{\cA \cG}{2L}\Big(\int_0^1|\varphi'|\big|\widehat{b(\varphi)}-\widehat{\varphi'}\big|\One_{\alpha\in Q}\,d\alpha\Big)^{\!2}.\qquad
    \label{key est part 1}
\end{align}
Now letting $\sigma:=+1$ or $\sigma:=-1$ depending on whether the function $f$ fulfills the property $(iii.1)$ or $(iii.2)$ of Definition \ref{tracing def}, we have a.e.\ on $Q$ that
\begin{align}
  \cH|\varphi'|\big|\widehat{b(\varphi)}-\widehat{\varphi'}\big|
    &\geq \sigma|\varphi'|\Skp{\nf(\varphi)}{\widehat{b(\varphi)}-\widehat{\varphi'}} \nonumber \\
    &= |\varphi'|\cdot\sigma\Skp{\nf(\varphi)}{\widehat{b(\varphi)}} - \sigma\skp{\nf(\varphi)}{\varphi'} \nonumber \\
    &=|\varphi'|-\sigma\,\partial_\alpha f(\varphi). \label{key est step}
\end{align}

Since $\hqq\circ f$ is Lipschitz continuous (with Lipschitz constant $\cH$),\linebreak $\hqq\circ f\circ\varphi$ is absolutely continuous, and so its classical derivative exists a.e.~on $[0,1]$. 
We have $\partial_\alpha\hqq(f(\varphi))=\partial_\alpha f(\varphi)$ wherever $f(\varphi)\in(q_1,q_2)$,~and $\partial_\alpha\hqq(f(\varphi))=0$ wherever $f(\varphi)\notin(q_1,q_2)$ (except possibly at $\alpha=0,1$) because $\hqq$ does not take values outside of $[q_1,q_2]$. This shows that $\partial_\alpha\hqq(f(\varphi))$ $=[\partial_\alpha f(\varphi)]\One_{f(\varphi)\in(q_1,q_2)}$, and so \eqref{key est step} implies that
\begin{align}
\cH\int_0^1|\varphi'|\big|\widehat{b(\varphi)}-\widehat{\varphi'}\big|\One_{\alpha\in Q}\,d\alpha
    &\geq \int_0^1\Big(|\varphi'|-\sigma\,\partial_\alpha f(\varphi)\Big)\One_{\alpha\in Q}\,d\alpha \nonumber\\
    &= L-\sigma\int_0^1[\partial_\alpha f(\varphi)]\One_{f(\varphi)\in(q_1,q_2)}\,d\alpha \nonumber\\
    &= L-\sigma\int_0^1\partial_\alpha\hqq(f(\varphi))\,d\alpha \nonumber\\
    &= L-\sigma\Delta \nonumber\\*
    &\geq L-|\Delta|. \label{key est part 2}
\end{align}
Multiplying \eqref{key est part 1} by $\cH^2$ and plugging in \eqref{key est part 2}, we thus obtain
\[ \cH^2S(\gamma)
   \geq \frac{\cA \cG}{2L}(L-|\Delta|)^2 = \frac{\cA \cG}{2}\Big(L-2|\Delta|+\frac{|\Delta|^2}{L}\Big)
   \geq \cA\cG\big(\tfrac12L-|\Delta|\big), \]
i.e.\ $L\leq\frac{2\cH ^2}{\cA\cG}S(\gamma)+2|\Delta|$, and \eqref{key est short} is proven.
\end{proof}

\begin{remark} \label{constant comparison}
If $K_1\subset K_2$, $(\tilde q_1,\tilde q_2)\subset(q_1,q_2)$, and if~$f$ traces the flowlines of~$b$ between $q_1$ and $q_2$, then
\[
  \cA(K_1)                      \geq \cA(K_2),      \hspace{.7cm}
  \cG(K_1)                      \geq \cG(K_2),      \hspace{.7cm}
  \cH (f,\tilde q_1,\tilde q_2) \leq \cH(f,q_1,q_2).
\]
\end{remark}

\subsection{Proof of Proposition \ref{blm prop 1}} \label{subsec proof prop 3}

\begin{proof}
We will again prove the stronger condition of Remark \ref{blm remark}~(ii).
Let $x\in\psi(M,\R)\cap\E$ and $\eta>0$ be given. By Corollary \ref{move cor} there exists another admissible manifold $M'$ such that $x\in M'$. For this manifold $M'$, Lemma~\ref{man2func cor} now provides us with an
$\eps>0$ and a function $f\colon D\to\R$ such that the properties (i)-(v) of Lemma \ref{man2func cor} are fulfilled. By decreasing $\eps>0$ if necessary, we may assume that $\bar B_\eps(x)\subset D$. As in Lemma \ref{man2func cor} we set $E:=f^{-1}\big((-\eps,\eps)\big)$.

The set $f^{-1}\big(\{-\tfrac{\eps}{2},\tfrac{\eps}{2}\}\big)$ is compact since it is closed in $D$ and a subset of the compact set $\bar E\subset D$ (see Lemma~\ref{man2func cor}~(iii)). Since it is disjoint from the closed set $E^c$ we thus have
\[ \Delta := \dist\!\big(f^{-1}\big(\{-\tfrac{\eps}{2},\tfrac\eps2\}\big),\,E^c\big)>0. \]
Lemma~\ref{lower semi lemma} (ii) and Definition \ref{drift vector field} provide us with constants $\cB :=\cB(\bar B_\eps(x))>0$ and $\cA :=\cA(\bar E)>0$, and Lemma~\ref{man2func cor} (iv) and (v) imply that the constants $\cG:=\cG(\bar E)$ and $\cH :=\cH (f,-\eps,\eps)$ defined in Lemma~\ref{key estimate} fulfill $\cG>0$ and $\cH<\infty$, so that all the requirements are met to apply Lemma \ref{key estimate} to any interval $(q_1,q_2)\subset(-\eps,\eps)$.
%Lastly, the set $K:=f^{-1}\big([-\frac\eps2,\frac\eps2]\big)\subset E$ is compact (since it is closed in $D$ and $\bar E\subset D$ is compact), and so $c:=\max_{w\in K}|\nabla f(w)|$ is well-defined and finite.
Finally, we define
\begin{equation} \label{r bound 1}
\nu:=\min\!\bigg\{\eps,\ \frac{\cA \cG\Delta}{5\cB\cH^2}\,,\ \frac{\eta}{4}\Big(\cH+\frac{\cB\cH ^2}{\cA\cG}\Big)^{\!\!-1}\bigg\},
\end{equation}
and we let $r\in(0,\nu]$ be so small that $\bar B_r(x)\subset f^{-1}\big((-\tfrac\eps2,\tfrac\eps2)\big)\subset E$ (which is possible because $f(x)=0$ by Lemma \ref{man2func cor}~(i)), and that for $\forall w\in\bar B_r(x)\cap\E$ $\exists\gamma\in\Gamma_x^w\colon\,\length(\gamma)\leq\nu$ (which is possible by Assumption \hE).

Now let $x_1,x_2\in\bar B_r(x)\cap\E$. For $i=1,2$ let $\bar\gamma^i\in\Gamma^{x_i}_x$ with $\length(\bar\gamma^i)\leq\nu$ and thus in particular $\bar\gamma^i\subset\bar B_\nu(x)\subset\bar B_\eps(x)$, and let $\bar\gamma:=-\bar\gamma^1+\bar\gamma^2\in\Gx$. Since $\bar\gamma\subset\bar B_\eps(x)$, Lemma \ref{lower semi lemma} (ii) shows that
\begin{equation} \label{rough inf est 2}
   \inf_{\gamma\in\Gx}\hS(\gamma) \leq\hS(\bar\gamma)\leq\cB\length(\bar\gamma)
   \leq 2\cB\nu.
\end{equation}
Next, let $(\varphi_n)_{n\in\N}\subset\CI$ be some parameterizations of a minimizing sequence $(\gamma_n)_{n\in\N}$ of $\prob$. We claim that
\begin{equation} \label{length bound 2}
\exists n_0\in\N\ \forall n\geq n_0:\ \max_{\alpha\in[0,1]}f(\varphi_n(\alpha))<\eps.
\end{equation}
Indeed, if this were not the case then we could extract a subsequence $(\varphi_{n_k})_{n\in\N}$ such that $\max_{\alpha\in[0,1]}f(\varphi_{n_k}(\alpha))\geq\eps$ for $\forall k\in\N$.
Since $x_1,x_2\in\bar B_r(x)\subset f^{-1}\big((-\tfrac\eps2,\tfrac\eps2)\big)$, we have $f(\varphi_{n_k}(0))=f(x_1)<\frac\eps2$ and $f(\varphi_{n_k}(1))=f(x_2)<\frac\eps2$, and thus for $\forall k\in\N$ there would then be two numbers $0<\check\alpha_k<\hat\alpha_k<1$ such that $f(\varphi_{n_k}(\check\alpha_k))=\frac\eps2$, $f(\varphi_{n_k}(\hat\alpha_k))=\eps$, and $f(\varphi_{n_k}(\alpha))\in(\frac\eps2,\eps)$ for $\forall\alpha\in(\check\alpha_k,\hat\alpha_k)$. Applying Lemma \ref{key estimate} with $(q_1,q_2)=(\frac\eps2,\eps)$, we would then have
\begin{align*}
  \frac{2\cH^2}{\cA\cG}\hS(\gamma_{n_k})
    &\geq \length\!\big(\gamma_{n_k}|_{f^{-1}((\eps/2,\eps))}\big)
-2\big|h_{\eps/2}^\eps(\underbrace{f(x_1)}_{\leq\frac\eps2})
-h_{\eps/2}^\eps(\underbrace{f(x_2)}_{\leq\frac\eps2})\big| \\
    &= \int_0^1|\varphi_{n_k}'|\One_{f(\varphi_{n_k})\in(\eps/2,\eps)}\,d\alpha
                       - 2\Big|\frac\eps2-\frac\eps2\Big| \\
     &\geq \int_{\check\alpha_k}^{\hat\alpha_k}|\varphi_{n_k}'|\,d\alpha \\
     &\geq \bigg|\int_{\check\alpha_k}^{\hat\alpha_k}\varphi_{n_k}'\,d\alpha\bigg|
     =     \bigg|\!\!\underbrace{\varphi_{n_k}(\hat\alpha_k)}_{\in f^{-1}(\{\eps\})\subset E^c} \!\!-\ \underbrace{\varphi_{n_k}(\check\alpha_k)}_{\in f^{-1}(\{\frac\eps2\})}\,\bigg|
    \geq \Delta\,.
\end{align*}
(Note that Lemma \ref{key estimate} gives us this estimate for constants $\cA $, $\cG$ and $\cH $ that are defined using $q_1=\frac\eps2$ and $q_2=\eps$, but the above estimate still holds as is since by Remark \ref{constant comparison} the term $\frac{2\cH^2}{\cA\cG}$ becomes larger by switching to our constants.)
Taking the limit $k\to\infty$ and using \eqref{rough inf est 2}, we thus find that
\[
  \Delta \leq \frac{2\cH^2}{\cA\cG}\cdot2\cB\nu,
\]
which contradicts \eqref{r bound 1}. This proves \eqref{length bound 2}, and with analogous arguments one can show that $\min_{\alpha\in[0,1]}f(\varphi_n(\alpha))>-\eps$ for large enough $n\in\N$.

After passing on to a tailsequence we may thus assume that $\gamma_n\subset f^{-1}\big((-\eps,\eps)\big)$ for $\forall n\in\N$. Using this additional knowledge, we can now apply Lemma \ref{key estimate} one more time (this time with $(q_1,q_2)=(-\eps,\eps)$) 
%and use the estimate \eqref{hqq est 1}
to obtain
\begin{align}
  \length(\gamma_n)
    &= \length\!\big(\gamma_n|_{f^{-1}((-\eps,\eps))}\big) \nonumber \\
    &\leq \frac{2\cH ^2}{\cA \cG}\hS(\gamma_n) + 2\big|h_{-\eps}^\eps(f(x_1))-h_{-\eps}^\eps(f(x_2))\big| \nonumber \\
    &= \frac{2\cH ^2}{\cA \cG}\hS(\gamma_n) + 2\big|f(x_1)-f(x_2)\big| \nonumber \\
    &\leq \frac{2\cH ^2}{\cA \cG}\hS(\gamma_n) + 2|x_1-x_2|\max_{w\in\bar B_r(x)}|\nabla f(w)| \nonumber \\
    &\leq \frac{2\cH ^2}{\cA \cG}\hS(\gamma_n) + 4\cH r \label{length bound blm1}
\end{align}
for $\forall n\in\N$, and thus $\sup_{n\in\N}\length(\gamma_n)<\infty$. We can now apply Proposition~\ref{min prop} and then use \eqref{length bound blm1}, the minimizing property of $(\gamma_n)_{n\in\N}$, \eqref{rough inf est 2} and \eqref{r bound 1} to conclude that the problem $\prob$ has a strong minimizer $\gamma^\star\in\Gx$ that fulfills
\begin{align*}
  \length(\gamma^\star)
    &\leq \liminf_{n\to\infty}\length(\gamma_n) \\
    &\leq 4\cH r + \frac{2\cH ^2}{\cA \cG}\liminf_{n\to\infty}\hS(\gamma_n) \\
    &=    4\cH r + \frac{2\cH ^2}{\cA \cG}\inf_{\gamma\in\Gx}\hS(\gamma) \\
    &\leq 4\cH\nu + \frac{2\cH ^2}{\cA \cG}\cdot2\cB\nu \\
    &=    4\nu\Big(\cH+\frac{\cB \cH ^2}{\cA\cG}\Big) \\
    &\leq \eta.
\qedhere
\end{align*}
\end{proof}

\subsection{Proof of Proposition \ref{blm prop 2}} \label{subsec proof prop 4}

If $b(x)=0$ then the strategy in the proof of Proposition \ref{blm prop 1} (laying out a ``$b$-coordinate grid'' around $x$) breaks down because $x$ cannot lie on an admissible manifold. Using the following lemma, we can however lay out multiple $b$-coordinate grids, each with $x$ on its boundary, that together cover a punctuated neighborhood of $x$.
We then have to refine our estimates for the curve lengths carefully, by slicing that neighborhood into appropriate regions and adding up the bounds that we obtain for each of them. The following lemma provides us with the necessary tools for this technique.

\begin{lemma} \label{superprop}
a) Let $x\in D$, and let the assumptions of Proposition \ref{blm prop 2} (i) or (ii) for $x$ to have weak local minimizers be fulfilled.
Then there exist an $\eps>0$ and functions $f_1,\dots,f_m\in C(D,[0,\infty))$ such that for $\forall i=1,\dots,m$\\[.15cm]
\begin{tabular}{rl}
(i) & $f_i(x)=0$, \\[.1cm]
(ii) & \begin{minipage}[t]{11.5cm} $f_i$ traces the flowlines of $b$ between the values $0$ and $\eps$, \end{minipage} \\[.1cm]
(iii) & $\bar E_i$ is a compact subset of $D$, where $E_i:=f_i^{-1}\big((0,\eps)\big)$, and \\[.1cm]
(iv) &  $\forall w\in\bar E_i\setminus\{x\}\colon\ \,b(w)\neq0$.%\\[.1cm]
\end{tabular}\\[.15cm]
\enlh
Furthermore,\\[.05cm]
\begin{tabular}{rl}
\hspace{.04cm}
(v) &  \begin{minipage}[t]{11.5cm} $\exists\cE>0\ \,\forall w\in \bar B_\eps(x)\colon\ \,\max\{f_1(w),\dots,f_m(w)\}\geq\cE|w-x|$. \end{minipage}
\end{tabular}\\[.2cm]
b) In addition, if the assumptions of Proposition \ref{blm prop 2} (i) or (ii) for $x$ to have strong local minimizers are fulfilled, then\\[.2cm]
\begin{tabular}{rl}
(vi) &  $\forall i=1,\dots,m\colon$ $\sup_{w\in E_i}|\nf_i(w)|<\infty$, and \\[.1cm]
\hspace{-.19cm}
(vii) &  \begin{minipage}[t]{11.5cm} $\exists\cF\geq1\ \,\forall w\in \bar B_\eps(x)\colon$ $\max\{f_1(w),\dots,f_m(w)\}\leq \cF |w-x|$. \end{minipage}
\end{tabular}
\end{lemma}

Observe that since this lemma takes a vector field $b$ and provides us with corresponding functions $f_i$, the properties \eqref{strong condition 2}-\eqref{strong condition} (which do not concern~$b$) are not needed for its proof (they will only be used in the main part of the proof of Proposition \ref{blm prop 2}). The only additional condition that we will use for proving (vi)-(vii) is that in the saddle point case we have~$D\subset\R^2$.

\begin{proof}
Here we will only prove the statement for the case that $x$ is an attractor or a repellor of $b$, where -- as we will see -- only one flowline tracing function $f_1$ is enough, i.e.~we can take $m=1$. The much harder proof for the case of a saddle point is the content of Part \ref{part superprop}.

Let us first deal with the case in which $x$ is an attractor of $b$. Let $\eps>0$ be so small that $\bar B_\eps(x)\subset B_s$, where $B_s$ is the basin of attraction of~$x$, let $f_s\colon B_s\to[0,\infty)$ be the function given by Definition \ref{fs fu def}, and finally define
\begin{equation} \label{flowline tracing def attractor}
  f_1(w):=\begin{cases}
    f_s(w) & \text{if $w\in f_s^{-1}\big([0,\eps)\big)$,} \\
    \eps   & \text{else.}
  \end{cases}
\end{equation}
We will now show that $f_1$ fulfills the properties (i)-(vii) of Lemma~\ref{superprop}. \\[.2cm]
%Observe that the only assumption that we will actually use in this case is the one on the eigenvalues of $\nabla b(x)$; the condition \eqref{strong condition} is only used later in the main part of the proof of Proposition \ref{blm prop 2}.
%
(i) $f_1(x)=f_s(x)=0$.\\[.2cm]
(ii) To show that $f_1$ traces the flowlines of $b$ between the values $0$ and $\eps$, we have to check the three properties in Definition \ref{tracing def}.\\[.2cm]
(ii.1) Clearly, $f_1$ is continuous on $D_1:=f_s^{-1}\big([0,\eps)\big)$ and on $D_2:=D\setminus D_1$.\linebreak
$D_1$ is open since it can be written as $f_s^{-1}\big((-\infty,\eps)\big)$, and thus $D_2$ is closed in~$D$. To show that $f_1$ is continuous on all of $D$ it thus suffices to show that for any converging sequence $(w_n)_{n\in\N}\subset D_1$ with $w:=\lim_{n\to\infty}w_n\in D_2$ we have $\lim_{n\to\infty}f_1(w_n)=f_1(w)$. To do so, first note that by \eqref{fs norm est 1 lower} we have $D_1\subset\bar B_\eps(x)$, which implies that $w\in\bar B_\eps(x)\subset B_s$ and thus $\lim_{n\to\infty}f_s(w_n)=f_s(w)$. Now since $f_s(w_n)\in[0,\eps)$ for $\forall n\in\N$, we have $f_s(w)\in[0,\eps]$, and thus $w\in D_2$ implies $f_s(w)=\eps$. We can now conclude that $\lim_{n\to\infty}f_1(w_n)=\lim_{n\to\infty}f_s(w_n)=f_s(w)=\eps=f_1(w)$.\\[.2cm]
(ii.2) We have $E_1:=f_1^{-1}\big((0,\eps)\big)=f_s^{-1}\big((0,\eps)\big)$ and thus $f_1|_{E_1}=f_s|_{E_1}$. Also, we have $E_1\subset B_\eps(x)\setminus\{x\}\subset B_s\setminus\{x\}$ by \eqref{fs norm est 1 lower} and since $f_s(x)=0$. Therefore by Lemma \ref{distance function}~(i), $f_s$ and thus also $f_1$ is $C^1$ on $E_1$.\\[.2cm]
(ii.3) Since $f_1=f_s$ on the open set $E_1\subset B_s\setminus\{x\}$, we have $\nf_1|_{E_1}=\nf_s|_{E_1}$ and thus $\forall w\in E_1\colon\ \skp{\nf_1(w)}{b(w)}=\skp{\nf_s(w)}{b(w)}=-|b(w)|$ by \eqref{fs decreasing}.\\[.2cm]
(iii) We have $\bar E_1\subset\bar B_\eps(x)\subset B_s\subset D$, and so $\bar E_1$ is a compact subset of $D$.\\[.2cm]
(iv) The relation shown in part (iii) implies $\bar E_1\setminus\{x\}\subset B_s\setminus\{x\}$, and since $x$ is the only point in $B_s$ with zero drift, this shows that $\forall w\in\bar E_1\setminus\{x\}\colon\ b(w)\neq0$. \\[.2cm]
(v) Let $w\in\bar B_\eps(x)$. If $w\in f_s^{-1}([0,\eps))$ then $f_1(w)=f_s(w)\geq|w-x|$ by \eqref{fs norm est 1 lower}. Otherwise we have $f_1(w)=\eps\geq|w-x|$. Thus we can choose $\cE:=1$.\\[.2cm]
(vi) In the proof of Lemma \ref{distance function}~(i), an integrable bound on the integrand of \eqref{diff integrand b} was found that is uniform on a neighborhood of some fixed $w\in B_s\setminus\{x\}$. We can use even easier arguments to find an integrable bound that is uniform on some punctuated ball $\bar B_{\eta}(x)\setminus\{x\}$ (at $x$ the argument breaks down since $\frac{b}{|b|}$ is undefined). This proves that $|\nf_s|$ is bounded on $\bar B_\eta(x)\setminus\{x\}$, and since $\nf_s$ is continuous on $B_s\setminus\{x\}$, $|\nf_s|$ is thus bounded also on the set $\bar B_{\eps}(x)\setminus\{x\}$ which includes $E_1$. Since we saw in (ii.3) that $\nf_1|_{E_1}=\nf_s|_{E_1}$, this shows that $|\nf_1|$ is bounded on $E_1$.\\[.2cm]
(vii) Let $\cD>0$ be the constant given by \eqref{fs norm est 1 upper} that corresponds to $K:=\bar B_\eps(x)$. Then for $\forall w\in\bar B_\eps(x)$ we have $f_1(w)\leq f_s(w)\leq\cD|w-x|$, i.e.~we can take $\cF:=\cD$.\\[.2cm]
This completes the proof for the case of an attractor. If $x$ is a repellor then we replace $f_s$ by $f_u$ everywhere in our proof, and the only difference will be that in part (ii.3) we have $\forall w\in E_1\colon\,\skp{\nf_1(w)}{b(w)}=+|b(w)|$ by~\eqref{fu increasing}.
\end{proof}

We are now ready to prove Proposition \ref{blm prop 2}. In the part proving that $x$ has \textit{strong} local minimizers we must assume that the reader has read the proof of Lemma \ref{min prop0b} in Appendix \ref{weak convergence appendix}, since we will re-use its terminology without further notice.

\begin{proof}[Proof of Proposition \ref{blm prop 2}]
\textit{Preparations.} Let $x\in\E$, and let the conditions of Proposition \ref{blm prop 2} (i) or (ii) for $x$ to have weak local minimizers be fulfilled. Let $\eps,\cE>0$ and the functions $f_1,\dots,f_m\colon D\to[0,\infty)$ be given as in Lemma \ref{superprop}~a), let $E_i:=f_i^{-1}\big((0,\eps)\big)$ for $\forall i=1,\dots,m$, and define $F:=\max\{f_1,\dots,f_m\}$. By decreasing $\eps$ and $\cE$ if necessary, we may assume that $\bar B_{2\eps}(x)\subset D$ and $\cE\in(0,1)$.
Since $b(x)=0$ and since our assumptions imply that $\nabla b(x)$ is an invertible matrix, $b$ is locally invertible at $x$ and we can further decrease~$\eps$ until
\begin{equation} \label{b bound 1}
|b(w)|\geq A|w-x| \qquad\text{for}\ \forall w\in\bar B_\eps(x)\ \text{and some}\ A>0.
\end{equation}

If the additional conditions for $x$ to have \textit{strong} local minimizers are fulfilled then we will at this point first choose $\rho,\cC,\delta>0$ such that \eqref{strong condition} is fulfilled (where we may assume that $\rho\in(0,1]$ and thus also that $\delta\in(0,1]$), and then further decrease~$\eps$ until \eqref{strong condition 2} holds for some $\cJ>0$ (where we may assume that $\eps\in(0,\rho/\cJ]$). Observe that we will not use these properties \eqref{strong condition 2}-\eqref{strong condition} during the first part of our proof (where we show that $x$ has weak local minimizers). This ends our definition of~$\eps$.

In either case, for every $i=1,\dots,m$, the set $f_i^{-1}(\{\tfrac{\cE\eps}{2}\})$ is compact since it is closed in $D$ and a subset of the compact set $\bar E_i\subset D$ (see Lemma \ref{superprop}~(iii)). Since it is disjoint from the closed set $f_i^{-1}\big((0,\cE\eps)\big)^c$ we thus have
\[ \Delta := \min_{1\leq i\leq m}\dist\!\Big(f_i^{-1}(\{\tfrac{\cE\eps}{2}\}),\ f_i^{-1}\big((0,\cE\eps)\big)^c\Big)>0. \]
Next we let $\cB :=\cB(\bar B_{2\eps}(x))>0$ as given by Lemma \ref{lower semi lemma} (ii). Also, defining $E:=\bigcup_{i=1}^mE_i\supset F^{-1}\big((0,\eps)\big)$, the set $\bar E=\bigcup_{i=1}^m\bar E_i$ is a compact subset of $D$ by Lemma \ref{superprop} (iii), and so Definition \ref{drift vector field} provides us with a constant $\cA :=\cA (\bar E)>0$. Defining $E_i':=f_i^{-1}\big((\tfrac{\cE \eps}{2},\cE \eps)\big)\subset E_i$ for $\forall i=1,\dots,m$, the constant $\cG:=\min_{1\leq i\leq m}\cG\big(\overline{E_i'}\big)$ defined in Lemma \ref{key estimate} fulfills $\cG>0$ by Lemma \ref{superprop} (i), (iii) and (iv), and the constant $\cH :=\max_{1\leq i\leq m}\cH \big(f_i,\tfrac{\cE \eps}{2},\cE \eps\big)$ defined in Lemma \ref{key estimate} is finite since $\nabla f_i$ is continuous on $E_i\supset\overline{E'_i}$ by Lemma \ref{superprop}~(ii), and since $\overline{E'_i}$ is compact by Lemma \ref{superprop} (iii). Finally, we define
\begin{equation} \label{r bound 2}
\nu:=\min\!\bigg\{\eps,\,\frac{\cA \cG\Delta}{5\cB\cH^2}\bigg\},
\end{equation}
and we let $r\in(0,\nu]$ be so small that
\begin{equation} \label{Br in F}
\bar B_r(x)\subset F^{-1}\big([0,\tfrac{\cE\eps}{2})\big)
\end{equation}
(this is possible since $F\geq0$, $F$ is continuous, and $F(x)=0$ by Lemma~\ref{superprop}~(i)), that
\begin{equation} \label{b bound 2}
 \min_{w\in\bar B_\eps(x)\setminus B_{r}(x)}|b(w)|\leq \min_{w\in \bar E\setminus B_\eps(x)}|b(w)|
\end{equation}
(this is possible since $b(x)=0$, and since $\bar E\setminus B_\eps(x)$ is a compact set on which $b\neq0$ by Lemma \ref{superprop} (iii)-(iv)), and that for $\forall w\in\bar B_r(x)\cap\E$ $\exists\gamma\in\Gamma_x^w$: $\length(\gamma)\leq\nu$ (this is possible by Assumption~\hE).

If the additional conditions for $x$ to have \textit{strong} local minimizers are fulfilled 
then we will in fact show the stronger property in Remark \ref{blm remark}~(ii), so let $\eta>0$ be given. Under these conditions, Lemma \ref{superprop} (vi) says that the constant $\cHH :=\max_{1\leq i\leq m}\cH (f_i,0,\eps)$ defined in Lemma \ref{key estimate} is finite, and Lemma \ref{superprop}~(vii) gives us a constant $\cF>0$. We then decrease $r$ further so that
\begin{equation} \label{r bound 2b}
\frac{2ar^\delta}{1-2^{-\delta}} \leq\eta,
\text{\quad where \quad}
a:=\frac{2^{4+\delta}m\cJ^{1+\delta}\cC\cHH^2\cF}{\cA\cE^{2+\delta} A} + 4m\cF\eps^{1-\delta}.
\end{equation}
Again observe that we will not use the constants $\cHH$ and $\cF$ and the estimate \eqref{r bound 2b} during the first part of our proof. This ends our definition of $r$.\\[.2cm]
\textit{Weak local minimizers.}
Now let $x_1,x_2\in \bar B_r(x)\cap\E$, let $(\gamma_n)_{n\in\N}\subset\Gx$ be a minimizing sequence of $\prob$, and let us assume that each curve~$\gamma_n$ visits the point $x$ at most once (otherwise we may cut out the piece between the first and the last hitting point of $x$, which can only decrease the action of the curve). Denoting by $(\tilde\varphi_n)_{n\in\N}\subset\AC{0,1}$ their arclength parameterizations given by Lemma \ref{arclength lemma}~(i),
we first claim that for sufficiently large $n\in\N$ we have
\begin{equation} \label{length bound 3}
\max_{\alpha\in[0,1]}F(\tilde\varphi_n(\alpha))<\cE \eps.
\end{equation}
Indeed, if this were not the case then we could extract a subsequence $(\tilde\varphi_{n_k})_{k\in\N}$ such that for some $i_0$ and $\forall k\in\N$ we had $\max_{\alpha\in[0,1]}f_{i_0}(\tilde\varphi_{n_k}(\alpha))$ $\geq\cE\eps$. Since by \eqref{Br in F} we have $f_{i_0}(x_1)\leq F(x_1)<\tfrac12\cE\eps$ and similarly $f_{i_0}(x_2)<\tfrac12\cE\eps$, we could then use the same arguments as in the proof of Proposition~\ref{blm prop 1} (only here with Lemma \ref{key estimate} applied to $f_{i_0}$ and $(q_1,q_2)=(\tfrac12\cE\eps,\cE \eps)$) and Remark~\ref{constant comparison} to conclude that
\[
  \Delta
    \leq \frac{\,2\cH\big(f_{i_0},\tfrac{\cE \eps}{2},\cE \eps\big)^2}
              {\cA\big(\overline{E_{i_0}'}\big)\cG\big(\overline{E_{i_0}'}\big)}
         \cdot2\cB\nu
    \leq \frac{2\cH^2}{\cA\cG}\cdot2\cB\nu,
\]
contradicting \eqref{r bound 2}. This proves \eqref{length bound 3} for large enough $n\in\N$, and so after passing on to a tailsequence we may assume that \eqref{length bound 3} holds for $\forall n\in\N$.

In particular, this implies that $\gamma_n\subset\bar B_\eps(x)$ for $\forall n\in\N$. Indeed, otherwise there would be a point $w$ on $\gamma_n$ such that $|w-x|=\eps$, and Lemma \ref{superprop} (v) and \eqref{length bound 3} would then imply that $\cE\eps=\cE |w-x|\leq F(w)<\cE\eps$. As a result, we are allowed to apply the estimate in Lemma \ref{superprop} (v) (and later also the one in Lemma \ref{superprop} (vii)) to all points on the curves $\gamma_n$.\\[.2cm]
We will now use Lemma \ref{min prop0b} to construct a converging subsequence. In order to control the lengths of $\gamma_n$ away from $x$, we use \eqref{length bound 3}, the definition of $F$, Lemma~\ref{key estimate} (whose conditions can be checked as above) and \eqref{hqq est 1} to obtain for $\forall i=1,\dots,m$ and $\forall u\in(0,\cE\eps)$ constants $cst(i,u)>0$ (independent of~$x_1$ and~$x_2$) such that
\begin{align}
\int_{\gamma_n}\One_{F(z)>u}\,|dz| 
    &=    \int_{\gamma_n}\One_{F(z)\in(u,\cE\eps)}\,|dz| \nonumber \\
    &\leq \sum_{i=1}^m \int_{\gamma_n}\One_{f_i(z)\in(u,\cE\eps)}\,|dz|\nonumber \\
    &\leq \sum_{i=1}^m \Big[ cst(i,u)\hS(\gamma_n)+2\big|h_{u}^{\cE\eps}(f_i(x_1))-h_{u}^{\cE\eps}(f_i(x_2))\big| \Big] \nonumber \\
    &\leq \Big(\sum_{i=1}^m cst(i,u)\Big)\hS(\gamma_n) + 2m(\cE\eps-u) \nonumber \\
    &\leq \Big(\sum_{i=1}^m cst(i,u)\Big)\sup_{j\in\N}\hS(\gamma_j) + 2m\cE\eps=:\eta(u). \label{length bound 4}
\end{align}
For $u\geq\cE\eps$ this estimate holds with $\eta(u):=0$ by \eqref{length bound 3}.
We could now use that $\bar B_u(x)^c\subset F^{-1}\big((\cE u,\infty)\big)$
by Lemma \ref{superprop}~(v) to check the condition \eqref{length away from x cond}, but in preparation for the second part of this proof we will instead make use of the remark at the beginning of the proof of Lemma \ref{min prop0b}, which says that the estimate \eqref{length bound 4} is enough as is, and we will consider the construction and terminology of that proof, using \textit{our} function $F$ (instead of the function $F(w)=|w-x|$), $c:=\cE$, $K:=\bar B_\eps(x)$, and $u_k:=\tilde r2^{-k}$, where
\begin{equation} \label{tilde r def}
  \tilde r:=\max_{w\in\bar B_r(x)} F(w).
\end{equation}
Thus, by Lemma \ref{min prop0b} there exist parameterizations $\varphi_n\in\Cxx$ of $\gamma_n$ such that a subsequence of $(\varphi_n)_{n\in\N}$ converges to a parameterization $\varphi^\star\in\Cxx$ of a curve $\gamma^\star\in\tGxxx$. We have $\gamma^\star\subset\bar B_\eps(x)=K$ since $\gamma_n\subset\bar B_\eps(x)$ for $\forall n\in\N$, and in particular we can apply the estimate in Lemma \ref{superprop} (v) (and later also the one in Lemma \ref{superprop} (vii)) to every point on $\gamma^\star$. By \eqref{general outer length}, i.e.~the generalized version of \eqref{length away from x result}, we therefore have
\[
\length\!\big(\gamma^\star|_{\bar B_{u}(x)^c}\big)
= \int_{\gamma^\star}\One_{|z-x|>u}\,|dz|
\leq \int_{\gamma^\star}\One_{F(z)>\cE u}\,|dz|
\leq \eta(\cE u)
=: \tilde\eta(u)
\]
for $\forall u>0$. Finally, by Lemmas \ref{lower semi lemma} (iv) and \ref{weak vs strong inf} we have
\begin{equation} \label{pre min prop}
S(\gamma^\star) \leq \liminf_{n\to\infty}S(\gamma_n)=\inf_{\gamma\in\Gx}\hS(\gamma)=\inf_{\gamma\in\tGxx}\hS(\gamma),
\end{equation}
and since $\gamma^\star\in\tGxx$, we must have equality, i.e.~$\gamma^\star$ is a weak minimizer of $\prob$. This concludes the proof that $x$ has weak local minimizers. \\[.3cm]
\textit{Strong local minimizers.} Now let the additional conditions of part (i) or (ii) be fulfilled. To show that $x$ has in fact \textit{strong} local minimizers, it remains to show that $\varphi^\star\in\AC{0,1}$ (so that $\gamma^\star\in\Gx$) and that $\length(\gamma^\star)\leq\eta$.

\indent To show that ${\varphi^\star}'\in L^1(0,1)$ and to estimate $\length(\gamma^\star)$, we now begin by proving some properties of the function $F\circ\varphi^\star$. First, note that replacing $\tilde\varphi_n$ by its reparametrized version $\varphi_n$ in \eqref{length bound 3} and then taking the limit $n\to\infty$ implies that
\begin{equation} \label{length bound 3B}
\max_{\alpha\in[0,1]}F(\varphi^\star(\alpha))\leq\cE\eps<\eps.
\end{equation}
Second, taking the limit $n\to\infty$ in \eqref{Fphik geq B1} implies that for $\forall k\in\N$ we have
\begin{equation} \label{Fphik geq C}
\begin{split}
    \text{either} \hspace{.7cm}&\forall s\in [0,d_k^-]\colon\,\ F(\varphi^\star(s))\geq u_k\\
    \text{or}     \hspace{1.3cm}&\text{$\varphi^\star$ is constant on $[0,d_k^-]$}
\end{split}
\end{equation}
(or both), and the same is true with $[0,d_k^-]$ replaced by $[d_k^+,1]$.
Third, we have
\begin{align}
  \forall n\in\N\,\ \forall k\in\N_0\colon &\quad F(\varphi_n(\tfrac12))\leq F(\varphi_n(d_{k+1}^-))\leq F(\varphi_n(d_k^-)), \label{F increasing 1} \\
  \forall k\in\N_0\colon &\quad F(\varphi^\star(\tfrac12))\leq F(\varphi^\star(d_{k+1}^-))\leq F(\varphi^\star(d_k^-)), \label{F increasing 2}
\end{align}
and the same relations hold with $d_k^-$ and $d_{k+1}^-$ replaced by $d_k^+$ and $d_{k+1}^+$.

Indeed, the left inequality in \eqref{F increasing 1} is clear: $F(\varphi_n(\tfrac12))=F\big(\tilde\varphi_n(\alpha_n(\tfrac12))\big)$ $=F(\tilde\varphi_n(\alpha_{\min}^n))\leq F\big(\tilde\varphi_n(\alpha_n(d_{k+1}^-))\big)
=F(\varphi_n(d_{k+1}^-))$. The second inequality in \eqref{F increasing 1} can be seen as follows: If $\alpha_n(d_k^-)=\alpha_n(d_{k+1}^-)$ then we have $F(\varphi_n(d_{k+1}^-))= F(\varphi_n(d_k^-))$, so \eqref{F increasing 1} holds. Also, if $I_{n,k+1}=\varnothing$ then \linebreak $F(\varphi_n(d_{k+1}^-))=F\big(\tilde\varphi_n(\alpha_n(d_{k+1}^-))\big)
=F(\tilde\varphi_n(\alpha_{\min}^n))\leq F\big(\tilde\varphi_n(\alpha_n(d_k^-))\big)=$ \linebreak $F(\varphi_n(d_k^-))$, and \eqref{F increasing 1} holds as well. Otherwise we have $\alpha_n(d_k^-)<\alpha_n(d_{k+1}^-)$ $=\min I_{n,k+1}$, so that $\alpha_n(d_k^-)\notin I_{n,k+1}$ and thus $F(\varphi_n(d_k^-))=F\big(\tilde\varphi_n(\alpha_n(d_k^-))\big)$ $>u_{k+1}\geq F\big(\tilde\varphi_n(\alpha_n(d_{k+1}^-))\big)=F(\varphi_n(d_{k+1}^-))$. This ends the proof of \eqref{F increasing 1}, and \eqref{F increasing 2} now follows by taking the limit $n\to\infty$. The modified statements with $d_k^-$ and $d_{k+1}^-$ replaced by $d_k^+$ and $d_{k+1}^+$ can be shown analogously.\\[.2cm]
\indent Next, we will prove a minimizing property of $\varphi^\star$, namely that for each pair of numbers $0\leq s_1<s_2<\tfrac12$ or $\tfrac12<s_1<s_2\leq1$ we have
\begin{equation} \label{minimizing property}
  \hS(\varphi^\star|_{[s_1,s_2]}) = \inf_{\gamma\in\Gxy{\varphi^\star(s_1)}{\varphi^\star(s_2)}}\hS(\gamma).
\end{equation}
We will prove this for the case $0\leq s_1<s_2<\frac12$, the other case can be shown analogously.
To do so, we denote the left-hand side of \eqref{minimizing property} by $S^\star$.
If the statement were wrong then we could find a curve $\gamma_0\in\Gxy{\varphi^\star(s_1)}{\varphi^\star(s_2)}$ whose action fulfills $\sigma:=S^\star-\hS(\gamma_0)>0$.
By the minimizing property of $(\gamma_n)_{n\in\N}$ and the relation $S^\star=\hS(\varphi^\star|_{[s_1,s_2]})\leq\liminf_{n\to\infty}\hS(\varphi_n|_{[s_1,s_2]})$ (which follows from Lemma \ref{lower semi lemma 2}~(i)), respectively, we could now choose an $n\in\N$ so large that
\[
   \hS(\gamma_n)<\inf_{\gamma\in\Gx}\hS(\gamma)+\tfrac14\sigma
    \qquad\text{and}\qquad
   \hS(\varphi_n|_{[s_1,s_2]})\geq S^\star-\tfrac14\sigma,
\]
and since $\lim_{n\to\infty}\varphi_n(s_i)=\varphi^\star(s_i)$ for $i=1,2$, Assumption \hE\ would allow us to choose $n\in\N$ so large that there exist curves
\[
    \bar\gamma^1\in\Gamma^{\varphi_n(s_1)}_{\varphi^\star(s_1)}
    \qquad\text{and}\qquad
    \bar\gamma^2\in\Gamma^{\varphi_n(s_2)}_{\varphi^\star(s_2)}
\]
with $\length(\bar\gamma^{1,2})\leq\min\{\frac{\sigma}{4\cB},\eps\}$.

%The actions of the straight connection lines $\bar\gamma_1$ and $\bar\gamma_2$ from $\varphi_n(s_1)$ to $\varphi^\star(s_1)$ and from $\varphi^\star(s_2)$ to $\varphi_n(s_2)$, respectively, are by Lemma \ref{lower semi lemma} (ii) bounded by $\hS(\bar\gamma_1)\leq \cB |\varphi_n(s_1)-\varphi^\star(s_1)|\leq\frac14\sigma$, and similarly $\hS(\bar\gamma_2)\leq \frac14\sigma$.
Now $\gamma^\star\subset\bar B_\eps(x)$ and $\length(\bar\gamma^{1,2})\leq\eps$ imply that $\bar\gamma^{1,2}\subset\bar B_{2\eps}(x)$, and so by Lemma \ref{lower semi lemma}~(ii) we have the estimates $S(-\bar\gamma^1)\leq\cB\length(\bar\gamma^1)\leq\tfrac14\sigma$ and similarly $S(\bar\gamma^2)\leq\tfrac14\sigma$.
Therefore the curve $\hat\gamma\in\Gx$, constructed by removing from $\gamma_n$ the piece given by $\varphi_n|_{[s_1,s_2]}$ and replacing it by the curve $-\bar\gamma^1+\gamma_0+\bar\gamma^2$, would have the action
\begin{align*}
 \hS(\hat\gamma)
   &= \hS(\gamma_n)-\hS(\varphi_n|_{[s_1,s_2]}) + \hS(-\bar\gamma^1)+\hS(\gamma_0)+\hS(\bar\gamma^2) \hspace{-7cm}&&\\
   &< \Big(\inf_{\gamma\in\Gx}\hS(\gamma)+\tfrac14\sigma\Big)-\big(S^\star-\tfrac14\sigma\big)
      +\tfrac14\sigma+(S^\star-\sigma)+\tfrac14\sigma \\
   &= \inf_{\gamma\in\Gx}\hS(\gamma),
\end{align*}
which is a contradiction, and \eqref{minimizing property} is proven.\\[.3cm]
\indent We are now ready to show that ${\varphi^\star}'\in L^1(0,1)$ and to estimate $\length(\gamma^\star)$. Fix $\forall k\in\N_0$, and let $E_i^k:=f_i^{-1}\big((u_{k+2},\eps)\big)\subset E_i\subset E$ for $\forall i=1,\dots,m$.
Using \eqref{length bound 3B}, \eqref{Fphik geq C}, Lemma \ref{key estimate} applied to the curve given by $\varphi^\star|_{Q_k^-}\in\bar C(d_k^-,d_{k+1}^-)$, Remark \ref{constant comparison}, and \eqref{hqq est 2}, we find that
% with $[\alpha_1,\alpha_2]=Q_k^+=[d_{k+1}^+,d_k^+]$, $[q_1,q_2]=[\cH r2^{-k},\eps]$ and $K_l$
\begin{align}
  \int_{Q_k^-}|{\varphi^\star}'|\,d\alpha
    &= \int_{d_k^-}^{d_{k+1}^-}|{\varphi^\star}'|\One_{F(\varphi^\star)\in[u_{k+1},\eps)}\,d\alpha \nonumber \\
    &\leq \sum_{i=1}^m \int_{d_k^-}^{d_{k+1}^-}|{\varphi^\star}'|\One_{f_i(\varphi^\star)\in(u_{k+2},\eps)}\,d\alpha \label{omittable line} \\
    &\leq \sum_{i=1}^m \bigg[ \frac{2\cH (f_i,u_{k+2},\eps)^2}{\cA (E_i^k)\cG(E_i^k)}\hS\big(\varphi^\star|_{Q_k^-}\big) \nonumber \\*
    &     \hspace{1.3cm} + 2\,\Big|h_{u_{k+2}}^\eps\big(f_i(\varphi^\star(d_k^-))\big)-h_{u_{k+2}}^\eps\big(f_i(\varphi^\star(d_{k+1}^-))\big)\Big|\bigg] \nonumber \\
    &\leq \sum_{i=1}^m \bigg[ \frac{2\cHH ^2}{\cA \cG(E_i^k)}\hS\big(\varphi^\star|_{Q_k^-}\big)  + 2\,\Big|f_i\big(\varphi^\star(d_k^-)\big)-f_i\big(\varphi^\star(d_{k+1}^-)\big)\Big|\bigg]. \label{final est part one}
\end{align}
To estimate $\cG(E_i^k)$, first we argue that
\begin{equation*}
  E_i^k \subset [\bar B_\eps(x)\cap E_i^k]\cup[\bar B_\eps(x)^c\cap E_i^k]
        \subset [\bar B_\eps(x)\cap B_{u_{k+2}/\cF}(x)^c]\cup[B_\eps(x)^c\cap \bar E],
\end{equation*}
where we used that $E_i^k\subset E\subset \bar E$, and that for $\forall w\in\bar B_\eps(x)\cap E_i^k$ we have $|w-x|\geq\frac{1}{\cF }F(w)\geq\frac{1}{\cF }f_i(w)>\frac{1}{\cF}u_{k+2}$, i.e.~$w\in\bar B_\eps(x)\cap B_{u_{k+2}/\cF}(x)^c$. Furthermore, by \eqref{tilde r def} and Lemma \ref{superprop}~(v) and (vii) we have $\cE r\leq\tilde r\leq\cF r$ and thus in particular $\frac{u_{k+2}}{\cF}\leq\frac{\tilde r}{\cF}\leq r$. Thus, together with \eqref{b bound 2} and \eqref{b bound 1} we find that
\begin{align}
  \cG(E_i^k)
     &= \min\big\{|b(w)|\,;\,w\in E_i^k\} \nonumber\\
     &\geq \min\big\{|b(w)|\,;\,w\in [\bar B_\eps(x)\cap B_{u_{k+2}/\cF}(x)^c]\cup[B_\eps(x)^c\cap \bar E]\big\} \nonumber\\
     &\geq \min\big\{|b(w)|\,;\,w\in [\bar B_\eps(x)\cap B_{u_{k+2}/\cF}(x)^c]\cup[\bar B_\eps(x)\cap B_{r}(x)^c]\big\} \nonumber\\
     &= \min\big\{|b(w)|\,;\,w\in \bar B_\eps(x)\cap B_{u_{k+2}/\cF}(x)^c\big\}\nonumber\\
     &\geq \frac{Au_{k+2}}{\cF} = \frac{A\tilde r}{\cF}2^{-(k+2)}
      \geq \frac{A\cE r}{\cF}2^{-(k+2)}. \label{mu estimate}
\end{align}
Assume now that for the given $k\in\N_0$ \eqref{Fphi leq 1} holds (recall that we denote our limit by $\varphi^\star$ instead of $\varphi$).
Using \eqref{mu estimate}, $f_i\geq0$, the definition of $F$, and \eqref{F increasing 2} and \eqref{Fphi leq 1}, we can then continue the estimate \eqref{final est part one} and find that
\begin{align}
\int_{Q_k^-}|{\varphi^\star}'|\,d\alpha
    &\leq  \frac{2m\cHH^2\cF 2^{k+2}}{\cA\cE Ar}\hS\big(\varphi^\star|_{Q_k^-}\big)
           + 2\sum_{i=1}^m\Big[ f_i\big(\varphi^\star(d_k^-)\big)+f_i\big(\varphi^\star(d_{k+1}^-)\big) \Big] \nonumber\\
    &\leq  \frac{2^{k+3}m\cHH^2\cF}{\cA\cE Ar}\hS\big(\varphi^\star|_{Q_k^-}\big)
           + 2m\Big[ F\big(\varphi^\star(d_k^-)\big)+F\big(\varphi^\star(d_{k+1}^-)\big) \Big] \nonumber\\
    &\leq  \frac{2^{k+3}m\cHH^2\cF}{\cA\cE Ar}\hS\big(\varphi^\star|_{Q_k^-}\big)
           + 2m\cdot 2u_k. \label{Q int 1}
\end{align}
By \eqref{strong condition 2} there exist curves $\bar\gamma^1_k\in\Gamma^{\varphi^\star(d_k^-)}_x$ and $\bar\gamma^2_k\in\Gamma^{\varphi^\star(d_{k+1}^-)}_x$ with
\begin{subequations}
\begin{align}
\length(\bar\gamma^1_k)&\leq\cJ\big|\varphi^\star(d_k^-)-x\big|\hspace{-2.95cm}&&\leq\cJ\eps\leq\rho, \label{length gamma k 1}\\
\length(\bar\gamma^2_k)&\leq\cJ\big|\varphi^\star(d_{k+1}^-)-x\big|\hspace{-2.95cm}&&\leq\cJ\eps\leq\rho, \label{length gamma k 2}
\end{align}
\end{subequations}
and thus in particular $\bar\gamma^{1,2}_k\subset\bar B_\rho(x)$.
Let
\[
  \bar\gamma_k:=-\bar\gamma^1_k+\bar\gamma^2_k
            \in\Gamma_{\varphi^\star(d_k^-)}^{\varphi^\star(d_{k+1}^-)},
\]
which fulfills $\bar\gamma_k\subset\bar B_\rho(x)$, and let $\bar\varphi_k\in\AC{0,1}$ be a parameterization of $\bar\gamma_k$ with $\bar\varphi_k(\frac12)=x$.
The minimizing property \eqref{minimizing property}, \eqref{strong condition}, \eqref{length gamma k 1}-\eqref{length gamma k 2}, Lemma \ref{superprop}~(v), and again \eqref{F increasing 2} and \eqref{Fphi leq 1} now tell us that
\begin{align}
  \hS\big(\varphi^\star|_{Q_k^-}\big)
    &\leq \hS(\bar\gamma_k) \nonumber\\
    &=    \int_0^1\ell(\bar\varphi_k,\bar\varphi_k')\,d\alpha \nonumber\\
    &\leq \cC \int_0^1|\bar\varphi_k-x|^\delta|\bar\varphi_k'|\,d\alpha \nonumber\\
    &\leq \cC \max_{\alpha\in[0,1]}\big|\bar\varphi_k(\alpha)-x\big|^\delta
               \cdot\int_0^1|\bar\varphi_k'|\,d\alpha \nonumber\\
    &= \cC \max_{\alpha\in[0,1]}\bigg|\int_{1/2}^\alpha\bar\varphi_k'\,d\tilde\alpha\,\bigg|^\delta
               \cdot\int_0^1|\bar\varphi_k'|\,d\alpha \nonumber\\
    &\leq \cC\bigg[\int_0^1|\bar\varphi_k'|\,d\alpha\bigg]^{1+\delta}
     = \cC\length(\bar\gamma_k)^{1+\delta} \nonumber\\
    &= \cC\big(\length(\bar\gamma^1_k)+\length(\bar\gamma^2_k)\big)^{1+\delta} \nonumber\\
%    &= \cC \max\big\{|\bar\varphi_k(0)-x|,\,|\bar\varphi_k(1)-x|\big\}^\delta\cdot|\bar\varphi_k(0)-\bar\varphi_k(1)| \\
%    &= \cC \max\big\{|\bar\varphi_k(0)-x|,\,|\bar\varphi_k(1)-x|\big\}^\delta \\*
%    &\hspace{4cm} \cdot\big(|\bar\varphi_k(0)-x|+|\bar\varphi_k(1)-x|\big) \\
%    &\leq 2\cC \max\big\{\big|\bar\varphi_k(0)-x\big|,\,\big|\bar\varphi_k(1)-x\big|\big\}^{1+\delta} \\
    &\leq \cC\Big[ \cJ\big|\varphi^\star(d_k^-)-x\big|
                  +\cJ\big|\varphi^\star(d_{k+1}^-)-x\big|\Big]^{1+\delta} \nonumber\\
    &\leq \cC\Big[ \frac\cJ\cE F(\varphi^\star(d_k^-))
                  +\frac\cJ\cE F(\varphi^\star(d_{k+1}^-))\Big]^{1+\delta} \nonumber\\
    &\leq \cC\Big(\frac{2\cJ u_k}{\cE}\Big)^{\!1+\delta}
     =    \cC\Big(\frac{2\cJ r2^{-k}}{\cE}\Big)^{\!1+\delta}. \label{Q int 2}
%     =    \frac{2\cC}{\cE}(r2^{-k})^{1+\delta}.
\end{align}
Therefore, if \eqref{Fphi leq 1} holds then by \eqref{Q int 1}, \eqref{Q int 2} and \eqref{r bound 2b} we have the estimate
\begin{align}
  \int_{Q_k^-}|{\varphi^\star}'|\,d\alpha
    &\leq \frac{2^{k+3}m\cHH^2\cF}{\cA\cE Ar}\cdot \cC\Big(\frac{2\cJ r2^{-k}}{\cE}\Big)^{\!1+\delta} + 4m\cF r2^{-k} \nonumber\\
    &\leq \bigg(\frac{2^{4+\delta}m\cJ^{1+\delta}\cC\cHH^2\cF}{\cA\cE^{2+\delta} A} + 4m\cF\eps^{1-\delta} \bigg)r^{\delta}2^{-\delta k} \nonumber\\
    &= ar^\delta2^{-\delta k}. \label{varphistar length part}
\end{align}
But if instead \eqref{Fphi leq 2} holds then ${\varphi^\star}'$ vanishes a.e.~on $[d_k^-,\tfrac12]\supset Q_k^-$ and thus \eqref{varphistar length part} is trivial. Therefore \eqref{varphistar length part} always holds, and analogously the same estimate can be established for $Q_k^+$. We thus obtain
\begin{align}
  \int_0^1|{\varphi^\star}'|\,d\alpha
     &= \sum_{k=0}^\infty\Big(\int_{Q_k^-}|{\varphi^\star}'|\,d\alpha + \int_{Q_k^+}|{\varphi^\star}'|\,d\alpha \Big) \nonumber \\
    &\leq 2ar^\delta\sum_{k=0}^\infty 2^{-\delta k}
     =    \frac{2ar^\delta}{1-2^{-\delta}} \leq\eta \label{min length eta}
\end{align}
by \eqref{r bound 2b}, i.e.~${\varphi^\star}'\in L^1(0,1)$ and $\length(\gamma^\star)\leq\eta$.
To prove the absolute continuity of ${\varphi^\star}$, it remains to show that
\begin{equation}\label{abs cont}
  \varphi^\star(s)-\varphi^\star(0) = \int_0^s{\varphi^\star}'\,d\alpha \qquad\text{for }\forall s\in[0,1].
\end{equation}
This is true for $\forall s\in[0,\frac12)$ since $\varphi^\star$ is absolutely continuous on each $J_k$, and for $s=\frac12$ by taking the limit $s\nearrow\frac12$ in \eqref{abs cont} and using dominated convergence. Analogously, one can show that $\varphi^\star(1)-\varphi^\star(s)=\int_s^1{\varphi^\star}'\,d\alpha$ for all $s\in[\frac12,1]$, and therefore for $s\in(\frac12,1]$ we have
\begin{align*}
 \varphi^\star(s)-\varphi^\star(0)
 &= \big(\varphi^\star(1)-\varphi^\star(\tfrac12)\big)+\big(\varphi^\star(\tfrac12)-\varphi^\star(0)\big)-\big(\varphi^\star(1)-\varphi^\star(s)\big) \\
 &= \int_{1/2}^1{\varphi^\star}'\,d\alpha + \int_0^{1/2}{\varphi^\star}'\,d\alpha - \int_s^1{\varphi^\star}'\,d\alpha \\
 &= \int_0^s{\varphi^\star}'\,d\alpha
 \end{align*}
as well. This concludes the proof of the absolute continuity of $\varphi^\star$, so that $\gamma^\star\in\Gx$, i.e.\ $x$ has \emph{strong} local minimizers. This terminates the proof of Proposition \ref{blm prop 2}.
\end{proof}

\pagebreak
\begin{appendices}
  \section{Proof of Lemma \ref{distance function}}
\label{fs fu appendix}

\begin{proof}
It is enough to show these properties for $f_s$; the analogous properties for $f_u$ then follow by replacing $b$ by $-b$.
To show that $f_s$ is finite-valued, first recall~\cite[Thm.~7.1]{Verhulst} that
\begin{equation} \label{stable exp decay 1}
 \exists c,\eps,\alpha>0\ \forall v\in\bar B_\eps(x)\ \forall t\geq0:\ |\psi(v,t)-x|\leq c|v-x|e^{-\alpha t}\leq c\eps,
\end{equation}
where we will assume that $\eps$ is so small that that $\bar B_{c\eps}(x)\subset D$.
Thus, since for any given $w\in B_s$ there exists a $T\geq0$ such that $\psi(w,T)\in B_\eps(x)$, $|\psi(w,t)-x|$ decays exponentially as $t\to\infty$, and since $\exists a>0\ $ $\forall v\in\bar B_\eps(x):$ $|b(v)|\leq a|v-x|$, also $|b(\psi(w,t))|$ decays exponentially, proving that the integral in \eqref{fs def 1} converges.
The continuity of $f_s$ will follow from~(i) and~(iv). \\[.2cm]
(i) Let $w\in B_s\setminus\{x\}$. Then formally we can differentiate
\begin{align}
 \nf_s(w)
   &= \nabla_{\!\!w\,} \int_0^\infty|b(\psi(w,t))|\,dt 
    = \int_0^\infty\nabla_{\!\!w\,}|b(\psi(w,t))|\,dt \nonumber\\
   &= \int_0^\infty\frac{b(\psi(w,t))^T \nabla b(\psi(w,t)) \nabla\psi(w,t)}{|b(\psi(w,t))|}\,dt. \label{diff integrand b}
\end{align}
To make the exchange of integration and differentiation rigorous and to show that $\nf_s(w)$ is continuous, it suffices to show that there exists a function $p\in L^1([0,\infty),\R)$ such that the integrand of \eqref{diff integrand b}, let us call it $q(w,t)$, fulfills $|q(v,t)|\leq p(t)$ for $\forall t\geq0$ and all $v$ in some ball $\bar B_\eta(w)$. To find such a bound for $q$, first we use that $\big|\frac{b}{|b|}\big|\leq1$. Second, if we choose $T$ as before and $\eta>0$ so small that
\begin{equation} \label{T step into Beps}
 \forall v\in\bar B_\eta(w)\colon\ \,\psi(v,T)\in B_\eps(x)
\end{equation}
then by \eqref{stable exp decay 1} and \eqref{T step into Beps} we have
\[ \forall v\in\bar B_\eta(w)\ \forall t\geq0\colon\quad\psi(v,t)\in K':=\psi\big(\bar B_\eta(w),[0,T]\big)\cup\bar B_{c\eps}(x)\subset D, \]
and since $K'$ is compact, $|\nabla b(\psi(v,t))|$ can be bounded by a constant as well. Therefore it suffices to show that we can decrease $\eta>0$ so much that
\begin{equation} \label{grad decay b}
  \exists \tilde c,\tilde\alpha>0\ \forall t\geq0\ \forall v\in\bar B_\eta(w)\colon\ \,|\nabla\psi(v,t)| \leq \tilde ce^{-\tilde\alpha t}.
\end{equation}
To do so, first recall %~\cite{Perko}
that $X_v(t):=\nabla\psi(v,t)$ is the solution of the ODE
\begin{align*}
 \partial_t X_v(t) &= \nabla b(\psi(v,t))X_v(t) \\
                   &= AX_v(t)+C_v(t)X_v(t)\qquad\forall t\geq0, \\
 X_v(0)            &= I,
\end{align*}
where $A:=\nabla b(x)$ and $C_v(t):=\nabla b(\psi(v,t))-A$. Since $\lim_{t\to\infty}\psi(v,t)=x$ uniformly for $v\in\bar B_\eta(w)$ by \eqref{stable exp decay 1} and \eqref{T step into Beps}, we have $\lim_{t\to\infty}C_v(t)=0$ uniformly for $v\in\bar B_\eta(w)$, and so \eqref{grad decay b} is a straight forward generalization of the proof of \cite[Thm.~6.3]{Verhulst} (where now one has to keep track of the uniformity of all estimates in $v$).\\[.2cm]
(ii) \vspace{-.87cm}
\begin{align}
\quad\ \Skp{\nf_s(w)}{b(w)}
   &=\partial_t f_s(\psi(w,t))\big|_{t=0}
    =\lim_{h\to0}\tfrac 1h\big[f_s(\psi(w,h))-f_s(w)\big] \nonumber\\
   &=\lim_{h\to0}\frac1h\bigg[\int_0^\infty\big|b\big(\psi(w,t+h)\big)\big|\,dt
                            -\int_0^\infty|b(\psi(w,t))|\,dt\bigg] \nonumber\\
   &=-\lim_{h\to0}\frac{1}{h}\int_0^h|b(\psi(w,t))|\,dt 
    =-|b(w)| \label{fs diff calculation}
\end{align}\\[.1cm]
%\begin{align}
%\dot\psi(w,t)&=\partial_\tau\psi(w,t+\tau)|_{\tau=0}=\partial_\tau\psi(\psi(w,\tau),t)|_{\tau=0}
%  \nonumber\\
% &= \nabla\psi(\psi(w,\tau),t)\dot\psi(w,\tau)|_{\tau=0}=\nabla\psi(w,t)b(w)\hspace{-.4cm} \label{fs diff calculation 1}\\[.3cm]
% \hspace{-.8cm}\stackrel{\eqref{diff integrand b}}{\Longrightarrow}\quad 
%\Skp{\nabla f_s(w)}{b(w)}
%   &= \int_0^\infty\frac{b\big(\psi(w,t)\big)^T \nabla b\big(\psi(w,t)\big) \dot\psi(w,t)}{|b(\psi(w,t))|}\,dt \nonumber\\
%   &= \int_0^\infty\partial_t\big|b(\psi(w,t))\big|\,dt
%    = \big|b(\psi(w,t))\big|\Big|_{t=0}^\infty \label{fs diff calculation 2} \\
%   &= |b(x)|-|b(w)|=-|b(w)| \nonumber
%\end{align}
%
\noindent(iii) \vspace{-.67cm}
\begin{equation} \label{proof lower bound}
  f_s(w) \geq \bigg|\int_0^\infty\dot\psi(w,t)\,dt\bigg| = \Big|\psi(w,t)\big|_{t=0}^{\infty}\Big|=|x-w|.
\end{equation}
(iv) We set $\tilde a:=\max_{v\in\bar B_{c\eps}(x)}\frac{|b(v)|}{|v-x|}$ and use \eqref{stable exp decay 1} to find that for $\forall w\in\bar B_\eps(x)$ we have
\[ f_s(w) \leq \tilde a\int_0^\infty|\psi(w,t)-x|\,dt \leq c\tilde a|w-x|\int_0^\infty e^{-\alpha t}\,dt = \frac{c\tilde a}{\alpha}|w-x|. \]
Since $\frac{f_s(w)}{|w-x|}$ is continuous on $K\setminus B_\eps(x)$ by part (i), \eqref{fs norm est 1 upper} holds with
\[ \cD:=\max\!\bigg\{\frac{c\tilde a}{\alpha},\ \max_{w\in K\setminus B_\eps(x)}\frac{f_s(w)}{|w-x|}\bigg\}. \qedhere \]
\end{proof}

\end{appendices}

\newpage
\part{Proof of a Technical Lemma} \label{part superprop}

\section{Proof of Lemma \ref{superprop} \texorpdfstring{--}{-} Main Arguments} \label{superprop main section}

Since the case in which $x$ is an attractor or a repellor of $b$ was already proven in Section \ref{subsec proof prop 4}, let us now consider the case in which $x$ is a saddle point of~$b$. We assume that all the conditions of Proposition \ref{blm prop 2}~(ii) for $x$ to have weak minimizers are fulfilled, i.e.~that $\nabla b(x)$
%\[ A:=\nabla b(x) \]
has only eigenvalues with nonzero real parts, and that there exist admissible manifolds $M_i$, $i\in I:=\{1,\dots,m\}$, such that \eqref{main condition} is fulfilled.

Our proof is structured as follows. In Section \ref{superprop proof step 1} we review some details of the Stable Manifold Theorem, make several definitions and choose some constants to prepare for the estimates to come. In Section~\ref{superprop proof step 2} we use Lemma \ref{move M lemma} to modify the given admissible manifolds $M_i$ in such a way that they obtain certain additional properties. Finally, in Section~\ref{superprop proof step 3} we define the functions $f_i$ explicitly and prove that they have the desired properties. The proofs of various technical statements in this chapter are deferred to Appendix \ref{proof appendix} in order to not interrupt the flow of the main arguments, and it is recommended to skip those proofs on first reading.

\subsection{Setting Things Up} \label{superprop proof step 1}

By our assumption on $\nabla b(x)$ we can write
\begin{equation} \label{A def}
   A:=\nabla b(x)=R\begin{pmatrix}P&0\\0&Q\end{pmatrix}R^{-1}
\end{equation}
for some matrices $R\in\R^{n\times n}$, $P\in\R^{n_s\times n_s}$ and $Q\in\R^{n_u\times n_u}$, where $n_s,n_u\in\N$ fulfill $n_s+n_u=n$, and where all the eigenvalues of $P$ have negative real parts and all those of $Q$ have positive real parts.

Let $M_s^{loc}$ and $M_u^{loc}$ be the \textit{local} stable and unstable manifolds of $b$ at the point $x$, respectively, as given by the Stable Manifold Theorem (see e.g.~\cite[Sec.~2.7]{Perko} or \cite[Sec.\,13.4]{CL}). These are $C^1$-manifolds of dimension $n_s$ and $n_u$, respectively, which for some constant $a_0>0$ with $\bar B_{a_0}(x)\subset D$ have the following properties \eqref{Ms Mloc relation}-\eqref{G def} which are explained in more detail in Appendix~\ref{Msloc ps comments sec}. Note that each of the properties involving $a_0$ remains valid if $a_0$ is decreased.

First, $M_s^{loc}$ and $M_u^{loc}$ are related to the \textit{global} stable and unstable manifolds $M_s$ and $M_u$ defined in \eqref{Ms def}-\eqref{Mu def} via the equations
\begin{equation}
  M_s=\psi\big(M_s^{loc},(-\infty,0]\big)
  \qquad\text{and}\qquad
  M_u=\psi\big(M_u^{loc},[0,\infty)\big), \label{Ms Mloc relation}
\end{equation}
so that in particular $M_s^{loc}\q\subset\q M_s$ and $M_u^{loc}\q\subset\q M_u$. On the other hand, we~have
\begin{subequations}
\begin{align}
  \forall w\in\bar B_{a_0}(x)\setminus M_s^{loc}\ \,\exists t>0\colon\ \ &\psi(w,t)\notin\bar B_{a_0}(x),   \label{stable man prop} \\
  \forall w\in\bar B_{a_0}(x)\setminus M_u^{loc}\ \,\exists t<0\colon\ \ &\psi(w,t)\notin\bar B_{a_0}(x).   \label{unstable man prop} 
\end{align}
\end{subequations}
Furthermore,
\begin{equation}  \label{compact Msloc parts}
  M_s^{loc}\cap\bar B_{a_0}(x) \quad \text{ and }\quad
  M_u^{loc}\cap\bar B_{a_0}(x) \quad \text{ are compact,}
\end{equation}
and by choosing $M_s^{loc}$ and $M_u^{loc}$ sufficiently small we may assume that
\begin{equation} \label{Msloc Muloc intersection}
  M_s^{loc}\cap M_u^{loc}=\{x\}
\end{equation}
and that
\begin{align} 
  &\te_0:=\sup\Big\{\,\skp{y_s}{y_u}\ \Big|\ 
   |y_s|=|y_u|=1;\ 
    y_s\in T_{w_s}M_s^{loc},\ y_u\in T_{w_u}M_u^{loc}\ \nonumber\\*
  &\hspace*{5.65cm}
    \text{for some}\ w_s\in M_s^{loc},\ w_u\in M_u^{loc}\,\Big\}\in[0,1). \nonumber\\
  & \label{sup skp}
\end{align}
During the proof of the Stable Manifold Theorem we
learn
%\hyperlink{my footnote}{\footnotemark[\value{footnote}]}
how to construct a function $p_s\in C^1\big(\bar B_{a_0}(x),M_s^{loc}\big)$\footnote{By this we mean that $p_s$ is the restriction to $\bar B_{a_0}(x)$ of a $C^1$-function that is defined on a larger open ball.} that projects $\bar B_{a_0}(x)$ along $T_xM_u^{loc}$ onto $M_s^{loc}$, i.e.~one~has
\begin{align}
  \forall v\in \bar B_{a_0}(x)\colon\ \ &p_s(v)-v\in T_xM_u^{loc}, \label{as prop 0}\\
  \forall v\in M_s^{loc}\cap\bar B_{a_0}(x)\colon\ \ &p_s(v)=v. \label{as prop 1}
\end{align}
For $\forall v\in\bar B_{a_0}(x)$ and $\forall t\in\R$ the function
\begin{equation} \label{chiv def}
  \chi_s^v(t):=\psi(p_s(v),t)
\end{equation}
fulfills\footnote{See \cite[Appendix~4]{Verhulst} for a quick derivation~of~\eqref{projection equation}.}
\begin{equation} \label{projection equation}
  \chi_s^v(t) = x + U_t(v-x) + \int_0^tU_{t-\tau}g(\chi_s^v(\tau))\,d\tau - \int_t^\infty V_{t-\tau}g(\chi_s^v(\tau))\,d\tau,
\end{equation}
where we define
\pb
\begin{align}
   \qquad
   U_t:=R\begin{pmatrix}e^{tP}&0\\0&0\end{pmatrix}\!R^{-1},
   \qquad%\text{and}\quad\,
   V_t:=R\begin{pmatrix}0&0\\0&e^{tQ}\end{pmatrix}\!R^{-1}\quad\ \ 
                 &\forall t\in\R, \label{Ut Vt semigroup def} \\*
   g(w):=b(w)-A(w-x) \hspace{2.8cm}&\forall w\in D. \label{G def}
\end{align}
Similarly, there exists a function $p_u\in C^1\big(\bar B_{a_0}(x),M_u^{loc}\big)$ that projects $\bar B_{a_0}(x)$ along $T_xM_s^{loc}$ onto $M_u^{loc}$, and the function $\chi_u^v(t):=\psi(p_u(v),t)$ fulfills a relation analogous to \eqref{projection equation}.\\[.2cm]
Let us now adjust Definition~\ref{fs fu def} and Lemma~\ref{distance function} to the present situation where $x$ is a saddle point.

\begin{definition} \label{fs fu def saddle}
Let $x\in D$ be such that $b(x)=0$ and that all the eigenvalues of the matrix $\nabla b(x)$ have nonzero real part.
Then we define the functions $f_s\colon M_s\to[0,\infty)$ and $f_u\colon M_u\to[0,\infty)$ by
\begin{subequations}
\begin{align}
  f_s(w) &:= \int_0^\infty|b(\psi(w,t))|\,dt = \int_0^\infty|\dot\psi(w,t)|\,dt,\qquad w\in M_s, \\*
  f_u(w) &:= \int_{-\infty}^0\!|b(\psi(w,t))|\,dt = \int_{-\infty}^0\!|\dot\psi(w,t)|\,dt,\qquad w\in M_u.
\end{align}
\end{subequations}
\end{definition}

\begin{lemma} \label{distance function saddle}
The functions $f_s$ and $f_u$ of Definition \ref{fs fu def saddle} are finite-valued and have the following properties:\\[.2cm]
\begin{tabular}{rl}
(i)  & \begin{minipage}[t]{11.30cm}
For $\forall w\in M_s$, the function $t\mapsto f_s(\psi(w,t))$ is non-increasing\linebreak (decreasing if $w\neq x$) and $C^1$, with $\partial_t f_s(\psi(w,t))=-|b(\psi(w,t))|$;\\[.15cm]
for $\forall w\in M_u$, the function $t\mapsto f_u(\psi(w,t))$ is non-decreasing\linebreak (increasing if $w\neq x$) and $C^1$, with $\partial_t f_u(\psi(w,t))=+|b(\psi(w,t))|$.
\end{minipage} \\[1.8cm]
\,\,(ii) & \begin{minipage}[t]{11.38cm}
\vspace{-.87cm}
\begin{subequations}
\begin{eqnarray}
\hspace{-6.05cm}
\forall w\in M_s\hspace{1pt}\colon &\!\!\!\hspace{2.5pt}f_s(w)\geq|w-x|, \label{fs est 1 saddle lower} \\
\hspace{-6.05cm}
\forall w\in M_u\colon &\!\!\!f_u(w)\geq|w-x|. \label{fs est 2 saddle lower}
\end{eqnarray}
\end{subequations}
\end{minipage} \\[.8cm]
\end{tabular}
Furthermore, after decreasing $a_0>0$ sufficiently, we have the following:\\[.15cm]
\begin{tabular}{rl}
(iii) & \begin{minipage}[t]{11.38cm}
There exist functions $\tilde f_s,\tilde f_u\in C\big(\bar B_{a_0}(x),[0,\infty)\big)$ that are $C^1$ on $\bar B_{a_0}(x)\setminus\{x\}$ such that
\vspace{-.15cm}
\begin{subequations}
\begin{eqnarray}
\hspace*{-4.5cm}
\forall w\in M_s^{loc}\cap\bar B_{a_0}(x)\colon\ \,\hspace{1pt}
 f_s(w)&\!\!\!=\tilde f_s(w), \label{fs C1 extension} \\
\hspace*{-4.5cm}
\forall w\in M_u^{loc}\cap\bar B_{a_0}(x)\colon\ \,
 f_u(w)&\!\!\!\!=\tilde f_u(w). \label{fu C1 extension}
\end{eqnarray}
\end{subequations}
\end{minipage} \\[2cm]
(iv) & \begin{minipage}[t]{11.38cm}
There $\exists\cI\geq1$ such that
\vspace{-.15cm}
\begin{subequations}
\begin{eqnarray}
\hspace*{-3.75cm}
\forall w\in M_s^{loc}\cap\bar B_{a_0}(x)\colon\ \,\hspace{1pt}
 f_s(w)&\!\!\!\leq \cI |w-x|, \label{fs est 1 saddle upper} \\
\hspace*{-3.75cm}
\forall w\in M_u^{loc}\cap\bar B_{a_0}(x)\colon\ \,
 f_u(w)&\!\!\!\!\leq \cI |w-x|. \label{fs est 2 saddle upper}
\end{eqnarray}
\end{subequations}
\end{minipage}
\end{tabular}\\[-.1cm]
\nopagebreak
\end{lemma}
\begin{proof}
See Appendix \ref{fs fu appendix b}.
\end{proof}
\noindent Now consider for $\forall a>0$ the level sets
\[ \Msa:=f_s^{-1}(\{a\}) \quad\text{and}\quad \Mua:=f_u^{-1}(\{a\}), \]
which by \eqref{fs est 1 saddle lower}-\eqref{fs est 2 saddle lower} and because of $f_s(x)=f_u(x)=0$ fulfill
\begin{equation} \label{Ms Mu in ball}
  \forall a>0\colon\ \  \Msa\cup\Mua\subset\bar B_a(x)\setminus\{x\}.
\end{equation}
We will now continue to decrease $a_0>0$ to make our construction in Sections~\ref{superprop proof step 2} and \ref{superprop proof step 3} work. First, we have the following.
\begin{lemma} \label{Msa compact lemma}
We can decrease $a_0>0$ so much that for $\forall a\in(0,a_0]$
\begin{subequations}
\begin{align}
  \Msa &\text{ and $f_s^{-1}\big([0,a_0]\big)$ are compact subsets of $M_s^{loc}$,}
        \label{Ms Mu in Mloc}\\
  \Mua &\text{ and $f_u^{-1}\big([0,a_0]\big)$ are compact subsets of $M_u^{loc}$,}
        \label{Ms Mu in Mloc 2}
\end{align}
\end{subequations}\\[-1.05cm]
\begin{equation} \label{Msa flow = Ms}
  \psi(\Msa,\R)=M_s\setminus\{x\}, \quad\ 
  \psi(\Mua,\R)=M_u\setminus\{x\},
\end{equation}
and that in the two-dimensional case \emph($n=2$\emph) the sets $M_s^a$ and $M_u^a$ each consist of exactly two points.
\end{lemma}
\begin{proof}
See Appendix \ref{Ms compact appendix}.
\end{proof}
Second, since $b(x)=0$, by Remark~\ref{on manifold b neq0} we have $x\notin M_i$ for $\forall i\in I$, i.e.~$f_{M_i}(x)\neq 0$, and so we can make $a_0>0$ so small that
\begin{equation} \label{fMi neq0}
  \forall i\in I\ \  \forall w\in\bar B_{a_0}(x)\colon\ \  f_{M_i}(w)\neq 0.
\end{equation}
In fact, using the notation
\[
  I^+:=\big\{i\in I\,\big|\,f_{M_i}(x)>0\big\} \qquad\text{and}\qquad I^-:=\big\{i\in I\,\big|\,f_{M_i}(x)<0\big\},
\]
we have $I^+\cup I^-=I$, and \eqref{fMi neq0} and the continuity of the functions $f_{M_i}$ imply
\begin{subequations}
\begin{align}
  \forall i\in I^+\ \, \forall w\in\bar B_{a_0}(x)\colon&\ \ f_{M_i}(w)>0, \label{fMi>0} \\
  \forall i\in I^-\ \, \forall w\in\bar B_{a_0}(x)\colon&\ \ f_{M_i}(w)<0. \label{fMi<0}
\end{align}
\end{subequations}
Third, since $\nabla b(x)$ is an invertible matrix, the function $b$ is locally invertible at $x$ by the Inverse Function Theorem, and its local inverse is $C^1$ as well. Since $b(x)=0$, we can thus decrease $a_0>0$ so much that
\begin{equation} \label{b bounds}
  \exists\dC,\dB>0\ \,\forall w\in\bar B_{a_0}(x)\colon\quad \dC|b(w)| \leq |w-x| \leq \dB|b(w)|.
\end{equation}
In particular, we have
\begin{equation} \label{b neq 0}
  \forall w\in\bar B_{a_0}(x)\setminus\{x\}\colon\ \ b(w)\neq0.
\end{equation}
Fourth, observe the following refined version of the triangle inequality.
\begin{lemma} \label{triangle inequality lemma}
$\forall\te\in[0,1)\;\exists d\in(0,1)\;\forall v,w\in\Rn\colon$
\begin{equation} \label{triangle ineq eq}
  \skp{v}{w}\leq\te|v||w|\quad\Rightarrow\quad|v+w|
    \leq\max\!\big\{|v|,|w|\big\}+d\min\!\big\{|v|,|w|\big\}
\end{equation}
\end{lemma}
\begin{proof}
See Appendix \ref{triangle inequality sec}.
\end{proof}
Let $\dD\in(0,1)$ be the constant $d$ given by Lemma~\ref{triangle inequality lemma} that corresponds to the value $\te=\te_0\in[0,1)$ defined in \eqref{sup skp}, let $\dA,\alpha>0$ such that
\begin{equation} \label{semigroup estimate}
  \forall t\geq0\colon\ |U_t|\leq \dA e^{-\alpha t}
  \qquad\text{and}\qquad
  \forall t\leq0\colon\ |V_t|\leq \dA e^{\alpha t},
\end{equation}
and choose $\kappa>0$ so small that
\begin{equation} \label{eps estimates}
  \frac{2\dA\kappa}{\alpha}\leq\frac12,
  \quad
  \bigg[(|A|+\kappa)\frac{8\dB\dA}{\alpha}+2\dB\bigg]\kappa \leq \tfrac14(1-\dD)
  \quad\text{and}\quad
  8\dB\dA\kappa\leq1.
\end{equation}
Then since the function $g$ defined in \eqref{G def} is $C^1$ and fulfills $\nabla g(x)=0$, we can further decrease $a_0>0$ so much that $\forall w\in\bar B_{a_0}(x)\colon\ |\nabla g(w)|\leq\kappa$. As a consequence, we have
\begin{align}
  \forall w_1,w_2\in\bar B_{a_0}(x)\colon\quad&|g(w_1)-g(w_2)|\leq\kappa|w_1-w_2|, \label{G est 1}\\
\intertext{and (taking $w_2=x$ and using $g(x)=0$) thus in particular}
  \forall w\in\bar B_{a_0}(x)\colon\quad&|g(w)|\leq\kappa|w-x|. \label{G est 2}
\end{align}
This completes our definition of $a_0$. Now since $x\in M_s^{loc}\cap M_u^{loc}$, by \eqref{as prop 1} we have $p_s(x)=p_u(x)=x$, and so we can choose $a_1\in(0,a_0]$ so small that
\begin{equation}
  p_s(\bar B_{a_1}(x))\cup p_u(\bar B_{a_1}(x))\subset\bar B_{a_0}(x). \label{ps to Ba0}
\end{equation}
% In particular, we have
%\begin{equation} \label{fs pu cont}
%  \text{\qquad$f_s\circ p_s=\tilde f_s\circ p_s$\qquad and\qquad $f_u\circ p_u=\tilde f_u\circ p_u$\qquad on\ \,$\bar B_{a_1}(x)$,}
%\end{equation}
%where $\tilde f_s$ and $\tilde f_u$ are the functions given in Lemma \ref{distance function saddle}~(iii).
%
\begin{lemma}  \label{Hartman-Grobman Lemma}
We can decrease $a_1>0$ so much that
$\forall\eta>0\ \,\exists\mu>0\colon$ \\[.2cm]
(i) all the flowlines starting from a point $w\in\bar B_\mu(x)\setminus M_s^{loc}$ will leave $B_{a_1}(x)$ at some time $T_1(w)>0$ as $t\to\infty$, and we have
\begin{equation} \label{HG inclusion 1}
  \psi\big(w,[0,T_1(w)]\big)\subset N_\eta\big(M_u^{loc}\cap\bar B_{a_1}(x)\big)\cap\bar B_{a_1}(x);
\end{equation}
(ii) all the flowlines starting from a point $w\in\bar B_\mu(x)\setminus M_u^{loc}$ will leave $B_{a_1}(x)$ at some time $T_2(w)<0$ as $t\to-\infty$, and we have
\begin{equation} \label{HG inclusion 2}
  \psi\big(w,[T_2(w),0]\big)\subset N_\eta\big(M_s^{loc}\cap\bar B_{a_1}(x)\big)\cap\bar B_{a_1}(x).
\end{equation}
\end{lemma}
\begin{proof}
See Appendix \ref{Hartman-Grobman sec}.
The lemma is obtained from the linear case $b(w)=A(w-x)$ by applying the Hartman-Grobman-Theorem \cite[p.119]{Perko}.
\end{proof}

\begin{definition}
For $\forall i\in I$ we denote by~$z_i$ and~$t_i$ the functions that Lemma~\ref{man2func lemma} associates to the admissible manifolds $M_i$.
\end{definition}

It remains to choose one last sufficiently small constant, $\tilde a>0$.
To prepare, the next lemma groups the points $w\in\Msa\cup\Mua\subset (M_s\cup M_u)\setminus\{x\}\subset\bigcup_{i=1}^m\psi(M_i,\R)$ (here we used the condition \eqref{main condition}) according to the index $i$ such that $w\in\psi(M_i,\R)$, and it gives us a bound on $|t_i(w)|$.

\begin{lemma} \label{compact hitting partition}
$\forall a\in(0,a_0]$ $\exists\!$ compact $K_1^a,\dots,K_m^a\subset D$ $\exists\eta_a,T_a>0$ such that
\begin{gather}
\bigcup_{i\in I^+}K_i^a=\Msa \qquad\text{and}\qquad \bigcup_{i\in I^-}K_i^a=\Mua, \label{Ki +- union} \\[.2cm]
\forall i\in I\colon\ \ \bar N_{\eta_a}(K_i^a)\subset \psi(M_i,[-T_a,T_a]).\label{Ki nbhd}
\end{gather}
In the two-dimensional case we can use the sets
\begin{subequations}
\begin{align}
  K_i^a &= \psi(M_i,\R)\cap\Msa
  \quad\text{for $i\in I^+$}, \label{2D Kia} \\
  K_i^a &= \psi(M_i,\R)\cap\Mua
  \quad\text{for $i\in I^-$}.
\end{align}
\end{subequations}
\end{lemma}
\begin{proof}
See Appendix \ref{compact hitting partition sec}.
\end{proof}
\noindent Now let us define the compact set
\begin{equation} \label{K def}
  K:=\bar B_{a_0}(x)\cup\bigcup_{i=1}^m\psi\big(M_i,[-T_{a_0},T_{a_0}]\big).
\end{equation}
By Remark~\ref{on manifold b neq0} no point in $M_i$ and thus also in $\psi(M_i,\R)$ has zero drift, and using \eqref{b neq 0} we thus find that the set $b^{-1}\big(\Rn\setminus\{0\}\big)\cup\{x\}$ is open and contains~$K$. Therefore we can choose $\tilde a>0$ so small that
\begin{gather}
0<\tilde a<a_1\leq a_0, \label{a relations} \\
\bar N_{2\tilde a}(K)\subset  b^{-1}\big(\Rn\setminus\{0\}\big)\cup\{x\} \subset D. \label{a tilde def}
\end{gather}
Finally, in the two-dimensional case ($n=2$) we decrease $\tilde a>0$ at this point as described on pages \pageref{prop (vi) Step 2}-\pageref{prop (vi) Step 3 end} (\textit{Steps 2-3} of our proof of Lemma \ref{superprop}~(vi)). We emphasize that our construction on those pages will not make use of anything we do beyond this point, and that the sole reason for postponing this step is to not unnecessarily distract the reader now with further details.
This completes our preparation process.

\subsection{Modification of the Admissible Manifolds} \label{superprop proof step 2}

We begin the second part of our proof with the definition of the sets $\hMtas$ and $\hMtau$.

\begin{lemma} \label{Msc cap Msc-hat lemma}
There exists a $\rho_0>0$ such that the compact sets
\begin{align}
   \hMtas := p_s^{-1}(\Msan)\cap\bar N_{\rho_0}(\Msan)
   &\quad\ \ \text{and}\quad\ \ 
   \hMtau := p_u^{-1}(\Muan)\cap\bar N_{\rho_0}(\Muan) \label{Msa Msu def} \\
\intertext{fulfill}
   & \nonumber \\[-.85cm]
   \hMtas\cap M_s=\Msan
   &\quad\ \ \text{and}\quad\ \ 
   \hat M_u^{\tilde a}\cap M_u=\Muan. \label{Msc cap Msc-hat}
\end{align}
\end{lemma}
\begin{proof}
See Appendix \ref{Msc cap Msc-hat sec}.
\end{proof}
\noindent Note that since $p_s$ and $p_u$ are only defined on $\bar B_{a_0}(x)$, we have
\begin{equation} \label{Msc in ball}
  \hMtas\subset\bar B_{a_0}(x)  \qquad\text{and}\qquad \hMtau \subset\bar B_{a_0}(x). 
\end{equation}
Our goal in this section is to use Lemma \ref{move M lemma} to turn the admissible manifolds $M_i$ into new ones, $M_i'$, whose union covers $\hMtas\cap\bar N_\rho(M_s^{\tilde a})$ and $\hMtau\cap\bar N_\rho(M_u^{\tilde a})$ for some sufficiently small $\rho>0$, see \eqref{Mi' def} and \eqref{Mi' cover}.
The essential ingredients for defining the functions $\beta_i$ needed for Lemma \ref{move M lemma} are the functions given by the following Lemma. Observe the resemblance with Lemma~\ref{man2func lemma}.

\begin{lemma} \label{zs ts lemma}
There exist open sets $D_s\supset M_s\setminus\{x\}$ and $D_u\supset M_u\setminus\{x\}$ and functions $z_s\in C^1\big(D_s,\hat M_s^{\tilde a}\big)$, $t_s\in C^1(D_s,\R)$, $z_u\in C^1\big(D_u,\hat M_u^{\tilde a}\big)$ and\linebreak $t_u\in C^1(D_u,\R)$ such that
\begin{subequations}
\begin{align}
\hspace{1.36cm}
  \forall w\in D_s\colon\ \ &\psi\big(z_s(w),t_s(w)\big)=w,
 \label{ts zs property 1} \\*
  \forall w\in D_u\colon\ \ &\psi\big(z_u(w),t_u(w)\big)=w,
\end{align}
\end{subequations}\\[-1.4cm]
\begin{subequations}
\begin{align}
\hspace{-1.4cm}
   \forall w\in D_s\cap\hat M_s^{\tilde a}\colon\ \ &z_s(w)=w, \label{zsx=x}\\*[-.05cm]
\hspace{-1.4cm}
   \forall w\in D_u\cap\hat M_u^{\tilde a}\colon\ \ &z_u(w)=w.
\end{align}
\end{subequations}
Furthermore, $z_s$ and $z_u$ are constant on the flowlines of $b$, i.e.~we have
\begin{subequations}
\begin{align}
&\forall w\in D_s\hspace{1pt}\ \,\forall t\in\R\colon\quad\ \psi(w,t)\in D_s\hspace{1pt}
   \ \ \Rightarrow\ \ \hspace{1pt} z_s(\psi(w,t))=z_s(w), \label{zs cst on flowlines} \\
&\forall w\in D_u\ \,\forall t\in\R\colon\quad\ \psi(w,t)\in D_u
   \ \ \Rightarrow\ \  z_u(\psi(w,t))=z_u(w).
\end{align}
\end{subequations}
\end{lemma}
\begin{proof}
See Appendix \ref{zs ts lemma sec}. The proof resembles the one of Lemma \ref{man2func lemma},
with the additional difficulty that now our target manifolds $\hat M_s^{\tilde a}$ and $\hat M_u^{\tilde a}$ are not admissible, and so a single flowline might intersect them more than once.
\end{proof}
\begin{remark} \label{zs K remark} We may assume that
\begin{subequations}
\begin{align}
  &\forall i\in I^+\colon\ \ K_i^{\tilde a }= z_s(K_i^{a_0}), \label{Ki tilde a def 1}\\
  &\forall i\in I^-\colon\ \ K_i^{\tilde a }= z_u(K_i^{a_0}). \label{Ki tilde a def 2}
\end{align}
\end{subequations}
\end{remark}
\begin{proof}
See Appendix \ref{zs K remark app}.
\end{proof}

\noindent The next lemma provides us with sets $G_i$ that we will need momentarily.

\begin{lemma} \label{Ui supset lemma}
For $\forall i\in I$ there exists an open set $G_i\subset D$ such that
\begin{equation} \label{Ui def}
  \hspace{-1.05cm}\forall i\in I\ \ \hspace{-1pt}\colon \ \ G_i\supset\psi(K_i^{\tilde a},[-T_{\tilde a},T_{\tilde a}]);
\end{equation}
\vspace{-.90cm}
\begin{subequations}
\begin{align}
  \forall i\in I^+\colon \ \ G_i\cap f^{-1}_{M_i}\big([0,\infty)\big)\hspace{.3cm}&\subset N_{\tilde a}(K), \label{Ui prop 1} \\
  \forall i\in I^-\colon \ \ G_i\cap f^{-1}_{M_i}\big((-\infty,0]\big)&\subset N_{\tilde a}(K).
\end{align}
\end{subequations}
\end{lemma}
\begin{proof}
See Appendix \ref{Ui supset lemma sec}.
\end{proof}
Now let some $i\in I$ be given. Assuming for the moment that $i\in I^+$, we have $K_i^{\tilde a}\subset M_s^{\tilde a}$ by \eqref{Ki +- union} and thus $\psi(K_i^{\tilde a},[-T_{\tilde a},T_{\tilde a}])\subset \psi(M_s^{\tilde a},\R)=M_s\setminus\{x\}\subset D_s$ by \eqref{Msa flow = Ms} and the choice of $D_s$ in Lemma \ref{zs ts lemma}, and combining this with \eqref{Ui def} we find that
\begin{equation} \label{Kai VU 1}
  \psi(\Kati,[-T_{\tilde a},T_{\tilde a}])\subset D_s\cap G_i.
\end{equation}
Since $\Kati\subset\Msan\subset\hMtas$ by \eqref{Ki +- union} and \eqref{Msc cap Msc-hat}, \eqref{zsx=x} and \eqref{zs cst on flowlines} imply that $\forall w\in\psi(\Kati,[-T_{\tilde a},T_{\tilde a}])\colon\ z_s(w)\in\Kati$, and since $z_s$ is continuous
% on the compact set $\psi(\Kati,[-T_{\tilde a},T_{\tilde a}])$,
there is an open set $W_i$ with
\begin{equation} \label{comp open pre}
  \psi(\Kati,[-T_{\tilde a},T_{\tilde a}])\subset W_i\subset D_s\cap G_i
\end{equation}
that is so small that
\[ \forall w\in W_i\colon\quad z_s(w)\in N_{\eta_{\tilde a}}(\Kati)\subset\psi(M_i,[-T_{\tilde a},T_{\tilde a}]), \]
where in the last step we used \eqref{Ki nbhd}. In particular,
\begin{equation} \label{ti<T}
  \forall w\in W_i\colon\quad z_s(w)\in\psi(M_i,\R) \text{\ \ and\ \ }  t_i(z_s(w))\in [-T_{\tilde a},T_{\tilde a}].
\end{equation}
Furthermore, since $\Kati$ and $[-T_{\tilde a},T_{\tilde a}]$ are compact and $W_i$ is open,
because of \eqref{comp open pre} we can choose a $\rho\in(0,\rho_0]$ small enough that
\begin{equation} \label{comp open}
  \psi\big(\bar N_\rho(\Kati),[-T_{\tilde a},T_{\tilde a}]\big)\subset W_i\subset D_s\cap G_i.
\end{equation}
Finally, we let $\nu_i\in C^1(D,[0,1])$ be a function with $\supp(\nu_i)\subset W_i$ such that 
\begin{equation} \label{gamma=1}
\forall w\in\psi\big(\bar N_\rho(\Kati),[-T_{\tilde a},T_{\tilde a}]\big)\colon\ \ \nu_i(w)=1
\end{equation}
and define
\begin{equation}\label{beta i def}
  \beta_i(w):=
    \begin{cases}
      \nu_i(w)t_i(z_s(w)) &\text{if $w\in W_i$,} \\
      0                              &\text{if $w\in D\setminus W_i$},
    \end{cases}
\end{equation}
which is well-defined by \eqref{ti<T}. Then $\beta_i\in C^1(D,\R)$, and by Lemma \ref{move M lemma} the set $M_i'$, defined by
\begin{equation} \label{Mi' def}
  M_i':=\psi_{\beta_i}(M_i,1),
\end{equation}
is an admissible manifold again.

If $i\in I^-$ then an analogous strategy for defining $M_i'$ can be applied (with $M_s^{\tilde a}$, $D_s$ and $z_s$ replaced by $M_u^{\tilde a}$, $D_u$ and $z_u$, respectively), and the relations \eqref{Kai VU 1}-\eqref{Mi' def} hold in their correspondingly modified form. In this way we can define $M_i'$ successively for $\forall i\in I$, at each step potentially decreasing the previously obtained $\rho$ (this is possible since \eqref{comp open}-\eqref{gamma=1} remain true if $\rho$ is decreased).

\begin{definition}
For $\forall i\in I$ we denote by~$z_i'$ and~$t_i'$ the functions that Lemma~\ref{man2func lemma} associates to the admissible manifolds $M_i'$.
\end{definition}

\noindent The new admissible manifolds $M_i'$ have the following properties.

\begin{lemma}[Properties of $M_i'$] \label{Mi' properties}
\hspace*{0cm}\\[.35cm]
\hspace*{.14cm}
\begin{tabular*}{13cm}[t]{rl}
  (i) &  \hspace{-.12cm}$\forall i\in I\hspace{7pt}\colon\ \,\psi(M_i',\R)=\psi(M_i,\R)$.\\[.25cm]
  (ii) &
\end{tabular*}\\[-1.03cm]
\begin{subequations}
\begin{align}
  & \hspace{-4.84cm} \forall i\in I^+\colon\ \,\hMtas\cap\bar N_\rho(\Kati)\subset M_i', \label{Msan NKati in Mi} \\
  & \hspace{-4.84cm}\forall i\in I^-\colon\ \,\hMtau\cap\bar N_\rho(\Kati)\subset M_i',
\end{align}
\end{subequations}\\[-1.05cm]
\begin{equation} \label{Mi' cover}
  \hspace{.05cm}
  \hMtas\cap\bar N_\rho(\Msan)\subset \bigcup_{i\in I^+}M_i',  \quad\ \,
  \hMtau\cap\bar N_\rho(\Muan)\subset \bigcup_{i\in I^-}M_i'.
\end{equation}\\[-.15cm]
\begin{tabular*}{13cm}[t]{rl}
 (iii) & $\forall i\in I\ \forall w\in\bar N_{\tilde a}(M_i')\setminus\{x\}\colon\ \,b(w)\neq0$. \\[.5cm]
(iv) & %\vspace{-.9cm}
\end{tabular*}\\[-1.13cm]
\begin{subequations}
\begin{align}
  &\hspace{-2.53cm}\forall i\in I^+\ \forall z\in M_i'\colon\quad\int_0^\infty   |b(\psi(z,\tau))|\,d\tau\geq\tilde a, \label{mod man dist} \\
  &\hspace{-2.53cm}\forall i\in I^-\ \forall z\in M_i'\colon\quad\int_{-\infty}^0   |b(\psi(z,\tau))|\,d\tau\geq\tilde a.
\end{align}
\end{subequations}\\[2cm]\markit %[-.35cm]
\hspace*{-.04cm}
\begin{tabular*}{13cm}[t]{rl}
(v) & \begin{minipage}[t]{11.1cm}
\hspace{-.05cm}
For $\forall\tilde\rho\in(0,\rho]$ $\exists\mu>0$ such that
\end{minipage}
\end{tabular*}
\vspace{-.6cm}
\begin{align}
   \hspace{.35cm}\forall w\in\bar B_\mu(x)\setminus M_u^{loc}\ \ \exists t<0\colon\,\ & \psi(w,t)\in\hMtas, \label{flow into hMtas} \\*
&\big|p_s(\psi(w,t))-\psi(w,t)\big|\leq\tilde\rho; \label{projection not far} \\[.2cm]
   \hspace{.35cm}\forall w\in\bar B_\mu(x)\setminus M_s^{loc}\ \ \exists t>0\colon\,\ & \psi(w,t)\in\hMtau, \\*
&\big|p_u(\psi(w,t))-\psi(w,t)\big|\leq\tilde\rho.
\end{align}
\hspace*{-.04cm}
\begin{tabular*}{13cm}[t]{rl}
(vi) & \begin{minipage}[t]{11.1cm}
\hspace{-.2cm}
There $\exists\eps>0$ such that
\end{minipage}
\end{tabular*}
\vspace{-.6cm}
\begin{align}
   \hspace{.7cm}\forall w\in\bar B_\eps(x)\setminus M_u^{loc}\ \,\exists i\hspace{1pt}\in I^+\colon\,\ &w\in\psi(M_i',(0,\infty)), \label{tiw pos}\\*
&z_i'(w)\in\hMtas, \label{zi in hMtas} \\*
&\psi\big(w,[-t_i'(w),0]\big)\subset\bar B_{a_0}(x); \label{flow segment in B 1}\\[.2cm]
   \hspace{.7cm}\forall w\in\bar B_\eps(x)\setminus M_s^{loc}\ \,\exists j\in I^-\colon\,\ &w\in\psi(M_j',(-\infty,0)), \label{tiw neg}\\*
&z_j'(w)\in\hMtau,\\*
&\psi\big(w,[0,-t_j'(w)]\big)\subset\bar B_{a_0}(x). \label{flow segment in B 2}
\end{align}
\end{lemma}

\begin{proof}
In part (ii) we will only show \eqref{Msan NKati in Mi} and the first relation in \eqref{Mi' cover}, in parts (iii)-(iv) we will only treat the case $i\in I^+$, and in parts (v)-(vi) we will only show the properties \eqref{flow into hMtas}-\eqref{projection not far} and \eqref{tiw pos}-\eqref{flow segment in B 1}, respectively. The remaining properties can then be shown analogously. Throughout the proofs of parts (i)-(iv) we will repeatedly make use of the following three properties:\\[.2cm]
\noindent First, for any given $\beta\in C^1(D,\R)$ we have
\begin{align}
  \psi_{\beta}(w,t)&\hspace{3pt}=\psi(w,s_w(t))\quad\forall w\in D\ \,\forall t\in\R,\quad \text{where} \label{psi beta b rep} \\
 s_w(t)&:=\int_0^t\beta(\psi_{\beta}(w,\tau))\,d\tau. \label{sw def}
\end{align}
Indeed, if $\beta(w)=0$ then $\psi_\beta(w,t)=w$ for $\forall t\in\R$, and \eqref{psi beta b rep}-\eqref{sw def} are trivial. Otherwise we have for $\forall\tau\in s_w(\R)$
\begin{align*}
  \tfrac{d}{d\tau}\psi_{\beta}\big(w,s_w^{-1}(\tau)\big)
    &= \dot\psi_{\beta}\big(w,s_w^{-1}(\tau)\big)\cdot(s_w^{-1})'(\tau) \\
    &= (\beta b)\big(\psi_{\beta}(w,s_w^{-1}(\tau))\big)\cdot\big[\beta\big(\psi_{\beta}(w,s_w^{-1}(\tau))\big)\big]^{-1} \\
    &= b\big(\psi_{\beta}(w,s_w^{-1}(\tau))\big)
\end{align*}
and $\psi_{\beta}\big(w,s_w^{-1}(0)\big)=\psi_{\beta}(w,0)=w$, showing that $\psi_{\beta}\big(w,s_w^{-1}(\tau)\big)=\psi(w,\tau)$. We will for $\forall i\in I$ denote by $s_w^i$ the functions defined in \eqref{sw def}, with $\beta=\beta_i$.\\[.2cm]
\noindent Second, since by \eqref{beta i def} the functions $\beta_i$ vanish outside of $W_i$, we have
\begin{align}
  \forall w\in W_i\ \,\forall \tau\in\R\colon\ \ &\psi_{\beta_i}(w,\tau)\in W_i. \label{Vi inv} \\
\intertext{Since by \eqref{comp open pre} we have $W_i\subset D_s$ for $\forall i\in I^+$, and since by \eqref{zs cst on flowlines} $z_s$ is constant on the flowlines of $b$ and thus on those of $\beta_i b$, this implies that}
  \forall i\in I^+\ \,\forall w\in W_i\ \,\forall \tau\in\R\colon\ \ &z_s(\psi_{\beta_i}(w,\tau))=z_s(w). \label{zs inv on Vi}\\
\intertext{Third, let $i\in I^+$ and $u\in W_i\subset D_s$. Since $z_s$ takes values in $\hMtas\subset\bar B_{a_0}(x)$ by \eqref{Msc in ball}, we have $f_{M_i}(z_s(u))\!>\!0$ by \eqref{fMi>0}, and by \eqref{ti<T}, \eqref{prime lemma eq 2} and \eqref{beta i def} this implies that}
   \forall i\in I^+\ \,\forall u\in W_i\colon\ \ &t_i(z_s(u))\in(0,T_{\tilde a}], \label{ti>0 on Vi} \\
   \forall i\in I^+\ \,\forall u\in D\colon\ \ &\beta(u)\in[0,T_{\tilde a}]. \label{beta u in int}
\end{align}
Now let us begin with the proofs of the properties (i)-(vi).\\[.2cm]
(i) Since \eqref{psi beta b rep} implies $\psi_{\beta_i}(w,1)\in\psi(w,\R)$ for $\forall w\in D$, we have by \eqref{Mi' def}
\[
  \psi(M_i',\R)
   =\psi\big(\psi_{\beta_i}(M_i,1),\R\big)
   \subset\psi\big(\psi(M_i,\R),\R\big)
   =\psi(M_i,\R).
\]
for $\forall i\in I$. The reverse inclusion follows analogously from the equation $M_i=\psi_{\beta_i}(M_i',-1)$.\\[.2cm]
(ii) Let $i\in I^+$ and $w\in\hMtas\cap\bar N_\rho(\Kati)$. Then for $\forall t\in[-1,0]$ we have $|s_w^i(t)|\leq T_{\tilde a}$ by \eqref{sw def} and \eqref{beta u in int}, and thus $\psi_{\beta_i}(w,t)=\psi(w,s_w^i(t))\in\psi\big(\bar N_\rho(\Kati),[-T_{\tilde a},T_{\tilde a}]\big)\subset W_i\subset D_s$ by \eqref{psi beta b rep} and \eqref{comp open}.
By \eqref{Vi inv}, \eqref{gamma=1}, \eqref{beta i def}, \eqref{zs inv on Vi} and \eqref{zsx=x} we therefore have
\begin{equation*}
  \forall t\in[-1,0]\colon\ \ 
  \beta_i(\psi_{\beta_i}(w,t))
    = t_i\big(z_s(\psi_{\beta_i}(w,t))\big)
    = t_i(z_s(w))
    = t_i(w),
\end{equation*}
which implies $s_w^i(-1)=-t_i(w)$ by \eqref{sw def}.
We can now conclude that $\psi_{\beta_i}(w,-1)=\psi(w,s_w^i(-1))=\psi(w,-t_i(w))=z_i(w)$, i.e.~$w=\psi_{\beta_i}(z_i(w),1)\in\psi_{\beta_i}(M_i,1)=M_i'$.
This shows \eqref{Msan NKati in Mi}, and taking the union over all $i\in I^+$ on both sides and using \eqref{Ki +- union} implies the first relation in \eqref{Mi' cover}.\\[.2cm]
(iii) Let $i\in I^+$. It is enough to show
\begin{equation} \label{small claim 2}
  M_i'\subset N_{\tilde a}(K)
\end{equation}
since then by \eqref{a tilde def} we can conclude that
\begin{equation*}
  \bar N_{\tilde a}(M_i')\subset \bar N_{2\tilde a}(K)\subset
b^{-1}(\Rn\setminus\{0\})\cup\{x\},
\end{equation*}
which is (iii). To show \eqref{small claim 2}, let $w\in M_i'$. By definition of $M_i'$ in \eqref{Mi' def} and by \eqref{psi beta b rep} there is a $v\in M_i$ such that $w=\psi_{\beta_i}(v,1)=\psi(v,s_v^i(1))$, which implies that $w\in\psi(M_i,\R)$ and $t_i(w)=s_v^i(1)$. \\[.15cm]
\textit{Case 1:} $\beta_i(v)=0$. Then $\psi_{\beta_i}(v,t)=v$ for $\forall t\in\R$, so $w=v\in M_i\subset K\subset N_{\tilde a}(K)$ by \eqref{K def}. \\[.15cm]
\textit{Case 2:} $\beta_i(v)\neq0$. Then $\beta_i(\psi_{\beta_i}(v,t))\neq0$ for $\forall t\in\R$, and in particular $\beta_i(w)\neq0$. Therefore we have $w\in W_i\subset G_i$ by \eqref{beta i def} and \eqref{comp open pre}. Furthermore, we have $t_i(w)=s_v^i(1)\geq0$ by \eqref{sw def} and \eqref{beta u in int}, and thus $f_{M_i}(w)\geq0$ by \eqref{prime lemma eq 2}. By \eqref{Ui prop 1} we can now conclude that
\begin{equation}
  w\in G_i\cap f_{M_i}^{-1}\big([0,\infty)\big)\subset N_{\tilde a}(K)
\end{equation}
also in this case, completing the proof of \eqref{small claim 2} and thus of (iii). \\[.2cm]
(iv) Again let $i\in I^+$, and suppose that \eqref{mod man dist} is not true, i.e.~that $\exists z\in M_i'$ such that
\begin{equation} \label{wrong ass tilde a}
  \int_0^\infty|\dot\psi(z,\tau)|\,d\tau<\tilde a.
\end{equation}
Then for $s,t\geq T>0$ we have
\begin{equation*}
\big|\psi(z,t)-\psi(z,s)\big|=\bigg|\int_s^t\dot\psi(z,\tau)\,d\tau\bigg|
  \leq\int_T^\infty|\dot\psi(z,\tau)|\,d\tau\to0
\end{equation*}
as $T\to\infty$, and thus $\exists \tilde x\in\bar D\colon\,\lim_{t\to\infty}\psi(z,t)=\tilde x$. Furthermore, since 
\begin{equation*}
 \tilde a > \int_0^\infty|\dot\psi(z,\tau)|\,d\tau\geq\bigg|\int_0^\infty\dot\psi(z,\tau)\,d\tau\bigg|
 =\Big|\lim_{t\to\infty}\psi(z,t)-\psi(z,0)\Big|=|\tilde x-z|
\end{equation*}
and $z\in M_i'$, \eqref{small claim 2} and \eqref{a tilde def} tell us that $\tilde x\in N_{\tilde a}(M_i')\subset N_{2\tilde a}(K)\subset D$. Therefore the limit
\begin{equation} \label{psi dot lim}
\lim_{t\to\infty}\dot\psi(z,t)=\lim_{t\to\infty}b(\psi(z,t))=b(\tilde x)
\end{equation}
exists, and since also the limit $\lim_{t\to\infty}\psi(z,t)$ exists, the limit \eqref{psi dot lim} must be zero, i.e.~$b(\tilde x)=0$. Since $\tilde x\in N_{\tilde a}(M_i')$, part (iii) of this lemma thus says that $\tilde x=x$, i.e.~$\lim_{t\to\infty}\psi(z,t)=x$. In other words, we have $z\in M_s$, and our assumption \eqref{wrong ass tilde a} can be rephrased as $f_s(z)<\tilde a$. \\[.15cm]
Now since $z\in M_i'=\psi_{\beta_i}(M_i,1)$, there $\exists v\in M_i$ such that $z=\psi_{\beta_i}(v,1)$. \\[.15cm]
\textit{Case 1:} $\beta_i(v)=0$. Then $\psi_{\beta_i}(v,t)=v$ for $\forall t\in\R$ and thus $z=v\in M_i$. But on the other hand by \eqref{fs est 1 saddle lower} we have $|z-x|\leq f_s(z)<\tilde a< a_0$, which by \eqref{fMi neq0} implies that $f_{M_i}(z)\neq0$, contradicting $z\in M_i$. \\[.15cm]
\textit{Case 2:} $\beta_i(v)\neq0$. Then by \eqref{beta i def} we have $v\in W_i$, and \eqref{Vi inv} and \eqref{ti>0 on Vi} imply that $t_i\big(z_s(\psi_{\beta_i}(v,\tau))\big)>0$ for $\forall \tau\in\R$. Therefore by \eqref{sw def}, \eqref{beta i def} and \eqref{zs inv on Vi} we have
\begin{equation*}
  s_v^i(1) =    \int_0^1\beta_i(\psi_{\beta_i}(v,\tau))\,d\tau
         \leq \int_0^1t_i\big(z_s(\psi_{\beta_i}(v,\tau))\big)\,d\tau
         = t_i(z_s(v)).
\end{equation*}
Since $\psi(v,-t_s(v))=z_s(v)$ and $v\in M_i$ implies that $t_i(z_s(v))=-t_s(v)$, this means that $s_v^i(1)\leq-t_s(v)$, and so using Lemma~\ref{distance function saddle}~(i) we find that
\begin{eqnarray}
  \tilde a&>&f_s(z) = f_s(\psi_{\beta_i}(v,1))=f_s\big(\psi(v,s_v^i(1))\big) \nonumber \\
   &\geq& f_s\big(\psi(v,-t_s(v))\big)=f_s(z_s(v)). \label{tilde a contradiction}
\end{eqnarray}
Finally, since $z=\psi_{\beta_i}(v,1)=\psi(v,s_v^i(1))$ and $z\in M_s$, we have
\[ z_s(v)=\psi\big(v,-t_s(v)\big)=\psi\big(z,-s_v^i(1)-t_s(v)\big)\in M_s, \]
and since $z_s(v)\in\hMtas$ by definition of $z_s$, \eqref{Msc cap Msc-hat} thus implies that $z_s(v)\in M_s^{\tilde a}$. But this means that $f_s(z_s(v))=\tilde a$, contradicting \eqref{tilde a contradiction}.\\[.2cm]
%
%(iv) First, note that \eqref{psi beta b rep} implies $\psi_{\beta_i}(w,-1)\subset\psi(w,\R)$ for $\forall w\in D$. Therefore we have for $\forall i\in I$ that
%\[
%  \psi(M_i,\R)
%   =\psi\big(\psi_{\beta_i}(M_i',-1),\R\big)
%   \subset\psi\big(\psi(M_i',\R),\R\big)
%   =\psi(M_i',\R).
%\]
%(In fact, one can show equality.) Therefore \eqref{main condition} implies
%\begin{equation*}
%  (M_s\cup M_u)\setminus\{x\}\subset\bigcup_{i=1}^m\psi(M_i,\R)\subset\bigcup_{i=1}^m\psi(M_i',\R).
%\end{equation*}
%
%
%(iv) Since $\Msan$ is a compact subset of the open set $B_{a_0}(x)\cap D_s$ by \eqref{Ms Mu in Mloc}, \eqref{Ms Mu in ball} and \eqref{Ds Du region}, there is a $\tilde\rho\in(0,\rho]$ such that
%\begin{equation} \label{rho tilde prop}
%  (XXX)\quad\bar N_{\tilde \rho}(\Msan)\subset B_{a_0}(x)\cap D_s.
%\end{equation}
(v) Let $\tilde\rho\in(0,\rho]$ be given. Since by \eqref{as prop 1} and \eqref{compact Msloc parts} we have $p_s(w)-w=0$ on the compact set $M_s^{loc}\cap\bar B_{a_0}(x)$, there is an $\eta>0$ such that
\begin{equation} \label{leq tilde rho}
  \forall w\in\bar N_\eta\big(M_s^{loc}\cap\bar B_{a_0}(x)\big)\cap\bar B_{a_0}(x)\colon\ \ 
  |p_s(w)-w|\leq\tilde\rho.
\end{equation}
Now define the function $g(w):=f_s(p_s(w))\geq0$ for $\forall w\in\bar B_{a_1}(x)$, which is continuous by \eqref{ps to Ba0} and Lemma \ref{distance function saddle}~(iii).
The compact set $g^{-1}\big([0,\tilde a]\big)\cap\partial B_{a_1}(x)$ is disjoint from the compact set $M_s^{loc}\cap\bar B_{a_0}(x)$, since any point $w$ that is contained in both sets would have to fulfill $\tilde a\geq g(w)=f_s(p_s(w))=f_s(w)\geq|w-x|=a_1$ (where we used \eqref{as prop 1} and \eqref{fs est 1 saddle lower}), contradicting \eqref{a relations}. Thus we can decrease $\eta>0$ so much that
\begin{equation} \label{ball surface disjoint}
  \big[g^{-1}\big([0,\tilde a]\big)\cap\partial B_{a_1}(x)\big]\cap\bar N_\eta\big(M_s^{loc}\cap\bar B_{a_0}(x)\big)=\varnothing.
\end{equation}
Applying Lemma \ref{Hartman-Grobman Lemma} to this choice of $\eta$, we obtain a $\mu>0$ such that all the flowlines starting from some point $w\in\bar B_\mu(x)\setminus M_u^{loc}$ will leave $B_{a_1}(x)$ at some time $T_2(w)<0$ as $t\to-\infty$, and \eqref{HG inclusion 2} holds. Since $g(x)=f_s(p_s(x))=f_s(x)=0$ by \eqref{as prop 1}, we can decrease $\mu>0$ so much that
\begin{equation} \label{F at x}
  \forall w\in\bar B_\mu(x)\colon\ \,g(w)<\tilde a.
\end{equation}
Now let $w\in\bar B_\mu(x)\setminus M_u^{loc}$. By \eqref{HG inclusion 2} and \eqref{a relations} we have $\psi(w,T_2(w))\in\bar N_\eta\big(M_s^{loc}\cap\bar B_{a_0}(x)\big)$, and thus $\psi(w,T_2(w))\notin g^{-1}\big([0,\tilde a]\big)\cap\partial B_{a_1}(x)$ by \eqref{ball surface disjoint}. Since $\psi(w,T_2(w))\in\partial B_{a_1}(x)$ by definition of $T_2(w)$, this means that\linebreak $\psi(w,T_2(w))\notin g^{-1}\big([0,\tilde a]\big)$, i.e.~$g\big(\psi(w,T_2(w))\big)>\tilde a$. Since $g(\psi(w,0))<\tilde a$ by \eqref{F at x}, there $\exists t\in(T_2(w),0)$ such that $\tilde a=g(\psi(w,t))$ $=f_s\big(p_s(\psi(w,t))\big)$, i.e.
\begin{equation} \label{ps psi in Msan}
 p_s(\psi(w,t))\in\Msan
\end{equation}
and thus $\psi(w,t)\in p_s^{-1}(\Msan)$. Furthermore, by \eqref{leq tilde rho}, \eqref{HG inclusion 2} and \eqref{a relations} we have $\big|p_s(\psi(w,t))-\psi(w,t)\big|\leq\tilde\rho$, i.e.\ \eqref{projection not far}, and thus $\psi(w,t)\in\bar N_{\tilde\rho}(\Msan)\subset\bar N_\rho(\Msan)\subset\bar N_{\rho_0}(\Msan)$ by \eqref{ps psi in Msan}. Combining the last two statements and using \eqref{Msa Msu def}, we find that $\psi(w,t)\in p_s^{-1}(\Msan)\cap\bar N_{\rho_0}(\Msan)=\hMtas$, which is \eqref{flow into hMtas}.\\[.2cm]
(vi) Continuing the construction of part (v) (e.g.\ for the choice $\tilde\rho:=\rho$), we have found that $\psi(w,t)\in\hMtas\cap\bar N_\rho(\Msan)$. Therefore by \eqref{Mi' cover} there $\exists i\in I^+$ such that $z:=\psi(w,t)\in M_i'$ and thus $w=\psi(z,-t)\in\psi\big(M_i',(0,\infty)\big)$, with $z_i'(w)=z=\psi(w,t)\in\hMtas$ and $t_i'(w)=-t$. Finally, since $[-t_i'(w),0]=[t,0]\subset[T_2(w),0]$, \eqref{HG inclusion 2} implies that $\psi\big(w,[-t_i'(w),0]\big)\subset\bar B_{a_1}(x)\subset\bar B_{a_0}(x)$.
This shows that \eqref{tiw pos}-\eqref{flow segment in B 1} hold for $\eps:=\mu$.
\end{proof}

\subsection{Definition of the Functions \texorpdfstring{$f_i$}{fi}; Proof of Their Properties}
\label{superprop proof step 3}
We are now ready to define the functions $f_i$ that we are looking for.
\begin{definition}
We define the functions $f_1,\dots,f_m\colon D\to[0,\infty)$ as follows: If $i\in I^+$ then we define
\begin{subequations}
\begin{equation} \label{fi def 1}
  f_i(w):=
    \begin{cases}
      \tilde a & \text{if\ \ $f_{M_i'}(w)<0$,} \\
      \max\!\Big\{0,\;\tilde a-\int_0^{t_i'(w)}\big|b\big(\psi(z_i'(w),\tau)\big)\big|\,d\tau\Big\}
               & \text{if\ \ $w\in\psi(M_i',[0,\infty))$,} \\
      0        & \text{else;}
    \end{cases}
\end{equation}
and if $i\in I^-$ then we define
\begin{equation} \label{fi def 2}
  f_i(w):=
    \begin{cases}
      \tilde a & \text{if\ \ $f_{M_i'}(w)>0$,} \\
      \max\!\Big\{0,\;\tilde a-\int_{t_i'(w)}^0\big|b\big(\psi(z_i'(w),\tau)\big)\big|\,d\tau\Big\}
               & \text{if\ \ $w\in\psi(M_i',(-\infty,0])$,} \\
      0        & \text{else.}
    \end{cases}
\end{equation}
\end{subequations}
\end{definition}
These functions are well-defined: If $w\in\psi(M_i',[0,\infty))$ then $t_i'(w)\geq0$ and thus $f_{M_i'}(w)\geq0$ by \eqref{prime lemma eq 2}; and similarly, if $w\in\psi(M_i',(-\infty,0])$ then $f_{M_i'}(w)\leq0$. Note that the two integrals in \eqref{fi def 1}-\eqref{fi def 2} are the lengths of the flowline segments between $w$ and $z_i'(w)$.

Now let $\eps>0$ be the value given to us in Lemma \ref{Mi' properties}~(vi), and let us reduce it if necessary so that $\eps\leq\tilde a$.

We will now show that the functions $f_i$ fulfill the properties (i)-(vii) of Lemma \ref{superprop}. The properties (ii)-(iv) and (vi) will in fact be proven for~$\tilde a$ instead of $\eps$, i.e.~we will show stronger statements than required (since $\tilde a\geq\eps$), and for that purpose we denote
\[ E_i':=f_i^{-1}\big((0,\tilde a)\big) \qquad\text{for $i\in I$.} \]
In parts (i)-(iv) and (vi) we will restrict ourselves to the case $i\in I^+$ (the proofs for the case $i\in I^-$ can be done analogously).
\paragraph{Proof of properties (i)-(iv).}
(i) Recalling \eqref{Mi' def} and the construction of $f_{M_i'}$ in the proof of Lemma \ref{move M lemma}, and using that $b(x)=0$, we find that
\begin{equation} \label{fMi'x>0}
  \forall i\in I^+\colon\quad f_{M_i'}(x)=f_{M_i}(\psi_{\beta_i}(x,-1))=f_{M_i}(x)>0. 
\end{equation}
Also, since by Remark~\ref{on manifold b neq0} $M_i'$ and thus also $\psi(M_i',\R)$ does not contain any points with zero drift, we have $x\notin\psi(M_i',[0,\infty))$.
Therefore $f_i(x)$ is defined by the third line in \eqref{fi def 1}, and so we have $f_i(x)=0$.\\[.2cm]
(ii) To show that the function $f_i$ traces the flowlines of~$b$ between the values~$0$ and $\tilde a$, we have to check the three properties in Definition \ref{tracing def}.\\[.2cm]
(ii.1) The definition of $f_i$ in \eqref{fi def 1} divides $D$ into three parts, let us call them $D_1,D_2$ and~$D_3$. To show that $f_i$ is continuous on $D$, we will show that $f_i$ is continuous on the closures in $D$ of each of the three parts, i.e.~on $\overline D_1^D$, $\overline D_2^D$ and $\overline D_3^D$.

First consider $D_1=f_{M_i'}^{-1}\big((-\infty,0)\big)$. For $\forall w\in\overline D_1^D\setminus D_1\subset f_{M_i'}^{-1}(\{0\})=M_i'$ we have $t_i'(w)=0$ by \eqref{man move remark eq}, and $f_i(w)$ is defined by the second line in \eqref{fi def 1}, so
\[ f_i(w) = \max\!\Big\{0,\;\tilde a-\int_0^0\big|b\big(\psi(z_i'(w),\tau)\big)\big|\,d\tau\Big\}=\max\{0,\tilde a\}=\tilde a. \]
This shows that $f_i$ is constant and thus continuous on $\overline D_1^D$.

Regarding $D_3$, observe that by \eqref{prime lemma eq 2} we have
$\psi(M_i',(-\infty,0))\subset$\linebreak $f_{M_i}^{-1}\big((-\infty,0)\big)$, and so we can write
\begin{align*}
D_3&:=D\setminus\big[f_{M_i'}^{-1}\big((-\infty,0)\big)\cup\psi\big(M_i',[0,\infty)\big)\big]\\
 &\hspace{3pt}=D\setminus\big[\underbrace{f_{M_i'}^{-1}\big((-\infty,0)\big)}_{\text{open}}
     \,\cup\,\underbrace{\psi\big(M_i',\R\big)}_{\text{open by L.\ref{man2func lemma}}}\big].
\end{align*}
This shows that $D_3$ is closed in $D$, i.e.~that $\overline D_3^D=D_3$, and so $f_i$ is constant and thus continuous also on $\overline D_3^D$.

It remains to show that $f_i$ is continuous on $\overline D_2^D$. Suppose that this were not the case. Then there would be a sequence $(w_n)_{n\in\N}\subset D_2=\psi(M_i',[0,\infty))$ that converges to some $w\in D$ and for which we have
\begin{equation} \label{fi ass}
  \limsup_{n\to\infty}\big|f_i(w_n)-f_i(w)\big|>0.
\end{equation}
Since $f_i|_{D_2}$ is continuous, we must have $w\notin D_2$.
By passing on to a subsequence we may assume that $z_i'(w_n)$ converges to some $z\in M_i'$ as $n\to\infty$ (since $M_i'$ is compact), and that $t_i'(w_n)$ converges to some $t\in[0,\infty]$ (since $t_i'(w_n)\geq0$ for $\forall n\in\N$).

Now if we had $t<\infty$ then letting $n\to\infty$ in the equation $w_n=\psi\big(z_i'(w_n),t_i'(w_n)\big)$ would tell us that $w=\psi(z,t)\in\psi(M_i',[0,\infty))=D_2$. Thus we have $t=\infty$, and with Fatou's Lemma and \eqref{mod man dist} we find
\begin{align*}
  \liminf_{n\to\infty}\int_0^{t_i'(w_n)}\!\big|b\big(\psi(z_i'(w_n),\tau)\big)\big|\,d\tau  
    &\geq \int_0^\infty\!\!\lim_{n\to\infty}\One_{\tau\in[0,t_i'(w_n)]}
          \big|b\big(\psi(z_i'(w_n),\tau)\big)\big|\,d\tau \\
    &= \int_0^\infty|b(\psi(z,\tau))|\,d\tau
     \geq \tilde a
\end{align*}\vspace{-.3cm}
\[ \Longrightarrow\quad \lim_{n\to\infty} f_i(w_n) = \lim_{n\to\infty}\max\!\Big\{0,\;\tilde a-\int_0^{t_i'(w_n)}\big|b\big(\psi(z_i'(w_n),\tau)\big)\big|\,d\tau\Big\}=0. \]

To find the value of $f_i(w)$, first note that for $\forall n\in\N$ we have $t_i'(w_n)\geq0$ and thus $f_{M_i'}(w_n)\geq0$ by \eqref{prime lemma eq 2}, and taking the limit $n\to\infty$ shows that $f_{M_i'}(w)\geq0$, i.e.\ $w\notin D_1$. Since also $w\notin D_2$, this shows that $f_i(w)$ is defined by the third line in \eqref{fi def 1}, so that $f_i(w)=0=\lim_{n\to\infty}f_i(w_n)$, in contradiction to \eqref{fi ass}.
This shows that $f_i$ is continuous on $\overline D_2^D$, and thus on all of~$D$.\\[.2cm]
(ii.2) To show that $f_i$ is $C^1$ on $E_i'=f_i^{-1}\big((0,\tilde a)\big)$, note that $E_i'\subset\psi(M_i',[0,\infty))$ by \eqref{fi def 1}, so that
\begin{equation} \label{fi on Ei'}
  \forall w\in E_i'\colon\quad
  f_i(w)
    = \tilde a - \int_0^{t_i'(w)}\big|b\big(\psi(z_i'(w),\tau)\big)\big|\,d\tau 
  \ \,\in(0,\tilde a)
\end{equation}
and thus
\begin{align}
  \nf_i(w)
    &= -\big|b\big(\psi(z_i'(w),t_i'(w))\big)\big|\nabla t_i'(w) \nonumber \\*
    & \hspace{.9cm}{}- \int_0^{t_i'(w)}\Big(\frac{b^T\nabla b}{|b|}\Big)\big(\psi(z_i'(w),\tau)\big)\nabla\psi(z_i'(w),\tau)\,d\tau\,\cdot\nabla\! z_i'(w) \nonumber \\
    &= -\big|b(w)\big|\nabla t_i'(w) \nonumber \\*
    & \hspace{.9cm}{}- \int_0^{t_i'(w)}\Big(\frac{b^T\nabla b}{|b|}\Big)\big(\psi(z_i'(w),\tau)\big)\nabla\psi(z_i'(w),\tau)\,d\tau\,\cdot\nabla\! z_i'(w) \label{grad fi formula}
\end{align}
for $\forall w\in E_i'$. The last term is well-defined and continuous in $w$ since $z_i'(w)\in M_i'$ implies that $b(z_i'(w))\neq0$ by Remark~\ref{on manifold b neq0} and thus $b\big(\psi(z_i'(w),\tau)\big)\neq0$ for $\forall\tau\in\R$.\\[.2cm]
(ii.3) Now using \eqref{grad fi formula}, \eqref{grad z grad t eq 1} and \eqref{grad z grad t eq 2}, we find for $\forall w\in E_i'$ that
\begin{align*}
  \skp{\nf_i(w)}{b(w)}
    &= -|b(w)|\underbrace{\skp{\nabla t_i'(w)}{b(w)}}_{=1} \\*
    & \hspace{.43cm}{}- \int_0^{t_i'(w)}\!\!\Big(\frac{b^T\nabla b}{|b|}\Big)\!\big(\psi(z_i'(w),\tau)\big)\nabla\psi(z_i'(w),\tau)\,d\tau\, \underbrace{\nabla\! z_i'(w)b(w)}_{=0} \\
    &= -|b(w)|.
\end{align*}
\textit{Remark:} For $i\in I^-$ we would obtain $\forall w \in E_i'\colon\,\skp{\nf_i(w)}{b(w)}=+|b(w)|$.
\\[.2cm]
(iii) By \eqref{fi on Ei'} we have for $\forall w\in E_i'$
\begin{align*}
  \tilde a
    &> \int_0^{t_i'(w)}\big|b\big(\psi(z_i'(w),\tau)\big)\big|\,d\tau 
     = \int_0^{t_i'(w)}|\dot\psi(z_i'(w),\tau)|\,d\tau \\
    &\geq \bigg|\int_0^{t_i'(w)}\dot\psi(z_i'(w),\tau)\,d\tau\bigg| 
     = \big|\psi\big(z_i'(w),t_i'(w)\big)-\psi(z_i'(w),0)\big| 
     = |w-z_i'(w)|
\end{align*}
and thus $w\in N_{\tilde a}(M_i')$, so that
\begin{equation} \label{Ei subset}
  E_i'\subset N_{\tilde a}(M_i')\subset N_{2\tilde a}(K)
\end{equation}
by \eqref{small claim 2}. Since $K$ is compact, this shows that $\bar E_i'$ is compact as well, with $\bar E_i'\subset\bar N_{2\tilde a}(K)\subset D$ by \eqref{a tilde def}. \\[.2cm]
(iv) By \eqref{Ei subset} we have $\bar E_i'\subset\bar N_{\tilde a}(M_i')$ and thus $\bar E_i'\setminus\{x\}\subset\bar N_{\tilde a}(M_i')\setminus\{x\}$, and so by Lemma \ref{Mi' properties} (iii) we have $\forall w\in\bar E_i'\setminus\{x\}\colon\ b(w)\neq0$.

\paragraph{Proof of property (v).}
Now let $F:=\max\{f_1,\dots,f_m\}$.
It suffices to show the estimate $F(w)\geq\cE|w-x|$ for $\forall w\in\bar B_\eps(x)\setminus (M_s^{loc}\cup M_u^{loc})$ since this set is dense in $\bar B_\eps(x)$ and since both $F$ and $|\cdot{}-x|$ are continuous by part (ii.1).

Let $w\in\bar B_\eps(x)\setminus(M_s^{loc}\cup M_u^{loc})$ be fixed.
%
%If $w=x$ then by part~(i) we have $F(x)=0=\cE|w-x|$, independently of our future choice of~$\cE$. Therefore let us assume that $w\neq x$ and thus $w\notin M_s^{loc}\cap M_u^{loc}$ by \eqref{Msloc Muloc intersection}.
%
%Suppose now that $w\in M_s^{loc}$. Then $w\in\bar B_\eps(x)\setminus M_u^{loc}$, and by \eqref{tiw pos}-\eqref{zi in hMtas} there $\exists i\in I^+$ such that $w\in\psi(M_i',(0,\infty))$, with $t_i(w)>0$ and $z_i(w)\in\hMtas$. Since also $z_i(w)=\psi(w,-t_i(w))\in M_s$, by \eqref{Msc cap Msc-hat} we have $z_i(w)\in\Msan$, i.e.
%\begin{equation} \label{fs zi w}
%   \tilde a
%     = f_s(z_i(w))
%     = \int_0^\infty\big|b\big(\psi(z_i(w),t)\big)\big|\,dt.
%\end{equation}
%Since because of \eqref{tiw pos} $f_i(w)$ is defined by the second line in \eqref{fi def 1}, \eqref{fs zi w} and \eqref{fs est 1 saddle lower} imply that
%\begin{align*}
%   F(w)
%     &\geq f_i(w)
%      \geq \tilde a-\int^{t_i(w)}_0\big|b\big(\psi(z_i(w),t)\big)\big|\,dt
%      = \int_{t_i(w)}^\infty\big|b\big(\psi(z_i(w),t)\big)\big|\,dt \\
%     &= \int_0^\infty\big|b\big(\psi\big(z_i(w),t_i(w)+t\big)\big)\big|\,dt
%      = \int_0^\infty\big|b(\psi(w,t))\big|\,dt
%      = f_s(w) \\
%     &\geq |w-x|.
%\end{align*}
%Analogously the estimate $F(w)\geq|w-x|$ can also be shown to hold under the assumption $\forall w\in M_u^{loc}$, and so the desired inequality $F(w)\geq\cE|w-x|$ holds for $\forall w\in\bar B_\eps(x)\cap(M_s^{loc}\cup M_u^{loc})$ and any choice $\cE\in(0,1]$.
%
%It remains to consider the case $w\in\bar B_\eps(x)\setminus (M_s^{loc}\cup M_u^{loc})$.
%
Then by Lemma~\ref{Mi' properties}~(vi) there exist $i\in I^+$ and $j\in I^-$ such that \eqref{tiw pos}-\eqref{flow segment in B 2} hold. We abbreviate $\Tm:=-t_i'(w)<0$, $\Tp:=-t_j'(w)>0$, and
\[
   \phi(t):=\psi(w,t) \qquad\text{for $\forall t\in\R$.}
\]
Because of \eqref{tiw pos} and \eqref{tiw neg}, $f_i(w)$ and $f_j(w)$ are defined by the second lines in \eqref{fi def 1} and \eqref{fi def 2}, respectively, and we can begin our estimate as follows
\begin{align}
F(w)
  &\geq \max\!\big\{f_i(w),f_j(w)\big\} \nonumber\\*
  &\geq\max\!\bigg\{
  \tilde a-\int^{t_i'(w)}_0\big|b\big(\psi(z_i'(w),t)\big)\big|\,dt,\ \,\tilde a-\int_{t_j'(w)}^0\big|b\big(\psi(z_j'(w),t)\big)\big|\,dt\bigg\}\nonumber\\
  &=\max\!\bigg\{\tilde a-\int_{-t_i'(w)}^0\big|b\big(\psi(z_i'(w),t_i'(w)+t)\big)\big|\,dt,\nonumber\\*
  &\hspace{4.95cm}\tilde a-\int_0^{-t_j'(w)}\big|b\big(\psi(z_j'(w),t_j'(w)+t)\big)\big|\,dt\bigg\}\nonumber\\
  &=\max\!\bigg\{\tilde a-\int_{-t_i'(w)}^0|b(\psi(w,t))|\,dt,\ \,
  \tilde a-\int^{-t_j'(w)}_0|b(\psi(w,t))|\,dt\bigg\}\nonumber\\
  &=\max\!\bigg\{\tilde a-\int_{\Tm}^0|\dot\psi(w,t)|\,dt,\ \,
  \tilde a-\int^{\Tp}_0|\dot\psi(w,t)|\,dt\bigg\}\nonumber\\
  &=\max\!\bigg\{\tilde a-\int_{\Tm}^0|\dot \phi|\,dt,\ \,
  \tilde a-\int^{\Tp}_0|\dot \phi|\,dt\bigg\}. \label{big est beginning}
\end{align}
We must now show that the last line in \eqref{big est beginning} is bounded below by $\cE|w-x|$ for some constant $\cE>0$. The trick will be to write
\begin{equation} \label{phi rep}
   \phi-x=(\phi_s-x)+(\phi_u-x)+r
\end{equation}
for some small remainder $r$ (which vanishes if $b$ is linear), where $\phi_s$ is a flowline in $M_s$ and $\phi_u$ is a flowline in $M_u$. The flowlines $\phi_s$ and $\phi_u$ are in several ways easier to deal with, mostly since we can apply $f_s$ and $f_u$ to them, respectively. \\[.2cm]
To define $\phi_s$ and $\phi_u$, first note that since $\phi(\Tm)=\psi(w,-t_i'(w))=z_i'(w)\in\hMtas$ by \eqref{zi in hMtas} and similarly $\phi(T_2)\in\hMtau$, by \eqref{Msa Msu def} we have
\begin{equation} \label{ws wu def}
  w_s:=p_s(\phi(\Tm))\in\Msan
  \qquad\text{and}\qquad
  w_u:=p_u(\phi(T_2))\in\Muan.
\end{equation}
We now define the functions $\phi_s\in C^1(\R,M_s)$, $\phi_u\in C^1(\R,M_u)$ and finally $r\in C^1(\R,\Rn)$ by
\begin{subequations}
\begin{align}
  \phi_s(t)&:=\psi(w_s,t-\Tm), \label{phis def} \\
  \phi_u(t)&:=\psi(w_u,t-\Tp) \label{phiu def}
\end{align}
\vspace{-.6cm}
\end{subequations}
\begin{equation} \label{r def}
\hspace{.51cm}\text{and}\qquad
  r(t)     :=\phi(t)-\phi_s(t)-\phi_u(t)+x
\end{equation}
for $\forall t\in\R$, i.e.\ \eqref{phi rep}, which fulfill
\begin{align}
  \phi_s(\Tm)=w_s
  &\qquad\text{and}\qquad
  \phi_u(\Tp)=w_u.
  \label{ws wu eq}
\end{align}
Note that for $\forall\tau\in\R$ we have
\begin{subequations}
\begin{align}
 \int_\tau^\infty|\dot\phi_s(t)|\,dt
   &=\int_0^\infty|\dot\phi_s(t+\tau)|\,dt
    =\int_0^\infty\big|b\big(\psi(w_s,t+\tau-\Tm)\big)\big|\,dt \nonumber\\
   &=\int_0^\infty\big|b\big(\psi(\phi_s(\tau),t)\big)\big|\,dt
    =f_s(\phi_s(\tau)), \label{a tilde int 1}\\
 \int_{-\infty}^\tau|\dot\phi_u(t)|\,dt
   &=\dots=f_u(\phi_u(\tau)), \label{a tilde int 2}
\end{align}
\end{subequations}
and thus by \eqref{ws wu def} and \eqref{ws wu eq} in particular
\begin{equation} \label{fs-phis fu-phiu = tilde a}
  \int_{\Tm}^\infty|\dot\phi_s(t)|\,dt=f_s(\phi_s(\Tm))=\tilde a
  \qquad\text{and}\qquad
  \int_{-\infty}^{\Tp}|\dot\phi_u|\,dt=f_u(\phi_u(\Tp))=\tilde a.
\end{equation}
Furthermore, by Lemma~\ref{distance function saddle}~(i)
\begin{subequations}
\begin{align}
  &\text{$f_s\circ\phi_s$\, is $C^1$ and non-increasing,} \label{fs phis dec} \\
%  \qquad\text{and}\qquad
  &\text{$f_u\circ\phi_u$ is $C^1$ and non-decreasing.} \label{fu phiu inc}
\end{align}
\end{subequations}
Thus, by \eqref{fs est 1 saddle lower}-\eqref{fs est 2 saddle lower}, \eqref{fs-phis fu-phiu = tilde a} and \eqref{fs phis dec}-\eqref{fu  phiu inc} we have
\begin{align}
  \forall t\geq \Tm\colon\quad &|\phi_s(t)\hspace{1pt}-x|\leq f_s(\phi_s(t))\hspace{2pt}\leq f_s(\phi_s(\Tm))\hspace{2pt}=\tilde a, \label{phis leq tilde a}\\
  \forall t\leq \Tp\colon\quad &|\phi_u(t)-x|\leq f_u(\phi_u(t))\leq f_u(\phi_u(\Tp))=\tilde a, \label{phiu leq tilde a}
\end{align}
which together with \eqref{Ms Mu in Mloc}-\eqref{Ms Mu in Mloc 2}, \eqref{flow segment in B 1} and \eqref{flow segment in B 2} implies
\begin{align}
  \phi_s([\Tm,\infty)) \subset M_s^{loc},\qquad\ \,
  \phi_u((-\infty,\Tp]) &\subset M_u^{loc}, \label{phis in Msloc}\\
  \phi([\Tm,\Tp])      \cup
  \phi_s([\Tm,\infty)) \cup
  \phi_u((-\infty,\Tp])
  &\subset \bar B_{a_0}(x). \label{can apply estimates}
\end{align}
The relation \eqref{can apply estimates} will be necessary to justify the use of various estimates that are only valid on $\bar B_{a_0}(x)$.

As another consequence, choosing $t=\Tm$ in \eqref{phiu leq tilde a} and using \eqref{fs-phis fu-phiu = tilde a} shows that $f_u(\phi_u(\Tm))\leq\tilde a=f_s(\phi_s(\Tm))$, and similarly we find that $f_s(\phi_s(\Tp))\leq f_u(\phi_u(\Tp))$. Therefore we have $f_u(\phi_u(\Tm))-f_s(\phi_s(\Tm))\leq0\leq f_u(\phi_u(\Tp))-f_s(\phi_s(\Tp))$, and thus there $\exists \bar t\in[\Tm,\Tp]$ such that
\begin{equation} \label{t1 def}
  f_u(\phi_u(\bar t\qb))=f_s(\phi_s(\bar t\qb)).
\end{equation}
Our next goal is to find small bounds on $\int_{\Tm}^{\Tp}|r|\,dt$ and $\int_{\Tm}^{\Tp}|\dot r|\,dt$. We begin by recalling Duhamel's formula, which says that
\begin{align}
  \!\!\phi(t) &= x+e^{tA}(w-x) + \int_0^te^{(t-\tau)A}g(\phi(\tau))\,d\tau \nonumber \\
          &= x+(U_t+V_t)(w-x) + \int_0^t(U_{t-\tau}+V_{t-\tau})g(\phi(\tau))\,d\tau
  \quad\forall t\in\R, \label{dAlembert}
\end{align}
where the matrix groups $(U_t)_{t\in\R}$ and $(V_t)_{t\in\R}$ are the ones defined in \eqref{Ut Vt semigroup def}.
Since $\phi(\Tm)\in\bar B_{a_0}(x)$ by \eqref{can apply estimates}, we can choose $v:=\phi(\Tm)$ in \eqref{chiv def}-\eqref{projection equation}, and since by \eqref{chiv def}, \eqref{ws wu def} and \eqref{phis def} we then have
$\chi_s^v(t)=\psi(p_s(v),t)=\psi\big(p_s(\phi(\Tm)),t\big)=\psi(w_s,t)=\phi_s(t+\Tm)$
for $\forall t\in\R$, \eqref{projection equation} tells us that
\begin{align*}
  \phi_s(t+\Tm) &= x+U_t(\phi(\Tm)-x) +\int_0^tU_{t-\tau}g(\phi_s(\tau+\Tm))\,d\tau \\*
                &\hspace{3.47cm} -\int_t^\infty V_{t-\tau}g(\phi_s(\tau+\Tm))\,d\tau
\end{align*}
for $\forall t\in\R$. We now replace $t$ by $t-\Tm$, use \eqref{dAlembert} to obtain an expression for $\phi(\Tm)$, and use that $U_{\tau_1}U_{\tau_2}=U_{\tau_1+\tau_2}$ and $U_{\tau_1}V_{\tau_2}=0$ for $\forall \tau_1,\tau_2\in\R$, to obtain
\begin{align}
\phi_s(t)
 &=x+U_{t-\Tm}\bigg[(U_{\Tm}+V_{\Tm})(w-x)
                  +\int_0^{\Tm}(U_{\Tm-\tau}+V_{\Tm-\tau})g(\phi(\tau))\,d\tau\bigg] \nonumber\\*
 &\hspace{.4cm} + \int_0^{t-\Tm}      U_{t-\Tm-\tau}g(\phi_s(\tau+\Tm))\,d\tau
                - \int_{t-\Tm}^\infty V_{t-\Tm-\tau}g(\phi_s(\tau+\Tm))\,d\tau \nonumber\\
 &= x+U_t(w-x) - \int_{\Tm}^0U_{t-\tau}g(\phi(\tau))\,d\tau \nonumber\\*
 &\hspace{.4cm} + \int_{\Tm}^t U_{t-\tau}g(\phi_s(\tau))\,d\tau
                 - \int_t^\infty V_{t-\tau}g(\phi_s(\tau))\,d\tau. \label{phi_s rep}\\
\intertext{Similarly, one can obtain the formula}
 \phi_u(t)
 &= x+V_t(w-x) + \int_0^{\Tp}V_{t-\tau}g(\phi(\tau))\,d\tau \nonumber \\*
 &\hspace{.4cm} - \int_t^{\Tp}     V_{t-\tau}g(\phi_u(\tau))\,d\tau
                  + \int_{-\infty}^t U_{t-\tau}g(\phi_u(\tau))\,d\tau. \label{phi_u rep}
\end{align}
Subtracting \eqref{phi_s rep} and \eqref{phi_u rep} from \eqref{dAlembert}, we thus obtain for $\forall t\in[\Tm,\Tp]$
\begin{align}
 r(t) &= \phi(t)-\phi_s(t)-\phi_u(t)+x \nonumber \\
      &= (U_t+V_t)(w-x)+\int_0^t(U_{t-\tau}+V_{t-\tau})g(\phi(\tau))\,d\tau \nonumber \\
      &\hspace{.4cm}-U_t(w-x) +\int_{\Tm}^0U_{t-\tau}g(\phi(\tau))\,d\tau
                    -\int_{\Tm}^tU_{t-\tau}g(\phi_s(\tau))\,d\tau\nonumber \\*
      &\hspace{7.33cm} +\int_t^\infty V_{t-\tau}g(\phi_s(\tau))\,d\tau\nonumber \\
      &\hspace{.4cm}-V_t(w-x) -\int^{\Tp}_0V_{t-\tau}g(\phi(\tau))\,d\tau
                     +\int^{\Tp}_tV_{t-\tau}g(\phi_u(\tau))\,d\tau\nonumber \\*
      &\hspace{7.2cm} -\int^t_{-\infty} U_{t-\tau}g(\phi_u(\tau))\,d\tau \nonumber \\
      &=\int_{\Tm}^{\Tp}\big(\One_{\tau<t}U_{t-\tau}-\One_{\tau\geq t}V_{t-\tau}\big)
         \big(g(\phi(\tau))-g(\phi_s(\tau))-g(\phi_u(\tau))\big)\,d\tau\nonumber \\*
      &\hspace{3.23cm}-\int_{-\infty}^{\Tm}U_{t-\tau}g(\phi_u(\tau))\,d\tau+\int_{\Tp}^\infty
         V_{t-\tau}g(\phi_s(\tau))\,d\tau \nonumber \\
      &=\int_{-\infty}^\infty\big(\One_{\tau<t}U_{t-\tau}-\One_{\tau\geq t}V_{t-\tau}\big)\Delta(\tau)\,d\tau, \label{r rep} \\
\intertext{where for $\forall \tau\in\R$ we define}
\Delta(\tau)\,\,
      &\!\!:= \One_{\Tm\leq \tau\leq \Tp}\cdot\big(g(\phi(\tau))-g(\phi_s(\tau))-g(\phi_u(\tau))\big)\nonumber \\*
  &\hspace{5.3cm}   -\One_{\tau<\Tm}g(\phi_u(\tau))-\One_{\tau>\Tp}g(\phi_s(\tau)).\nonumber 
\end{align}
Combining \eqref{r rep} with \eqref{semigroup estimate}, we obtain the estimate
\begin{equation} \label{first r est}
  |r(t)| \leq \dA\int_{-\infty}^\infty e^{-\alpha|t-\tau|}|\Delta(\tau)|\,d\tau \qquad\text{for $\forall t\in[\Tm,\Tp]$.}
\end{equation}
Now let $C_1\subset[\Tm,\Tp]$ and $C_2:=[\Tm,\Tp]\setminus C_1$ be two measurable sets to be chosen later, and let $C_1^-:=C_1\cup(-\infty,\Tm)$ and $C_2^+:=C_2\cup(\Tp,\infty)$.
Then we have for $\forall\tau\in\R$
\begin{align}
\Delta(\tau)
  &= \One_{\tau\in C_1}\cdot\big(g(\phi(\tau))-g(\phi_s(\tau))\big)
     +\One_{\tau\in C_2}\cdot\big(g(\phi(\tau))-g(\phi_u(\tau))\big) \nonumber\\*
  &\hspace{5.37cm}   -\One_{\tau\in C_1^-}g(\phi_u(\tau))-\One_{\tau\in C_2^+}g(\phi_s(\tau)),\nonumber\\
\intertext{and thus by \eqref{can apply estimates}, \eqref{G est 1}-\eqref{G est 2}, \eqref{r def} and \eqref{b bounds}}
|\Delta(\tau)|
  &\leq\One_{\tau\in C_1}\cdot\kappa\big|\underbrace{\phi(\tau)-\phi_s(\tau)}_{=r(\tau)+\phi_u(\tau)-x}\big|
      +\One_{\tau\in C_2}\cdot\kappa\big|\underbrace{\phi(\tau)-\phi_u(\tau)}_{=r(\tau)+\phi_s(\tau)-x}\big| \nonumber\\*
  &\hspace{3.9cm}    +\One_{\tau\in C_1^-}\cdot\kappa|\phi_u(\tau)-x|+\One_{\tau\in C_2^+}\cdot\kappa|\phi_s(\tau)-x| \nonumber\\
  &\leq\One_{\tau\in C_1}\cdot\kappa\big(|r(\tau)|+|\phi_u(\tau)-x|\big)
      +\One_{\tau\in C_2}\cdot\kappa\big(|r(\tau)|+|\phi_s(\tau)-x|\big) \nonumber\\*
  &\hspace{3.9cm}    +\One_{\tau\in C_1^-}\cdot\kappa|\phi_u(\tau)-x|+\One_{\tau\in C_2^+}\cdot\kappa|\phi_s(\tau)-x| \nonumber\\
  &\leq\kappa\Big(\One_{\tau\in[\Tm,\Tp]}|r(\tau)|+2\cdot\One_{\tau\in C_1^-}|\phi_u(\tau)-x|+2\cdot\One_{\tau\in C_2^+}|\phi_s(\tau)-x|\Big) \nonumber\\
  &\leq\kappa\Big(\One_{\tau\in[\Tm,\Tp]}|r(\tau)|+2\dB\cdot\One_{\tau\in C_1^-}\big|b(\phi_u(\tau))\big|+2\dB\cdot\One_{\tau\in C_2^+}\big|b(\phi_s(\tau))\big|\Big), \nonumber\\
  &=\kappa\Big(\One_{\tau\in[\Tm,\Tp]}|r(\tau)|+2\dB\cdot\One_{\tau\in C_1^-}|\dot \phi_u(\tau)|+2\dB\cdot\One_{\tau\in C_2^+}|\dot\phi_s(\tau)|\Big). \label {Delta estimate}
\end{align}
We can now use \eqref{first r est}, \eqref{Delta estimate} and the first estimate in \eqref{eps estimates} to obtain
\begin{align}
\int_{\Tm}^{\Tp}|r(t)|\,dt
  &\leq \dA\int_{\Tm}^{\Tp}dt\int_{-\infty}^\infty d\tau\, e^{-\alpha|t-\tau|}|\Delta(\tau)| \nonumber\\
  &\leq \dA\int_{-\infty}^\infty d\tau\,|\Delta(\tau)|\int_{-\infty}^\infty dt\,e^{-\alpha|t-\tau|} \nonumber\\
  &= \frac{2\dA}{\alpha}\int_{-\infty}^\infty |\Delta(\tau)|\,d\tau \nonumber\\
  &\leq\frac{2\dA\kappa}{\alpha}\bigg[\int_{\Tm}^{\Tp}|r|\,dt + 2\dB\int_{C_1^-}|\dot \phi_u|\,dt + 2\dB\int_{C_2^+}|\dot \phi_s|\,dt\bigg] \nonumber\\
  &\leq\frac12\int_{\Tm}^{\Tp}|r|\,dt + \frac{4\dB\dA\kappa}{\alpha}\bigg[\int_{C_1^-}|\dot \phi_u|\,dt + \int_{C_2^+}|\dot \phi_s|\,dt\bigg] \nonumber\\[.2cm]
\Longrightarrow\quad\int_{\Tm}^{\Tp}|r|\,dt
  &\leq\frac{8\dB\dA\kappa}{\alpha}
       \bigg[\int_{C_1^-}|\dot \phi_u|\,dt + \int_{C_2^+}|\dot \phi_s|\,dt\bigg]. \label{r int estimate}
\end{align}
To turn this into an estimate for $\int_{\Tm}^{\Tp}|\dot r|\,dt$, we start from the relation
\begin{align}
  \dot r
    &= \dot\phi - \dot\phi_s - \dot\phi_u \nonumber\\*
    &= b(\phi)-b(\phi_s)-b(\phi_u) \nonumber\\
    &= \big(A(\phi-x)+g(\phi)\big) - \big(A(\phi_s-x)+g(\phi_s)\big) - \big(A(\phi_u-x)+g(\phi_u)\big) \nonumber\\
    &= A(\phi-\phi_s-\phi_u+x)+\big(g(\phi)-g(\phi_s)-g(\phi_u)\big) \nonumber\\
    &= Ar+\Delta, \label{rdot eq}
\end{align}
where the last step is valid only on $[\Tm,\Tp]$. Using \eqref{rdot eq}, \eqref{Delta estimate}, \eqref{r int estimate} and the second estimate in \eqref{eps estimates}, we thus obtain
\begin{align}
\int_{\Tm}^{\Tp}|\dot r|\,dt
  &\leq |A|\int_{\Tm}^{\Tp}|r|\,dt + \int_{\Tm}^{\Tp}|\Delta|\,dt \nonumber \\
  &\leq(|A|+\kappa)\int_{\Tm}^{\Tp}|r|\,dt + 2\dB\kappa\int_{C_1}|\dot \phi_u|\,dt
    + 2\dB\kappa\int_{C_2}|\dot \phi_s|\,dt \nonumber \\
  &\leq\Big[(|A|+\kappa)\frac{8\dB\dA}{\alpha}+2\dB\Big]\kappa
       \bigg[\int_{C_1^-}|\dot \phi_u|\,dt+\int_{C_2^+}|\dot \phi_s|\,dt\bigg] \nonumber \\
  &\leq\tfrac14(1-\dD)\bigg[\int_{C_1^-}|\dot \phi_u|\,dt+\int_{C_2^+}|\dot \phi_s|\,dt\bigg]. \label{int r dot est}
\end{align}
Since by \eqref{phis in Msloc} we have $\phi_s([\Tm,\Tp])\subset M_s^{loc}$ and $\phi_u([\Tm,\Tp])\subset M_u^{loc}$ and thus also
\[
  \forall t\in[\Tm,\Tp]\colon\quad
    \dot\phi_s(t)\in T_{\phi_s(t)}M_s^{loc}
    \quad\text{and}\quad
    \dot\phi_u(t)\in T_{\phi_u(t)}M_u^{loc},
\]
\eqref{sup skp} tells us that
\[
   \forall t\in[\Tm,\Tp]\colon\quad
   \Skp{\dot\phi_s(t)}{\dot\phi_u(t)}\leq\te_0|\dot\phi_s(t)||\dot\phi_u(t)|.
\]
Therefore, if we choose
\begin{subequations}
\begin{align}
  C_1&:=\big\{t\in[\Tm,\Tp]\,\big|\,|\dot \phi_u(t)|\leq|\dot \phi_s(t)|\big\}, \label{C1 def}\\
  C_2&:=\big\{t\in[\Tm,\Tp]\,\big|\,|\dot \phi_u(t)|>|\dot \phi_s(t)|\big\}, \label{C2 def}
\end{align}
\end{subequations}
then by our choice of $\dD$ using Lemma \ref{triangle inequality lemma} we have on $[\Tm,\Tp]$ that
\begin{equation} \label{phis phiu est}
  |\dot\phi_s+\dot\phi_u|\leq\One_{t\in C_1}\big(|\dot \phi_s|+\dD|\dot \phi_u|\big)+\One_{t\in C_2}\big(\dD|\dot \phi_s|+|\dot\phi_u|\big),
\end{equation}
and using \eqref{phis phiu est}, \eqref{int r dot est}, \eqref{fs-phis fu-phiu = tilde a} and \eqref{C1 def}-\eqref{C2 def}, we obtain the estimate
\begin{align}
\int_{\Tm}^{\Tp}|\dot \phi|\,dt
  &= \int_{\Tm}^{\Tp}|\dot \phi_s+\dot \phi_u+\dot r|\,dt \nonumber\\
  &\leq\int_{\Tm}^{\Tp}|\dot \phi_s+\dot \phi_u|\,dt
       +2\int_{\Tm}^{\Tp}|\dot r|\,dt -\int_{\Tm}^{\Tp}|\dot r|\,dt \nonumber\\
  &\leq \int_{C_1}(|\dot \phi_s|+\dD|\dot \phi_u|)\,dt
       +\int_{C_2}(\dD|\dot \phi_s|+|\dot \phi_u|)\,dt \nonumber\\*
  & \hspace{.8cm}+\tfrac12(1-\dD)\bigg[\int_{C_1^-}|\dot \phi_u|\,dt+\int_{C_2^+}|\dot \phi_s|\,dt\bigg]-\int_{\Tm}^{\Tp}|\dot r|\,dt \nonumber\\
  &= \int_{C_1\cup C_2^+}|\dot \phi_s|\,dt + \int_{C_1^-\cup C_2}|\dot \phi_u|\,dt \nonumber\\*
  &\hspace{.8cm} -\tfrac12(1+\dD)\bigg[\int_{\Tp}^\infty|\dot \phi_s|\,dt
     +\int_{-\infty}^{\Tm}|\dot \phi_u|\,dt\bigg] \nonumber\\*
  &\hspace{.8cm} -\tfrac12(1-\dD)\bigg[\int_{C_1}|\dot \phi_u|\,dt + \int_{C_2}|\dot \phi_s|\,dt\bigg]
     -\int_{\Tm}^{\Tp}|\dot r|\,dt \nonumber\\
  &= \tilde a+\tilde a-\tfrac12(1+\dD)\bigg[\int_{\Tp}^\infty|\dot \phi_s|\,dt
     +\int_{-\infty}^{\Tm}|\dot \phi_u|\,dt\bigg] \nonumber\\*
  &\hspace{.8cm} -\tfrac12(1-\dD)\int_{\Tm}^{\Tp}\min\!\big\{|\dot \phi_u|,|\dot \phi_s|\big\}\,dt
    -\int_{\Tm}^{\Tp}|\dot r|\,dt. \label{int phi dot est}
\end{align}
To control the next-to-last integral, note that by \eqref{fs phis dec}-\eqref{fu phiu inc} and \eqref{t1 def} we have
\[
   \min\!\big\{f_u(\phi_u),f_s(\phi_s)\big\}
     =f_u(\phi_u)\One_{(-\infty,\bar t\,]}+f_s(\phi_s)\One_{(\bar t,\infty)},
\]
and thus using \eqref{phis in Msloc}-\eqref{can apply estimates}, \eqref{b bounds} and Lemma \ref{distance function saddle}~(ii) and (iv) we find that
\begin{align}
\int_{\Tm}^{\Tp}\min\!\big\{|\dot \phi_u|,|\dot \phi_s|\big\}\,dt
 &\geq \frac{1}{\dB}\int_{\Tm}^{\Tp}\min\!\big\{|\phi_u-x|,|\phi_s-x|\big\}\,dt \nonumber\\*
 &\geq \frac{1}{\dB\cI}\int_{\Tm}^{\Tp}\min\!\big\{f_u(\phi_u),f_s(\phi_s)\big\}\,dt \nonumber\\
&=\frac{1}{\dB\cI}\bigg[\int_{\Tm}^{\bar t}f_u(\phi_u)\,dt+\int^{\Tp}_{\bar t}f_s(\phi_s)\,dt\bigg]\nonumber\\
&\geq\frac{1}{\dB\cI}\bigg[\int_{\Tm}^{\bar t}|\phi_u-x|\,dt+\int^{\Tp}_{\bar t}|\phi_s-x|\,dt\bigg]\nonumber\\
 &\geq\frac{\dC}{\dB\cI}\bigg[\int_{\Tm}^{\bar t}|\dot \phi_u|\,dt+\int^{\Tp}_{\bar t}|\dot \phi_s|\,dt\bigg]. \label{min est}
\end{align}
We can now re-order the terms in \eqref{int phi dot est}, use \eqref{min est}, define $\dG:=$\linebreak $\min\!\big\{\tfrac12(1+\dD),\tfrac12(1-\dD)\frac{\dC}{\dB\cI},
\tfrac12\big\}>0$, and use \eqref{a tilde int 1}-\eqref{a tilde int 2} and \eqref{t1 def} to obtain
\begin{align}
2\tilde a-\int_{\Tm}^{\Tp}|\dot \phi|\,dt
  &\geq \tfrac12(1+\dD)\bigg[\int_{-\infty}^{\Tm}|\dot \phi_u|\,dt+\int^\infty_{\Tp}|\dot \phi_s|\,dt\bigg] \nonumber\\*
  & \hspace{.4cm}+\tfrac12(1-\dD)\int_{\Tm}^{\Tp}\min\!\big\{|\dot \phi_u|,|\dot \phi_s|\big\}\,dt+\int_{\Tm}^{\Tp}|\dot r|\,dt \nonumber\\
 &\geq \tfrac12(1+\dD)\bigg[\int_{-\infty}^{\Tm}|\dot \phi_u|\,dt+\int^\infty_{\Tp}|\dot \phi_s|\,dt\bigg] \nonumber\\*
  & \hspace{.4cm}+\tfrac12(1-\dD)\frac{\dC}{\dB\cI}\bigg[\int_{\Tm}^{\bar t}|\dot \phi_u|\,dt
+\int_{\bar t}^{\Tp}|\dot \phi_s|\,dt\bigg]+\int_{\Tm}^{\Tp}|\dot r|\,dt \nonumber\\
 &\geq \dG\bigg[ \int_{-\infty}^{\bar t}|\dot\phi_u|\,dt +\int_{\bar t}^\infty|\dot\phi_s|\,dt
           +2\int_{\Tm}^{\Tp}|\dot r|\,dt\bigg]\nonumber\\
&= \dG\bigg[f_u(\phi_u(\bar t\qb))+f_s(\phi_s(\bar t\qb))+2\int_{\Tm}^{\Tp}|\dot r|\,dt\bigg]\nonumber\\
&= 2\dG\bigg[f_s(\phi_s(\bar t\qb))+\int_{\Tm}^{\Tp}|\dot r|\,dt\bigg]. \label{2a-int est}
\end{align}
Observe that the left-hand side of \eqref{2a-int est} is the sum of the two expressions in the last line of \eqref{big est beginning} that we have to estimate. Instead of splitting the integral on the left of \eqref{2a-int est} into the two integrals in \eqref{big est beginning} however, we will have to take an extra step first and split it into two equal parts instead.
%While the choice of $t_1$ was useful for dealing with integrals over $|\dot\phi_s|$ and $|\dot\phi_u|$, it is inadequate for integrals over $|\dot\phi|$.
In other words, we define $\hat t\in[\Tm,\Tp]$ as the unique value that fulfills
\begin{equation} \label{t2 def}
   \int_{\Tm}^{\hat t}|\dot \phi|\,dt=\int_{\hat t}^{\Tp}|\dot \phi|\,dt
\end{equation}
and thus in particular
\begin{equation} \label{t2 def conseq}
  \int_{\Tm}^{\hat t}|\dot \phi|\,dt
     = \frac12\int_{\Tm}^{\Tp}|\dot \phi|\,dt
     = \frac12\bigg[\int_{\bar t}^{\Tp}|\dot\phi|\,dt+\int_{\Tm}^{\bar t}|\dot\phi|\,dt\bigg].
\end{equation}
We must now further estimate the right-hand side of \eqref{2a-int est} by a multiple of $|\phi(\hat t\qh)-x|$. We begin by using \eqref{a tilde int 1} and \eqref{fs-phis fu-phiu = tilde a} 
to find
\begin{align*}
  \int_{\bar t}^{\Tp}\!|\dot \phi|\,dt - \!\int_{\Tm}^{\bar t}\!|\dot \phi|\,dt
    &=\int_{\bar t}^{\Tp}|\dot \phi_s+\dot \phi_u+\dot r|\,dt
      -\int_{\Tm}^{\bar t}|\dot \phi_s+\dot \phi_u+\dot r|\,dt \\
    &\leq \int_{\bar t}^{\Tp} \!\!\big(|\dot \phi_s|+|\dot \phi_u|+|\dot r|\big)\,dt
         -\int^{\bar t}_{\Tm}\!\!\big(|\dot \phi_s|-|\dot \phi_u|-|\dot r|\big)\,dt\\
    &\leq2\int_{\bar t}^\infty\!|\dot \phi_s|\,dt-\int_{\Tm}^\infty\!|\dot \phi_s|\,dt
         +\int_{-\infty}^{\Tp}\!|\dot \phi_u|\,dt+\int_{\Tm}^{\Tp}\!|\dot r|\,dt \\
    &=2f_s(\phi_s(\bar t\qb))-\tilde a+\tilde a+\int_{\Tm}^{\Tp}|\dot r|\,dt \\
    &=2f_s(\phi_s(\bar t\qb))+\int_{\Tm}^{\Tp}|\dot r|\,dt.
\end{align*}
Analogously one can obtain the estimate
\begin{align*}
 \int_{\Tm}^{\bar t}|\dot \phi|\,dt - \int_{\bar t}^{\Tp}|\dot \phi|\,dt
   &\leq 2f_u(\phi_u(\bar t\qb))+\int_{\Tm}^{\Tp}|\dot r|\,dt \\
   &=2f_s(\phi_s(\bar t\qb))+\int_{\Tm}^{\Tp}|\dot r|\,dt,
\end{align*}
where we used \eqref{t1 def}, and putting both together we find that
\[ \bigg|\int_{\bar t}^{\Tp}|\dot \phi|\,dt - \int_{\Tm}^{\bar t}|\dot \phi|\,dt\bigg|
 \leq 2f_s(\phi_s(\bar t\qb))+\int_{\Tm}^{\Tp}|\dot r|\,dt. \]
This and \eqref{t2 def conseq} then lead us to the estimate
\begin{align*}
|\phi(\hat t\qh)-\phi(\bar t\qb)|
  &= \bigg|\int_{\bar t}^{\hat t}\dot \phi\,dt\bigg|
   \leq\bigg|\int_{\bar t}^{\hat t}|\dot \phi|\,dt\bigg|
   =\bigg|\int_{\Tm}^{\hat t}|\dot \phi|\,dt - \int_{\Tm}^{\bar t}|\dot \phi|\,dt\bigg| \\
  &= \frac12\bigg|\int_{\bar t}^{\Tp}|\dot \phi|\,dt-\int_{\Tm}^{\bar t}|\dot \phi|\,dt\bigg|
   \leq f_s(\phi_s(\bar t\qb))+\frac12\int_{\Tm}^{\Tp}|\dot r|\,dt,
\end{align*}
which in turn allows us to bound $|\phi(\hat t\qh)-x|$ by terms only involving $\bar t$,
\begin{align}
|\phi(\hat t\qh)-x|
  &\leq |\phi(\bar t\qb)-x|+|\phi(\hat t\qh)-\phi(\bar t\qb)| \nonumber\\
  &\leq |\phi(\bar t\qb)-x|+f_s(\phi_s(\bar t\qb))+\frac12\int_{\Tm}^{\Tp}|\dot r|\,dt \nonumber\\
  &= \big|\phi_s(\bar t\qb)+\phi_u(\bar t\qb)+r(\bar t\qb)-2x\big|+f_s(\phi_s(\bar t\qb))+\frac12\int_{\Tm}^{\Tp}|\dot r|\,dt \nonumber\\
  &\leq|\phi_s(\bar t\qb)-x|+|\phi_u(\bar t\qb)-x|+|r(\bar t\qb)|+f_s(\phi_s(\bar t\qb))
        +\frac12\int_{\Tm}^{\Tp}|\dot r|\,dt \nonumber\\
  &\leq f_s(\phi_s(\bar t\qb))+f_u(\phi_u(\bar t\qb))+|r(\bar t\qb)|+f_s(\phi_s(\bar t\qb))
        +\frac12\int_{\Tm}^{\Tp}|\dot r|\,dt \nonumber\\
  &= 3f_s(\phi_s(\bar t\qb))+\frac12\int_{\Tm}^{\Tp}|\dot r|\,dt +|r(\bar t\qb)|, \label{phi t2 est}
\end{align}
where we used \eqref{fs est 1 saddle lower}-\eqref{fs est 2 saddle lower} and again \eqref{t1 def}. To estimate $|r(\bar t\qb)|$ further, we start from \eqref{first r est} and \eqref{Delta estimate}, where this time we choose $C_1:=[\Tm,\bar t]$ and $C_2:=(\bar t,\Tp]$, and then use \eqref{a tilde int 1}-\eqref{a tilde int 2}, the first estimate in \eqref{eps estimates}, and again \eqref{t1 def}:
\begin{align*}
\sup_{\Tm\leq t\leq \Tp}|r(t)|
  &\leq \sup_{\Tm\leq t\leq \Tp} \dA\int_{-\infty}^\infty e^{-\alpha|t-\tau|}|\Delta(\tau)|\,d\tau \\
  &\leq \kappa \dA\sup_{\Tm\leq t\leq \Tp}\bigg[\int_{\Tm}^{\Tp}e^{-\alpha|t-\tau|}|r(\tau)|\,d\tau \\
  &\hspace{.4cm} +2\dB\int_{-\infty}^{\bar t}e^{-\alpha|t-\tau|}|\dot \phi_u(\tau)|\,d\tau 
   +2\dB\int_{\bar t}^\infty e^{-\alpha|t-\tau|}|\dot \phi_s(\tau)|\,d\tau\bigg] \\
  &\leq \kappa \dA\bigg[\sup_{\Tm\leq \tau\leq \Tp}|r(\tau)|
         \cdot\int_{-\infty}^\infty e^{-\alpha|\tau|}\,d\tau
        +2\dB\int_{-\infty}^{\bar t}|\dot \phi_u(\tau)|\,d\tau \\*
  &\hspace{6.4cm}      +2\dB\int^\infty_{\bar t}|\dot \phi_s(\tau)|\,d\tau\bigg] \\
  &= \kappa \dA\bigg[\frac2\alpha\sup_{\Tm\leq t\leq \Tp}|r(t)|
                  +2\dB f_u(\phi_u(\bar t\qb))+2\dB f_s(\phi_s(\bar t\qb))\bigg] \\
  &\leq \frac12\sup_{\Tm\leq t\leq \Tp}|r(t)| + 4\dB\dA\kappa\cdot f_s(\phi_s(\bar t\qb)).
\end{align*}
Solving and using also the third estimate in \eqref{eps estimates}, we thus find that
\[
   \sup_{\Tm\leq t\leq \Tp}|r(t)|
   \leq 8\dB\dA\kappa\cdot f_s(\phi_s(\bar t\qb))
   \leq f_s(\phi_s(\bar t\qb)),
\]
and so \eqref{phi t2 est} can be estimated further by
\begin{align} \label{phi t2 est 2}
  |\phi(\hat t\qh)-x|
    &\leq 3f_s(\phi_s(\bar t\qb))+\frac12\int_{\Tm}^{\Tp}|\dot r|\,dt +|r(\bar t\qb)| \nonumber \\
    &\leq 4f_s(\phi_s(\bar t\qb))+\frac12\int_{\Tm}^{\Tp}|\dot r|\,dt.
\end{align}
Combining \eqref{t2 def conseq}, \eqref{2a-int est} and \eqref{phi t2 est 2}, we obtain
\begin{align}
  \tilde a-\int_{\Tm}^{\hat t}|\dot \phi|\,dt
    &=\frac12\bigg[2\tilde a-\int_{\Tm}^{\Tp}|\dot \phi|\,dt\bigg]
     \geq \dG\bigg[f_s(\phi_s(\bar t\qb))+\int_{\Tm}^{\Tp}|\dot r|\,dt\bigg] \nonumber\\
    &\geq \frac{\dG}{4}\bigg[4f_s(\phi_s(\bar t\qb))
          +\frac12\int_{\Tm}^{\Tp}|\dot r|\,dt\bigg]
     \geq \tfrac14\dG|\phi(\hat t\qh)-x|, \label{a - t2int 1}
\end{align}
and by \eqref{t2 def} thus also
\begin{equation} \label{a - t2int 2}
   \tilde a-\int^{\Tp}_{\hat t}|\dot \phi|\,dt
     \geq \tfrac14\dG|\phi(\hat t\qh)-x|.
\end{equation}
To replace $\hat t$ by $0$ in \eqref{a - t2int 1}-\eqref{a - t2int 2} and finally prove the desired lower bound for the last line in \eqref{big est beginning}, let $\cE:=\min\{\tfrac14\dG,1\}>0$. If $\hat t\geq0$ then \eqref{a - t2int 1} implies
\begin{align}
\tilde a-\int_{\Tm}^0|\dot \phi|\,dt
  &= \bigg[\tilde a-\int_{\Tm}^{\hat t}|\dot \phi|\,dt\bigg] + \int_0^{\hat t}|\dot \phi|\,dt \nonumber\\
  &\geq \tfrac14\dG|\phi(\hat t\qh)-x|+\bigg|\int_0^{\hat t}\dot \phi\,dt\bigg| \nonumber\\
  &= \tfrac14\dG|\phi(\hat t\qh)-x|+|\phi(\hat t\qh)-\phi(0)| \nonumber\\
  &\geq \cE\big(|\phi(\hat t\qh)-x|+|\phi(\hat t\qh)-\phi(0)|\big) \nonumber\\
  &\geq\cE|\phi(0)-x|=\cE|w-x|, \label{a - 0-int 1}
\end{align}
and similarly, if $\hat t\leq0$ then \eqref{a - t2int 2} implies
\begin{equation} \label{a - 0-int 2}
\tilde a-\int_0^{\Tp}|\dot \phi|\,dt \geq \cE|w-x|.
\end{equation}
In any case, at least one of the estimates \eqref{a - 0-int 1} and \eqref{a - 0-int 2} has to hold, and so we can conclude that
\[ \max\!\bigg\{\tilde a-\int^0_{\Tm}|\dot \phi|\,dt,\ \,\tilde a-\int^{\Tp}_0|\dot \phi|\,dt\bigg\}
   \geq \cE|w-x|. \]
With this we can now finally complete the estimate \eqref{big est beginning} and prove that
$F(w)\geq\cE|w-x|$ for $\forall w\in\bar B_\eps(x)\setminus(M_s^{loc}\cup M_u^{loc})$ and thus for $\forall w\in\bar B_\eps(x)$, which is what we had to show.\\[.3cm]
\noindent From now on let us assume that the state space is two-dimensional,\linebreak i.e.~$D\subset\R^2$.
\vspace{-.2cm}

\paragraph{Proof of property (vi).} Again we will assume that $i\in I^+$.
The proof is divided into two parts: First we show in \textit{Step 1} that
\begin{equation} \label{prop (vi) first claim}
  \bar E_i'\setminus\{x\}\subset\psi(M_i',\R),
\end{equation}
so that for any choice of $\mu>0$, $\bar E_i'\setminus B_\mu(x)$ is a compact subset of $\psi(M_i',\R)$ by what we showed in part~(iii).
Since the expression for $\nf_i|_{E_i'}$ given in \eqref{grad fi formula} extends to a continuous function on all of $\psi(M_i',\R)$ and is thus bounded on $\bar E_i'\setminus B_\mu(x)$, this implies that $\nf_i$ is bounded on $E_i'\setminus B_\mu(x)$. It then remains to show in~\textit{Steps 2-12} that for some $\mu>0$ we have
\begin{equation} \label{prop (vi) claim B}
  \sup_{w\in E_i'\cap B_\mu(x)} |\nf_i(w)|<\infty.
\end{equation}
\textit{Step 1:}
To show \eqref{prop (vi) first claim}, let $w\in\bar E_i'\setminus\{x\}$, and let $(w_n)_{n\in\N}\subset E_i'$ with\linebreak $w_n\to w$. By passing on to a subsequence we may assume that $\forall n\in\N\colon$\!\!\!\linebreak $|w_n-x|\geq\frac12|w-x|$ and that $\lim_{n\to\infty}z_i'(w_n)=z$ for some $z\in M_i'$ (since $M_i'$~is compact). We begin by showing that there exist $\delta>0$ and $n_0\in\N$~such~that
\begin{equation} \label{prop (vi) second claim}
  \forall n\geq n_0\colon\quad \psi\big(z_i'(w_n),\,[0,t_i'(w_n)]\big)\cap B_\delta(x)=\varnothing.
\end{equation}
To see this, first recall that by \eqref{mod man dist} there $\exists t'>0$ such that
\[
  \int_0^{t'}|b(\psi(z,\tau))|\,d\tau
    \geq \tilde a - \tfrac15|w-x|.
\]
Since the expression on the left is a continuous function of $z$ and since $b(z)\neq0$ by Remark~\ref{on manifold b neq0}, there $\exists\nu>0$ such that
\begin{equation} \label{segment less a tilde}
  \forall z'\in\bar B_\nu(z)\colon\ \,
  b(z')\neq0
  \quad\text{and}\quad
  \int_0^{t'}|b(\psi(z',\tau))|\,d\tau
  \geq\tilde a-\tfrac14|w-x|.
\end{equation}
Since the compact set $\psi\big(\bar B_\nu(z),[0,t']\big)$ does not contain any roots of $b$, it does not contain~$x$, and thus we can choose a $\delta\in\big(0,\tfrac14|w-x|\big]$ such that
\begin{equation} \label{flow away from x}
  \psi\big(\bar B_\nu(z),[0,t']\big)\cap B_\delta(x) = \varnothing.
\end{equation}
Finally, let $n_0\in\N$ be so large that
\begin{equation} \label{zi close to z}
  \forall n\geq n_0\colon\quad z_i'(w_n)\in\bar B_\nu(z).
\end{equation}
Now suppose that \eqref{prop (vi) second claim} were wrong, i.e.\ that for some $n\geq n_0$ there were a $t''\in[0,t_i'(w_n)]$ such that $\psi(z_i'(w_n),t'')\in B_\delta(x)$.
Then by \eqref{flow away from x}-\eqref{zi close to z} it would have to fulfill $t''>t'$, i.e.\ $0<t'<t''\leq t_i'(w)$. Furthermore, we would have
\begin{align*}
  \int_{t''}^{t_i'(w_n)}\big|b\big(\psi(z_i'(w_n),\tau)\big)\big|\,d\tau
    &\geq \bigg|\int_{t''}^{t_i'(w_n)}b\big(\psi(z_i'(w_n),\tau)\big)\,d\tau\bigg| \\
    &=    \bigg|\int_{t''}^{t_i'(w_n)}\dot\psi(z_i'(w_n),\tau)\,d\tau\bigg| \\
    &=    \big| \psi(z_i'(w_n),t_i'(w_n)) - \psi(z_i'(w_n),t'') \big| \\
    &=    \big| w_n - \psi(z_i'(w_n),t'') \big| \\
    &\geq |w_n-x| - \big| \psi(z_i'(w_n),t'') - x \big| \\
    &>    \tfrac12|w-x|-\delta \\
    &\geq \tfrac14|w-x|.
\end{align*}
Together with \eqref{fi on Ei'}, \eqref{segment less a tilde} and \eqref{zi close to z} this would then lead to the contradiction
\begin{align*}
  \tilde a
    &>    \int_0^{t_i'(w_n)}\big|b\big(\psi(z_i'(w_n),\tau)\big)\big|\,d\tau \\
    &\geq \int_0^{t'}\big|b\big(\psi(z_i'(w_n),\tau)\big)\big|\,d\tau
     +    \int_{t''}^{t_i'(w_n)}\big|b\big(\psi(z_i'(w_n),\tau)\big)\big|\,d\tau  \\
    &>    \big(\tilde a-\tfrac14|w-x|\big) + \tfrac14|w-x|
     =    \tilde a,
\end{align*}
concluding the proof of \eqref{prop (vi) second claim}.

Now let $n\geq n_0$ and $t\in(0,t_i'(w_n)]$. The vector $v:=\psi(z_i'(w_n),t)\in\psi(M_i',[0,\infty))$ fulfills $z_i'(v)=z_i'(w_n)$ and $t_i'(v)=t\in(0,t_i'(w_n)]$, and so by \eqref{fi on Ei'} we have
\[
   0 < \int_0^{t_i'(v)}\big|b\big(\psi(z_i'(v),\tau)\big)\big|\,d\tau
   \leq \int_0^{t_i'(w_n)}\big|b\big(\psi(z_i'(w_n),\tau)\big)\big|\,d\tau
   < \tilde a,
\]
i.e.\ $v\in E_i'$.
This shows that $\psi\big(z_i'(w_n),\,(0,t_i'(w_n)]\big)\subset E_i'$, which together with \eqref{prop (vi) second claim} implies that $\psi\big(z_i'(w_n),\,[0,t_i'(w_n)]\big)\subset \bar E_i'\setminus B_\delta(x)$. Since $\dF:=\min\!\big\{|b(v)|\,\big|\,v\in\bar E_i'\setminus B_\delta(x)\big\}>0$ by what we showed in parts~(iii) and~(iv), by \eqref{fi on Ei'} we therefore have
\[
   \tilde a>\int_0^{t_i'(w_n)}\big|b\big(\psi(z_i'(w_n),\tau)\big)\big|\,d\tau
   \geq t_i'(w_n)\cdot\dF,
\]
i.e.\ $t_i'(w_n)\in\big[0,\frac{\tilde a}{\dF}\big)$. We can thus extract a subsequence $(w_{n_k})_{k\in\N}$ such that $\lim_{k\to\infty}t_i'(w_{n_k})=t'''$ for some $t'''\in\big[0,\frac{\tilde a}{\dF}\big]$. Taking the limit $k\to\infty$ in the relation $w_{n_k}=\psi\big(z_i'(w_{n_k}),t_i'(w_{n_k})\big)$ now tells us that $w=\psi(z,t''')\in\psi(M_i',\R)$, terminating the proof of \eqref{prop (vi) first claim}.\\[.2cm]
\textit{Step 2:} \label{prop (vi) Step 2}
To prepare for the proof of \eqref{prop (vi) claim B}, we begin by defining an invertible affine transformation $L\colon\Rn\to\Rn$ that shifts $x$ to the origin and then turns space so that $T_xM_u^{loc}$ coincides with the $y$-axis.
To do so, let $\tilde R$ be an orthogonal matrix such that $A=\tilde R\tramat\tilde R^T$ for some $p,q>0$ and $r\in\R$, define $L$ by
\begin{equation} \label{L def}
   L(w)=\tilde R^T(w-x), \qquad L^{-1}(v):=\tilde Rv+x,
\end{equation}
and define the transformed drift $\tilde b\in C^1(L(D),\Rn)$ by
\[ \tilde b(v):=\tilde R^Tb(L^{-1}(v)). \]
Since $\tilde b(0)=\tilde R^Tb(x)=0$ and $\nabla\tilde b(0)=\tilde R^T\nabla b(x)\tilde R=\tilde R^TA\tilde R=\tramat$, we can write $\tilde b(v)=\tramat v+\tilde g(v)$ for some $C^1$-function $\tilde g$ with
\begin{equation} \label{g=dg=0}
   \tilde g(0)=0
   \quad\text{and}\quad
   \nabla\tilde  g(0)=0,
\end{equation}
and so the flow $\chi(v,t):=L\big(\psi(L^{-1}(v),t)\big)$ for $\forall v\in L(D)$ $\forall t\in\R$, which fulfills
\begin{equation} \label{chi L prop}
  \chi(L(w),t) = L(\psi(w,t))
  \qquad \forall w\in D\ \forall t\in\R,
\end{equation}
is the solution of the system
\begin{subequations}
\begin{align}
   \dot\chi(v,t)
    &= \tilde R^Tb\big(\psi(L^{-1}(v),t)\big)
     = \tilde R^Tb\big(L^{-1}(\chi(v,t))\big)
     = \tilde b(\chi(v,t)) \label{chi dot} \\*
    &= \tramat\chi(v,t)+\tilde g(\chi(v,t)),\label{chi dot 2} \\*[.1cm]
  \chi(v,0)
    &= L\big(\psi(L^{-1}(v),0)\big)=v. \label{chi init}
\end{align}
\end{subequations}
Writing this system componentwise with $\tilde g=(g_1,g_2)$, $\chi=\chi(v,t)=(\chi_1,\chi_2)$ and $v=(v_1,v_2)$, we have
\begin{subequations}
\begin{align}
  \dot \chi_1&= -p\chi_1    \hspace{1.074cm}+g_1(\chi_1,\chi_2), \label{dot chi1}\\
  \dot \chi_2&= \hspace{9pt}r\chi_1+q\chi_2 +g_2(\chi_1,\chi_2), \label{dot chi2}\\[.1cm]
  \chi_1(v,0)&=v_1, \qquad
  \chi_2(v,0)=v_2.\label{chi1 init}
\end{align}
\end{subequations}
\noindent
\textit{Step 3:}
Next, we will have to choose some constants. Let
\begin{align}
  \tilde \te&:=\tfrac{|r|}{p+q}+1, \nonumber \\
  \dH&:=\tfrac2p(|r| + q\tilde \te)+2, \label{d7 def} \\
  \te&:=\max\!\Big\{ \tfrac{p+|r|}{q}+1,\ \tilde \te + 1 + (4+\dH+2\tilde \te)^{1+2p/q} \Big\} > \tilde \te + 2,
     \label{K const def}
\end{align}
and for some small $\sigma>0$ to be chosen momentarily we define the open double wedge
\[
   W_{\sigma,\te} := \big\{(s,y)\in\R^2\,\big|\,0<|s|<\sigma,\ |\tfrac ys|< \te \big\}
                \subset B_{\sigma(1+\te)}(0).
\]
To choose $\sigma$, note that since $g_1$ and $g_2$ are $C^1$-functions that by \eqref{g=dg=0} fulfill $g_{1,2}(0,0)=0$ and
\begin{equation} \label{grad gi}
   \nabla g_{1,2}(0,0)=0,
\end{equation}
we have $g_{1,2}(s,y)=o(|s|+|y|)$, and since on $W_{\sigma,\te}$ we have $|s|+|y|<(1+\te)|s|$, this implies that
\begin{equation} \label{gi order}
   g_{1,2}(s,y)=o(|s|)
\qquad\text{as $(s,y)\to0$ in $W_{\sigma,\te}$.}
\end{equation}
Therefore we can pick $\sigma>0$ so small that
\begin{equation} \label{xdot pos}
  \big|\tfrac1{ps}g_{1,2}(s,y)\big|\leq\tfrac12 \qquad\text{for $\forall(s,y)\in W_{\sigma,\te}$,}
\end{equation}
and then the function $h\colon W_{\sigma,\te}\to\R$ given by
\begin{equation} \label{h def}
    h(s,y)
     := \bigg[\frac rp + \frac{qy}{ps}\bigg]
        - \frac{\frac rp + \frac{qy}{ps}+\frac1{ps}g_2(s,y)}{1-\frac1{ps}g_1(s,y)}
\end{equation}
is well-defined and $C^1$. Furthermore, we have
\begin{align*}
    s\partial_yh(s,y)
      &=s\Bigg[\frac{q}{ps}-\frac{\frac{q}{ps} +\frac1{ps}\partial_yg_2(s,y)}{1-\frac1{ps}g_1(s,y)}
-\frac{\frac rp + \frac{qy}{ps}+\frac1{ps}g_2(s,y)}
             {\big(1-\frac1{ps}g_1(s,y)\big)^2ps}
        \partial_yg_1(s,y)\Bigg] \\
      &=\frac q p 
       -\frac{\frac qp +\frac1p \partial_yg_2(s,y)}{1-\frac1{ps}g_1(s,y)}
       -\frac{\frac rp + \frac{qy}{ps}+\frac1{ps}g_2(s,y)}
             {\big(1-\frac1{ps}g_1(s,y)\big)^2p}
        \partial_yg_1(s,y),
\end{align*}
and since by \eqref{grad gi}-\eqref{gi order} the last expression converges to $0$ as $(s,y)\to0$ in $W_{\sigma,\te}$, we can choose $\sigma>0$ so small that
\begin{equation} \label{hy est}
  |\partial_yh(s,y)|\leq \tfrac{q}{2p}|s|^{-1}
   \qquad \text{for $\forall (s,y)\in W_{\sigma,\te}$.}
\end{equation}
Finally, writing $M_s\setminus\{x\}=\psi(w'_1,\R)\cup\psi(w'_2,\R)$ for some points $w'_1,w'_2\in D$, by \eqref{chi L prop} the points $L(w'_1)$ and $L(w'_2)$ lie on the global stable manifold of the saddle point $\chi=0$ of the system \eqref{chi dot}-\eqref{chi init}. Since by \eqref{chi dot 2} the local stable manifold of that system at the origin is tangent to the eigenvector $(p+q,-r)$ of the matrix $\tramat$ and is thus contained in $W_{\tilde\te,\sigma}$ near the origin, there therefore $\exists T>0$ such that
\begin{equation} \label{chi flow in small wedge}
   \chi\big(L(w'_k),[T,\infty)\big)\subset W_{\sigma,\tilde\te}
   \qquad \text{for $k=1,2$.}
\end{equation}
Since our construction in \textit{Steps 2-3} was solely based on the given vector field~$b$, we can use it to decrease $\tilde a$ one final time, as explained at the end of Section \ref{superprop proof step 1}, so that $\tilde a<\min\!\big\{f_s(\psi(w'_1,T)),\,f_s(\psi(w'_2,T))\big\}$. (To prepare also for the case $i\in I^-$, we must at this point also further decrease $\tilde a$ according to an analogous construction with the stable and unstable direction exchanged.)

Since $f_s^{-1}((0,\tilde a])\subset M_s\setminus\{x\}=\psi(w'_1,\R)\cup\psi(w'_2,\R)$ and since by Lemma \ref{distance function saddle}~(i) our choice of $\tilde a$ implies that for $k=1,2$ and $\forall t<T$ we have $f_s(\psi(w'_k,t))\geq f_s(\psi(w'_k,T))>\tilde a$, \eqref{chi L prop} and \eqref{chi flow in small wedge} then imply that
\begin{align}
   L\big(f_s^{-1}((0,\tilde a])\big)
   &\subset L\big(\psi(w'_1,[T,\infty))\cup\psi(w'_2,[T,\infty))\big) \nonumber \\
   &=       \chi\big(L(w'_1),[T,\infty)\big)\cup\chi\big(L(w'_2),[T,\infty)\big) \nonumber \\
   &\subset W_{\sigma,\tilde\te}. \label{Msloc in WdK}
\end{align}
We now denote by $w_1,w_2\in D$ the two points given by Lemma~\ref{Msa compact lemma} such that
\begin{equation} \label{Ma tilde 2 point}
   M_s^{\tilde a} = \{w_1,w_2\},
\end{equation}
and we denote for $k=1,2$
\label{prop (vi) Step 3 end}
\begin{align} \label{Lwk in Wst}
  (\tilde s_k,\tilde y_k) &:= L(w_k) \in L(M_s^{\tilde a}) = L\big(f_s^{-1}(\{\tilde a\})\big) \subset W_{\sigma,\tilde \te}.
\end{align}
\textit{Step 4:}
For initial values $(s_0,y_0)\in W_{\sigma,\te}$ now consider the solution $y(s):=y(s_0,y_0;s)$ of the ODE
\begin{subequations}
\begin{align}  y'(s)&=\frac{rs+qy+g_2(s,y)}{-ps+g_1(s,y)}
        =-\frac{\frac rp + \frac{qy}{ps}+\frac1{ps}g_2(s,y)}{1-\frac1{ps}g_1(s,y)}
         \label{y(x) ode pre} \\*
       &=-\bigg[\frac rp+\frac{qy}{ps}\bigg]+h(s,y), \label{y(x) ode}\\*
  y(s_0)&=y_0. \label{y(x) init}
\end{align}
\end{subequations}
The right-hand sides in \eqref{y(x) ode pre}-\eqref{y(x) ode} are well-defined, equal and $C^1$ on $W_{\sigma,\te}$ by \eqref{xdot pos}-\eqref{h def}, and so $y(s)$ is well-defined until its graph reaches the boundary of $W_{\sigma,\te}$.

The meaning of the system \eqref{y(x) ode pre}-\eqref{y(x) init} is the following: Consider a solution $\chi(v,t)$ of \eqref{chi dot}-\eqref{chi init} starting from some point $v=(s_0,y_0)\in W_{\sigma,\te}$ such that for some $\hat t>0$ we have
\begin{equation} \label{flow in WdK}
   \chi(v,[0,\hat t\qh])\subset W_{\sigma,\te}.
\end{equation}
If $s_0>0$ then this implies that
\begin{equation} \label{chi1 one sign}
  \forall t\in[0,\hat t\qh]\colon\ \,\chi_1(v,t)>0
  \quad\text{and thus}\quad
  \dot\chi_1(v,t)<0
\end{equation}
by \eqref{dot chi1} and \eqref{xdot pos}. This shows that
\begin{equation} \label{integration bounds}
   0<\chi_1(v,\hat t\qh)<\chi_1(v,0)=s_0,
\end{equation}
and that on $[0,\hat t\qh]$ the function $\chi(v,\cdot\,)$ takes values on the graph of some function $y(s)=y(s_0,y_0;s)$, i.e.\ we have
\begin{align}
   \chi_2(v,t)&=y(\chi_1(v,t)), \label{chi2=y(chi1)} \\
   \dot\chi_2(v,t)&=y'(\chi_1(v,t))\dot\chi_1(v,t) \label{y' chain}
\end{align}
for $\forall t\in[0,\hat t\qh]$. Since $\dot\chi_1(v,t)\neq0$ by \eqref{chi1 one sign}, together with \eqref{dot chi1}-\eqref{dot chi2} this shows that
\[
  y'(\chi_1)=\frac{\,\dot\chi_2\,}{\dot\chi_1}
    = \frac{r\chi_1+q\chi_2+g_2(\chi_1,\chi_2)}{-p\chi_1+g_1(\chi_1,\chi_2)}
    = \frac{r\chi_1+qy(\chi_1)+g_2(\chi_1,y(\chi_1))}{-p\chi_1+g_1(\chi_1,y(\chi_1))},
\]
i.e.\ $y(s)$, $s\in[\chi_1(v,\hat t\qh),s_0]$, is the solution of the ODE \eqref{y(x) ode pre}-\eqref{y(x) init}, where the initial condition \eqref{y(x) init} follows from setting $t=0$ in \eqref{chi2=y(chi1)} and using that $v=(s_0,y_0)$.

If $s_0<0$ then all inequalities in \eqref{chi1 one sign}-\eqref{integration bounds} are reversed, and so \eqref{chi2=y(chi1)}-\eqref{y' chain} hold as well, only that then $y(s)$ is defined on the interval $[s_0,\chi_1(v,\hat t\qh)]$.\\[.2cm]
\textit{Step 5:}
Now let us choose a $\mu>0$ for which we will be able to show \eqref{prop (vi) claim B}. Denoting
\begin{equation} \label{J def}
   J:=\big\{k\in\{1,2\}\,\big|\,w_k\notin\psi(M_i',\R)\big\},
\end{equation}
we have for $\forall k\in J$ that $\psi(w_k,\R)\cap M_i'=\varnothing$,
i.e.\ $\forall\tau\in\R\colon f_{M_i'}(\psi(w_k,\tau))\neq0$. Thus, if we had $f_{M_i'}(\psi(w_k,-1))<0$ then we would have $f_{M_i'}(\psi(w_k,\tau))<0$ for $\forall\tau\in\R$, and letting $\tau\to\infty$ and using that $w_k\in\Msan\subset M_s$ would imply that $f_{M_i'}(x)\leq0$, contradicting \eqref{fMi'x>0}.
This shows that
\begin{equation} \label{fMi'>0 pre}
  \forall k\in J\colon\quad f_{M_i'}(\psi(w_k,-1))>0.
\end{equation}
Furthermore, since for $\forall k\in\{1,2\}$ we have
\[ 
  \int_{-1}^\infty|b(\psi(w_k,\tau))|\,d\tau
  > \int_0^\infty|b(\psi(w_k,\tau))|\,d\tau
  = f_s(w_k)=\tilde a,
\]
there $\exists T'>0$ so large that
\begin{equation} \label{long flow pre}
  \forall k\in\{1,2\}\colon\quad \int_{-1}^{T'}|b(\psi(w_k,\tau))|\,d\tau > \tilde a.
\end{equation}
Using also the value $\rho>0$ constructed in the steps leading to Lemma~\ref{Mi' properties}, by \eqref{fMi'>0 pre}-\eqref{long flow pre} there thus exists a $\tilde\rho>0$ such that
\begin{equation} \label{nu def est}
  \tilde\rho < \min\!\big\{\rho,\,\tfrac13|w_1-w_2|,\, |\tilde s_1|,\, |\tilde s_2|\big\},
\end{equation}
\vspace{-.8cm}
\begin{align}
  \forall k\in J\ \,\forall v\in\bar B_{\tilde\rho}(w_k)\colon&\ f_{M_i'}(\psi(v,-1))>0, \label{fMi'>0} \\
  \forall k\in\{1,2\}\ \,\forall v\in\bar B_{\tilde\rho}(w_k)\colon&\ \int_{-1}^{T'}|b(\psi(v,\tau))|\,d\tau >\tilde a. \label{long flow}
\end{align}
Finally, we have $x\notin M_i'$ by Remark~\ref{on manifold b neq0}, and because of \eqref{long flow} the sets $\bar B_{\tilde\rho}(w_k)$ and thus $\psi\big(\bar B_{\tilde\rho}(w_k),[-1,T']\big)$ cannot contain $x$. Therefore we can choose $\mu\in(0,a_0]$ so small that \eqref{flow into hMtas}-\eqref{projection not far} hold, and that
\begin{gather}
\bar B_{\cI\mu}(x)\cap M_i'=\varnothing, \label{away from Mi'} \\
B_\mu(x)\cap\psi\big(\bar B_{\tilde\rho}(w_k),[-1,T']\big)=\varnothing. \label{ball away from flow}
\end{gather}
\textit{Step 6:} 
To show \eqref{prop (vi) claim B}, let now $w\in E_i'\cap B_\mu(x)$. We must find a bound on $|\nf_i(w)|$ that is independent of our choice of $w$. We begin by showing that there exist $\eta>0$ and $k\in\{1,2\}$ such that
\begin{align}
   &\hspace{-1.55cm} B_\eta(w)\subset E_i'\cap B_\mu(x), \label{claim step 2b1} \\
     \forall u\in B_\eta(w)\colon\,\  & v_u:=L(z_i'(u)) \in W_{\sigma,\te-1},  \label{claim step 2b2}\\*
     & v_u\hspace{3pt}=(\tilde s_k,y_u) \label{claim step 2b3}
\intertext{for some $y_u\in\R$ with}
     &|y_u-\tilde y_k| \leq \tilde\rho. \label{yu ytildek dist}
\end{align}
To do so, let $\eta>0$ be so small that \eqref{claim step 2b1} holds and that
\begin{equation} \label{zi' cont}
  \forall u\in B_\eta(w)\colon\quad
    |z_i'(u)-z_i'(w)|\leq\tilde\rho,
\end{equation}
and let $u\in B_\eta(w)\subset E_i'\cap B_\mu(x)\subset\bar B_{a_0}(x)$. First observe that this implies that $u\notin M_u^{loc}$. Indeed, otherwise we would by Lemma \ref{distance function saddle}~(i), (ii) and~(iv) have for $\forall\tau\leq0$
\[
  |\psi(u,\tau)-x| \leq f_u(\psi(u,\tau)) \leq f_u(u)\leq\cI|u-x|\leq \cI\mu
\]
and thus $\psi(u,\tau)\!\notin\!M_i'$ by \eqref{away from Mi'}. But this would show that $u\!\notin\!\psi(M_i',[0,\infty))$, which by \eqref{fi def 1} contradicts $u\in E_i'=f_i^{-1}\big((0,\tilde a)\big)$.

Since $u\notin M_u^{loc}$, by \eqref{flow into hMtas}-\eqref{projection not far} there $\exists t<0$ such that $v:=\psi(u,t)$ fulfills
\begin{gather}
  v\in\hMtas, \label{flow into hMtas v} \\
  |p_s(v)-v|\leq\tilde\rho. \label{projection not far v}
\end{gather}
In particular, because of \eqref{flow into hMtas v} and \eqref{Msa Msu def} we have $p_s(v)\in\Msan$, and so by \eqref{Ma tilde 2 point} there $\exists k\in\{1,2\}$ such that $p_s(v)=w_k$ and thus by \eqref{projection not far v}
\begin{equation} \label{c near wk}
  v\in\bar B_{\tilde\rho}(w_k).
\end{equation}
Suppose we had $k\in J$.
Then by \eqref{fMi'>0} we would have
\[
   0<f_{M_i'}(\psi(v,-1))
    =f_{M_i'}(\psi(u,t-1))
    =f_{M_i'}\big(\psi\big(z_i'(u),\,t_i'(u)+t-1\big)\big)
\]
and thus $t_i'(u)+t-1>0$ by \eqref{fM sgn 1}. Furthermore, by \eqref{c near wk} and \eqref{ball away from flow} we would have $u\notin\psi(v,[-1,T'])=\psi\big(u,[t-1,t+T']\big)$ and thus $0\notin[t-1,t+T']$, and since $t-1<t<0$, this would show that $t+T'<0$. To summarize, we would have
\[ -t_i'(u)<t-1<t+T'<0, \]
and so by \eqref{fi on Ei'}, \eqref{c near wk} and \eqref{long flow} we would arrive at the contradiction
\begin{align}
\tilde a
    &> \int^{t_i'(u)}_0\big|b\big(\psi(z_i'(u),\tau)\big)\big|\,d\tau \nonumber\\
    &= \int_{-t_i'(u)}^0\big|b\big(\psi(z_i'(u),t_i'(u)+\tau)\big)\big|\,d\tau \nonumber \\
    &= \int_{-t_i'(u)}^0|b(\psi(u,\tau))|\,d\tau \label{fi change direction} \\
    &\geq \int_{t-1}^{t+T'}|b(\psi(u,\tau))|\,d\tau \\
    &= \int_{-1}^{T'}|b(\psi(v,\tau))|\,d\tau
    >\tilde a.
\end{align}
Therefore we have $k\notin J$ and thus $w_k\in\psi(M_i',\R)\cap\Msan=\psi(M_i,\R)\cap\Msan$ $=K_i^{\tilde a}$ by \eqref{J def}, \eqref{Ma tilde 2 point}, Lemma \ref{Mi' properties}~(i) and \eqref{2D Kia}. By \eqref{flow into hMtas v}, \eqref{c near wk}, \eqref{nu def est} and \eqref{Msan NKati in Mi} we thus have $v\in\hMtas\cap N_{\tilde\rho}(K_i^{\tilde a})\subset\hMtas\cap N_\rho(K_i^{\tilde a})\subset M_i'$, and so the relation $u=\psi(v,-t)$ shows that $z_i'(u)=v$. Therefore we have $p_s(v)-v=w_k-z_i'(u)$, and so \eqref{projection not far v} and \eqref{as prop 0} say that
\begin{gather}
  |w_k-z_i'(u)|\leq\tilde\rho, \label{translated prop 1} \\*
  w_k-z_i'(u)\in T_xM_u^{loc}. \label{translated prop 2}
\end{gather}
To see that $k$ is independent of our choice of $u\in B_\eta(w)$, we apply the above arguments to $w$ instead of $u$ and find that for some~$k'$ \eqref{translated prop 1}-\eqref{translated prop 2} hold with $w_{k}-z_i'(u)$ replaced by $w_{k'}-z_i'(w)$. Since
\[
   |w_k-w_{k'}|\leq|w_k-z_i'(u)|+|z_i'(u)-z_i'(w)|+|z_i'(w)-w_{k'}|\leq3\tilde\rho<|w_1-w_2|
\]
by \eqref{translated prop 1}, \eqref{zi' cont} and \eqref{nu def est}, we must have $k'=k$.

Now \eqref{L def} and \eqref{translated prop 2} imply that for $\forall u\in B_\eta(w)$ we have
\[
   L(z_i'(u))-L(w_k)
    = \tilde R^T(z_i'(u)-w_k)
    \in \tilde R^T T_xM_u^{loc}=T_0L(M_u^{loc}).
\]
Since $L(M_u^{loc})$ is just the local unstable manifold at $\chi=0$ of the transformed system \eqref{chi dot}-\eqref{chi init} and is thus tangent to the $y$-axis at the origin, this means that the first components of $L(z_i'(u))$ and $L(w_k)=(\tilde s_k,\tilde y_k)$ coincide, which is \eqref{claim step 2b3}. Furthermore, since
\[
   |L(z_i'(u))-L(w_k)|
    = |\tilde R^T(z_i'(u)-w_k)|
    = |z_i'(u)-w_k|
    \leq \tilde\rho
\]
by \eqref{translated prop 1}, their $y$-components differ by at most $\tilde\rho$, i.e.\ \eqref{yu ytildek dist}, and together with \eqref{Lwk in Wst}, \eqref{nu def est} and \eqref{K const def} this implies $\big|\frac{y_u}{\tilde s_k}\big|\leq\big|\frac{\tilde y_k}{\tilde s_k}\big|+\big|\frac{\tilde\rho}{\tilde s_k}\big|<\tilde \te+1< \te-1$, which is \eqref{claim step 2b2}.\\[.2cm]
\textit{Step 7:}
W.l.o.g.\ let us from now on assume that $\tilde s_k>0$.
In this step we will show that then for $\forall(s_0,y_0)\in W_{\sigma,\te}$ with $0<s_0<\tilde s_k$ the function $y(s_0,y_0;s)$ is well-defined (and has its graph in $W_{\sigma,\te}$) at least for $s\in[s_0,\tilde s_k]$, thus allowing us to define the function
\begin{equation} \label{f hat def}
   \hat f(s_0,y_0)
     :=\int_{s_0}^{\tilde s_k}\sqrt{1+[\partial_sy(s_0,y_0;s)]^2}\,ds
\end{equation}
which we may in short write as $\hat f(v)$ for $v=(s_0,y_0)\in W_{\sigma,\te}$.

To see this, we will show that as $s$ increases from $s_0$, the graph of $y(s):=y(s_0,y_0;s)$ is repelled from the upper and lower boundaries of $W_{\sigma,\te}$ and must thus reach the right boundary of $W_{\sigma,\te}$ at $s=\sigma>\tilde s_k$. Indeed, suppose that at some $s>0$ the graph of $y(s)$ has reached a point $(s,y)$ with $\frac ys\geq \te-1$. Then by \eqref{K const def} we have
$\frac ys\geq \te-1\geq\frac{p+|r|}{q}\geq\frac{p-r}{q}$ and thus $\frac rp+\frac{qy}{ps}\geq1$, so that
\[
  \partial_s\bigg[\frac{y(s)}{s}\bigg]
     = \frac 1s\Big[y'-\frac ys\Big]
     = \frac1s\bigg[{-\frac{\frac rp + \frac{qy}{ps}
                     +\frac 1{ps}g_2(s,y)}{1-\frac 1{ps}g_1(s,y)}}-\frac ys\bigg]
     \leq \tfrac1s\big[{-\tfrac{1-1/2}{1+1/2}} - 0 \big]
     <0
\]
by \eqref{y(x) ode pre} and \eqref{xdot pos}. Similarly, one can show that if $\frac{y(s)}s\leq -(\te-1)$ then $\partial_s\big[\frac{y(s)}{s}\big]>0$.

Furthermore, observe that for any point $(s_0,y_0)\in W_{\sigma,\te}$ such that $y(s):=y(s_0,y_0;s)$ is defined for all $s$ in some interval $[s_1,\tilde s_k]\ni s_0$, the uniqueness of the solutions of \eqref{y(x) ode pre}-\eqref{y(x) init} implies that $y(s)=y(s_1,y(s_1);s)$, so that
\begin{align}
  \int_{s_1}^{\tilde s_k}\sqrt{1+[y'(s)]^2}\,ds
    &= \int_{s_1}^{\tilde s_k}\sqrt{1+\big[\partial_sy(s_1,y(s_1);s)\big]^2}\,ds \nonumber\\
    &= \hat f(s_1,y(s_1)). \label{hat f shift}
\end{align}
\textit{Step 8:}
We will now show that $\hat f$ is $C^1$ on $W_{\sigma,\te}$, and that for $\forall (s_0,y_0)\in W_{\sigma,\te}$ we have the bounds
\begin{subequations}
\begin{align}
  |\partial_{s_0}\hat f(s_0,y_0)| &\leq 5 + \tfrac8p(|r| + q\te), \label{fhat bound 1}\\
  |\partial_{y_0}\hat f(s_0,y_0)| &\leq 3, \label{fhat bound 2}
\end{align}
\end{subequations}
which are the core of this proof.

To do so, first note that since the right-hand side of \eqref{y(x) ode} is $C^1$ on $W_{\sigma,\te}$, $y(s):=y(s_0,y_0;s)$ is $C^1$ with respect to the initial data~$y_0$, with
\[
   \partial_s\big[\partial_{y_0}y(s)\big]
    =\partial_{y_0}y'(s)
    =\big[{-\tfrac q p}s^{-1}+\partial_yh(s,y(s))\big]\partial_{y_0}y(s)
\]
for $\forall s\in[s_0,\tilde s_k]$ by \eqref{y(x) ode}, and since $\partial_{y_0}y(s_0)=1$ by \eqref{y(x) init}, we find that
\begin{align*}
  \partial_{y_0}y(s)
   &= \exp\!\Big(\int_{s_0}^{s}\big[{-\tfrac q p} {s'}^{-1}+\partial_yh(s',y(s'))\big]\,ds'\Big),\\
\partial_{y_0}y'(s)
   &= \big[{-\tfrac q p}s^{-1}+\partial_yh(s,y(s))\big]\exp\!\Big(\int_{s_0}^s\big[{-\tfrac q p} {s'}^{-1}+\partial_yh(s',y(s'))\big]\,ds'\Big)
\end{align*}
for $\forall s\in[s_0,\tilde s_k]$. We can now invoke \eqref{hy est} to obtain
\begin{align}
  |\partial_{y_0}y'(s)|
    &\leq \frac{3q}{2p}s^{-1}\exp\!\Big({-\frac q{2p}}\int_{s_0}^s{s'}^{-1}\,ds'\Big)
    = \frac{3q}{2p}s^{-1}\Big(\frac{s_0}s\Big)^{\!\frac q{2p}} \nonumber\\
\Rightarrow\
\int_{s_0}^{\tilde s_k}|\partial_{y_0}y'(s)|\,ds
   &\leq \frac{3q}{2p}\int_{s_0}^{\tilde s_k}s^{-1}\Big(\frac{s_0}{s}\Big)^{\!\frac{q}{2p}}\,ds
    = \frac{3q}{2p}\int_1^{\tilde s_k/s_0}s^{-(1+\frac{q}{2p})}\,ds \nonumber\\
   &\leq \frac{3q}{2p}\int_1^\infty s^{-(1+\frac{q}{2p})}\,ds =3, \label{dy0 int est}
\end{align}
which by \eqref{f hat def} leads us to our first bound
\[
  \big|\partial_{y_0}\hat f(s_0,y_0)\big|
    = \bigg|\int_{s_0}^{\tilde s_k}\frac{y'(s)}{\sqrt{1+[y'(s)]^2}}\cdot\partial_{y_0}y'(s)\,ds\bigg|
    \leq \int_{s_0}^{\tilde s_k}|\partial_{y_0}y'(s)|\,ds
    \leq 3,
\]
i.e.\ \eqref{fhat bound 2}. For the other bound \eqref{fhat bound 1}, note that for small $\Delta$ we have
\[
   y\big(s_0+\Delta,\,y(s_0,y_0;s_0+\Delta);\,s\big)=y(s_0,y_0;s),
\]
and differentiating with respect to $s$ and then computing the $\Delta$-derivative at $\Delta=0$ leads us to
\begin{align}
  &\hspace{1.2cm}
\partial_sy\big(s_0+\Delta,\,y(s_0,y_0;s_0+\Delta);\,s\big)=\partial_sy(s_0,y_0;s) \nonumber \\
  &\hspace{-.7cm}\Rightarrow\hspace{.5cm}
  (\partial_{s_0}\partial_sy)(s_0,y_0;\,s)
 +(\partial_{y_0}\partial_sy)(s_0,y_0;\,s)\cdot(\partial_sy)(s_0,y_0;\,s_0) =0 \nonumber \\
  &\hspace{-.7cm}\Rightarrow\hspace{3.2cm}
 \partial_{s_0}y'(s) =-\partial_{y_0}y'(s)\cdot y'(s_0). \label{ds0y dy0y rel}
\end{align}
By \eqref{y(x) ode pre} and \eqref{xdot pos}, $|y'(s_0)|$ can be bounded by
\begin{align} \label{y' upper bound}
  \sup_{s_0\leq s\leq\tilde s_k}|y'(s)|
   &= \sup_{(s,y)\in W_{\sigma,\te}}
    \bigg| \frac{\frac rp + \frac{qy}{ps}+\frac 1{ps}g_2(s,y)}
                {1-\frac 1{ps}g_1(s,y)} \bigg| \nonumber \\[.15cm]
   &\leq \frac{\frac{|r|}p + \frac{q}{p}\te+\frac12}{1-\frac12}
    =    \tfrac2p(|r| + q\te)+1,
\end{align}
and so \eqref{f hat def}, \eqref{ds0y dy0y rel}, \eqref{dy0 int est} and \eqref{y' upper bound} lead to the estimate
\begin{align*}
  \big|\partial_{s_0}\hat f(s_0,y_0)\big|
    &= \bigg|{-\sqrt{1+[y'(s_0)]^2}}
+\int_{s_0}^{\tilde s_k}\frac{y'(s)}{\sqrt{1+[y'(s)]^2}}\cdot\partial_{s_0}y'(s)\,ds\bigg| \\
    &\leq 1+|y'(s_0)| + \int_{s_0}^{\tilde s_k}|\partial_{s_0}y'(s)|\,ds \\
    &\leq 1+|y'(s_0)|+|y'(s_0)|\int_{s_0}^{\tilde s_k}|\partial_{y_0}y'(s)|\,ds \\
    &\leq 1+4|y'(s_0)| \\
    &\leq 1+4\big[\tfrac2p(|r| + q\te)+1\big] \\
    &= 5 + \tfrac8p(|r| + q\te).
\end{align*}
\textit{Step 9:}
Now let us consider the function
\[ \tilde y(s):=y(\tilde s_k,\tilde y_k;s) \]
that passes through the point $(\tilde s_k,\tilde y_k)=L(w_k)$. Since $w_k\in M_s^{\tilde a}=f_s^{-1}(\{\tilde a\})$ by \eqref{Ma tilde 2 point}, Lemma \ref{distance function saddle}~(i) implies that $\psi(w_k,[0,\infty)) \subset f_s^{-1}\big((0,\tilde a]\big)$ and thus
\[
   \chi\big(L(w_k),[0,\infty)\big)
     = L\big(\psi(w_k,[0,\infty))\big)
     \subset L\big(f_s^{-1}\big((0,\tilde a]\big)\big)
     \subset W_{\sigma,\tilde \te}
\]
by \eqref{chi L prop} and \eqref{Msloc in WdK}. Since by \eqref{chi L prop} and \eqref{L def} we have
\begin{equation} \label{chi limit 0}
   \lim_{t\to\infty}\chi(L(w_k),t)
     = \lim_{t\to\infty} L(\psi(w_k,t))
     = L(x)
     = 0,
\end{equation}
by our remarks at the end of \textit{Step 4} this shows that $\tilde y(s)$ is defined for $\tilde s_k\geq s>\lim_{t\to\infty}\chi_1(L(w_k),t)=0$, i.e.\ for $\forall s\in(0,\tilde s_k]$, with graph in $W_{\sigma,\tilde \te}$, i.e.
\begin{equation} \label{tilde y in wedge}
   \big\{(s,\tilde y(s)) \,\big|\, s\in(0,\tilde s_k]\big\}\subset W_{\sigma,\tilde \te}.
\end{equation}
Using that $\dot\chi(L(w_k),t)=\tilde R^Tb(\psi(w_k,t))$ by \eqref{chi L prop} and \eqref{L def}, abbreviating $\chi=\chi(L(w_k),\tau)$ etc., using \eqref{y' chain} for $\tilde y(s)$, and finally making the substitution $s=\chi_1(L(w_k),\tau)$ and recalling that $\dot\chi_1(L(w_k),\cdot\,)<0$ by \eqref{chi1 one sign} and our assumption $\tilde s_k>0$, we thus obtain for $\forall t>0$
\begin{align}
  \int^t_0|b(\psi(w_k,\tau))|\,d\tau
    &= \int^t_0|\dot\chi(L(w_k),\tau)|\,d\tau \nonumber\\
    &=  \int_0^t\!\sqrt{\dot\chi_1^2+\dot\chi_2^2}\ d\tau \nonumber\\
    &=  \int_0^t\sqrt{1+[\tilde y'(\chi_1)]^2}\ |\dot\chi_1|\,d\tau \nonumber\\
    &=  \int_{\chi_1(L(w_k),t)}^{\chi_1(L(w_k),0)}\sqrt{1+[\tilde y'(s)]^2}\,ds \label{chi change of var}.
\end{align}
Now using that $\chi(L(w_k),0)=L(w_k)=(\tilde s_k,\tilde y_k)$ and \eqref{chi limit 0}, taking the limit $t\to\infty$ implies
\begin{equation} \label{tilde y int is tilde a}
   \int_0^{\tilde s_k}\sqrt{1+[\tilde y'(s)]^2}\,ds
     = \int^\infty_0|b(\psi(w_k,\tau))|\,d\tau
     = f_s(w_k)
     = \tilde a.
\end{equation}
\textit{Step 10:}
Next, let $u\in B_\eta(w)$ be fixed, and denote
\begin{align}
  s_t  &:= \chi_1(v_u,t)
  \qquad\text{for }\forall t\in(0,t_i'(u)], \label{st def} \\
  y(s) &:= y(\tilde s_k,y_u;s),
\end{align}
i.e.\ $y(s)$  is the curve passing through the point $(\tilde s_k,y_u)=v_u=L(z_i'(u))$ (recall \eqref{claim step 2b2}-\eqref{claim step 2b3}).
We claim for $\forall t\in(0,t_i'(u)]$ that
\begin{equation} \label{2h statement}
  \text{if}\quad
  \chi(v_u,[0,t])\subset W_{\sigma,\te}
  \quad\text{then}\quad
\begin{cases}
  \tilde a-f_i(u) \geq \hat f(s_t,y(s_t))  & \text{for}\ t<t_i'(u), \\
  \tilde a-f_i(u) = \hat f(L(u))        & \text{for}\ t=t_i'(u).
\end{cases} \vspace{.2cm}
\end{equation}
Indeed, if $\chi(v_u,[0,t])\subset W_{\sigma,\te}$ then \eqref{fi on Ei'}, a calculation analogous to \eqref{chi change of var}, and \eqref{hat f shift} show that
\begin{align}
  \tilde a - f_i(u)
    &=    \int^{t_i'(u)}_0\big|b\big(\psi(z_i'(u),\tau)\big)\big|\,d\tau \nonumber\\
    &\geq    \int^t_0\big|b\big(\psi(z_i'(u),\tau)\big)\big|\,d\tau \nonumber\\
    &=  \int_{\chi_1(L(z_i'(u)),t)}^{\chi_1(L(z_i'(u)),0)}\sqrt{1+[y'(s)]^2}\,ds \nonumber\\
    &=  \int_{s_t}^{\tilde s_k}\sqrt{1+[y'(s)]^2}\,ds \label{chi change of var 2}\\
    &=  \hat f(s_t,y(s_t)), \nonumber
\end{align}
where the integration bounds in \eqref{chi change of var 2} followed from \eqref{st def} and the relation $\chi(v_u,0)=v_u=(\tilde s_k,y_u)$. If $t=t_i'(u)$ then we have equality, and thus the second statement in \eqref{2h statement} follows if we can show that $\big(s_{t_i'(u)},y(s_{t_i'(u)})\big)=L(u)$.

To do so, note that by \eqref{st def} and \eqref{chi2=y(chi1)} we have $y(s_t)=y(\chi_1(v_u,t))=\chi_2(v_u,t)$ and thus
\begin{equation} \label{st yst}
   (s_t,y(s_t)) = \chi(v_u,t)
   \qquad\text{for }
   \forall t\in(0,t_i'(u)],
\end{equation}
and therefore by \eqref{chi L prop} in particular
\begin{equation}
  \big(s_{t_i'(u)},y(s_{t_i'(u)})\big)
    = \chi\big(L(z_i'(u)),t_i'(u)\big)
    = L\big(\psi(z_i'(u),t_i'(u))\big)
    = L(u) \label{chi L z}.
\end{equation}
\textit{Step 11:}
Next we claim that
\begin{equation} \label{flowline in WdK}
  \chi\big(v_u,[0,t_i'(u)]\big)\subset W_{\sigma,\te}.
\end{equation}
Suppose that this were false. Since $v_u\in W_{\sigma,\te-1}$ by \eqref{claim step 2b2}, the exit time
\[ \hat t:=\min\!\big\{ t\in[0,t_i'(u)]\,\big|\,\chi(v_u,t)\notin W_{\sigma,\te-1}\big\} >0 \]
would then be well-defined and fulfill
\begin{align*}
  \chi(v_u,[0,\hat t\qh)) &\subset W_{\sigma,\te-1}, \\
  \chi(v_u,\hat t\qh)     &\notin W_{\sigma,\te-1}.
\end{align*}
Since $\tilde s_k>0$, we would then have \eqref{chi1 one sign} at least for $t\in[0,\hat t\qh)$, and since $\chi(v_u,\hat t\qh)$ is not the origin (which would imply that also $0=v_u=L(z_i'(u))$ and thus $z_i'(u)=x$ in contradiction to Remark \ref{on manifold b neq0}), it would have to lie on the top or bottom border of $W_{\sigma,\te-1}$.
As a result, we would have
\begin{equation} \label{curve in Wst}
   \chi(v_u,[0,\hat t\qh])\subset W_{\sigma,\te},
\end{equation}
and so $y(s)$ is defined (and has graph in $W_{\sigma,\te}$) for $s\in[\chi_1(v_u,\hat t\qh),\tilde s_k]=[s_{\hat t},\tilde s_k]$.
Furthermore, since $\chi(v_u,\hat t\qh)=(s_{\hat t},y(s_{\hat t}))$ by \eqref{st yst}, we would have $|y(s_{\hat t})|=(\te-1)|s_{\hat t}|$ and thus
\begin{equation} \label{ysK tysK bound}
  |y(s_{\hat t})-\tilde y(s_{\hat t})| \geq |y(s_{\hat t})|-|\tilde y(s_{\hat t})|\geq(\te-1-\tilde \te)|s_{\hat t}|
\end{equation}
by \eqref{tilde y in wedge}. Since by \eqref{yu ytildek dist} and \eqref{nu def est} we also have
\[ |y(\tilde s_k)-\tilde y(\tilde s_k)| = |y_u-\tilde y_k| \leq \tilde\rho < \tilde s_k, \]
and since $\te-1-\tilde \te\geq1$ by \eqref{K const def}, the continuity of the function $s\mapsto s^{-1}|y(s)-\tilde y(s)|$ on $[s_{\hat t},\tilde s_k]$ would imply that there $\exists\bar s\in[s_{\hat t},\tilde s_k]$ such that
\begin{equation} \label{s bar def}
   |y(\bar s)-\tilde y(\bar s)|=\bar s.
\end{equation}
Now by \eqref{y(x) ode} we have for $\forall s\in[s_{\hat t},\tilde s_k]$
\begin{align*}
  \partial_s\big[s^{q/p}(y(s)-\tilde y(s))\big] \hspace{-2.3cm}& \\*
    &= \tfrac qp s^{q/p\,-1}(y-\tilde y) \\*
    &\quad  + s^{q/p}\bigg[\bigg({-\bigg[\frac rp+\frac{qy}{ps}\bigg]}+h(s,y)\bigg)
          -\bigg({-\bigg[\frac rp+\frac{q\tilde y}{ps}\bigg]+h(s,\tilde y)\bigg)\bigg]} \\
    &= s^{q/p}(h(s,y)-h(s,\tilde y)) \\
    &= s^{q/p}(y-\tilde y)\,\partial_yh(s,y^*)
\end{align*}
for some $y^*(s)$ between $y(s)$ and $\tilde y(s)$, and thus
\[
  s^{q/p}(y(s)-\tilde y(s))
    = s_{\hat t}^{q/p}(y(s_{\hat t})-\tilde y(s_{\hat t}))\exp\!\Big(\int_{s_{\hat t}}^{s}\partial_yh(s',y^*(s'))\,ds'\Big).
\]
Since with $(s',y(s'))$ and $(s',\tilde y(s'))$ also $(s',y^*(s'))$ is in $W_{\sigma,\te}$, we can use the estimate \eqref{hy est} to find
\begin{align*}
  s^{q/p}|y(s)-\tilde y(s)|
    &\leq s_{\hat t}^{q/p}|y(s_{\hat t})-\tilde y(s_{\hat t})|
          e^{\tfrac q{2p}\int_{s_{\hat t}}^{s} {s'}^{-1}\,ds'} \\*
    &= s_{\hat t}^{q/p}|y(s_{\hat t})-\tilde y(s_{\hat t})| \big(\tfrac{s}{s_{\hat t}}\big)^{q/2p} \\[.1cm]
  \hspace{-1cm}\Rightarrow\qquad
  s^{q/2p}|y(s)-\tilde y(s)|
     &\leq s_{\hat t}^{q/2p}|y(s_{\hat t})-\tilde y(s_{\hat t})| \\*
     &\leq (\te-1-\tilde \te)^{-q/2p}|y(s_{\hat t})-\tilde y(s_{\hat t})|^{1+q/2p}
\end{align*}
by \eqref{ysK tysK bound}. Setting $s:=\bar s$ and using \eqref{s bar def} and \eqref{K const def} would now imply
\begin{align}
  \bar s^{1+q/2p}
     &\leq (\te-1-\tilde \te)^{-q/2p}|y(s_{\hat t})-\tilde y(s_{\hat t})|^{1+q/2p} \nonumber \\
   \Rightarrow\qquad
   |y(s_{\hat t})-\tilde y(s_{\hat t})|
      &\geq (\te-1-\tilde \te)^{\frac{q/2p}{1+q/2p}}\bar s
      \geq (4+\dH+2\tilde \te)\bar s. \label{yshatt est}
\end{align}
Since by \eqref{tilde y in wedge}, by the equivalent of \eqref{y' upper bound} for $\tilde y$ and $\tilde\te$ instead of $y$ and $\te$, and by \eqref{d7 def} we have
\[
  \frac1{\bar s}\int_0^{\bar s} \sqrt{1+[\tilde y'(s)]^2}\,ds
    \leq 1+\sup_{0<s\leq\bar s}|\tilde y'(s)| 
    \leq 1+\tfrac2p(|r| + q\tilde \te)+1
     = \dH
\]
and by \eqref{hat f shift} and \eqref{tilde y int is tilde a} thus
\begin{align}
  \hat f(\bar s,\tilde y(\bar s))
    &=\int_{\bar s}^{\tilde s_k} \sqrt{1+[\tilde y'(s)]^2}\,ds \nonumber \\
    &= \int_0^{\tilde s_k} \sqrt{1+[\tilde y'(s)]^2}\,ds
       - \int_0^{\bar s} \sqrt{1+[\tilde y'(s)]^2}\,ds \nonumber \\
    &\geq \tilde a - \dH\bar s, \label{fhat sbar est}
\end{align}
we could finally use \eqref{curve in Wst} and \eqref{2h statement}, twice \eqref{hat f shift}, \eqref{fhat bound 2}, \eqref{fhat sbar est}, \eqref{yshatt est}, twice \eqref{tilde y in wedge} and \eqref{s bar def} to obtain the contradiction
\begin{align*}
  \tilde a
    &> \tilde a-f_i(u) 
     \geq  \hat f(s_{\hat t},y(s_{\hat t}))  \\
    &= \int_{s_{\hat t}}^{\tilde s_k} \sqrt{1+[y'(s)]^2}\,ds \\
    &= \int_{s_{\hat t}}^{\bar s} \sqrt{1+[y'(s)]^2}\,ds + \int_{\bar s}^{\tilde s_k} \sqrt{1+[y'(s)]^2}\,ds \\
    &= \int_{s_{\hat t}}^{\bar s} \sqrt{1+[y'(s)]^2}\,ds + \hat f(\bar s,y(\bar s)) \\
    &\geq \int_{s_{\hat t}}^{\bar s} |y'(s)|\,ds
          + \big[\hat f(\bar s,y(\bar s)) - \hat f(\bar s,\tilde y(\bar s))\big]
          + \hat f(\bar s,\tilde y(\bar s)) \\
    &\geq |y(\bar s)-y(s_{\hat t})| - 3|y(\bar s)-\tilde y(\bar s)| + (\tilde a - \dH\bar s) \\
    &\geq \big[|y(s_{\hat t})-\tilde y(s_{\hat t})|-|\tilde y(s_{\hat t})-\tilde y(\bar s)| - |\tilde y(\bar s)-y(\bar s)|\big] \\
    & \hspace{5.42cm}{}- 3|y(\bar s)-\tilde y(\bar s)| + \tilde a - \dH\bar s \\
    &\geq (4+\dH+2\tilde \te)\bar s - |\tilde y(s_{\hat t})|-|\tilde y(\bar s)| -4 |y(\bar s)-\tilde y(\bar s)| + \tilde a - \dH\bar s \\
    &\geq (4+\dH+2\tilde \te)\bar s - \tilde \te s_{\hat t}-\tilde \te\bar s -4 \bar s + \tilde a - \dH\bar s\\
    &\geq \tilde a,
\end{align*}
concluding the proof of \eqref{flowline in WdK}.\\[.2cm]
\textit{Step 12:}
We can now put everything together: By \eqref{flowline in WdK} the condition in \eqref{2h statement} is fulfilled for $t=t_i'(u)$, and so we have $\tilde a-f_i(u)=\hat f(L(u))$. This relation was shown for $\forall u\in B_\eta(w)$, and differentiating it at $u=w$ shows that
\[ |\nf_i(w)| = |\nhf(L(w))\tilde R^T|=|\nhf(L(w))|. \]
Since $L(w)=\chi\big(L(z_i'(w)),t_i'(w)\big)\in W_{\sigma,\te}$ by \eqref{chi L z} and \eqref{flowline in WdK}, \eqref{fhat bound 1}-\eqref{fhat bound 2} thus give us the upper bound
\[ |\nf_i(w)| \leq \big[5+\tfrac8p(|r|+q\te)\big]+3 \]
which is independent of our choice of $w\in E_i'\cap B_\mu(x)$. This terminates our proof of property (vi).

\paragraph{Proof of property (vii).}
\hspace{-1.5pt}Let $\cF:=\sup\!\big\{|\nf_i(v)|\,\big|\,v\in E_i',\,i\in I\big\}$,
which is finite by what we showed in~part~(vi) and which fulfills $\cF\geq1$ by our calculation for part~(ii.3) and by part (iv).
Let $w\in\bar B_\eps(x)$ and $i\in I$; we must show that $f_i(w)\leq\cF|w-x|$.

If $f_i(w)=0$ then the estimate is trivial.
Otherwise the function $h\in C([0,1],[0,\tilde a])$, defined by $h(\te):=f_i(x+\te(w-x))$, fulfills
\[  h(1)=f_i(w)>0=f_i(x)=h(0)  \]
by property~(i), and thus the values
\begin{align*}
  \te_1&:=\max\!\big\{\te\in[0,1]&&\hspace{-2.9cm}|\,h(\te)=0\big\}, \\*
  \te_2&:=\min\big\{\te\in[\te_1,1]&&\hspace{-2.9cm}|\,h(\te)=f_i(w)\big\}
\end{align*}
fulfill $\te_1<\te_2$. For $\forall\te\in(\te_1,\te_2)$ we then have
\[ 0=h(\te_1)<h(\te)<h(\te_2)=f_i(w)\leq\tilde a, \]
i.e.~$x+\te(w-x)\in f_i^{-1}\big((0,\tilde a)\big)=E_i'$, so $h$ is $C^1$ on $(\te_1,\te_2)$ by what was shown in part (ii.2). Thus by the mean value theorem $\exists\hat\te\in(\te_1,\te_2)$ such that
\begin{align*}
\hspace{2.4cm}
  f_i(w)&=h(\te_2)-h(\te_1)
         =h'(\hat\te)\cdot(\te_2-\te_1) \\
        &\leq\big|\nf_i(x+\hat\te(w-x))\big||w-x|\cdot|\te_2-\te_1| \\
        &\leq\cF|w-x|\cdot1.
\hspace{6.45cm}\openbox
\end{align*}

\begin{appendices}
  
\newpage

\section{Proof of Lemma \ref{superprop} \texorpdfstring{--}{-} Some Technical Details} \label{proof appendix}

In this appendix, let us denote
\begin{subequations}
\begin{align}
E_s&:=\big\{(v_1,\dots,v_n)\in\Rn\,\big|\,v_{n_s+1}=\dots=v_n=0\big\}, \label{Es def} \\
E_u&:= \big\{(v_1,\dots,v_n)\in\Rn\,\big|\,v_1=\dots=v_{n_s}=0\big\}. \label{Eu def}
\end{align}
\end{subequations}

\subsection{Remarks on the Construction of \texorpdfstring{$M_s^{loc}$, $M_u^{loc}$, $p_s$ and $p_u$}{Msloc, Muloc, ps and pu}}
\label{Msloc ps comments sec}

First let us quickly review the proof of the Stable Manifold Theorem found in \cite[Sec.~2.7]{Perko} and \cite[Sec.\,13.4]{CL}. Both sources begin the construction of $M_s^{loc}$ by using the transformation $w=x+R\tilde w$, $\tilde b(\tilde w):=R^{-1}b(x+R\tilde w)$ to reduce it to the case where $x=0$ and $R=I$. Our formulas for general $x$ and $R$ can thus be obtained either by reversing this transformation, or directly by generalizing the construction in \cite{Perko,CL}. Their analogues for $M_u^{loc}$ are then obtained by reversing time and replacing $b$ by $-b$.

In a first step, the method of successive approximations is used \cite[p.\,109-110]{Perko} to construct for every $v$ in some ball $B_\delta(x)\subset D$ a function $\chi_s^v$ with
\begin{equation} \label{exp conv}
  \lim_{t\to\infty}\chi_s^v(t)=x
\end{equation}
that solves \eqref{projection equation} and thus $\dot\chi_s^v=b(\chi_s^v)$, i.e.~$\chi_s^v(t)=\psi(\chi_s^v(0),t)$. One then defines the function $p_s(v):=\chi_s^v(0)$ for $\forall v\in B_\delta(x)$ (implying \eqref{chiv def}), and finally one defines the manifold $M_s^{loc}$ as the image of the function $\phi_s\colon B_\eta^{n_s}(0)\to D$, $\phi_s(u):=p_s(x+R(u,0,\dots,0)^T)$, where $\eta:=\delta/|R|$, and where $B_\eta^{n_s}(0)$ denotes the ball in $\R^{n_s}$ with radius $\eta$ and center $0$. Analogously one can define the functions $\chi_u^v$, $p_u$ and $\phi_u$ and the manifold $M_u^{loc}$.

The functions $p_s$ and $p_u$ are shown to be $C^1$ with derivatives such that
\begin{equation}\label{phi deriv}
  \big(\nabla\phi_s(0),\nabla\phi_u(0)\big)=R
\end{equation}
(see \cite[last line on p.\,331, and Thm.\,4.2]{CL}), and since $\phi_s(0)=p_s(x)=x$ and $\phi_u(0)=p_u(x)=x$, this shows that $M_s^{loc}$ and $M_u^{loc}$ are proper $C^1$-manifolds with
%
%One can show that $\phi_s(0)=\phi_u(0)=x$ and $\big(\nabla\phi_s(0),\nabla\phi_u(0)\big)=R$, so that in particular
\begin{equation} \label{tangent space reps}
  T_xM_s^{loc}=RE_s
  \qquad\text{and}\qquad
  T_xM_u^{loc}=RE_u.
\end{equation}
More details on the remaining properties of the functions $p_s$ and $p_u$ can be found at the end of this section.\\[.2cm]
\eqref{Msloc Muloc intersection}: Next we claim that we can decrease $\eta>0$ so that \eqref{Msloc Muloc intersection} holds. Indeed, otherwise we could find sequences $(u_s^k)_{k\in\N}\subset B_\eta^{n_s}(0)\setminus\{0\}$ and $(u_u^k)_{k\in\N}\subset B_\eta^{n_u}(0)\setminus\{0\}$ converging to zero such that for $\forall k\in\N$  and $u^k:=(u_s^k,-u_u^k)$ we have
\begin{align*}
  0&=\phi_s(u_s^k)-\phi_u(u_u^k)\\
   &=\big(x+\nabla\phi_s(0)u_s^k\big)-\big(x+\nabla\phi_u(0)u_u^k\big)+o(|u_s^k|+|u_u^k|)\\
   &=\big(\nabla\phi_s(0),\nabla\phi_u(0)\big)u^k+o(|u^k|)\\
   &=Ru^k+o(|u^k|),
\end{align*}
and dividing by $|u^k|$ and multiplying by $R^{-1}$ would imply that $u^k/|u^k|\to0$.\\[.2cm]
\eqref{sup skp}: To ensure that also \eqref{sup skp} is fulfilled, note that the vectors $y_s$ and $y_u$ in \eqref{sup skp} are of the form
\[
   y_s(c_s,u_s):=\frac{\nabla\phi_s(u_s)c_s}{|\nabla\phi_s(u_s)c_s|}\,, \qquad
   y_u(c_u,u_u):=\frac{\nabla\phi_u(u_u)c_u}{|\nabla\phi_u(u_u)c_u|}
\]
for some $(c_s,u_s)\in\partial B_1^{n_s}(0)\times B^{n_s}_\eta(0)$ and $(c_u,u_u)\in \partial B_1^{n_u}(0)\times B^{n_u}_\eta(0)$. Since $y_s(c_s,0)\in T_xM_s^{loc}$ and $y_u(c_u,0)\in T_xM_u^{loc}$ and since $T_xM_s^{loc}\cap T_xM_u^{loc}=R(E_s\cap E_u)=\{0\}$ by \eqref{tangent space reps}, we have
$y_s(c_s,0)\neq y_u(c_u,0)$ and thus
\[
  \Skp{y_s(c_s,0)}{y_u(c_u,0)}<1
\qquad\text{for $\forall c_s\in\partial B_1^{n_s}(0)$ and $\forall c_u\in\partial B_1^{n_u}(0)$.}
\]
Thus the continuity of the function $f(c_s,u_s,c_u,u_u):=\Skp{y_s(c_s,u_s)}{y_u(c_u,u_u)}$ and the compactness of $\partial B_1^{n_s}(0)$ and $\partial B_1^{n_u}(0)$ imply that
\[
  \sup\!\big\{f(c_s,0,c_u,0)\,\big|\,c_s\in\partial B_1^{n_s}(0),\,c_u\in\partial B_1^{n_u}(0)\big\}<1,
\]
and so we can decrease $\eta>0$ so much that
\begin{align*}
  \te_0=\sup\!\big\{f(c_s,u_s,c_u,u_u)\,\big|\,&(c_s,u_s)\,\in\partial B_1^{n_s}(0)\qb\times B^{n_s}_\eta(0),\\
&(c_u,u_u)\in\partial B_1^{n_u}(0)\times B^{n_u}_\eta(0)\,\big\}<1,
\end{align*}
which is \eqref{sup skp}.\\[.2cm]
\eqref{Ms Mloc relation}-\eqref{unstable man prop}:
In \cite[Ch.\,13, Thm.\,4.1]{CL} it is shown that $\exists a_0\in\big(0,\frac{\eta}{|R^{-1}|}\big)$ such that the property \eqref{stable man prop} (and analogously \eqref{unstable man prop}) holds. The relation ``$\subset$'' in \eqref{Ms Mloc relation} is now a direct consequence of \eqref{stable man prop}-\eqref{unstable man prop}, while the relation ``$\supset$'' in \eqref{Ms Mloc relation} was already clear from \eqref{exp conv} and its counterpart $\lim_{t\to-\infty}\chi_u^v(t)=x$.\\[.27cm]
%
%In \cite[Ch.\,13, Thm.\,4.1]{CL} it is shown that $\exists a_0'\in\big(0,\frac{\eta}{|R^{-1}|}\big)$ such that the property \eqref{stable man prop} (and equivalently \eqref{unstable man prop}) holds with $a_0$ replaced by $a_0'$, and by \cite[Corollary on p.\,115]{Perko} there $\exists a_0\in(0,a_0']$ such that
%\begin{align*}
%   \psi\big(M_s^{loc}\cap\bar B_{a_0}(x),[0,\infty)\big)&\subset \bar B_{a_0'}(x), \\
%   \psi\big(M_u^{loc}\cap\bar B_{a_0}(x),(-\infty,0]\big)&\subset \bar B_{a_0'}(x),
%\end{align*}
%which together with the modified versions of \eqref{stable man prop}-\eqref{unstable man prop} (with $a_0$ replaced by $a_0'$) implies \eqref{Msloc Muloc flow inv 1}-\eqref{Msloc Muloc flow inv 2}. Since $a_0\leq a_0'$, \eqref{stable man prop}-\eqref{unstable man prop} also hold as they are, i.e.\ for $a_0$ instead of $a_0'$. The relation ``$\subset$'' in \eqref{Ms Mloc relation} is now a direct consequence of \eqref{stable man prop}-\eqref{unstable man prop}, while the relation ``$\supset$'' in \eqref{Ms Mloc relation} was already clear from \eqref{exp conv} and its counterpart $\lim_{t\to-\infty}\chi_u^v(t)=x$.
%
To see that
$p_s(\bar B_{a_0}(x))\subset M_s^{loc}$ (observe that $a_0<\frac{\eta}{|R^{-1}|}=\frac{\delta}{|R||R^{-1}|}\leq\delta$), first note that the construction of $\chi_s^v$ in~\cite{Perko} implies for $\forall v,w\in B_\delta(x)$ that
\begin{equation} \label{ps prop 2}
  \text{if\quad $v-w\in RE_u$\quad then\quad $\chi_s^v=\chi_s^w$\quad and thus\quad $p_s(v)=p_s(w)$.}
\end{equation}
Therefore, if we denote by $P_s$ the orthogonal projection onto $E_s$ and if for $\forall v\in\bar B_{a_0}(x)$ we let $u_v\in\R^{n_s}$ be the vector such that $(u_v,0)=P_sR^{-1}(v-x)$ then $|u_v|=|(u_v,0)|\leq|R^{-1}(v-x)|\leq|R^{-1}|a_0<\eta$, i.e.\ $u_v\in B^{n_s}_\eta(0)$, and since $v-(x+R(u_v,0))=R(I-P_s)R^{-1}(v-x)\in RE_u$, \eqref{ps prop 2} implies that $p_s(v)=p_s(x+R(u_v,0))=\phi_s(u_v)\in M_s^{loc}$. Similarly, one can show that $p_u(\bar B_{a_0}(x))\subset M_u^{loc}$.\\[.2cm]
\eqref{as prop 0}-\eqref{as prop 1}: From \eqref{projection equation} and \eqref{tangent space reps} one can see that for $\forall v\in B_\delta(x)$ we have $p_s(v)-v\in RE_u=T_xM_u^{loc}$, i.e.~\eqref{as prop 0}.
Therefore, if $v\in$ $M_s^{loc}\cap B_\delta(x)$ and thus $v=p_s(w)$ for some $w\in B_\delta(x)$, then $v-w=p_s(w)-w$ $\in RE_u$, and thus by \eqref{ps prop 2} we have $p_s(v)=p_s(w)=v$, which is \eqref{as prop 1}.\\[.2cm]
\eqref{compact Msloc parts}: Note that $M_s^{loc}\cap\bar B_{a_0}(x)=p_s(\bar B_{a_0}(x))\cap\bar B_{a_0}(x)$ (indeed, ``$\supset$'' is clear since $p_s$ maps into $M_s^{loc}$, ``$\subset$'' follows from \eqref{as prop 1}). The continuity of $p_s$ thus implies that $M_s^{loc}\cap\bar B_{a_0}(x)$ is compact, and an analogous representation shows that also $M_u^{loc}\cap\bar B_{a_0}(x)$ is compact.

\subsection{Proof of Lemma \ref{distance function saddle}} \label{fs fu appendix b}

\begin{proof}
We will only show these properties for $f_s$. Since $M_s^{loc}$ is an $n_s$-dimensional $C^1$-manifold, it can locally be described by a diffeomorphism $\zeta_s\colon U\to\zeta_s(U)=B_\mu(0)$, for some neighborhood $U\subset\bar B_{a_0}(x)$ of $x$ and some $\mu>0$, that fulfills $\zeta_s(x)=0$ and
\begin{subequations}
\begin{align}
  M_s^{loc}\cap U &= \zeta_s^{-1}(E_s), \label{zeta Es inv}\\
  \hspace{-.85cm}\text{i.e.}\quad
  \zeta_s(M_s^{loc}\cap U) &= E_s\cap\zeta_s(U),  \label{zeta Es}
\end{align}
\end{subequations}
where $E_s$ is given by \eqref{Es def}.

Indeed, in the notation of Appendix \ref{Msloc ps comments sec}, we can define $\zeta_s$ via its inverse
\begin{equation} \label{zeta inv def}
  \zeta_s^{-1	}(u_1,\dots,u_n):=
    \phi_s(u_1,\dots,u_{n_s})+R(0,\dots,0,u_{n_s+1},\dots,u_n)^T
\end{equation}
for $\forall u\in B_\mu(0)$, which is a diffeomorphism for sufficiently small $\mu\in(0,\eta]$ since $\nabla\q\zeta_s^{-1}(0)=R$ by \eqref{phi deriv}, and where we also choose $\mu$ so small that for $\forall u\in B_\mu(0)$ we have $\zeta_s^{-1}(u),\phi_s(u_1,\dots,u_{n_s})\in\bar B_{a_0}(x)$. The relation ``$\supset$'' in \eqref{zeta Es inv} is clear. To show the reverse relation ``$\subset$'', let $w\in M_s^{loc}\cap U$, and let $u\in B_\mu(0)$ be such that $w=\zeta_s^{-1}(u)$. Then $w-\phi_s(u_1,\dots,u_{n_s})\in RE_u$ by \eqref{zeta inv def}, and so \eqref{as prop 1}, \eqref{ps prop 2} and again \eqref{as prop 1} imply that
\[ \zeta_s^{-1}(u)=w=p_s(w)=p_s(\phi_s(u_1,\dots,u_{n_s}))=\phi_s(u_1,\dots,u_{n_s}). \]
By \eqref{zeta inv def} this shows that $u\in E_s$, i.e.\ $w\in\zeta_s^{-1}(E_s)$, terminating the proof of \eqref{zeta Es inv}.\\[.2cm]
Now consider the vector field $\tilde b\in C^1(U,\Rn)$ defined by
\[
  \tilde b(w) := b(w) - 2R\big(\begin{smallmatrix}0&0\\0&Q\end{smallmatrix}\big)R^{-1}\nabla\q\zeta_s(x)^{-1}\zeta_s(w),
 \qquad w\in U. 
\]
In this new vector field, $x$ is an attractor since by \eqref{A def}
\[
  \nabla\tilde b(x)= \nabla b(x)-2R\big(\begin{smallmatrix}0&0\\0&Q\end{smallmatrix}\big)R^{-1}\nabla\q\zeta_s(x)^{-1}\nabla\q\zeta_s(x)
  = R\big(\begin{smallmatrix}P&0\\0&-Q\end{smallmatrix}\big)R^{-1}
\]
has only eigenvalues with negative real parts. Also, we have $\tilde b(w)=b(w)$ for $\forall w\in M_s^{loc}\cap U$. Indeed, for $\forall w\in M_s^{loc}\cap U$ we have by \eqref{zeta Es} and \eqref{tangent space reps}
\begin{align*}
  \zeta_s(w)\in \zeta_s(M_s^{loc}\cap U)
     \subset E_s
     &%=T_0E_s
     =T_0(E_s\cap\zeta_s(U))
      =T_0\zeta_s(M_s^{loc}\cap U) \\
     &=\nabla\q\zeta_s(x)\,T_x(M_s^{loc}\cap U)
     =\nabla\q\zeta_s(x)\,RE_s,
\end{align*}
i.e.~$R^{-1}\nabla\q\zeta_s(x)^{-1}\zeta_s(w)\in\! E_s$, which implies that $\!\big(\begin{smallmatrix}0&0\\0&Q\end{smallmatrix}\big)\!R^{-1}\nabla\q\zeta_s(x)^{-1}\zeta_s(w)=0$.\linebreak
\indent Since $x$ is an attractor of $\tilde b$, there $\exists\nu>0$ such that $B_\nu(x)$ is contained in its basin of attraction, which in particular implies that $B_\nu(x)\subset U$ and that the flow $\tilde\psi(w,t)$ corresponding to $\tilde b$ is defined and in $U$ for $\forall w\in B_\nu(x)$ and $\forall t\in[0,\infty)$. Thus we can define a function $\tilde f_s\colon B_\nu(x)\to[0,\infty)$ based on this flow $\tilde\psi$ as in Definition~\ref{fs fu def}, which has all the properties of Lemma~\ref{distance function}. In particular, $\tilde f_s$ is continuous on $B_\nu(x)$ and $C^1$ on $B_\nu(x)\setminus\{x\}$.

Furthermore, by \cite[Corollary on p.\,115]{Perko} we can reduce $\nu>0$ so much that for $\forall w\in M_s^{loc}\cap B_\nu(x)$ we have $\psi(w,[0,\infty))\subset U\subset\bar B_{a_0}(x)$, and thus in fact $\psi(w,[0,\infty))\subset M_s^{loc}\cap U$ because of \eqref{stable man prop}.
Therefore, since $b=\tilde b$ on $M_s^{loc}\cap U$, any flowline $\psi(w,[0,\infty))$ starting from a point $w\in M_s^{loc}\cap B_\nu(x)$ coincides with the flowline $\tilde\psi(w,[0,\infty))$, which implies that $f_s(w)=\tilde f_s(w)$ for $\forall w\in M_s^{loc}\cap B_\nu(x)$.

In particular, $f_s$ is finite-valued on $M_s^{loc}\cap B_\nu(x)$, and if we decrease $a_0$ so much that $a_0\in(0,\nu)$ then (iii) and (iv) hold, where for $\cI$ we choose the constant $\cD\geq1$ given by Lemma \ref{distance function}~(iv) corresponding to the function $\tilde f_s$ and the compact set $K:=\bar B_{a_0}(x)$.
Furthermore, given any $\forall w\in M_s$, by \eqref{Ms def} and \eqref{Ms Mloc relation} there is a $T>0$ such that $\psi(w,T)\in M_s^{loc}\cap B_\nu(x)$ and thus
\[
  f_s(w) = \int_0^T|b(\psi(w,t))|\,dt+f_s(\psi(w,T))<\infty,
\]
so $f_s$ is finite-valued on all of $M_s$. The statements in (i) now follow from
\begin{align*}
  \partial_t f_s(\psi(w,t))
    &=\lim_{h\to0}\tfrac1h\big[f_s(\psi(w,t+h))-f_s(\psi(w,t))\big] \\
    &=\lim_{h\to0}\frac1h\bigg[\int_0^\infty\!\big|b\big(\psi(w,\tau+t+h)\big)\big|\,d\tau -\! \int_0^\infty\!\big|b\big(\psi(w,\tau+t)\big)\big|\,d\tau\bigg] \\
    &=-\lim_{h\to0}\frac1h\int^{t+h}_t|b(\psi(w,\tau))|\,d\tau =-|b(\psi(w,t))|.
\end{align*}
The proof of (ii) is identical to the one of Lemma \ref{distance function}~(iii), see \eqref{proof lower bound}.
\end{proof}

\subsection{Proof of Lemma \ref{Msa compact lemma}} \label{Ms compact appendix}

\begin{proof}
First we will show that
\begin{equation} \label{fs subset Msloc}
  f_s^{-1}\big([0,a_0]\big)\subset M_s^{loc}\cap\bar B_{a_0}(x),
\end{equation}
which in particular says that $f_s^{-1}\big([0,a_0]\big)$ is a subset of $M_s^{loc}$.
By \eqref{fs est 1 saddle lower} we have $f_s^{-1}\big([0,a_0]\big)\subset\bar B_{a_0}(x)$. Thus, if \eqref{fs subset Msloc} were wrong then there would be a $w\in f_s^{-1}\big([0,a_0]\big)\setminus M_s^{loc}\subset\bar B_{a_0}(x)\setminus M_s^{loc}$, and by \eqref{stable man prop} we could find a $t>0$ such that $\psi(w,t)\notin\bar B_{a_0}(x)$.
But then by \eqref{fs est 1 saddle lower} and Lemma~\ref{distance function saddle}~(i) we would have $a_0<|\psi(w,t)-x|\leq f_s(\psi(w,t))\leq f_s(\psi(w,0))=f_s(w)$, contradicting $w\in f_s^{-1}\big([0,a_0]\big)$, and \eqref{fs subset Msloc} is proven.

Now let $\tilde f_s\in C(\bar B_{a_0}(x),[0,\infty))$ be the function given by Lemma \ref{distance function saddle}~(iii) that fulfills $f_s=\tilde f_s$ on $M_s^{loc}\cap\bar B_{a_0}(x)$. Then by \eqref{fs subset Msloc} we have
\begin{align*}
 f_s^{-1}\big([0,a_0]\big)&=       f_s^{-1}\big([0,a_0]\big)\cap \big(M_s^{loc}\cap\bar B_{a_0}(x)\big) \\*
                  &=\tilde f_s^{-1}\big([0,a_0]\big)\cap \big(M_s^{loc}\cap\bar B_{a_0}(x)\big).
\end{align*}
Since $\tilde f_s^{-1}\big([0,a_0]\big)$ and by \eqref{compact Msloc parts} also $M_s^{loc}\cap\bar B_{a_0}(x)$ are compact, this shows that $f_s^{-1}\big([0,a_0]\big)$ is compact. The statements for $f_u^{-1}\big([0,a_0]\big)$, $\Msa=f_s^{-1}(\{a\})$ and $\Mua=f_u^{-1}(\{a\})$ follow from similar arguments.\\[.2cm]
Next let us show the first relation in \eqref{Msa flow = Ms}. The inclusion ``$\subset$'' is clear since $\Msa\subset M_s\setminus\{x\}$. To show the inclusion ``$\supset$'', let $a\in(0,a_0]$ and $w\in M_s\setminus\{x\}$. By \eqref{Ms def} and \eqref{Ms Mloc relation} there $\exists t\geq0$ so large that $\psi(w,t)\in M_s^{loc}\cap\bar B_{a/\cI}(x)$, which by \eqref{fs est 1 saddle upper} implies that
\begin{equation} \label{fs a est 1}
 f_s(\psi(w,t))\leq\cI|w-x|\leq a
\end{equation}
since $\frac{a}{\cI}\leq a\leq a_0$. Since by \eqref{Msloc Muloc intersection} we have $w_0:=\psi(w,t)\in M_s^{loc}\cap\bar B_{a/\cI}(x)\setminus\{x\}\subset \bar B_{a_0}(x)\setminus M_u^{loc}$, by \eqref{unstable man prop} there $\exists t'<0$ such that $\psi(w_0,t')\notin\bar B_{a_0}(x)$ and by \eqref{fs est 1 saddle lower} thus
\begin{equation} \label{fs a est 2}
  f_s(\psi(w,t+t'))=f_s(\psi(w_0,t'))\geq|\psi(w_0,t')-x|>a_0\geq a.
\end{equation}
Now by \eqref{fs a est 1}, \eqref{fs a est 2} and the continuity of $f_s(\psi(w,\cdot\,))$ shown in Lemma~\ref{distance function saddle}~(i), there $\exists t''\in[t+t',t]$ such that $f_s(\psi(w,t''))=a$, i.e.~$v:=\psi(w,t'')\in\Msa$, which implies $w=\psi(v,-t'')\in\psi(\Msa,\R)$. This proves that $M_s\setminus\{x\}\subset\psi(\Msa,\R)$.\\[.2cm]
Finally, observe that in the two-dimensional case $M_s\setminus\{x\}$ consists of only two distinct flowlines, each of which contain by Lemma \ref{distance function saddle}~(i) at most and by \eqref{Msa flow = Ms} at least one point in $M_s^a$. Thus $M_s\setminus\{x\}$ contains exactly two points in $M_s^a$, and since $M_s^a\subset M_s\setminus\{x\}$ by \eqref{Msa flow = Ms}, this shows that $M_s^a$ consists of exactly two points. Analogous arguments show this statement also for~$M_u^a$.
\end{proof}

\subsection{Proof of Lemma \ref{triangle inequality lemma}}
\label{triangle inequality sec}

\begin{proof}
Let $d:=-1+\sqrt{2+2\te}\in(0,1)$, which fulfills $d^2+2d-(1+2\te)=0$. Let $v,w\in\Rn$ fulfill $\skp{v}{w}\leq\te|v||w|$, and w.l.o.g.~let us assume that $|w|\leq|v|$. Now if $v=0$ then $w=0$, and the estimate is trivial. Otherwise
\begin{align*}
             && \frac{|w|}{|v|} & \leq 1=\frac{2(d-\te)}{1-d^2} \\
 &\Rightarrow & 2\te|v|+|w|  &\leq 2d|v|+d^2|w| \\
 &\Rightarrow & |v+w|^2 &=|v|^2+2\skp{v}{w}+|w|^2 \leq |v|^2+2\te|v||w|+|w|^2 \\*
              &&& \leq|v|^2+2d|v||w|+d^2|w|^2=\big(|v|+d|w|\big)^2.
\qedhere
\end{align*}
\end{proof}

\subsection{Proof of Lemma \ref{Hartman-Grobman Lemma}}
\label{Hartman-Grobman sec}

\begin{proof}
We will only show part (i); part (ii) can be proven analogously. According to the Hartman-Grobman-Theorem \cite[p.119]{Perko} there exists an open set $U\subset D$ containing $x$, and a homeomorphism $F\colon U\to F(U)\subset\Rn$ such that $F(x)=0$, and that for $\forall w\in U$ and every interval $J\subset\R$ with $0\in J$ and $\psi(w,J)\subset U$ we have $\forall t\in J\colon\ F(\psi(w,t))=e^{tA'}F(w)$, where $A':=\big(\begin{smallmatrix}P&0\\0&Q\end{smallmatrix}\big)$. In addition, we may assume that
\begin{equation}
  F^{-1}(E_u)\subset M_u^{loc},  \label{F inv Eu} \\
% \qquad\text{where} \\*
%  E_u         &:= \big\{(u_1,\dots,u_n)\in\Rn\,\big|\,u_1=\dots=u_{n_s}=0\big\}. \nonumber
\end{equation}
where $E_u$ is given by \eqref{Eu def}.

Indeed, by picking $\delta>0$ sufficiently small we can make sure that for $\forall w\in B_\delta(x)\cap F^{-1}(E_u)$ and $\forall t\leq0$, $|e^{tA'}F(w)|\leq\big(\sup_{\tau\leq0}|e^{\tau Q}|\big)|F(w)|$ is so small that $F^{-1}(e^{tA'}F(w))\in U\cap\bar B_{a_0}(x)$ and thus $\psi(w,t)=F^{-1}(e^{tA'}F(w))\in \bar B_{a_0}(x)$, which by \eqref{unstable man prop} implies that $w\in M_u^{loc}$. Therefore we have $B_\delta(x)\cap F^{-1}(E_u)\subset M_u^{loc}$, and so \eqref{F inv Eu} holds if we replace $F$ by $F|_{B_\delta(x)\cap U}$\hspace{1pt}.\\[.2cm]
\noindent Now let us decrease $a_1>0$ so much that $\bar B_{a_1}(x)\subset U$, let $\eta>0$, and define
\begin{equation} \label{K1 K2 def}
   K_1:=\bar B_{a_1}(x)\cap M_u^{loc}
    \qquad\text{and}\qquad
   K_2:=\bar B_{a_1}(x)\setminus N_\eta(K_1).
\end{equation}
Since $K_2$ is a compact subset of $U$, $F(K_2)$ is compact as well, and since by \eqref{F inv Eu} and \eqref{K1 K2 def} we have
\begin{align*}
  F(K_2)\cap E_u
   &= F(K_2\cap F^{-1}(E_u)) \\
   &\subset F\big((\bar B_{a_1}(x)\setminus K_1)\cap M_u^{loc}\big) \\
   &= F\big((\bar B_{a_1}(x)\setminus M_u^{loc})\cap M_u^{loc}\big)
    = \varnothing,
\end{align*}
there $\exists\nu>0$ such that
\begin{equation} \label{FK2 disjoint}
  F(K_2)\cap\bar N_\nu(E_u)=\varnothing.
\end{equation}
Finally, let $c:=\sup_{t\geq0}|e^{tP}|\in[1,\infty)$, and choose $\mu\in(0,a_1)$ so small that $\forall w\in\bar B_\mu(x)\colon\ |F(w)|<\frac{\nu}{c}$.

Now let $w\in\bar B_\mu(x)\setminus M_s^{loc}$. Since $\mu<a_1<a_0$, by \eqref{stable man prop} the flowline starting at $w$ will eventually leave $B_{a_1}(x)$ as $t\to\infty$. Denote the exit time by $T_1(w)>0$ and let $t\in[0,T_1(w)]$.
Then since $\psi(w,[0,t])\subset\bar B_{a_1}(x)\subset U$, we have $F(\psi(w,t))=e^{tA'}F(w)=u(t)+v(t)$, where
$u(t):=\bigl(\begin{smallmatrix}0&0\\0&e^{tQ}\end{smallmatrix}\bigr)F(w)\in E_u$ and $v(t):=\bigl(\begin{smallmatrix}e^{tP}&0\\0&0\end{smallmatrix}\bigr)F(w)$. Since $|v(t)|\leq|e^{tP}||F(w)|\leq c\, \cdot\frac{\nu}{c}=\nu$, this representation shows that
$F(\psi(w,t))\in\bar N_\nu(E_u)\subset\Rn\setminus F(K_2)$ by \eqref{FK2 disjoint}, and thus
$\psi(w,t)\in\bar B_{a_1}(x)\setminus K_2=\bar B_{a_1}(x)\cap N_\eta(K_1)$ by \eqref{K1 K2 def}. Since $t\in[0,T_1(w)]$ was arbitrary, we can conclude that $\psi\big(w,[0,T_1(w)]\big)\subset\bar B_{a_1}(x)\cap N_\eta(K_1)$, which is \eqref{HG inclusion 1}.
\end{proof}

\subsection{Proof of Lemma \ref{compact hitting partition}}
\label{compact hitting partition sec}

\begin{proof}
Let $a\in (0,a_0]$. By \eqref{Msa flow = Ms} and \eqref{main condition} we have $M_s^a\subset M_s\setminus\{x\}\subset\bigcup_{i\in I}\psi(M_i,\R)$, and in fact we have 
\begin{equation} \label{Msa plus union}
  M_s^a\subset\bigcup_{i\in I^+}\psi(M_i,\R).
\end{equation}
Indeed, if $w\in M_s^a$ and thus $w\in\psi(M_i,\R)$ for some $i\in I$ then by \eqref{prime lemma eq 1} we have $f_{M_i}(\psi(w,t))>0$ for $\forall t>-t_i(w)$,
and by \eqref{Ms def} and \eqref{fMi neq0} taking the limit $t\to\infty$ implies that $f_{M_i}(x)>0$, i.e.~$i\in I^+$.

In the two-dimensional case ($n=2$) this immediately shows that the sets $K_i^a$ defined in \eqref{2D Kia}, which by the last statement of Lemma~\ref{Msa compact lemma} contain at most two points and are thus compact, fulfill the first relation in \eqref{Ki +- union}. For $n\geq3$ we construct the sets $K_i^a$ for $i\in I^+$ as follows: Since the sets $\psi(M_i,\R)$ are open by Lemma~\ref{man2func lemma}, by \eqref{Msa plus union} we have that for $\forall w\in M_s^a$ $\exists i_w\in I^+$ $\exists r_w>0\colon\bar B_{r_w}(w)\subset \psi(M_{i_w},\R)$. Since $\{B_{r_w}(w)\,|\,w\in M_s^a\}$ is an open covering of the compact set $M_s^a$, there is a finite subcovering, i.e.~there is a finite set $F\subset M_s^a$ such that $\bigcup_{w\in F}B_{r_w}(w)\supset M_s^a$. Now defining the compact sets $K_i^a:=M_s^a\cap\big(\bigcup_{w\in F,\,i_w=i}\bar B_{r_w}(w)\big)$ for $\forall i\in I^+$, we have
\begin{equation} %\label{Ki union eq}
  \bigcup_{i\in I^+} K_i^a=M_s^a\cap\bigcup_{w\in F}\bar B_{r_w}(w)=M_s^a,
\end{equation}
which is the first relation in \eqref{Ki +- union}.
Analogously we can construct the sets $K_i^a$ for $\forall i\in I^-$ and show they fulfill the second relation in \eqref{Ki +- union}.

Since
\begin{equation} \label{Kia in Mi-flow}
   \forall i\in I\colon\ K_i^a\subset\psi(M_i,\R)
\end{equation}
(for $n=2$ this follows from \eqref{2D Kia}, for $n\geq3$ from the definition of the balls $\bar B_{r_w}(w)$) and since $\psi(M_i,\R)$ is open and $K_i^a$ compact, there $\exists\eta_a>0$ such that $\forall i\in I\colon\bar N_{\eta_a}(K_i^a)\subset\psi(M_i,\R)$. Since the sets $\bar N_{\eta_a}(K_i^a)$ are compact, $|t_i|$ is bounded on $\bar N_{\eta_a}(K_i^a)$ for $\forall i\in I$, say by some $T_a>0$, which implies \eqref{Ki nbhd}.
\end{proof}

\subsection{Proof of Lemma \ref{Msc cap Msc-hat lemma}}
\label{Msc cap Msc-hat sec}
\begin{proof}
We will only show how to construct a $\rho_0>0$ that fulfills the first statement in \eqref{Msc cap Msc-hat}. To begin, observe that $\Msan$ and $\psi\big(M_s^{a_0},[-T_{a_0},0]\big)$ are compact by \eqref{Ms Mu in Mloc} and disjoint: Indeed, every $w\in\psi\big(M_s^{a_0},[-T_{a_0},0]\big)$ can be written as $w=\psi(v,t)$ for some $v\in M_s^{a_0}$ and some $t\in[-T_{a_0},0]$, and so by Lemma \ref{distance function saddle}~(i) and \eqref{a relations} we have
\[ f_s(w)=f_s(\psi(v,t))\geq f_s(\psi(v,0))=f_s(v)=a_0>\tilde a \quad\Rightarrow\quad w\notin\Msan. \]
Since also $\Msan\subset\bar B_{\tilde a}(x)\subset B_{a_0}(x)$ by \eqref{Ms Mu in ball}, we can thus choose $\rho_0>0$ so small that
\begin{gather}
  \bar N_{\rho_0}(\Msan)\cap\psi\big(M_s^{a_0},[-T_{a_0},0]\big)=\varnothing,
  \label{Ds disjoint} \\
  N_{\rho_0}(\Msan)\subset\bar B_{a_0}(x). \label{Ds Du region}
\end{gather}\\[-.5cm]
Now define $\hMtas$ by \eqref{Msa Msu def}. This set is compact since both $\Msan$ (by \eqref{Ms Mu in Mloc}) and the domain $\bar B_{a_0}(x)$ of the continuous function $p_s$ are compact. We must show the first statement in \eqref{Msc cap Msc-hat}.

The relation $\Msan\subset\hMtas\cap M_s$ is easy: By \eqref{Ms Mu in Mloc} and \eqref{Ms Mu in ball} we have $\Msan\subset M_s^{loc}\cap\bar B_{a_0}$, and thus $\forall w\in \Msan\colon\ w=p_s(w)$ by \eqref{as prop 1}. This means that $\Msan\subset p^{-1}(\Msan)$, and thus $\Msan\subset\hMtas$ by \eqref{Msa Msu def}. The relation $\Msan\subset M_s$ is clear from \eqref{Msa flow = Ms}.

To show the reverse relation, i.e.~$\hMtas\cap M_s\subset\Msan$, let $w\in\hMtas\cap M_s$.
By \eqref{Msa Msu def} we have $w\in p_s^{-1}(\Msan)$, i.e.
\begin{equation} \label{fs ps = tilde a}
  f_s(p_s(w))=\tilde a.
\end{equation}
Suppose we had $f_s(w)>a_0$. Since $f_s(\psi(w,t))=\int_t^\infty|b(\psi(w,\tau))|\,d\tau\to0$ %$f_s(\psi(w,t))\to f_s(x)=0$ 
as\linebreak $t\to\infty$, there would then be a $t>0$ such that $f_s(\psi(w,t))=a_0$, i.e.~$v:=\psi(w,t)\in M_s^{a_0}$. Since $w\in\bar N_{\rho_0}(\Msan)$ by \eqref{Msa Msu def}, \eqref{Ds disjoint} then implies that $w\notin\psi\big(M_s^{a_0},[-T_{a_0},0]\big)$, and so the representation $w=\psi(v,-t)$ shows that $-t\notin[-T_{a_0},0]$ and thus $t>T_{a_0}$. Now since $v\in M_s^{a_0}$, by \eqref{Ki +- union} and \eqref{Ki nbhd} there $\exists i\in I^+$ such that $v\in K_i^{a_0}\subset\psi(M_i,[-T_{a_0},T_{a_0}])$. Therefore we can write $w=\psi(v,-t)=\psi\big(z_i(v),t_i(v)-t\big)$, which implies that $t_i(w)=t_i(v)-t$ $<T_{a_0}-T_{a_0}=0$ and thus $f_{M_i}(w)<0$ by \eqref{prime lemma eq 2}. Since $w\in\bar N_{\rho_0}(\Msan)\subset\bar B_{a_0}(x)$ by \eqref{Msa Msu def} and \eqref{Ds Du region}, \eqref{fMi>0} thus implies that $i\notin I^+$, a contradiction.

Therefore we must have $f_s(w)\leq a_0$ and thus $w\in M_s^{loc}\cap\bar B_{a_0}(x)$ by \eqref{Ms Mu in Mloc} and \eqref{fs est 1 saddle lower}. We can now use \eqref{as prop 1} to rewrite \eqref{fs ps = tilde a} as $f_s(w)=\tilde a$, i.e.~$w\in\Msan$.
\end{proof}

\subsection{Proof of Lemma \ref{zs ts lemma}}
\label{zs ts lemma sec}

\begin{proof}
We will only construct the functions $z_s$ and $t_s$ and the set $D_s$; the functions $z_u$ and $t_u$ and the set $D_u$ are defined analogously. We begin by defining
\begin{equation} \label{t tilde def}
  \tilde t(w):=\inf\!\big\{ t\in\R\,\big|\,\psi(w,t)\in\hMtas\big\} \qquad \text{for }\forall w\in D,
\end{equation}
which we interpret as $+\infty$ if $\psi(w,t)\notin\hMtas$ for $\forall t\in\R$. We claim that for $\forall v\in M_s\setminus\{x\}\ \exists\delta_v>0$ such that
\begin{align*}
 (i)   &\hspace{0.35cm}\text{the infimum in \eqref{t tilde def} is achieved for $\forall w\in B_{\delta_v}(v)$,}\\
 (ii)  &\hspace{0.35cm}\text{$\tilde t$ is $C^1$ on $B_{\delta_v}(v)$,} \\
 (iii) &\hspace{0.35cm}\forall w\in B_{\delta_v}(v)\cap\hMtas\colon\ \tilde t(w)=0.
\end{align*}
 Once this is established we can define the $C^1$-functions
\begin{align*}
   t_s(w)&:=-\tilde t(w), \\
   z_s(w)&:=\psi(w,\tilde t(w)) \\[-1.07cm]
   &\hspace{3.7cm}\text{for $\forall w\in D_s:=\bigcup_{v\in M_s\setminus\{x\}}B_{\delta_v}(v)$.}
\end{align*}
This definition then immediately implies \eqref{ts zs property 1}, and by property (i) we have $z_s(w)\in\hMtas$ for $\forall w\in D_s$. Property (iii) implies that for $\forall w\in D_s\cap\hMtas$ we have $\tilde t(w)=0$ and thus $z_s(w)=\psi(w,0)=w$, which is \eqref{zsx=x}. Finally, the relation
\begin{equation} \label{tilde t shift}
  \tilde t(\psi(w,\sigma))=\tilde t(w)-\sigma \qquad \text{for $\forall \sigma\in\R$}
\end{equation}
implies that
\begin{align*}
  z_s(\psi(w,\sigma))&=\psi\big(\psi(w,\sigma),\tilde t(\psi(w,\sigma))\big) \\*
                &=\psi\big(\psi(w,\sigma),\tilde t(w)-\sigma\big)
                 =\psi(w,\tilde t(w))
                 =z_s(w)
\end{align*}
wherever both sides are defined, which is \eqref{zs cst on flowlines}.\\[.2cm]
To prove the claims (i)-(iii) stated above, let $v\in M_s\setminus\{x\}$.\\[.2cm]
\textit{Case 1: $v\in\Msan$.}
Then since $\Msan\subset\bar B_{\tilde a}(x)\subset B_{a_1}(x)$ by \eqref{Ms Mu in ball} and \eqref{a relations}, there $\exists\mu,\nu>0$ such that
\begin{equation} \label{IFT setup 1}
  \forall (w,\tau)\in B_\mu(v)\times(-\nu,\nu)\colon\quad \psi(w,\tau)\in\bar N_{\rho_0}(\Msan)\cap\bar B_{a_1}(x)
\end{equation}
and thus in particular $p_s(\psi(w,\tau))\!\in\!\bar B_{a_0}(x)\cap M_s^{loc}$ by \eqref{ps to Ba0} and the definition of $p_s$.
Therefore by Lemma \ref{distance function saddle}~(iii)
%therefore since $\bar N_{\rho_0}(\Msan)\subset\bar B_{a_0}(x)$ by \eqref{Ds Du region}, 
the function $F(w,\tau):=f_s\big(p_s(\psi(w,\tau))\big)$ is well-defined and continuous on $B_\mu(v)\times(-\nu,\nu)$.
Observe that on this set we have
\begin{equation} \label{F equiv}
  F(w,\tau)=\tilde a
   \quad\Leftrightarrow\quad
  \psi(w,\tau)\in p_s^{-1}(\Msan)
   \quad\Leftrightarrow\quad
  \psi(w,\tau)\in\hMtas,
\end{equation}
where the last step follows from \eqref{Msa Msu def} and \eqref{IFT setup 1}.

Since $f_s(\psi(v,\cdot\,))$ is continuous by Lemma \ref{distance function saddle}~(i) and since $f_s(v)=\tilde a$, by decreasing $\nu>0$ we can also make sure that for $\forall\tau\in(-\nu,\nu)$ we have $\psi(v,\tau)\in f_s^{-1}\big([0,a_0]\big)\subset M_s^{loc}\cap \bar B_{a_0}(x)$ by \eqref{Ms Mu in Mloc} and \eqref{fs est 1 saddle lower}, and thus $F(v,\tau)=f_s(\psi(v,\tau))$ by \eqref{as prop 1}. Therefore by Lemma \ref{distance function saddle}~(i) we have
\begin{align}
                F(v,0) &= f_s(v)=\tilde a, \label{IFT setup 2}\\*
  \partial_\tau F(v,0) &= -|b(v)| < 0.  \label{IFT setup 3}
\end{align}
Because of \eqref{IFT setup 2} we can further decrease $\mu$ and $\nu$ so much that for $\forall(w,\tau)\in B_\mu(v)\times(-\nu,\nu)$ we have $f_s\big(p_s(\psi(w,\tau))\big)=F(w,\tau)\in(0,a_0)$ and thus $p_s(\psi(w,\tau))\in B_{a_0}(x)\setminus\{x\}$ by \eqref{fs est 1 saddle lower}, so that $F$ is $C^1$ on $B_\mu(v)\times(-\nu,\nu)$ by Lemma \ref{distance function saddle}~(iii).

Finally, by \eqref{IFT setup 3} we can further decrease $\mu$ and $\nu$ so much that for $\forall (w,\tau)\in B_\mu(v)\times(-\nu,\nu)$ we have $\partial_\tau F(w,\tau)<0$, so that
\begin{align}
  &\text{for $\forall w\in B_\mu(v)$ there is at most one value $\tau\in(-\nu,\nu)$} \nonumber\\*[-.35cm]
  & \label{at most one solution} \\*[-.35cm]
  &\text{such that $F(w,\tau)=\tilde a$.} \nonumber
\end{align}
We can now invoke the Implicit Function Theorem, and so there exists
a $\delta_v\in(0,\mu]$ and a function $\tau_v\in C^1\big(B_{\delta_v}(v),(-\nu,\nu)\big)$ such that for $\forall w\in B_{\delta_v}(v)$ we have $F(w,\tau_v(w))=\tilde a$, which by \eqref{at most one solution} and \eqref{F equiv} means that
\begin{align}
  &\text{for $\forall w\in B_{\delta_v}(v)$, $\tau_v(w)$ is the unique value in $(-\nu,\nu)$} \nonumber\\*[-.35cm]
  & \label{tau property} \\*[-.35cm]
  &\text{such that $\psi(w,\tau_v(w))\in\hMtas$.} \nonumber
\end{align}
Now since $v\in\Msan\subset\bigcup_{i\in I^+}\psi(M_i,\R)$ by \eqref{Msa plus union}, there $\exists i\in I^+$ such that $v\in\psi(M_i,\R)$, and \eqref{prime lemma eq 1} implies that for $t':=\min\{-t_i(v)-1,-\nu\}$ we have
\begin{equation} \label{delta v prep 2} % \label{f Mi<0 neu}
  f_{M_i}(\psi(v,t'))<0.
\end{equation}
%Also because of \eqref{prime lemma eq 1} we have $f_{M_i}(\psi(v,t))>0$ for $\forall t>-t_i(v)$, and letting $t\to\infty$ and considering \eqref{Ms def} and \eqref{fMi neq0} we find that $f_{M_i}(x)>0$, i.e.~$i\in I^+$.
By Lemma \ref{distance function saddle}~(i) we have $f_s(\psi(v,t))>f_s(\psi(v,0))=f_s(v)=\tilde a$ for $\forall t\in[t',-\nu]$, so that $\psi\big(v,[t',-\nu]\big)\cap\Msan=\varnothing$, and since also $\psi\big(v,[t',-\nu]\big)\subset M_s$, \eqref{Msc cap Msc-hat} thus tells us that
\begin{equation} \label{delta v prep 3}
  \psi\big(v,[t',-\nu]\big)\cap\hMtas=\varnothing.
\end{equation}
Now considering \eqref{delta v prep 2} and \eqref{delta v prep 3}, and that $\hMtas$ is compact, we can further decrease $\delta_v>0$ so much that
\begin{align}
  \forall w\in B_{\delta_v}(v)\colon\ \  &f_{M_i}(\psi(w,t'))<0, \label{fM<0 2 neu} \\
  \forall w\in B_{\delta_v}(v)\colon\ \  &\psi\big(w,[t',-\nu]\big)\cap\hMtas=\varnothing. \label{disjoint 2 neu}
\end{align}
Now let $w\in B_{\delta_v}(v)$. Then since $t\mapsto\sgn\!\big(f_{M_i}(\psi(w,t))\big)$ is non-decreasing by  \eqref{prime lemma eq 1}, \eqref{fM<0 2 neu} implies that $f_{M_i}(\psi(w,t))<0$ for $\forall t\in(-\infty,t']$.
Since by \eqref{Msc in ball} and \eqref{fMi>0} we have $f_{M_i}(u)>0$ for $\forall u\in\hMtas$, this means that $\psi(w,t)\notin\hMtas$ for $\forall t\in(-\infty,t']$, and by \eqref{disjoint 2 neu} in fact for $\forall t\in(-\infty,-\nu]$. Thus \eqref{tau property} implies that $\tau_v(w)$ is the unique value in all of $(-\infty,\nu)$ fulfilling $\psi(w,\tau_v(w))\in\hMtas$.

This in turn has three consequences: (i) the infimum in \eqref{t tilde def} is achieved for $\forall w\in B_{\delta_v}(v)$, with
\begin{equation} \label{tilde t = tau}
  \tilde t(w)=\tau_v(w) \qquad\text{for $\forall w\in B_{\delta_v}(v)$},
\end{equation}
which in turn implies that (ii) $\tilde t$ is $C^1$ on $B_{\delta_v}(v)$ since $\tau_v$ is; and (iii) since for $\forall w\in B_{\delta_v}(v)\cap\hMtas$ we have $\psi(w,0)=w\in\hMtas$, we can conclude that $0=\tau_v(w)=\tilde t(w)$ for those $w$. These are the three properties that we had to prove.\\[.2cm]
\textit{Case 2: $v\notin\Msan$.}
Then since $v\in M_s$, \eqref{Msc cap Msc-hat} implies that $v\notin\hMtas$. Since $\hMtas$ is compact, there thus exists a $\delta_v>0$ such that $B_{\delta_v}(v)\cap\hMtas=\varnothing$, and claim~(iii) will be trivially true.
Furthermore, by \eqref{Msa flow = Ms} there exist $u\in\Msan$ and $\sigma\in\R$ such that $v=\psi(u,-\sigma)$, i.e.~$\psi(v,\sigma)=u\in B_{\delta_u}(u)$, where $\delta_u$ is given by \textit{Case~1}. Let us decrease $\delta_v>0$ so much that $\forall w\in B_{\delta_v}(v)\colon\!\!$\linebreak $\psi(w,\sigma)\in B_{\delta_u}(u)$. Then by \eqref{tilde t shift} and \eqref{tilde t = tau} (applied to $B_{\delta_u}(u)$) we have
\[
  \tilde t(w)=\tilde t(\psi(w,\sigma))+\sigma=\tau_u(\psi(w,\sigma))+\sigma
\]
for $\forall w\in B_{\delta_v}(v)$, which implies property (ii), and
\[
  \psi(w,\tilde t(w))
 =\psi\big(\psi(w,\sigma),\tilde t(w)-\sigma\big)
 =\psi\big(\psi(w,\sigma),\tau_u(\psi(w,\sigma))\big)
 \in\hMtas
\]
by \eqref{tau property}, which is property (i).
\end{proof}

\subsection{Proof of Remark \ref{zs K remark}} \label{zs K remark app}
\begin{proof}
We will only prove \eqref{Ki tilde a def 1}, i.e.\ the case $i\in I^+$.
Note that $z_s(K_i^{a_0})$ is well-defined since for $i\in I^+$ we have $K_i^{a_0}\subset M_s^{a_0}\subset M_s\setminus\{x\}\subset D_s$ by \eqref{Ki +- union}, \eqref{Msa flow = Ms} and the definition of $D_s$.

The proof of Remark \ref{zs K remark} must be led separately for dimensions $n=2$ and $n\geq3$:
In the case $n=2$ we must show that our explicit definition \eqref{2D Kia} of $K_i^a$ that we will use later on fulfills \eqref{Ki tilde a def 1}; in the case $n\geq3$ we only need to show that given the sets $K_i^{a_0}$ constructed in Lemma \ref{compact hitting partition}, the sets $\tilde K_i^{\tilde a}:=z_s(K_i^{a_0})$ are an alternative choice that fulfill \eqref{Ki +- union}-\eqref{Ki nbhd} for some constants $\eta_{\tilde a},T_{\tilde a}>0$. A look at the last paragraph of the proof of Lemma~\ref{compact hitting partition} reveals that for the latter it suffices to show that the sets $\tilde K_i^{\tilde a}$ are compact and fulfill $\tilde K_i^{\tilde a}\subset\psi(M_i,\R)$ for $\forall i\in I^+$, and that $\bigcup_{i\in I^+}\tilde K_i^{\tilde a}=\Msan$. \\[.2cm]
\indent Beginning with the case $n=2$, first let $w\in K_i^{a_0}=\psi(M_i,\R)\cap M_s^{a_0}$. The three representations $z_s(w)=\psi(w,-t_s(w))=\psi\big(z_i(w),t_i(w)-t_s(w)\big)$ then show that $z_s(w)\in \hMtas\cap M_s\cap\psi(M_i,\R)=\Msan\cap\psi(M_i,\R)=K_i^{\tilde a}$ by \eqref{Msc cap Msc-hat} and \eqref{2D Kia}, proving the inclusion $z_s(K_i^{a_0})\subset K_i^{\tilde a}$.

For the reverse inclusion $K_i^{\tilde a}\subset z_s(K_i^{a_0})$ let $w\in K_i^{\tilde a}=\psi(M_i,\R)\cap\Msan$. Then we have $\psi(w,-t_i(w))=z_i(w)\in M_i\subset\bar B_{a_0}(x)^c$ by \eqref{fMi neq0} and thus $f_s\big(\psi(w,-t_i(w))\big)\geq|\psi(w,-t_i(w))-x|>a_0$. Since $f_s(\psi(w,0))=f_s(w)=\tilde a<a_0$, this shows that there $\exists t\in\R$ such that $f_s(\psi(w,t))=a_0$ and thus $v:=\psi(w,t)=\psi\big(z_i(w),t_i(w)+t\big)\in\psi(M_i,\R)\cap M_s^{a_0}=K_i^{a_0}$. Since $w\in M_s^{\tilde a}\subset\hMtas$ and $w,v\in M_s\setminus\{x\}\subset D_s$,
\eqref{zsx=x} and \eqref{zs cst on flowlines} now show that $w=z_s(w)=z_s(\psi(v,-t))=z_s(v)\in z_s(K_i^{a_0})$.\\[.2cm]
\indent
Moving on to the case $n\geq3$, first note that the sets $\tilde K_i^{\tilde a}$ are compact as the continuous images of compact sets. To see that $\tilde K_i^{\tilde a}\subset\psi(M_i,\R)$, note that if $w\in\tilde K_i^{\tilde a}=z_s(K_i^{a_0})$ then there $\exists v\in K_i^{a_0}$ such that
\[
   w=z_s(v)=\psi(v,-t_s(v))\in\psi(K_i^{a_0},\R)
    \subset\psi(\psi(M_i,\R),\R)=\psi(M_i,\R)
\]
by \eqref{Kia in Mi-flow}. Finally, to show $\bigcup_{i\in I^+}\tilde K_i^{\tilde a}=\Msan$, observe that since
\[
   \bigcup_{i\in I^+}\tilde K_i^{\tilde a}
   = \bigcup_{i\in I^+}z_s(K_i^{a_0})
   = z_s\bigg(\bigcup_{i\in I^+}K_i^{a_0}\bigg)
   = z_s(M_s^{a_0})
\]
by \eqref{Ki +- union}, we only need to prove that $z_s(M_s^{a_0})=\Msan$.

To do so, first observe that by \eqref{ts zs property 1} and \eqref{Msa flow = Ms} we have $z_s(M_s^{a_0})\subset\psi(M_s^{a_0},\R)\subset M_s$, and thus by definition of $z_s$ and by \eqref{Msc cap Msc-hat} we have $z_s(M_s^{a_0})\subset\hMtas\cap M_s=\Msan$.
To show the reverse inclusion, let $w\in\Msan$. Then by \eqref{Msa flow = Ms} we have $w\in M_s\setminus\{x\}=\psi(M_s^{a_0},\R)$, and so $\exists v\in M_s^{a_0}\ \exists t\in\R\colon\!\!\!$\linebreak $w=\psi(v,t)$ and thus $f_s(\psi(v,t))=f_s(w)=\tilde a$, i.e.\ $\psi(v,t)\in\Msan$. Since $f_s(\psi(v,\cdot\,))$ is decreasing by Lemma~\ref{distance function saddle}~(i), $t$ is in fact the unique value with this property. Since $v\in M_s$, by \eqref{Msc cap Msc-hat} this means that $t$ is the unique value such that $\psi(v,t)\in\hMtas$, which in the notation of Appendix~\ref{zs ts lemma sec} implies that $\tilde t(v)=t$ and thus $z_s(v)=\psi(v,\tilde t(v))=\psi(v,t)=w$. This shows that $w\in z_s(M_s^{a_0})$, completing our proof.
\end{proof}

\subsection{Proof of Lemma \ref{Ui supset lemma}}
\label{Ui supset lemma sec}

\begin{proof}
Again we will only consider the case $i\in I^+$. First we claim that 
\begin{equation} \label{small claim}
 \psi\big(K_i^{\tilde a},[-T_{\tilde a},T_{\tilde a}]\big)\cap f^{-1}_{M_i}\big([0,\infty)\big)\subset K.
\end{equation}
To see this, let $w\in\psi\big(K_i^{\tilde a},[-T_{\tilde a},T_{\tilde a}]\big)\cap f^{-1}_{M_i}\big([0,\infty)\big)$. If $w\in\bar B_{a_0}(x)$ then by \eqref{K def} we have $w\in K$. Therefore suppose now that $w\notin \bar B_{a_0}(x)$; we must show that $w\in K$ also in this case.

Let $v\in K_i^{\tilde a}$ and $t\in[-T_{\tilde a},T_{\tilde a}]$ such that $w=\psi(v,t)$. Since by Remark~\ref{zs K remark} we have $v\in K_i^{\tilde a}=z_s(K_i^{a_0})$, there $\exists u\in K_i^{a_0}\colon\ v=z_s(u)$, and we find that
\begin{equation} \label{y w t}
w=\psi(v,t)=\psi(z_s(u),t)=\psi\big(\psi(u,-t_s(u)),t\big)=\psi\big(u,t-t_s(u)\big).
\end{equation}
Since $u\in K_i^{a_0}\subset M_s^{a_0}$ by \eqref{Ki +- union}, and since $w\notin\bar B_{a_0}(x)\supset f_s^{-1}\big([0,a_0]\big)$ by \eqref{fs est 1 saddle lower}, we thus have
\[
  f_s(\psi(u,0))=f_s(u)=a_0<f_s(w)=f_s\big(\psi(u,t-t_s(u))\big),
\]
and so Lemma \ref{distance function saddle}~(i) implies that $0>t-t_s(u)$. Therefore by \eqref{y w t} and \eqref{Ki nbhd} we have
\[
w\in\psi\big(K_i^{a_0},(-\infty,0)\big)\subset\psi\big(\psi(M_i,[-T_{a_0},T_{a_0}]),(-\infty,0)\big)
=\psi\big(M_i,(-\infty,T_{a_0})\big)
\]
and thus $t_i(w)<T_{a_0}$. Furthermore, since $f_{M_i}(w)\geq0$ by our choice of $w$, by \eqref{prime lemma eq 2} we have $t_i(w)\geq0$. We can now conclude that $t_i(w)\in[0,T_{a_0})$ and thus $w\in\psi(M_i,[0,T_{a_0}))\subset K$ by \eqref{K def}, and \eqref{small claim} is proven.\\[.2cm]
Now we abbreviate $M_i^-:=f_{M_i}^{-1}\big((-\infty,0)\big)$, $M_i^+:=f_{M_i}^{-1}\big([0,\infty)\big)$, and $F:=\psi\big(K_i^{\tilde a},[-T_{\tilde a},T_{\tilde a}]\big)$, and finally we define the open set $G_i := M_i^-\cup N_{\tilde a}(F\cap M_i^+)$.
Then the relation \eqref{small claim} translates into
\begin{equation} \label{small claim translation}
   F\cap M_i^+\subset K,
\end{equation}
which by \eqref{a tilde def} implies that $N_{\tilde a}(F\cap M_i^+)\subset N_{\tilde a}(K)\subset D$ and thus $G_i\subset D$. Also, we have
\begin{equation*}
  G_i \supset [F\cap M_i^-]\cup[F\cap M_i^+]=F\cap[M_i^-\cup M_i^+]=F\cap D=F,
\end{equation*}
which is \eqref{Ui def}, and again using \eqref{small claim translation} we find that
\begin{align*}
  G_i\cap M_i^+ &= \big[M_i^-\cup N_{\tilde a}(F\cap M_i^+)\big]\cap M_i^+ \\
                &= \big[M_i^-\cap M_i^+\big]\cup\big[N_{\tilde a}(F\cap M_i^+)\cap M_i^+\big] \\
                &\subset\varnothing\cup N_{\tilde a}(F\cap M_i^+)\subset N_{\tilde a}(K),
\end{align*}
which is \eqref{Ui prop 1}.
\end{proof}

\end{appendices}

%\newpage
%\pdfbookmark[0]{References}{references}
\newpage\phantomsection

\end{document}